\documentclass[10pt]{amsart}
\usepackage[nohug,heads=littlevee]{diagrams}
\usepackage{mathrsfs}
\usepackage{bm}
\usepackage{stmaryrd}
\usepackage{amssymb}
\usepackage[latin1]{inputenc}
\usepackage{xr-hyper}
\usepackage[plainpages=false,pdfpagelabels]{hyperref}
\usepackage[left=2cm,top=2.5cm,right=2cm,bottom=1.4in,asymmetric]{geometry}
\usepackage{mathtools}
\usepackage[backend=bibtex]{biblatex}
\usepackage{enumitem}
\usepackage{color}
\newcommand{\defnword}[1]{\textbf{#1}}
\newcommand{\comment}[1]{}
\newcommand{\on}[1]{\operatorname{#1}}

\numberwithin{equation}{subsection}
\newtheorem{introthm}{Theorem}

\newtheorem{introconj}{Conjecture}

\newtheorem{introprop}{Proposition}
\newtheorem{proposition}[subsection]{Proposition}
\newtheorem{theorem}[subsection]{Theorem}
\newtheorem{conjecture}[subsection]{Conjecture}
\newtheorem*{thm*}{Theorem}
\newtheorem{lemma}[subsection]{Lemma}
\newtheorem*{lem*}{Lemma}
\newtheorem{corollary}[subsection]{Corollary}
\theoremstyle{definition}
\newtheorem{definition}[subsection]{Definition}

\newtheorem{example}[subsection]{Example}

\theoremstyle{remark}
\newtheorem{introrem}{Remark}

\newtheorem{remark}[subsection]{Remark}
\newtheorem{assumption}[subsection]{Assumption}

\newtheorem*{assump*}{Assumption}

\newarrow{Equals}{=}{=}{}{=}{}
\newarrow{Onto}{-}{-}{-}{-}{>>}

\newcommand{\Real}{\mathbb{R}}
\newcommand{\Int}{\mathbb{Z}}

\newcommand{\Comp}{\mathbb{C}}

\newcommand{\Adele}{\mathbb{A}}
\newcommand{\Field}{\mathbb{F}}
\newcommand{\Rat}{\mathbb{Q}}
\newcommand{\Lie}{\on{Lie}}

\newcommand{\Map}{\on{Map}}
\newcommand{\Fun}{\on{Fun}}
\newcommand{\form}{\mathfrak{F}}

\newcommand{\colim}{\on{colim}}

\newcommand{\gr}{\on{gr}}

\newcommand{\Rg}{\mathscr{O}}
\newcommand{\Reg}[1]{\Rg_{#1}}

\newcommand{\Hom}{\on{Hom}}
\newcommand{\End}{\on{End}}

\newcommand{\Ext}{\on{Ext}}
\newcommand{\Aut}{\on{Aut}}

\newcommand{\Sym}{\on{Sym}}
\newcommand{\define}{\stackrel{\mathrm{def}}{=}}

\newcommand{\Gm}{\mathbb{G}_m}

\newcommand{\Gmh}[1]{\mathbb{G}_{m,#1}}

\newcommand{\Pic}{\on{Pic}}

\newcommand{\dR}{\on{dR}}

\newcommand{\pow}[1]{[\vert#1\vert]}

\newcommand{\Rep}{\on{Rep}}

\newcommand{\Mod}[2][ ]{\on{Mod}_{#2}^{#1}}

\newcommand{\Spec}{\on{Spec}}
\newcommand{\Spf}{\on{Spf}}

\newcommand{\Fzip}{F\mathrm{Zip}}
\newcommand{\shvMap}{\mathbb{M}ap}

\newcommand{\Fil}{\on{Fil}}

\newcommand{\Sh}{\on{Sh}}
\newcommand{\Ss}{\mathcal{S}}

\newcommand{\hker}{\operatorname{hker}}
\newcommand{\hcoker}{\operatorname{hcoker}}

\newcommand{\GSp}{\on{GSp}}

\newcommand{\GL}{\on{GL}}
\newcommand{\SL}{\on{SL}}

\newcommand{\op}{\on{op}}

\bibliography{derived}

\begin{document}
\title[Derived special cycles]{Derived special cycles on Shimura varieties}
\author{Keerthi Madapusi}
\address{Keerthi Madapusi\\
Department of Mathematics\\
Maloney Hall\\
Boston College\\
Chestnut Hill, MA 02467\\
USA}
\email{madapusi@bc.edu}

\begin{abstract}
I employ methods from derived algebraic geometry to give a uniform moduli-theoretic construction of special cycle classes on integral models many Shimura varieties of Hodge type, including unitary, quaternionic, and orthogonal Shimura varieties. All desired properties of these cycles, even for those corresponding to degenerate Fourier coefficients under the Kudla correspondence, follow naturally from the construction. I formulate Kudla's modularity conjectures in this general framework, and give some preliminary evidence towards their validity. 
\end{abstract}

\maketitle

\setcounter{tocdepth}{1}
\tableofcontents

\section{Introduction} 

In this article, I introduce methods from derived algebraic geometry to the study of special cycles on Shimura varieties. More precisely, I present a uniform construction of special cycle classes on a wide range of Shimura varieties of Hodge type at places of good reduction, and conjecture, following Kudla, that they can be organized into modular generating series.

\subsection{Special cycle classes in the generic fiber}

To explain the results, let's begin with a Shimura datum $(G,X)$ where $G$ has a $\Rat$-representation $W$ with the following additional structure: There is a $\Rat$-algebra $D$ equipped with a positive involution $\iota$ acting $G$-equivariantly on $W$ such that there is an $\iota$-Hermitian form $\Phi$ on $W$, preserved by $G$. Moreover, let's demand that every point $x\in X$ determine a Hodge structure on $W$ polarized by the symmetric form underlying $\Phi$, and also require that this structure is pure of weights $(-1,1),(0,0),(1,-1)$. 

A basic example here is the orthogonal or GSpin Shimura datum considered in~\cite{HMP:mod_codim}. Other instances arise from unitary Shimura data that are related to the stacks of Kudla and Rapoport studied in~\cite{Kudla2009-lc}. A \emph{canonical} instance arises by looking at the adjoint representation $\Lie G$ equipped with its Killing form (with $D= \Rat$).

Fix some compact open subgroup $K\subset G(\Adele_f)$, and let $\Sh_K$ be the Shimura variety over the reflex field $E$ of $(G,X)$ associated with $K$, so that we have
\[
\Sh_K(\Comp) = G(\Rat)\backslash X\times G(\Adele_f)/K.
\]
The representation $W$ with its form $\Phi$ gives rise to a polarized variation of Hodge structures $\Fil^\bullet_{\mathrm{Hdg}}\mathbf{dR}(W)$ over $\Sh_K$, and the quotient $\gr^{-1}_{\mathrm{Hdg}}\mathbf{dR}(W)$, the \defnword{tautological bundle} (associated with $W$), is a vector bundle of some rank $d_+$. In particular, we have its top Chern class
\[
c_{\mathrm{top}}(\gr^{-1}_{\mathrm{Hdg}}\mathbf{dR}(W))\in \mathrm{CH}^{d_+}(\Sh_K).
\]

For $n\ge 1$, let $\mathcal{S}(W(\Adele_f)^m)^K$ be the space of $K$-invariant elements in the Schwartz space of compactly supported, locally constant functions on $W(\Adele_f)^n$. Let $\mathrm{Herm}_n^+(D)$ be the set of positive semi-definite $n\times n$ $\iota$-Hermitian matrices over $D$. Let $\bm{T}_{\Sh_K}$ be the tangent bundle of $\Sh_K$, and make the following technical assumption:
\begin{assumption}
\label{assumption}
The Kodaira-Spencer map 
\[
\bm{T}_{\Sh_K}\to \underline{\Hom}(\gr^0_{\mathrm{Hdg}}\mathbf{dR}(W),\gr^{-1}_{\mathrm{Hdg}}\mathbf{dR}(W))
\]
associated with the variation of Hodge structures $\Fil^\bullet_{\mathrm{Hdg}}\mathbf{dR}(W)$ is \emph{injective}.
\end{assumption}

Following Kudla, one can then attach to every $\varphi\in \mathcal{S}(W(\Adele_f)^n)^K$ and $T\in \mathrm{Herm}_n^+(D)$ a cycle class 
\begin{align*}
C_{\Sh_K,W}(T,\varphi)\in \mathrm{CH}^{nd_+}(\Sh_K)_{\Comp}.
\end{align*}

This goes as follows\footnote{The presentation here is a little different than the one found in sources like~\cite{kudla:special_cycles},~\cite{Zhang2009-pr} or~\cite{MR2551992}, but will be convenient for what I aim to do in this paper.}: For any $\underline{w}\in W^n$, let
\[
\Phi(\underline{w}') = \left(\frac{\Phi(w'_i,w'_j)}{2}\right)_{1\leq i,j\leq n}
\]
be the Gram matrix attached to $\underline{w}'$. Write $W(\Rat)^n_T\subset W(\Rat)^n = W^n$ for the subspace of those $n$-tuples whose Gram matrix is $T$.

Given a lattice $U_{\Int}\subset W^n$ whose base-change over $\Adele_f$ is stabilized by $K$ (we will call this a \defnword{$K$-stable} lattice) and a $K$-invariant coset $\mu\in W^n/U_{\Int}$, we consider the locally symmetric space 
\[
 X(U_{\Int},\mu,T)= \{(\underline{w},x,g)\in W(\Rat)^n_T\times X\times G(\Adele_f):\;\underline{w}\in g(\mu+U_{\Int}(\widehat{\Int})),\;h_x(i)\cdot w = w\},
\]
from which we obtain a finite and unramified map $Z_K(U_{\Int},\mu,T)\to \Sh_K$ over $E$ with
\[
Z_K(U_{\Int},\mu,T)(\Comp) = G(\Rat)\backslash X(U_{\Int},\mu,T)/K.
\]

One checks that, under Assumption~\ref{assumption}, $Z_K(U_{\Int},\mu,T)$ has relative codimension $d_+r(T)$ over $\Sh_K$, where $r(T)$ is the rank of the matrix $T\in M_n(D)$, and so gives a class $[Z_K(U_{\Int},\mu,T)]\in \mathrm{CH}^{d_+r(T)}(\Ss_K)$. There is now a canonical assignment
\begin{align}
\label{eqn:char0_cycle_class}
\mathcal{S}(W(\Adele_f)^n)^K\xrightarrow{\varphi\mapsto C_{\Sh_K,W}(T,\varphi)}\mathrm{CH}^{d_+n}(\Sh_K)_{\Comp}
\end{align}
determined by the property that when $\varphi$ is the characteristic function of the coset $\mu + U_{\Int}(\widehat{\Int})$, $C_{\Sh_K,W}(T,\varphi)$ is the \defnword{corrected class}
\[
 C_{\Sh_K,W}(T,\varphi) = [Z_K(U_{\Int},\mu,T)]\cdot c_{top}(\gr^{-1}_{\mathrm{Hdg}}\mathbf{dR}(W))^{n-r(T)}.
\]

This definition is justified---among other things---by the fact that it has the following properties:
\begin{enumerate}
	\item (Linear invariance) For $g\in \GL_n(D)$, we have
	\[
     C_{\Sh_K,W}({}^t\iota(g)Tg,L_g\varphi) = C_{\Sh_K,W}(T,\varphi),
	\]
	where $L_g\varphi(\underline{w}) = \varphi(g^{-1}\underline{w})$.
	\item (Product formula) For $n_i\in \Int_{\ge 1}$, $\varphi_i\in \mathcal{S}(W(\Adele_f)^{n_i})$ and $T_i\in \mathrm{Herm}^+_{n_i}(D)$, for $i=,1,2$, we have
	\[
     C_{\Sh_K,W}(T_1,\varphi_1)\cdot C_{\Sh_K,W}(T_2,\varphi_2) = \sum_{T = \begin{pmatrix}
     T_1&*\\
     *&T_2
     \end{pmatrix}}C_{\Sh_K,W}(T,\varphi_1\otimes \varphi_2).
	\]
\end{enumerate}

The first property is immediate from the definition, while the second property requires a bit of argument; see~\cite[Theorem 4.15]{Kudla2019-dq} for the case of orthogonal Shimura varieties associated with quadratic spaces over totally real fields. We will also recover both statements independently (and in quite a bit of generality) below using the methods of this paper.

The two properties above are consequences of---and evidence for---conjectures of Kudla on the modularity of generating series obtained from these cycle classes, an instance of an anticipated \emph{geometric} theta correspondence; see for instance~\cite{kudla:special_cycles} and also Conjecture~\ref{introconj} below.

\subsection{Cycle classes over the integral model: The main theorem}

Suppose now that $(G,X)$ is of \emph{Hodge type}: This means that $G$ admits a \emph{faithful} representation$(H,\psi)$ of \defnword{Siegel type}---that is, a representation with a symplectic form $\psi$ preserved by $G$ up to similitudes, such that the corresponding variation of Hodge structures is polarized by $\psi$ of weights $(-1,0),(0,-1)$. Then the Shimura variety can be viewed as a moduli space of abelian varieties with some additional structure. 

Kisin has constructed~\cite{kisin:abelian} a smooth integral model $\Ss_K$ over $\Reg{E}[D_K^{-1}]$, where $D_K\in \Int$ is the product of the primes $p$ such that $K_p$ is not hyperspecial (see~\cite{Kim2016-fb} for the case $p=2$). With an eye towards Kudla's \emph{arithmetic} modularity conjectures, one would now like to have well-behaved cycle classes over $\Ss_K$ restricting to the corrected classes $C_{\Sh_K,W}(T,\varphi)$ over the generic fiber. 

Given a $K$-stable lattice $W_{\Int}\subset W$\footnote{By this, I mean that the associated $\widehat{\Int}$-lattice in $W(\Adele_f)$ is $K$-stable.}, the filtered vector bundle $\Fil^\bullet_{\mathrm{Hdg}}\mathbf{dR}(W)$ has a natural extension $\Fil^\bullet_{\mathrm{Hdg}}\mathbf{dR}(W_{\Int})$ over $\Ss_K$, and so has a well-defined top Chern class in $\mathrm{CH}^{d_+}(\Ss_K)$. An immediate guess would now be to first extend $Z_K(U_{\Int},\varphi,T)$ to a finite scheme over $\Ss_K$ via normalization or flat Zariski closure, and to define the corrected classes via the procedure used in the generic fiber. However, there is no good reason for such a construction to be as well-behaved as its counterpart over the generic fiber, and indeed, both properties that justified the definition of the corrected cycles $C_{\Sh_K,W}(T,\varphi)$ will in general fail to hold for this na\"ive extension over the integral model. One needs to therefore find a different way of constructing the desired cycles.

This was the problem solved in~\cite{HMP:mod_codim} in the special case of orthogonal Shimura varieties. In \emph{loc. cit.}, we also verified that Kudla's modularity conjecture is valid for these Chow classes over the integral model as well, following the strategy of W. Zhang and Bruinier-Raum in the generic fiber; see~\cite{bruinier_raum}. The proofs of linear invariance and the product formula there---which play a crucial role in the proof of modularity---involve arguments modeled on the ideas in~\cite{Howard2019-td}, which seem particular to the situation where $d_+ = 1$, and do not seem to generalize in any obvious way.

The main result of this paper is that one can construct these cycle classes quite generally, under the following technical assumption often satisfied in practice (but we will not require any analogue of Assumption~\ref{assumption}):
\begin{assumption}
\label{assump:hodge_emb}
There exist symplectic representations $(H_1,\psi_1)$, $(H_2,\psi_2)$ of Siegel type and a $G$-equivariant embedding $W\subset \Hom(H_1,H_2)$.
\end{assumption}

\begin{introthm}
\label{introthm:main}
Suppose that Assumption~\ref{assump:hodge_emb} holds. Then for any $T\in\mathrm{Herm}_n^+(D)$, there exists a natural assignment
\begin{align*}
\mathcal{S}(W(\Adele_f)^n)^K&\to \mathrm{CH}^{nd_+}(\Ss_K)_{\Comp}\\
\varphi&\mapsto  \mathcal{C}_{\Ss_K,W}(T,\varphi)
\end{align*}
satisfying the linear invariance property and the product formula described above. Moreover, when Assumption~\ref{assumption} holds, the restriction of this assignment over the generic fiber agrees with~\eqref{eqn:char0_cycle_class}.
\end{introthm}

\begin{introrem}
\label{introrem:cycles}
Some details of the construction can be found in~\eqref{introsubsec:derived} and later subsections of this introduction. For now, here are some auxiliary remarks:
\begin{enumerate}[label=(\Alph*)]
    \item The cycle classes are independent of the choice of symplectic representations $H_1$ and $H_2$ in Assumption~\ref{assump:hodge_emb}.

    \item The theorem applies in essentially\footnote{To remove this adverb, one would have to extend the methods here to Shimura varieties of abelian type, as well as to their integral models at places of bad reduction.} every known context where one expects to have special cycle classes on Shimura varieties in a range of codimensions. These include:
    
    \begin{enumerate}[label=(\roman*)]
    	\item The classical orthogonal case considered in~\cite{kudla:special_cycles},~\cite{Zhang2009-pr} or~\cite{MR2551992}, involving quadratic spaces over totally real fields that are definite at all but one real place. In the case where the field is $\Rat$, the construction of the desired cycle classes over the integral model can be found in~\cite{HMP:mod_codim}, but already for arbitrary totally real fields, the results here appear to be new.
    	\item The more general orthogonal case considered in~\cite{Kudla2019-dq} involving quadratic spaces over totally real fields that are either definite or have signature $(n,2)$ at every real place. Here, the theorem recovers in particular the construction in~\emph{loc. cit.} of the corrected cycle classes over the generic fiber.
    	\item The unitary Shimura varieties of Kudla and Rapoport~\cite{Kudla2009-lc}: while the case of signature $(n,1)$ has been considered both in the cited work and in subsequent work such as~\cite{Howard2019-td}, the results here appear to be the first that produce a good theory of cycle classes in arbitrary codimension for unitary groups of arbitrary signature.\footnote{There is actually a much more direct construction of the cycle classes here, still using derived methods, but without any crystalline theory. This appeared in an earlier version of this paper, which can be found on the arXiv.} The results also work for unitary groups over totally real fields, as in the context found in~\cite{Mihatsch2021-mf}.
    	\item There are several new examples one can produce in this framework that seem to have not been systematically considered before; see~\ref{rem:adjoint} below for instance.
    \end{enumerate}

  \item The actual construction involves derived algebraic geometry, but has the advantage that it is entirely geometric and---outside of the eventually inessential symplectic representations arising in Assumption~\ref{assump:hodge_emb}---involves no non-canonical choices. The linear invariance property and the product formula can be understood directly in terms of the moduli interpretation of certain derived stacks.

  \item In particular, we recover the formulas for the corrected cycle classes in the generic fiber as a \emph{consequence} of the general construction, instead of their being instituted by fiat. Finding such an intrinsic definition was one of the original motivations for this work.

	\item In what is perhaps the most noticeable departure from the story of special cycles in the literature, I do \emph{not} treat the representation $W$ as being intrinsic to the Shimura datum. This allows for direct comparison between cycles associated with different representations,  leading, among other things, to a proof of the so-called pullback identities (see Corollary~\ref{cor:pullback_trivial} below) that works only over $\Ss_K$ without reference to any larger Shimura varieties. This flexibility is afforded almost entirely because of the use of derived methods, which erase the distinction between proper and improper intersections between cycles.
    
	\item\label{rem:adjoint} The eventual goal is to realize the cycle classes above as Fourier coefficients of a geometric theta series for a certain quasi-split unitary group $\mathrm{U}_D(n,n)$ (see Conjecture~\ref{introconj} below). Along the lines of the previous remark, however, I am not requiring $G$ to be part of a dual reductive pair with $\mathrm{U}_D(n,n)$. 

	This allows for consideration of perhaps some unusual examples such as the adjoint representation of any Shimura variety of Hodge type: this gives rise to classes in codimension $\dim \Sh_K$ on $\Ss_K$ that should be organized into Fourier coefficients of (vector valued half-integer weight) \emph{classical} modular forms for $\SL_2$. The dual reductive group here would be the orthogonal group associated with a certain twist of the Killing form on $\Lie G$, which has signature $(\dim\Sh_K,\dim \Lie G - \dim\Sh_K)$. Except in some low dimensional cases, this does not have an associated Shimura variety. What we are seeing here is a shadow of a geometric theta correspondence that should extend beyond the ambit of groups forming part of Shimura data.
\end{enumerate}

\end{introrem}

\subsection{Modularity}
\label{introsubsec:modularity}

One can now formulate the main conjecture, which is essentially due to Kudla, though it appears to be stated in the literature explicitly only for certain orthogonal and unitary Shimura varieties. My chief contribution here is making the objects involved defined unconditionally.
\begin{introconj}\label{introconj}
For $\varphi\in \mathcal{S}(W(\Adele_f)^n)^K$, the formal generating series 
\[
\Theta^{\mathrm{geom}}_{\Ss_K,W}(\varphi,\tau) = \sum_{T\in \mathrm{Herm}_n^+(D)}\mathcal{C}_{\Ss_K,W}(T,\varphi)\bm{q}^T 
\]
converges absolutely to a parallel weight $\mathrm{rank}_D(W)/2$ automorphic form relative to the Weil representation on the space $\mathcal{S}(W(\Adele)^n)$ for the unitary group $\mathrm{U}_D(n,n)$ attached to the standard split $2n$-dimensional skew $\iota$-Hermitian space over $D$.
\end{introconj}

I ask the reader to quickly glance at Section~\ref{sec:modularity} for the precise definitions involved in this conjecture.

The history of this conjecture and its known cases is quite long and storied, and I cannot recount it here; please refer to Kudla's article~\cite{kudla:special_cycles} and also C. Li's article~\cite{Li2023-uc}.  The following short remarks however can be seen as evidence:
\begin{enumerate}[label=(\Alph*)]
	\item The modularity of the associated generating series of cohomological classes in $H^{2nd_+}(\Sh_{K,\Comp},\Rat(nd_+))$ can be deduced from the constructions of Kudla and Millson~\cite{Kudla1986-ut},\cite{Kudla1987-ze},\cite{Kudla1988-gp},\cite{Kudla1990-om}. In \emph{loc. cit.} this is essentially done when the fixed points of $\iota$ are just the field of rational numbers. When $D = F$ is a totally real extension with $\iota$ trivial, and $W$ is an $(m+2)$-dimensional quadratic space over $F$ that is {}of signature $(m,2)$ or $(m+2,0)$ at every real place of $F$, the desired modularity is the content of~\cite{Rosu2020-an}; see also~\cite[\S 5.3]{Kudla2019-dq}. The argument in general has yet to be written up fully.

	\item When $F = \Rat$ in the last mentioned example, the proof of the conjecture is essentially the content of~\cite{HMP:mod_codim}, building on the seminal work of Borcherds~\cite{borcherds:gkz}, and using the modularity criteria of Zhang and Bruinier-Raum~\cite{Zhang2009-pr},~\cite{bruinier_raum}.

	\item In the case of divisors (that is, where $n= d_+ = 1$) on unitary Shimura varieties of signature $(n,1)$, the conjecture (and quite a bit more) is proven in~\cite{MR4183376}.

	\item For Shimura varieties associated with orthogonal (or GSpin) groups over totally real fields, Kudla shows in~\cite{Kudla2019-dq} that the conjecture is implied in the generic fiber by the Bloch-Beilinson conjectures.

	\item In this article, I show that the `easy' part of the modularity statement, namely automorphy with respect to the Siegel parabolic does indeed follow easily from the construction; see Proposition~\ref{prop:easy_modularity}. As in~\cite{Zhang2009-pr}, the remaining (very hard) part is establishing absolute convergence and automorphy with respect to a Weyl element.

	\item I also show that the `embedding trick' from~\cite{Kudla2019-dq}, which was also used in the proof of the main result of~\cite{HMP:mod_codim}, applies quite generally; see Proposition~\ref{prop:embedding_trick}. This allows in some cases to `reduce' the conjecture to the situation where $n$ is small with respect to the dimension of $\Ss_K$.
\end{enumerate}

\subsection{The role of derived algebraic geometry}
\label{introsubsec:derived}

The classes in Theorem~\ref{introthm:main} will be obtained as linear combinations of virtual fundamental classes of \emph{quasi-smooth derived stacks} over $\Ss_K$. 

Let's unpack this: A basic observation that can be found in~\cite{Gillet1987-ny} is that the $K$-theoretic class of the Koszul complex associated with a section of a vector bundle of rank $d$ over a flat regular $\Int$-scheme $X$ lies in the $d$-th Adams eigenspace $K_0(X)^{(d)}_{\Rat}$. Via the comparison between $K$-theory and the rational Chow groups of $X$ (see Appendix~\ref{app:cycle_classes}), this then produces a cycle class in $\mathrm{CH}^d(X)_{\Rat}$, which is of course the top Chern class of the vector bundle. On the other hand, if the section is \emph{regular}, it is also the cycle class associated with the zero section as a \emph{physical} regularly immersed closed subscheme of $X$. So at least in the case of such a regular section, we have a natural globalization of the Koszul complex in $K$-theory: the structure sheaf of a regular immersion into $X$ of codimension $d$---and part of the way towards the proof of Grothendieck-Riemann-Roch in~\cite{sga6} is the assertion that the class of such a structure sheaf also lies in the $d$-th Adams eigenspace.

Derived algebraic geometry allows one to dispense with the regularity hypothesis from this last assertion. Roughly speaking, just like Grothendieck's introduction of non-reduced rings allowed one to distinguish between a subspace cut out by an equation $x$ and another cut out by $x^2$, derived methods allow us further distinguish both from a `subspace' defined by the equations $x,x^2$. More concretely, for any sequence of elements $a_1,\ldots,a_d$ in a ring $R$, one can write down a \emph{derived} quotient $R/{}^{\mathbb{L}}(a_1,\ldots,a_d)$ obtained via the derived base change
\[
R\otimes^{\mathbb{L}}_{a_i\mapsfrom x_i,\Int[x_1,\ldots,x_d]}\Int.
\]
This is an example of an \emph{animated} or \emph{derived} commutative ring (and even an $R$-algebra in this case): it has the classical or discrete commutative ring $R/(a_1,\ldots,a_d)$ as its underlying ring of connected components, but can in general have non-zero higher homotopy or homology groups; in this case they are simply the cohomology groups of the corresponding Koszul complex. 

Before we go further, a brief pr\'ecis on derived algebraic geometry: It is geometry that is locally modeled by spaces whose rings of functions are animated commutative rings. Here, an animated commutative ring can be thought of as a simplicial commutative ring up to homotopy. We will take the functor of points approach in this paper: objects in derived algebraic geometry will be functors on animated commutative rings, but valued in the ($\infty$-)category of spaces or homotopy types, rather than sets. Any animated commutative ring $R$ gives rise via its corepresentable functor to an \emph{affine} derived scheme $\Spec R$, the \emph{spectrum} of $R$. A derived scheme (resp. derived algebraic space, resp. derived Deligne-Mumford stack) is a functor that satisfies \'etale descent, and admits a covering by affine derived schemes in a suitable topology: Zariski, Nisnevich or \'etale in each of the listed cases; see~\cite[\S 2.1]{Khan2020-pd}. Since any animated commutative ring $R$ can be viewed as a (pro-)nilpotent thickening of its classical truncation $\pi_0(R)$, we can think of a derived scheme as a classical scheme along with a `thickening' of its structure sheaf in higher `derived' directions.

We now return to our particular setting: $\Spec R/{}^{\mathbb{L}}(a_1,\ldots,a_d)$ is an instance of a \defnword{quasi-smooth} derived closed subscheme of $\Spec R$ of \defnword{virtual codimension $d$}. Such objects are studied in~\cite{Khan2018-dk}, and are a generalization of regular immersions to the derived context. One should view them as morphisms $f:Y\to X$ of derived schemes that are relatively affine, and where $f_*\Reg{Y}$ can locally on the source be represented by a finite Koszul complex on $X$ of length $d$. In particular, the zero section of \emph{any} section of a vector bundle over a scheme $X$ of rank $d$ will always give rise to a quasi-smooth closed derived subscheme $Z\hookrightarrow X$ of virtual codimension $d$. 

For future reference, I'll note that there is a more intrinsic definition of a quasi-smooth morphism $Z\to X$ involving the \emph{relative cotangent complex} $\mathbb{L}_{Z/X}$. This is a complex of quasi-coherent sheaves (or rather an object in the derived category of such sheaves), and it controls the deformation theory of $Z$ relative to $X$. A finitely presented morphism $Z\to X$ is quasi-smooth and unramified of virtual codimension $d$ precisely when the shifted complex $\mathbb{L}_{Z/X}[-1]$\footnote{We will always use \emph{co}homological conventions for complexes in this paper.} is (quasi-isomorphic to) a vector bundle of rank $d$ over $Z$. Note that this recovers the usual differential criterion for a closed immersion to be regular.

A. Khan has been able to show that every quasi-smooth morphism $Z\to X$ has a canonical \defnword{virtual fundamental class} $[Z/X]$ attached to it in a degree determined by its virtual codimension $d$. One can realize this fundamental class in various motivic cohomology theories as in~\cite{khan:virtual}, but for our purposes here it turns out to be more convenient to use the $K$-theoretic realization Khan studies in~\cite{khan:gtheory}, since it has a direct relationship with Chow groups with rational coefficients for the spaces of interest to us. Using the methods of~\cite{sga6} in a derived context, Khan shows that the structure sheaf of $Z$ lies in $K_0(X)^{(d)}_{\Rat}$ and thus gives rise to a cycle class in $\mathrm{CH}^d(X)_{\Rat}$.

With this in mind, Theorem~\ref{introthm:main} is now a consequence of the following geometric result:
\begin{introthm}
\label{introthm:derived_cycles}
Suppose that we have a $K$-invariant coset $\mu + U_{\Int}(\widehat{\Int})$ for a $K$-stable lattice $U_{\Int}\subset W^n$. Then, for any $T\in \mathrm{Herm}_n^+(D)$, there is a canonical quasi-smooth and finite unramified morphism $\mathcal{Z}_K(U_{\Int},\mu,T)\to \Ss_K$ of virtual codimension $nd_+$ with the following properties:
\begin{enumerate}
	\item (Cotangent complex) Its cotangent complex over $\Ss_K$ is the pullback of $\bm{co}(U_{\Int})[1]$, where
	\[
     \bm{co}(U_{\Int}) = (\gr^{-1}_{\mathrm{Hdg}}\mathbf{dR}(U_{\Int}))^\vee.
	\]
	\item (Generic fiber) Its underlying classical object in the generic fiber recovers $Z_K(U_{\Int},\mu,\Lambda)$.
	\item (Linear invariance) If $g\in \GL_n(D)$, there is a canonical equivalence of $\Ss_K$-stacks
	\[
      \mathcal{Z}_K(U_{\Int},\mu,T)\xrightarrow{\simeq}\mathcal{Z}_K(gU_{\Int},g\mu,{}^t\iota(g)Tg).
	\]
	Here, the action of $g$ is via the identification $W^n = D^n\otimes_DW$.
	\item (Product identity) If $U'_{\Int}\subset W^{n'}$ is another $K$-stable lattice and $\mu'+U'_{\Int}(\widehat{\Int})$ a $K$-invariant coset, then for $T'\in \mathrm{Herm}^+_{n'}(T)$ we have
	\[
     \mathcal{Z}_K(U_{\Int},\mu,T)\times_{\Ss_K}\mathcal{Z}_K(U'_{\Int},\mu',T')\simeq \bigsqcup_{\tilde{T} = \begin{pmatrix}
     T&*\\
     *&T'
     \end{pmatrix}}\mathcal{Z}_K(U_{\Int}\oplus U'_{\Int},\mu+\mu',\tilde{T}).
	\]
 \end{enumerate}
\end{introthm}

When $\varphi$ is the characteristic function of $\mu + U_{\Int}(\widehat{\Int})$, the cycle class $\mathcal{C}_K(T,\varphi)$ is now recovered as the virtual fundamental class of $\mathcal{Z}_K(U_{\Int},\mu,T)$.

The previous theorem is deduced from a somewhat more general result. To phrase this cleanly, let $\mathrm{Latt}^0_{G,\mu,K}$ be the category of pairs $(W,W_{\Int})$, where $W$ is a representation of $G$ satisfying Assumption~\ref{assump:hodge_emb}, and $W_{\Int}\subset W$ is a $K$-stable lattice. Let $\mathrm{Aff}^{\op}_{/\Ss_K}$ be the $\infty$-category of pairs $(R,x)$, where $R$ is an animated commutative ring and $x\in \Ss_K(R)$ is an $R$-valued point: It is opposite to the $\infty$-category of derived affine schemes $x:\Spec R\to \Ss_K$ over $\Ss_K$. Let $\mathrm{qSmAb}_{/\Ss_K}$ be the $\infty$-category of locally (on the source) quasi-smooth and finite unramified maps $Y\to \Ss_K$ that represent functors 
\[
\mathrm{Aff}^{\op}_{/\Ss_K}\to \Mod[\mathrm{cn}]{\Int}
\]
to the derived category of connective complexes of abelian groups.

I deduce Theorem~\ref{introthm:derived_cycles} from the following result, which is the main content of Section~\ref{sec:shimura}:
\begin{introthm}
\label{introthm:W_scheme}
There is a functor
\begin{align*}
\mathrm{Latt}^0_{G,\mu,K}&\to \mathrm{qSmAb}_{/\Ss_K}\\
(W,W_{\Int})&\mapsto \mathcal{Z}_K(W_{\Int})
\end{align*}
along with a functorial equivalence
\[
\mathbb{L}_{\mathcal{Z}_K(W_{\Int})/\Ss_K}[-1]\xrightarrow{\simeq}\Reg{\mathcal{Z}_K(W_{\Int})}\otimes_{\Reg{\Ss_K}}\bm{co}(W_{\Int})
\]
and a functorial isomorphism
\[
\mathcal{Z}_K(W_{\Int})(\Comp) \simeq G(\Rat)\backslash X(W_{\Int})/K,
\]
where $X(W_{\Int}) = \{(w,x,g)\in W(\Adele_f)\times X\times G(\Adele_f):\; w\in gW_{\Int}(\widehat{\Int})\;;h_x(i)\cdot w = w\}$.
\end{introthm}

\subsection{The classical truncation}
\label{subsec:classical}

One can restrict $\mathcal{Z}_K(W_{\Int})$ to the usual category of commutative rings $R$ equipped with a map $\Spec R\to \Ss_K$: this gives the \emph{classical truncation} $\mathcal{Z}^{\mathrm{cl}}_K(W_{\Int})$. This can be defined without any reference to derived algebraic geometry under a certain further assumption. 

To begin, we reinterpret the $\Comp$-points of $\mathcal{Z}_K(W_{\Int})$ in a moduli theoretic way. For any point $z = [(x,g)]\in \Ss_K(\Comp)$, we obtain a $\Int$-Hodge structure $\mathbf{HS}_z(W_{\Int})$ described as follows: Its underlying $\Int$-module is $W\cap gW_{\Int}(\widehat{\Int})$ with Hodge structure given by the Deligne cocharacter $h_x$. Its base-change over $\Comp$ is the filtered vector space $\Fil^\bullet_{\mathrm{Hdg}}\mathbf{dR}_z(W)$, with underlying vector space $\mathbf{dR}_z(W) = W_{\Comp}$. Theorem~\ref{introthm:W_scheme} now tells us that we must have a canonical identification
\begin{align}\label{introeqn:zk_C_points}
\mathcal{Z}_K(W_{\Int})(z) \simeq \mathbf{HS}^{(0,0)}_z(W_{\Int}),
\end{align}
where the right hand side is the $\Int$-module of weight $(0,0)$ elements of the Hodge structure.

The choice of symplectic representations in Assumption~\ref{assump:hodge_emb}, along with appropriate choices of $K$-stable lattices $H_{1,\Int}\subset H_1$ and $H_{2,\Int}\subset H_2$ such that $W_{\Int}\subset \Hom(H_{1,\Int},H_{2,\Int})$ gives rise to abelian schemes $\mathcal{A}_1,\mathcal{A}_2\to \Ss_K$, whose homological realizations are the various sheaf-theoretic avatars of the pairs $(H_1,H_{1,\Int}),(H_2,H_{2,\Int})$. The construction as described here will depend on these choices, but I explain in Section~\ref{sec:shimura} why they are immaterial in the end, and in fact give rise to a functorially defined stack. 

Introduce the following additional hypothesis:
\begin{assumption}\label{introassump:lattice}
We can choose $(H_1,H_{1,\Int}),(H_2,H_{2,\Int})$ such that $W_{\Int}$ is a direct summand of $\Hom(H_{1,\Int},H_{2,\Int})$.
\end{assumption}

We now consider the subfunctor $\mathcal{Z}^{\mathrm{cl}}_K(W_{\Int})\subset \underline{\Hom}(\mathcal{A}_1,\mathcal{A}_2)$ determined by the following properties: 
\begin{enumerate}
	\item A geometric point $z = (x,f)$ of this Hom scheme parameterizing a point $x\in \Ss_K(k)$ with a homomorphism $f:\mathcal{A}_{1,x}\to \mathcal{A}_{2,x}$ lies in $\mathcal{Z}^{\mathrm{cl}}_K(W_{\Int})(k)$ if and only if its homological realization lies in $\mathbf{dR}_x(W_{\Int})$ if $k$ has characteristic $0$, and in $\mathbf{Crys}_{p,x}(W_{\Int})$ if $k$ has characteristic $p$.
	\item An arbitrary point $z\in \underline{\Hom}(\mathcal{A}_1,\mathcal{A}_2)(C)$ lies in $\mathcal{Z}^{\mathrm{cl}}_K(W_{\Int})$ if and only if its restriction to every geometric point of $\Spec C$ does so.
\end{enumerate}

Theorem~\ref{introthm:W_scheme} now has a concrete consequence for \emph{classical} schemes:

\begin{introthm}
\label{introthm:classical}
The map $\mathcal{Z}^{\mathrm{cl}}_K(W_{\Int})\to \Ss_K$ is a locally finite unramified map of classical schemes, whose $\Comp$-points are as described in Theorem~\ref{introthm:W_scheme}. Moreover:
\begin{enumerate}
	\item For every square-zero thickening $\tilde{C}\to C$ of classical commutative rings with kernel $I$, and every $x\in \Ss_K(C)$ lifting to $\tilde{x}\in \Ss_K(\tilde{C})$, there is an \emph{obstruction map}
\[
\mathrm{obst}_{\tilde{x}}:\mathcal{Z}^{\mathrm{cl}}_K(W_{\Int})(x)\to \gr^{-1}_{\mathrm{Hdg}}\mathbf{dR}_x(W_{\Int})\otimes_CI
\]
such that $\mathcal{Z}^{\mathrm{cl}}_K(W_{\Int})(\tilde{x}) = \ker\mathrm{obst}_{\tilde{x}}$.
\item Let $d_+(W)$ be the rank of the vector bundle $\gr^{-1}_{\mathrm{Hdg}}\mathbf{dR}(W)$. Given an open and closed subscheme $\mathcal{Z}^\circ \subset \mathcal{Z}^{\mathrm{cl}}_K(W_{\Int})$, the following are equivalent:
\begin{enumerate}[label=(\roman*)]
	\item $\mathcal{Z}^\circ$ is an equidimensional stack over $\Int$ of relative dimension $\dim \Sh_K - d_+(W)$.
	\item $\mathcal{Z}^\circ$ is an lci stack over $\Int$ of relative dimension $\dim \Sh_K - d_+(W)$.
\end{enumerate}
\end{enumerate}
\end{introthm}

In the case of orthogonal Shimura varieties, what we have here is a space of \emph{special endomorphisms} considered first in~\cite[\S 5]{mp:reg}. The corresponding analogue for the obstruction map is explicitly stated in~\cite[Prop. 6.5.1]{Howard2020-wv}: our vector bundle $\gr^{-1}_{\mathrm{Hdg}}\mathbf{dR}(W)$ here corresponds to the \emph{cotautological bundle} $\bm{\omega}^{-1}$ considered there. 

One can of course try to prove Theorem~\ref{introthm:classical} using only classical methods, but when I gave it a shot, I found a subtle technical point involving torsion in divided power envelopes rearing its head. The only way I know to get around it in general is to use derived methods as we do here. These methods also play a substantial role in the removal of Assumption~\ref{introassump:lattice}, though I'm not certain if they are absolutely essential.

\subsection{Construction of the derived special cycle: the global story}

So the task laid out in front of us now is upgrading the subfunctor $\mathcal{Z}^{\mathrm{cl}}_K(W_{\Int})$ defined above into the realm of derived algebraic geometry.\footnote{There is also the secondary task of removing Assumption~\ref{introassump:lattice}, but I will punt on that to the body of the paper.} A glance at the definition above indicates that we need to have a good handle on derived analogues of the relevant homological realization as well as of the space of homomorphisms between the abelian schemes $\mathcal{A}_1$ and $\mathcal{A}_2$.

Let us consider the second issue first. We want to get a derived version of the scheme $\underline{\Hom}(\mathcal{A}_1,\mathcal{A}_2)$ of homomorphisms. For this, recall the rigidity property of abelian varieties: A homomorphism between abelian schemes is simply a morphism that preserves zero sections. 

The key point now is that there is a good theory of mapping stacks between (projective) derived schemes (see Appendix~\ref{sec:lurie}): One can define a \emph{derived} stack $\tilde{\mathbb{H}}(\mathcal{A}_1,\mathcal{A}_2)\to \Ss_K$ that assigns to every $(C,z)\in \mathrm{Aff}^{\op}_{/\Ss_K}$ the \emph{space} of zero-section preserving maps from $\mathcal{A}_{1,z}$ to $\mathcal{A}_{2,z}$. By the rigidity property just mentioned, its underlying classical truncation is simply $\underline{\Hom}(\mathcal{A}_1,\mathcal{A}_2)$.

The study of the deformation theory of this functor is basically a glorified version of classical Kodaira-Spencer theory for abelian schemes. There is however a new wrinkle: In the classical realm, the deformation theory of the scheme $\underline{\Hom}(\mathcal{A}_1,\mathcal{A}_2)$ is governed by the \emph{vector bundle} $\pi_{2,*}\Omega^1_{\mathcal{A}_2/\Ss_K}\otimes_{\Reg{\Ss_K}}(R^1\pi_{1,*}\Reg{\mathcal{A}_1})^\vee$. When deforming in derived directions, we are forced to consider the contribution from the total relative cohomology of $\mathcal{A}_1$ (in degrees greater than $0$). This is a general phenomenon in derived algebraic geometry.

Here is the precise formulation:
\begin{introprop}
\label{introprop:h_tilde_cotangent}
Let $\pi_i:\mathcal{A}_i\to \Ss_K$ be the structure maps for $i=1,2$, and let $d_1$ be the relative dimension of $\mathcal{A}_1$ over $\Ss_K$. The relative cotangent complex of the derived scheme $\tilde{\mathbb{H}}(\mathcal{A}_1,\mathcal{A}_2)$ over $\Ss_K$ is the pullback of the perfect complex
\[
\pi_{2,*}\Omega^1_{\mathcal{A}_2/\Ss_K}\otimes_{\Reg{\Ss_K}}(\tau^{\ge 1}R\pi_{1,*}\Reg{\mathcal{A}_1})^\vee \simeq \pi_{2,*}\Omega^1_{\mathcal{A}_2/\Ss_K}\otimes_{\Reg{\Ss_K}}\bigoplus_{i=1}^{d_1}(R^i\pi_{1,*}\Reg{\mathcal{A}_1}[-i])^\vee.
\]
\end{introprop}

In particular, note that the derived stack $\widetilde{\mathbb{H}}(\mathcal{A}_1,\mathcal{A}_2)$ is \emph{not} in general locally quasi-smooth over $\Ss_K$: Its cotangent complex has Tor amplitude $[-d_1,-1]$.

Now, the theory of automorphic vector bundles on $\Ss_K$ gives us a quotient map
\[
\pi_{2,*}\Omega^1_{\mathcal{A}_2/\Ss_K}\otimes_{\Reg{\Ss_K}}(R^1\pi_{1,*}\Reg{\mathcal{A}_1})^\vee\twoheadrightarrow \bm{co}(W_{\Int}) \overset{\mathrm{defn}}{=} (\gr^{-1}_{\mathrm{Hdg}}\mathbf{dR}(W_{\Int}))^\vee
\]
of vector bundles over $\Ss_K$, which will in turn give us a map
\begin{align}
\label{eqn:cotangent_cplx_integrable?}
\mathbb{L}_{\widetilde{\mathbb{H}}(\mathcal{A}_1,\mathcal{A}_2)/\Ss_K}\twoheadrightarrow \Reg{\widetilde{\mathbb{H}}(\mathcal{A}_1,\mathcal{A}_2)}\otimes_{\Reg{\Ss_K}}\bm{co}(W_{\Int})[1]
\end{align}
of perfect complexes over $\widetilde{\mathbb{H}}(\mathcal{A}_1,\mathcal{A}_2)$. 

Therefore, the question of constructing the locally quasi-smooth stack $\mathcal{Z}_K(W_{\Int})$ now becomes one of \emph{integrating} this `quotient bundle' to the cotangent complex of a closed derived substack of $\widetilde{\mathbb{H}}(\mathcal{A}_1,\mathcal{A}_2)$. Towards this, I prove the following proposition in Section~\ref{sec:abvar_crystals} using a Beauville-Laszlo gluing argument:

\begin{introprop}
\label{introprop:gluing}
Suppose that $\mathcal{Z}_K^{\mathrm{cl}}(W_{\Int})\subset\underline{\Hom}(\mathcal{A}_1,\mathcal{A}_2)$ is represented by a closed substack, and that the following holds:
\begin{enumerate}
  \item Over the generic fiber $\mathcal{Z}_K^{\mathrm{cl}}(W_{\Int})_{\Rat}$ lifts to a closed derived substack $\mathcal{Z}_K(W_{\Int})_{\Rat}$ of $\widetilde{\mathbb{H}}(\mathcal{A}_1,\mathcal{A}_2)_{\Rat}$ with relative cotangent complex
  \[
    \mathbb{L}_{\mathcal{Z}_K(W_{\Int})_{\Rat}/\Sh_K}\simeq \Reg{\mathcal{Z}_K(W_{\Int})_{\Rat}}\otimes_{\Reg{\Ss_K}}\bm{co}(W_{\Int})[1].
  \]
  \item For every prime $p$, and every $n\geq 1$, $\mathcal{Z}_K^{\mathrm{cl}}(W_{\Int})_{\Int/p^n\Int}$ lifts to a closed derived substack $\mathcal{Z}_K(W_{\Int})_{\Int/p^n\Int}$ of $\widetilde{\mathbb{H}}(\mathcal{A}_1,\mathcal{A}_2)_{\Int/p^n\Int}$ with relative cotangent complex
  \[
    \mathbb{L}_{\mathcal{Z}_K(W_{\Int})_{\Int/p^n\Int}/\Ss_{K,\Int/p^n\Int}}\simeq \Reg{\mathcal{Z}_K(W_{\Int})_{\Int/p^n\Int}}\otimes_{\Reg{\Ss_K}}\bm{co}(W_{\Int})[1].
  \]
\end{enumerate}
Then these lifts are necessarily unique, and $\mathcal{Z}_K^{\mathrm{cl}}(W_{\Int})$ lifts to a closed derived substack\footnote{For this proposition to be literally true, one needs the identifications of the cotangent complexes to be done in a manner compatible with the surjection~\eqref{eqn:cotangent_cplx_integrable?} in a precise sense.
} $\mathcal{Z}_K(W_{\Int})$ of $\widetilde{\mathbb{H}}(\mathcal{A}_1,\mathcal{A}_2)$ with relative cotangent complex
  \[
    \mathbb{L}_{\mathcal{Z}_K(W_{\Int})/\Ss_K}\simeq \Reg{\mathcal{Z}_K(W_{\Int})}\otimes_{\Reg{\Ss_K}}\bm{co}(W_{\Int})[1].
  \]
\end{introprop}  

So we need to show that $\mathcal{Z}_K^{\mathrm{cl}}(W_{\Int})$ is representable by a closed substack, and also construct local lifts to derived substacks of $\widetilde{\mathbb{H}}(\mathcal{A}_1,\mathcal{A}_2)$ with the right cotangent complex. In fact, $\mathcal{Z}_K^{\mathrm{cl}}(W_{\Int})$ will turn out to be an open and closed substack of $\underline{\Hom}(\mathcal{A}_1,\mathcal{A}_2)$\footnote{Assumption~\ref{introassump:lattice} is required here.}, so even its representability effectively becomes a local question: One just has to pick out the correct connected components.

\subsection{Construction of the derived special cycle: the local story}

The construction of the local lifts of $\mathcal{Z}_K^{\mathrm{cl}}(W_{\Int})$ is tied up intimately with the first named issue of understanding derived analogues of homological realizations. For this, note that, for many points $z$ of $\Ss_K$, we can define spaces that are quite close to $\mathcal{Z}_K(W_{\Int})(z)$ using only information from these realizations.

When $k = \Comp$, we directly have $\mathcal{Z}_K(W_{\Int})(z) = \mathbf{HS}_{z}^{(0,0)}(W_{\Int})$.

For any finitely generated field $k$ in characteristic $0$, and $z\in \Sh_K(k)$, we have a filtered $k$-vector space $\Fil^\bullet_{\mathrm{Hdg}}\mathbf{dR}_z(W)$ equipped with an integrable connection $\nabla_z:\mathbf{dR}_z(W)\to \mathbf{dR}_z(W)\otimes\Omega^1_{k/\Rat}$, so we can take the space
\[
(\Fil^0_{\mathrm{Hdg}}\mathbf{dR}_z(W))^{\nabla = 0},
\] 
which is a finite dimensional $\Rat$-vector space into which $\mathcal{Z}_K(W_{\Int})(z)$ will embed.

When $k$ is a perfect field in characteristic $p$, we obtain for any $z\in \Ss_K(k)$ an $F$-crystal $\mathbf{Crys}_{p,z}(W_{\Int})$, which is a $W(k)$-module equipped with a Frobenius semi-linear automorphism $\varphi_0:\mathbf{Crys}_{p,z}(W_{\Int})[p^{-1}]\xrightarrow{\simeq}\mathbf{Crys}_{p,z}(W_{\Int})[p^{-1}]$. We can now take
\[
\mathbf{Crys}_{p,z}(W_{\Int})^{\varphi_0 = \mathrm{id}}
\]
which is a finite rank $\Int_p$-module into which $\mathcal{Z}_K(W_{\Int})(z)$ will embed. In fact, it embeds densely when $k$ is a finite extension of $\Field_p$. This is a variation of Tate's theorem for homorphisms of abelian varieties, and is shown in Proposition~\ref{prop:tate_conjecture} in the body of the paper.

One of the key points of this paper (and to my eye perhaps its most important contribution) is that it is possible to use derived methods to interpolate these (semi-)linearly defined spaces into geometric objects living locally over $\Ss_K$. These are defined purely in terms of families of what should be homological realizations associated with such objects (or with motives, more generally). Here, these realizations take the form of \emph{derived} filtered crystals. The reader can consult Sections~\ref{sec:infcrys} and~\ref{sec:fcrys} for precise definitions: the basic idea is to work with the usual definition of a crystal, but with respect to animated infinitesimal or divided power thickenings, following for instance~\cite{Mao2021-jt}. Among other things, such animated thickenings are better behaved than their classical counterparts when working with very singular or non-reduced schemes, \emph{especially} in the crystalline context.

The result is cleanest in the $p$-adic setting. Using the methods of Sections~\ref{sec:fcrys} and~\ref{sec:proofs}, we find that there is a $p$-adic transversally filtered derived $F$-crystal $\mathbf{Crys}_p(W_{\Int})$ over $\Ss^{\form}_{K,p}$\footnote{I use the notation $(\cdot)_p^\mathcal{F}$ to denote the $p$-adic formal completion of a scheme or stack, because we will be dealing with schemes valued in complexes of abelian groups, and I'd like to reserve the notation $(\cdot)_p^{\wedge}$ for the (derived) $p$-completion of such objects.}, which interpolates the spaces $\mathbf{Crys}_{p,z}(W_{\Int})$ considered above. We can now consider the space of sections of this object. We will also need to consider $p$-adic formal objects: These will be presehaves on the $\infty$-subcategory $\mathrm{Aff}_{/\Ss^{\form}_{K,p}}$ of $\mathrm{Aff}_{/\Ss_K}$ spanned (in the opposite category) by objects $(C,z)$ where $p$ is nilpotent in $\pi_0(C)$. All the relevant terminology extends to this situation painlessly.

Now, one can consider the \defnword{functor of sections} of the derived filtered $F$-crystal $\mathbf{Crys}_p(W_{\Int})$: to every $(C,z)$ in $\mathrm{Aff}_{/\Ss^{\form}_{K,p}}$ it assigns the space $\Map(\mathbf{1},\mathbf{Crys}_{p,z}(W_{\Int}))$ of maps in an appropriate $\infty$-category of derived filtered $F$-crystals over $C$, where $\mathbf{1}$ is the `unit' object. When $C = k$ is a perfect field, this recovers the space $\mathbf{Crys}_{p,z}(W_{\Int})^{\varphi_0 = \mathrm{id}}$ considered above.
 
Doing this systematically yields the key technical result of this paper, for which we will need one additional piece of notation. Let $\mathrm{qSmAb}_{/\Ss^{\form}_{K,p}}$ be the $\infty$-category of locally (on the source) quasi-smooth and finite unramified maps $\mathfrak{Y}\to \Ss^{\form}_{K,p}$ that represent functors $\mathrm{Aff}^{\op}_{/\Ss^{\form}_{K,p}}\to \Mod[\mathrm{cn}]{\Int}$.

\begin{introthm}
\label{introthm:filcrys}
For every prime $p$, there is a functor
\begin{align*}
\mathrm{Latt}^0_{G,\mu,K}&\to \mathrm{qSmAb}_{/\Ss^{\form}_{K,p}}\\
(W,W_{\Int})&\mapsto \mathcal{Z}_{K,crys}(W_{\Int})^{\form}_p
\end{align*}
along with a functorial equivalence
\[
\mathbb{L}_{\mathcal{Z}_{K,crys}(W_{\Int})^{\form}_p/\Ss^{\form}_{K,p}} \simeq \Reg{\mathcal{Z}_{K,crys}(W_{\Int})^{\form}_p}\otimes_{\Reg{\Ss_K}}\bm{co}(W_{\Int})[1],
\]
and is such that, for $(k,z)$ in $\mathrm{Aff}^{\op}_{/\Ss^{\form}_{K,p}}$ with $k$ a perfect field, we have
\[
\mathcal{Z}_{K,crys}(W_{\Int})^{\form}_p(z) \simeq \mathbf{Crys}_{p,z}(W_{\Int})^{\varphi_0 = \mathrm{id}}.
\]
\end{introthm}

Theorem~\ref{introthm:filcrys} is shown by verifying the criteria for representability in Lurie's generalization of Artin's representability theorem (see Appendix~\ref{sec:lurie}): This process is applied to certain truncated versions of the functor of sections $\mathcal{Z}_{K,crys}(W_{\Int})^{\form}_p$, and the statement for $\mathcal{Z}_{K,crys}(W_{\Int})^{\form}_p$ itself follows by taking limits.

We now get the desired local lifts $\mathcal{Z}_K(W_{\Int})_{\Int/p^n\Int}$ that feed into Proposition~\ref{introprop:gluing} as follows: Both the formal $p$-adic completion $\widetilde{\mathbb{H}}(\mathcal{A}_1,\mathcal{A}_2)^{\form}_p$ and the space $\mathcal{Z}_{K,crys}(W_{\Int})^{\form}_p$ above admit maps to a common space $\widetilde{\mathbb{H}}^{\form}_{p,crys}$, which is constructed in the same way as $\mathcal{Z}_{K,crys}(W_{\Int})^{\form}_p$, but now parameterizes \emph{maps} of filtered $F$-crystals associated with the crystalline cohomology of $\mathcal{A}_1$ and $\mathcal{A}_2$. The map from $\widetilde{\mathbb{H}}(\mathcal{A}_1,\mathcal{A}_2)^{\form}_p$ to $\widetilde{\mathbb{H}}^{\form}_{p,crys}$ is formally \'etale---this is a derived analogue of Grothendieck-Messing theory---and the map from $\mathcal{Z}_{K,crys}(W_{\Int})^{\form}_p$ is a closed immersion. The pullback of $\mathcal{Z}_{K,crys}(W_{\Int})^{\form}_p$ over $\widetilde{\mathbb{H}}(\mathcal{A}_1,\mathcal{A}_2)^{\form}_p$ now gives the desired local lift $\mathcal{Z}_K(W_{\Int})_{\Int/p^n\Int}$ for each $n\geq 1$. In other words, we are \emph{geometrically} constructing a stack of derived homomorphisms between the abelian schemes $\mathcal{A}_1$ and $\mathcal{A}_2$ with the correct crystalline homological realizations. 

As the previous paragraph hints at, one example in which the methods of Theorem~\ref{introthm:filcrys} apply is the case where we are looking to parameterize maps between the usual crystalline realizations of two abelian schemes; equivalently, we are looking at maps between the Dieudonn\'e $F$-crystals associated with their $p$-divisible groups. What the theorem is saying in this case is that we can get a good geometric space that parameterizes such maps---but we have to upgrade everything into the derived or animated realm to be able to do so. 

To obtain the \emph{generic} fiber $\mathcal{Z}_K(W_{\Int})_{\Rat}$, we get by with a bit less. The category of vector bundles with an integrable connection and a filtration compatible with Griffiths transversality generalizes to what I call here a \emph{transversally filtered infinitesimal crystal}. This is studied in Section~\ref{sec:infcrys}, where it is shown that the spaces of sections of such objects are reasonably well-behaved (though not representable) derived prestacks, which are good enough to allow use of the strategy from the paragraph before the last to construct the appropriate closed derived subscheme $\mathcal{Z}_K(W_{\Int})_{\Rat}$ of $\widetilde{\mathbb{H}}(\mathcal{A}_1,\mathcal{A}_2)_{\Rat}$. In fact, for the global applications, we could have also gotten away with a similarly limited result in the $p$-adic context as well---but I find the fact that one can establish representability purely locally quite striking, and so have included it here.

\subsection{Derived homomorphisms of abelian schemes}
\label{subsec:derived_hom}

The process used in the construction of $\mathcal{Z}_K(W_{\Int})$ has an interesting consequence for abelian schemes that is of independent interest.

Given abelian schemes $X,Y$ over a (discrete or classical) commutative ring $R$, we have the scheme $\underline{\Hom}(X,Y)$ of homomorphisms between them, and, just as we did above, we can consider the derived enhancement $\widetilde{\mathbb{H}}(X,Y)$ of zero section preserving maps. Unlike in the classical situation, the roles of $X$ and $Y$ in this derived scheme are quite asymmetric; by this I mean that we do not have an identification of the form
\[
\widetilde{\mathbb{H}}(X,Y)\simeq \widetilde{\mathbb{H}}(Y^\vee,X^\vee),
\]
where we are considering the dual abelian schemes on the right. This has to do with the fact (observed above already in Proposition~\ref{introprop:h_tilde_cotangent}) that the relative cotangent complex for $\widetilde{\mathbb{H}}(X,Y)$ over $R$ is the pullback of the complex
\[
\omega_Y\otimes_R\bigoplus_{i=1}^{\dim X}H^i(X,\Reg{X})^\vee[i],
\]
which is also asymmetric in $X$ and $Y$.

In particular, note that $\widetilde{\mathbb{H}}(X,Y)$ is not locally quasi-smooth over $R$ \emph{unless} $\dim X  = 1$; that is, when $X$ is an elliptic curve. One can exploit this last fact to give a much more direct construction of (analogues of) the derived stacks in Theorem~\ref{introthm:derived_cycles} over the moduli stacks of Kudla-Rapoport from~\cite{Kudla2009-lc}, though I do not do so here.
 
However, the methods used in the proof of Theorem~\ref{introthm:derived_cycles} show:
\begin{introthm}\label{introthm:end}
There is a closed derived subscheme $\mathbb{H}(X,Y)$ of $\widetilde{\mathbb{H}}(X,Y)$ that is locally quasi-smooth of virtual codimension $\dim(X)\dim(Y)$ over $R$ with underlying classical scheme $\underline{\Hom}(X,Y)$, and with relative cotangent complex
\[
\mathbb{L}_{\mathbb{H}(X,Y)/R} \simeq \Reg{\mathbb{H}}\otimes_R(\omega_Y\otimes_RH^1(X,\Reg{X}))[1]\simeq \Reg{\mathbb{H}}\otimes_R(\omega_Y\otimes_R\omega_{X^\vee})[1]. 
\]
\end{introthm}

Except in the case where $X$ is an elliptic curve, where it agrees with $\widetilde{\mathbb{H}}(X,Y)$, I have been unable to find a direct moduli-theoretic construction for $\mathbb{H}(X,Y)$.

A natural guess is to consider the functor $\underline{\mathbb{H}om}(X,Y)$ of maps between $X$ and $Y$ viewed as fppf sheaves of \emph{animated} abelian groups on the $\infty$-category $\mathrm{CRing}_{R/}$ of animated commutative rings over $R$. The problem is that this is a bit too complicated for our purposes: Its tangent spaces will see higher Ext groups of the form $\Ext^i(X,\mathbb{G}_a)$, which vanish in characteristic $0$ for $i>1$, but which can in characteristic $p$ be non-zero for arbitrarily large values of $i$. This is related to the non-vanishing of the cohomology of Eilenberg-MacLane spaces; see~\cite{Breen1969-qn}.

The construction I give here---just like the one for $\mathcal{Z}_K(W_{\Int})$---is obtained by gluing together various local incarnations into a closed derived subscheme of the mapping scheme $\widetilde{\mathbb{H}}(X,Y)$ via Proposition~\ref{introprop:gluing}, and sidesteps the issue of giving a geometric description for it. These local incarnations are obtained as spaces of maps between the degree $1$ cohomological realizations of $Y$ and $X$.

Before I wrap up here, let me remark that the functor $\mathbb{H}(X,Y)$ does behave like the space of maps in some stable $\infty$-category in an important way: It is additive in both variables; see Corollary~\ref{cor:hom_additive}. Moreover, if $X=Y$, then $\mathbb{H}(X,X)$ can be endowed with the structure of an associative algebra object in an appropriate $\infty$-category. 

Here's a wildly speculative interpretation of this phenomenon: abelian varieties are supposed to provide---under Scholze's general philosophy---instances of (the as of yet elusive notion of) global shtukas. What we have defined here should the space of derived maps between the shtukas underlying the abelian schemes $X$ and $Y$. The `forgetful' functor from abelian varieties to shtukas shuts the door on the pesky Eilenberg-MacLane cohomology classes, leaving us with a nice (quasi-)smooth space.

This speculation comports with the actual construction of $\mathbb{H}(X,Y)$, if one admits that derived filtered ($F$-)crystals can be viewed as families of local (integral) shtukas.

\subsection{A word (or two) on the technical details}
\label{introsubsec:technical}

The concepts and methods of proof used in Theorem~\ref{introthm:filcrys} form the technical heart of this paper. To begin, finding the correct definition of a filtered $F$-crystal turns out to be a bit subtle, and most of Section~\ref{sec:fcrys} is dedicated to getting there. What I arrive at is essentially the notion of a strongly divisible filtered $F$-crystal as defined for instance by Faltings~\cite{Faltings1989-af}. This theory is only well-behaved when the filtration is bounded between certain degrees (depending on $p$). It turns out, in any case, that we only need the Frobenius structure to be partially defined in two consecutive degrees, so this is not an issue here. The correct framework for such objects no doubt involves prismatic methods, but that's a story for another time.

As for the actual argument for the theorem (realized by Theorems~\ref{thm:FilFCrys_functorial_props} and~\ref{thm:FilFCrys_discrete} with proofs in Section~\ref{sec:proofs}), as mentioned above, the main point is to verify the criteria in Lurie's representability theorem for the derived mod-$p$ quotient of the functor $\mathcal{Z}_{K,crys}(W_{\Int})^{\form}_p$. For this, the assiduous reader might find it useful to know that the ideas from~\cite[\S 7]{Bhatt2022-sb} involved in the study of the crystalline Chern class played a key role in the development of the proofs here. The arguments in \emph{loc. cit.} can be interpreted as constituting the study of the space of sections of the crystalline realization of the Tate twist: After quotienting out by $p$, we get something equivalent to the \emph{scheme} $\mu_p$, which is of course quasi-smooth. Something similar happens in general: Working with the mod-$p$ quotient, I make use of a `conjugate' filtration provided by the $F$-crystal structure to break up the space of sections into manageable pieces, which can then be related to a much simpler space where representability---and hence the validity of Lurie's criteria---is more or less immediate. 

\subsection{What remains to be done (foundationally speaking)}

Outside of proving the modularity conjecture (!), there is still some basic foundational work that remains outstanding:

\begin{enumerate}[label=(\Alph*)]
	\item To get a fully robust theory, one would like to upgrade the cycle classes---and the conjecture---to those valued in an appropriate \emph{arithmetic} Chow group of a compactification of $\Ss_K$. Lifting to the Gillet-Soul\'e arithmetic Chow group of $\Ss_K$ is possible: In ongoing work with L. Garcia, we employ the description from~\cite{Gillet1990-ps} of the arithmetic Chow groups in terms of Hermitian vector bundles to upgrade the cycle classes here to arithmetic ones. But finding a suitable extension to the boundary appears to be a subtle question even for classical Chow groups; see~\cite{MR453649},\cite{MR4183376},\cite{MR4476226}.

	\item The constructions given here can be generalized to certain resolutions of Kisin-Pappas integral models for Shimura varieties from~\cite{Kisin2018-nm}. To do this in general would require a robust theory of filtered crystalline (or, better, prismatic) realizations over such models, as found for the unramified case in~\cite{Lovering2017-fy}. 

    Of course, even if one is able to construct quasi-smooth derived cycles over the models of Kisin-Pappas, it is only when we are working with a \emph{regular} integral model that one can extract cycle classes out of them. In this generality, one might be better off sticking with classes in $K$-theory.

	\item Along these lines, the methods used in Theorem~\ref{introthm:filcrys} can also be used to construct derived special cycles on (certain resolutions of) Rapoport-Zink spaces. This would for instance recover the `linear invariance' results of~\cite{Howard2019-td} for Rapoport-Zink spaces of unitary type in a moduli theoretic way.

   \item One should also extend the results here to the case of Shimura varieties of abelian type by incorporating the methods of~\cite{kisin:abelian} and~\cite{Lovering2017-me}. 
\end{enumerate}




\subsection{Structure of the paper}

\begin{itemize}
  \item In Section~\ref{sec:infcrys}, I lay out the apparatus necessary for setting up the story in the generic fiber: derived infinitesimal envelopes (which turn out to be classical objects under some mild Noetherian assumptions) and derived filtered infinitesimal crystals. This culminates in the main results, Propositions~\ref{prop:FIC_functorial_props} and~\ref{prop:FIC_discrete}.

  \item In Section~\ref{sec:fcrys}, we get introduced to derived PD thickenings following~\cite{Mao2021-jt}, and I use them to define (certain) derived filtered $F$-crystals and state the main Theorems~\ref{thm:FilFCrys_functorial_props} and~\ref{thm:FilFCrys_discrete}. The proofs of these theorems are the subject of Section~\ref{sec:proofs}.

  \item In Section~\ref{sec:abvar_crystals}, I introduce the derived mapping scheme $\widetilde{\mathbb{H}}(X,Y)$, and study its relationship with the local spaces constructed out of the cohomological realizations of $X$ and $Y$ via the the methods of Sections~\ref{sec:infcrys} and~\ref{sec:fcrys}. Denoting these local avatars by $\widetilde{\mathbb{H}}_\infty(X,Y)$ and $\widetilde{\mathbb{H}}^{\form}_{p,crys}(X,Y)$, I show that homological realization gives formally \'etale maps
  \[
    \widetilde{\mathbb{H}}(X,Y)\to \widetilde{\mathbb{H}}_\infty(X,Y)\;;\; \widetilde{\mathbb{H}}(X,Y)\to \widetilde{\mathbb{H}}^{\form}_{p,crys}(X,Y).
  \]
  This should be viewed as a derived analogue of Grothendieck-Messing theory. I then prove the gluing result in Proposition~\ref{introprop:gluing} (this is a consequence of the more general Corollary~\ref{cor:exist_Y}), and use it to construct the locally quasi-smooth scheme $\mathbb{H}(X,Y)$ in Theorem~\ref{introthm:end}.
 
  \item In Section~\ref{sec:shimura}, I apply the methods of the preceding sections to derived filtered crystals obtained from automorphic bundles on smooth integral canonical models of Shimura varieties. This gives us Theorem~\ref{thm:W_representability}, which recovers Theorem~\ref{introthm:derived_cycles} as a special case.

  \item In Section~\ref{sec:cycle_classes}, results quoted in Appendix~\ref{app:cycle_classes} are used to go from quasi-smooth maps to the cycle classes in Theorem~\ref{introthm:main}.

  \item Section~\ref{sec:modularity} is dedicated to the formulation and discussion of Conjecture~\ref{introconj}.
\end{itemize}

\subsection{The appendices}

There are several appendices. The first three contain definitions and results about $\infty$-categories that are well-known and `standard', but for which I couldn't find satisfactory references: In Appendix~\ref{sec:fibrations}, our goal is to know that sections of certain fibrations such as those that give rise to derived filtered crystals can be studied in terms of sections over the cosimplicial \u{C}ech nerve of a weakly final object. This is used in the paper to lift to the derived world the usual description of crystals in terms of vector bundles with integrable connection over infinitesimal or divided power envelopes in smooth schemes. Appendix~\ref{sec:completions} recollects a couple of facts about derived $p$-completions and is here for my psychological comfort. Appendix~\ref{sec:filtrations} covers filtrations and gradings from an $\infty$-category perspective. In particular, it studies dualizability of filtered modules over a filtered animated commutative ring and its relationship with dualizability in the category of graded modules over the associated graded animated commutative ring.

Appendix~\ref{sec:modules} recalls a convenient device from~\cite{lurie:spec_AB} to keep track of module objects in an $\infty$-category in terms of the associated Serre tensor construction.

Appendices~\ref{sec:lurie} and~\ref{sec:qsmooth} cover some background material on derived algebraic geometry for the convenience of the reader: In the former, I review the basics of the functorial approach, culminating in the statement of Lurie's representation theorems, and their use in obtaining representable derived mapping stacks between certain schemes. The latter appendix is dedicated to the definition and recollection of properties of quasi-smooth schemes from references like~\cite{Khan2018-dk}.

Finally, in Appendix~\ref{app:cycle_classes}, I recall results of Gillet-Soul\'e and Gillet on the relationship between $K$-groups and Chow groups following the treatment in~\cite{HMP:mod_codim}.

\subsection{Additional remarks}

\begin{introrem}
The idea of using derived algebraic geometry to construct cycle classes has already been employed by Feng-Yun-Zhang over global fields in finite characteristic~\cite{FYY2}, where they work with certain moduli spaces of shtukas. Of course the actual constructions here are quite a bit different in their technical manifestation, though they are certainly closely related in spirit.

As mentioned in the introduction of \emph{loc. cit.}, the results here confirm in our specific context the `hidden smoothness' philosophy of Deligne, Drinfeld and Kontsevich, which posits that singularities of naturally arising moduli problems can be resolved most transparently in the derived direction.
\end{introrem}

\section*{Acknowledgments}
I thank Ben Howard for posing the question of whether the construction of~\cite{HMP:mod_codim} had a derived moduli interpretation, David Treumann for patiently listening to my early half-baked ideas,  Adeel Khan for discussions about his work and about derived algebraic geometry in general, Mark Kisin for comments on an initial draft, and Michael Rapoport and Andreas Mihatsch for their remarks on the applicability of the framework used here. In particular, the impetus to precisely formulate the modularity conjecture here came from Rapoport, and this led to a substantial reworking of the presentation of the paper. Finally, it is a pleasure to acknowledge the obvious intellectual debt this paper owes to the work of Steve Kudla.

This work was partially supported by NSF grants DMS-1802169 and DMS-2200804.

\section*{Notation and other conventions}

\begin{enumerate}
  \item I adopt a resolutely $\infty$-categorical approach. This means that all operations, including (but not limited to) limits, colimits, tensor products, exterior powers etc. are always to be understood in a derived sense, unless otherwise stated.

  \item We will use $\mathrm{Spc}$ to denote the $\infty$-category of spaces or homotopy types: roughly speaking, this is the localization of the Quillen model category of simplicial sets with respect to homotopy equivalences. 

  \item At a first approximation, an $\infty$-category $\mathcal{C}$ is something that has a space of objects, and for any two objects $X,Y$ a space of morphisms $\Map_{\mathcal{C}}(X,Y)$.

  \item For any $\infty$-category $\mathcal{C}$ and an object $c$ of $\mathcal{C}$, we will write $\mathcal{C}_{c/}$ (resp. $\mathcal{C}_{/c}$) for the comma $\infty$-categories of arrows $c\to d$ (resp. $d\to c$).

  \item Given any two $\infty$-categories $\mathcal{C}$, $\mathcal{D}$, we have an $\infty$-category $\mathrm{Fun}(\mathcal{C},\mathcal{D})$ of functors with the arrows being the $\infty$-category analogues of natural transformations.

  \item This will apply in particular with $\mathcal{C}$ a discrete or classical category, including in particular the category $[1]$ with two objects $0,1$ and a single non-identity map $0\to 1$. In this case, $\mathrm{Fun}([1],\mathcal{C})$ is the $\infty$-category of arrows $X\to Y$ for objects $X$ and $Y$ in $\mathcal{C}$.
 
  \item We will often make reference to the process of \defnword{animation}, as described say in~\cite[Appendix A]{Mao2021-jt}. This is a systematic way to get well-behaved $\infty$-categories and functors between them, starting from `nice' classical categories with a set of compact, projective generators.

  \item An important example is $\mathrm{CRing}$, the $\infty$-category of \defnword{animated commutative rings}, obtained via the process of animation from the usual category of commutative rings. Objects here can be viewed as being simplicial commutative rings up to homotopy equivalence.

  \item We will follow homological notation for $\mathrm{CRing}$: For any $n\in \Int_{\ge 0}$, $\mathrm{CRing}_{\leq n}$ will be the subcategory of $\mathrm{CRing}$ spanned by those objects $R$ with $\pi_k(R) = 0$ for $k>n$; that is, by the \defnword{$n$-truncated objects}. If $n = 0$, we will write $\mathrm{CRing}_{\heartsuit}$ instead of $\mathrm{CRing}_{\leq 0}$: its objects are the \defnword{discrete} or classical commutative rings, and the category can be identified with the usual category of commutative rings. 

  \item Any animated commutative ring $R$ admits a \defnword{Postnikov tower} $\{\tau_{\leq n}R\}_{n\in \Int_{\ge 0}}$ where $R\to \tau_{\leq n}R$ is the universal arrow from $R$ into $\mathrm{CRing}_{\leq n}$ and the natural map $R\to \varprojlim_n \tau_{\leq n}R$ is an equivalence.

  \item We will also need the notion of a \defnword{stable $\infty$-category} from~\cite{Lurie2017-oh}: this is the $\infty$-category analogue of a triangulated category. The basic example is the $\infty$-category $\Mod{R}$, the derived $\infty$-category of $R$-modules. I will use \emph{cohomological} conventions for these objects.

  \item An important feature of a stable $\infty$-category $\mathcal{C}$ is that it has an initial and final object $0$, and, for any map $f:X\to Y$ in $\mathcal{C}$, we have the \defnword{homotopy cokernel} $\hcoker(f)$ defined as the pushout of $0\to Y$ along $f$. I will sometimes abuse notation and write $Y/X$ for this object.

  \item If $R\in \mathrm{CRing}_{\heartsuit}$ is a classical commutative ring, $M\in \Mod{R}$ is a complex of $R$-modules, and $a\in R$ is a non-zero divisor, we will write $M/{}^{\mathbb{L}}a$ for the derived tensor product
  \[
    M\otimes^{\mathbb{L}}_RR/aR.
  \]
  This will apply mostly in the situation where $R = \Int$.


  \item For any $M\in \Mod{\Int}$, we will write $\widehat{M}_p$ for the derived $p$-completion $\varprojlim_n M/{}^{\mathbb{L}}p^n$.

  \item In any stable $\infty$-category $\mathcal{C}$ and an object $X$ in $\mathcal{C}$, we set $X[1]=\hcoker(X\to 0)$: this gives a shift functor $\mathcal{C}\to \mathcal{C}$ with inverse $X\mapsto X[-1]$, and we set $\hker(f:X\to Y)  = \hcoker(f)[-1]$. 

  \item Given an animated commutative ring $R$, we will write $\Mod[\mathrm{cn}]{R}$ for the sub $\infty$-category spanned by the connective objects (that is, objects  with no cohomology in positive degrees), and $\Mod[\mathrm{perf}]{R}$ for the sub $\infty$-category spaned by the perfect complexes. 

  \item We have a truncation operator $\tau^{\leq 0}:\Mod{R}\to \Mod[\mathrm{cn}]{R}$ defined as the right adjoint to the natural functor in the other direction. This leads to truncation operators $\tau^{\leq n}$ and cotruncation operators $\tau^{\geq n}$ for any $n\in \Int$ in the usual way.

  \item If $f:X\to Y$ is a map in $\Mod[\mathrm{cn}]{R}$, we set $\hker^{\mathrm{cn}}(f) = \tau^{\leq 0}\hker(f)$: this is the \defnword{connective (homotopy) kernel}.

  \item For any stable $\infty$-category $\mathcal{C}$, the mapping spaces $\Map_{\mathcal{C}}(X,Y)$ between any two objects have canonical lifts to the $\infty$-category of connective spectra. We will be interested in stable $\infty$-categories like $\Mod{R}$, which are $\Mod[\mathrm{cn}]{\Int}$-modules, in the sense that the mapping spaces have canonical lifts to $\Mod[\mathrm{cn}]{\Int}$. In this case, we can extend the mapping spaces $\Map_{\mathcal{C}}(X,Y)$ from $\Mod[\mathrm{cn}]{\Int}$ to objects $\underline{\Map}_{\mathcal{C}}(X,Y)$ in $\Mod{\Int}$ by taking 
  \[
    \underline{\Map}_{\mathcal{C}}(X,Y) = \colim_{k\geq 0}\Map_{\mathcal{C}}(X,Y[k])[-k]\in \Mod{\Int}.
  \]
  When $\mathcal{C} = \Mod{R}$ for an animated commutative ring $R$, this lifts to the \defnword{internal Hom} in $\Mod{R}$.

  \item We will write $\bm{\Delta}$ for the usual \defnword{simplex} category with objects the sets $\{0,1,\ldots,n\}$ and morphisms given by the non-decreasing functions between them.

  \item A \defnword{cosimplicial object} $S^{(\bullet)}$ in an $\infty$-category  $\mathcal{C}$ is a functor
  \begin{align*}
  \bm{\Delta}&\mapsto \mathcal{C}\\
  [n]&\mapsto S^{(n)}.
  \end{align*}
  If $\mathcal{C}$ admits limits, we will write $\mathrm{Tot}S^{(\bullet)}$ for the limit of the corresponding functor: this is the \defnword{totalization} of $S^{(\bullet)}$.

  \item Given any $\infty$-category $\mathcal{C}$ with finite coproducts, and any object $S$ in $\mathcal{C}$ there is a canonical cosimplicial object $S^{(\bullet)}$ in $\mathcal{C}$, the \defnword{\u{C}ech conerve} with $S^{(n)} = \bigsqcup_{i\in [n]}S$. 
\end{enumerate}

For all other conventions, including those to do with filtrations and gradings, as well as terminology from derived algebraic geometry, please see the various appendices.

\section{Spaces of sections of derived filtered crystals: the infinitesimal context}\label{sec:infcrys}

The goal of this section is to formulate and prove a weaker version of Theorem~\ref{introthm:filcrys} that applies to derived filtered infinitesimal crystals, which is nevertheless sufficient for global applications. This turns out to be a matter of putting together the formalism of animation with the results of Bhatt from~\cite{Bhatt2012-kp}.

\subsection{}\label{subsec:anipair}
We begin by recalling some useful constructions from~\cite[\S 3.2]{Mao2021-jt}. Consider the (classical) category $\mathcal{D}$ of maps $R\to S$ of commutative rings. Within this, we find the full subcategory $\mathcal{D}^0$ consisting of all possible coproducts of the pair of maps $\Int[x]\xrightarrow{\mathrm{id}}\Int[x]$ and $\Int[x,y]\xrightarrow[y\mapsto 0]{x\mapsto x}\Int[x]$. 

Then $\mathcal{D}^0$, by construction, admits a set of compact projective generators, and so by the process of animation from~\cite[Appendix A]{Mao2021-jt} we obtain the $\infty$-category $\mathrm{Ani}(\mathcal{D}^0)$, which, by~\cite[Theorem 3.23]{Mao2021-jt}, is equivalent to the $\infty$-category $\mathrm{AniPair}$ spanned within $\mathrm{Fun}(\Delta[1],\mathrm{CRing})$ by those arrows $f:A\to A'$ such that $\pi_0(f):\pi_0(A)\to \pi_0(A')$ is a surjective map of classical commutative rings. We will use the notation $A\twoheadrightarrow A'$ to denote an object in $\mathrm{AniPair}$: Concretely, such a map can be represented by a map of simplicial commutative rings that is surjective on $\pi_0$.

Note that we have functors
\[
s:\mathrm{AniPair}\xrightarrow{(A'\twoheadrightarrow A)\mapsto A'} \mathrm{CRing}\;;\;t:\mathrm{AniPair}\xrightarrow{(A'\twoheadrightarrow A)\mapsto A} \mathrm{CRing}.
\]

\subsection{}\label{subsec:aniinfpair}
 There is a canonical functor
\[
\mathrm{AniPair}\xrightarrow{(A'\twoheadrightarrow A)\mapsto \mathrm{Fil}^\bullet_{(A'\twoheadrightarrow A)}A'} \mathrm{FilCRing}
\]
obtained by animating the functor sending a surjection $A'\twoheadrightarrow A$ with kernel $I$ to the ring $A'$ equipped with its $I$-adic filtration. 

We take $\mathrm{AniInfPair}$ to be the $\infty$-subcategory of $\mathrm{AniPair}$ spanned by the objects $A'\twoheadrightarrow A$ such that $\mathrm{Fil}^i_{(A'\twoheadrightarrow A)}A'\simeq 0$ for all $i$ sufficiently large. Write
\[
\mathrm{AniInfPair}^{[n]}\to \mathrm{AniInfPair}
\]
for the $\infty$-category spanned by those objects with $\Fil^i = 0$ for $i\geq n$.

Let $\mathrm{AniPair}_{\Rat/}$ be the $\infty$-subcategory of $\mathrm{AniPair}$ spanned by $A'\twoheadrightarrow A$ with $A'\in \mathrm{CRing}_{\Rat/}$. Then we analogously obtain $\infty$-subcategories $\mathrm{AniInfPair}_{\Rat/}$ and $\mathrm{AniInfPair}^{[n]}_{\Rat/}$ of $\mathrm{AniPair}_{\Rat/}$.

\begin{lemma}
\label{lem:adic_filtration}
Suppose that $A'\twoheadrightarrow A$ is in $\mathrm{AniInfPair}_{\Rat/}$. Then there is a canonical equivalence
\[
\gr^i_{(A'\twoheadrightarrow A)}A'\xrightarrow{\simeq}(\wedge^i_A\mathbb{L}_{A'/A})[-i],
\]
where on the right hand side, we are employing the \emph{derived} or animated exterior power functor. 
\end{lemma}
\begin{proof}
This is an observation from~\cite[Remark 4.13]{Bhatt2012-kp}. By the definition of the derived adic filtration, the object
\[
\gr^i_{(A'\twoheadrightarrow A)}A'\in \Mod[\mathrm{cn}]{A},
\]
is obtained by animating the functor
\begin{align*}
\mathcal{D}^0_{\Rat}&\to \mathrm{Mod}\\
(R'_{\Rat}\overset{g}{\twoheadrightarrow} R_{\Rat})&\to (R_{\Rat},(\ker g)^i/(\ker g)^{i+1}).
\end{align*}

Now, we have natural equivalences
\[
(\ker g)^i/(\ker g)^{i+1} \simeq \Sym^i_{R_{\Rat}}(\ker g/(\ker g)^2) \simeq \Gamma^i_{R_{\Rat}}(\mathbb{L}_{R_{\Rat}/R'_{\Rat}}[-1])\simeq \wedge^i\mathbb{L}_{R_{\Rat}/R'_{\Rat}}[-i].
\]
Here, the first equivalence uses the the regularity of the ideal $\ker g$, while the second combines the identification between symmetric and divided powers in characteristic $0$ with the natural identitication
\[
\mathbb{L}_{R_{\Rat}/R'_{\Rat}} \simeq (\ker g)/(\ker g)^2[1],
\]
which once again uses the regularity of $\ker g$. The final equivalence arises from an observation of Illusie, which is recalled for instance in~\cite[Prop. 25.2.4.2]{Lurie2018-kh}.
\end{proof}

\subsection{}
For any $g:A'\twoheadrightarrow A$ in $\mathrm{AniPair}_{\Rat}$, set
\[
\mathcal{I}^{[n]}(g) = \mathrm{hcoker}(\Fil^n_{(A'\twoheadrightarrow A)}A\to A').
\]
We will refer to the object
\[
\mathcal{I}^{[n]}(g)\twoheadrightarrow A\in \mathrm{AniInfPair}^{[n]}_{\Rat/}
\]
as the \defnword{$n$-th infinitesimal envelope of $g$}. This can also be viewed as the animation of the functor
\begin{align*}
\mathcal{D}^0_{\Rat}&\to \mathrm{AniInfPair}^{[n]}_{\Rat}\\
(g:R'_{\Rat}\to R_{\Rat})&\to (R'_{\Rat}/(\ker g)^n\twoheadrightarrow R_{\Rat}).
\end{align*}

\begin{lemma}
\label{lem:inf_env_discrete}
Suppose that we have $f:S\twoheadrightarrow R$ with $S,R\in \mathrm{CAlg}_{\heartsuit,\Rat/}$ and with $\ker f\subset S$ generated by a regular sequence. Then, for every $n\geq 1$, the natural map
\[
\mathcal{I}^{[n]}(f) \to S/(\ker f)^n
\]
is an equivalence.
\end{lemma}
\begin{proof}
This follows from Lemma~\ref{lem:adic_filtration} and the fact that we have equivalences
\[
\mathbb{L}_{R/S}[-1] \simeq (\ker f)/(\ker f)^2
\]
by the regularity of $\ker f$.
\end{proof}

\begin{proposition}
\label{prop:inf_env_trivial_quillen}
Suppose that we have a diagram
\[
\begin{diagram}
S\\
\dOnto^{g'}&\rdOnto^{g}\\
A'&\rOnto_f&A
\end{diagram}
\]
with $g',g,f$ in $\mathrm{AniPair}_{\Rat/}$. Suppose that one of the following conditions holds:
\begin{enumerate}
  \item $f$ belongs to $\mathrm{AniInfPair}_{\Rat/}$.
  \item $\pi_0(f):\pi_0(A')\to \pi_0(A)$ is an isomorphism
\end{enumerate}
Then the natural map of (derived) inverse limits
\[
\varprojlim_n\mathcal{I}^{[n]}(g') \to \varprojlim_n\mathcal{I}^{[n]}(g)
\]
is an equivalence.
\end{proposition}
\begin{proof}
Assume first that $f$ belongs to $\mathrm{AniInfPair}^{[m]}_{\Rat/}$. Then, for every $n\geq 1$, the composition
\[
\Fil^{nm}_{(S\twoheadrightarrow A)}S\to S \twoheadrightarrow \mathcal{I}^{[n]}(g')
\]
is equivalent to $0$.

Therefore, we see that
\[
K_n = \mathrm{hcoker}(\Fil^n_{(S\twoheadrightarrow A')}S\to \Fil^n_{(S\twoheadrightarrow A)}S)\simeq \mathrm{hker}(\mathcal{I}^{[n]}(g')\to \mathcal{I}^{[n]}(g))
\]
forms an inverse system $\{K_n\}_n$ that is essentially zero; that is, for every $n\geq 1$, the map
\[
K_{mn}\to K_n
\]
is equivalent to $0$, showing that we have a cofiber sequence
\[
0\simeq\varprojlim_n K_n \to \varprojlim_n\mathcal{I}^{[n]}(g')\to \varprojlim_n\mathcal{I}^{[n]}(g).
\]
This shows the proposition under assumption (1).

Now assume that $\pi_0(f)$ is an isomorphism. For $m\geq 1$, write $g'_m$ for the composition $S\to A'\to \mathcal{I}^{[m]}(f)$; then by the first case considered above, the natural map
\[
\varprojlim_n\mathcal{I}^{[n]}(g'_m) \to \varprojlim_n\mathcal{I}^{[n]}(g)
\]
is an equivalence. This implies therefore that the map from the double limit
\[
\varprojlim_m\varprojlim_n\mathcal{I}^{[n]}(g'_m)\to \varprojlim_n\mathcal{I}^{[n]}(g)
\]
is an equivalence.

To finish, we must show that the natural map
\[
\varprojlim_n\mathcal{I}^{[n]}(g') \to \varprojlim_n\varprojlim_m\mathcal{I}^{[n]}(g'_m)\simeq \varprojlim_m\varprojlim_n\mathcal{I}^{[n]}(g'_m)
\]
is an equivalence. This would follow if we knew that, for each $n\geq 1$, the map
\[
\mathcal{I}^{[n]}(g')\to \varprojlim_m\mathcal{I}^{[n]}(g'_m)
\]
is an equivalence. Using Lemma~\ref{lem:adic_filtration}, this amounts to showing that, for every $i\geq 0$, the natural map
\[
\wedge^i\mathbb{L}_{A'/S} \to \varprojlim_m\wedge^i\mathbb{L}_{\mathcal{I}^{[m]}(f)/S}
\]
is an equivalence. 

Consideration of the cofiber sequence
\[
\mathcal{I}^{[m]}(f)\otimes_{A'}\wedge^i\mathbb{L}_{A'/S}\to \mathbb{L}_{\mathcal{I}^{[m]}(f)/S}\to \mathbb{L}_{\mathcal{I}^{[m]}(f)/A'}
\]
shows that $\wedge^i\mathbb{L}_{\mathcal{I}^{[n]}(f)/A'}$ admits a natural finite filtration whose associated graded pieces are equivalent to
\[
\left(\mathcal{I}^{[m]}(f)\otimes_{A'}\wedge^j\mathbb{L}_{A'/S}\right)\otimes_{\mathcal{I}^{[m]}(f)}\wedge^k\mathbb{L}_{\mathcal{I}^{[m]}(f)/A'}
\]
for $j+k = i$.

Now, a result of Quillen shows that the natural map
\[
A' \to \varprojlim_m \mathcal{I}^{[m]}(f)
\]
is an equivalence; see~\cite[Prop. 4.11, Corollary 4.14]{Bhatt2012-kp}. In fact, the proof there actually shows that the inverse system
\[
\{\Fil^m_{A'\twoheadrightarrow A}A'\}_m
\] 
is \emph{strict essentially zero} in the terminology of~\cite[Defn. 3.7]{Bhatt2012-kp}; more precisely it shows that for any $m\geq 1$, we have
\[
H^i(\Fil^m_{A'\twoheadrightarrow A}A') = 0
\]
for $i>-m$ (that is, $\Fil^m_{A'\twoheadrightarrow A}A'$ is $(m-1)$-connective). Therefore, by~\cite[Lemma 3.10]{Bhatt2012-kp}, we find that the map
\[
\wedge^i\mathbb{L}_{A'/S}\to \varprojlim_m \left(\mathcal{I}^{[m]}(f)\otimes_{A'}\wedge^i\mathbb{L}_{A'/S}\right)
\]
is an equivalence. 

To finish, it suffices to show that
\[
\varprojlim_m \left[\left(\mathcal{I}^{[m]}(f)\otimes_{A'}\wedge^j\mathbb{L}_{A'/S}\right)\otimes_{\mathcal{I}^{[m]}(f)}\wedge^k\mathbb{L}_{\mathcal{I}^{[m]}(f)/A'}\right] \simeq 0
\]
whenever $k\ge 1$.

But once again by~\cite[Lemma 3.10]{Bhatt2012-kp}, it is enough to see that, for every $k\geq 1$, the inverse system
\[
\{\wedge^k\mathbb{L}_{\mathcal{I}^{[m]}(f)/A'}\}_{m}
\]
is strict essentially $0$. This follows from  two facts: First, that $\mathbb{L}_{\mathcal{I}^{[m]}(f)/A'}$ is $m$-connective---see~\cite[Corollary 25.3.6.4]{Lurie2018-kh}---and second, that the derived exterior power functor preserves $m$-connective objects, which follows from~\cite[Prop. 25.2.4.2]{Lurie2018-kh} (or more formally from the argument in Lemma 25.2.4.3 of \emph{loc. cit.}).
\end{proof}

While the derived infinitesimal envelopes of a surjective map of discrete rings are not usually discrete in general, under Noetherian assumptions, they are so in the limit.
\begin{proposition}
\label{prop:inf_env_carlsson}
Suppose that $S\in \mathrm{CRing}_{\heartsuit,\Rat/}$ is Noetherian, $J\subset S$ is an ideal, and that $g:S\twoheadrightarrow S/J$ is the quotient map. Then the natural map of inverse limits
\[
\varprojlim_n\mathcal{I}^{[n]}(g)\to \varprojlim_nS/J^n
\]
is an equivalence, and the right hand side is equivalent to the classical completion of $S$ along the ideal $J$.
\end{proposition}
\begin{proof}
Suppose that $J = (a_1,\ldots,a_r)$, and consider the surjection
\[
S[X_1,\ldots,X_r]\twoheadrightarrow S
\]
satisfying $X_i\mapsto a_i$. To conserve notation, for any $n\geq 1$, set
\[
S[X] = S[X_1,\ldots,X_r]\;;\; (X)^n = (X_1,\ldots,X_r)^n.
\]

Set 
\[
A = S/J\;; \;A' = S\otimes^{\mathbb{L}}_{S[X]}S[X]/(X).
\]
Then we have a map $f:A'\twoheadrightarrow A$ such that $\pi_0(f)$ is isomorphic to the identity map of $A$, and $g$ lifts to a map $g':S\twoheadrightarrow A'$ satisfying $f\circ g' = g$.

By Proposition~\ref{prop:inf_env_trivial_quillen}, there is an equivalence
\[
\varprojlim_n\mathcal{I}^{[n]}(g') \xrightarrow{\simeq}\varprojlim_n\mathcal{I}^{[n]}(g)
\]

By~\cite[Lemma 3.11]{Bhatt2012-kp}, the natural map
\[
\varprojlim_n S\otimes^{\mathbb{L}}_{S[X]}S[X]/(X)^n \to \varprojlim_n S/J^n
\]
is an equivalence, and the right hand side is actually the classical completion of $S$ along $J$.

Therefore, to finish the proof it is sufficient to know that the natural map
\[
\mathcal{I}^{[n]}(g') \to S\otimes^{\mathbb{L}}_{S[X]}S[X]/(X)^n 
\]
is an equivalence for every $n\geq 1$. This can be verified using Lemma~\ref{lem:adic_filtration}, which shows that both sides admit compatible, finite filtrations where the $i$-th associated graded piece is equivalent to 
\[
(\wedge^i\mathbb{L}_{A'/S})[-i]\simeq S\otimes^{\mathbb{L}}_{S[X]}(\wedge^i\mathbb{L}_{S/S[X]})[-i].
\]
\end{proof}

\subsection{}\label{subsec:infcrystals}
In what follows we will use the language of fibrations from Appendix~\ref{sec:fibrations}, and also the $\infty$-category theoretic treatment of filtered rings and modules from Appendix~\ref{sec:filtrations}. We now consider the two functors
\[
s:\mathrm{AniInfPair}^{[n]}\xrightarrow{(A'\twoheadrightarrow A)\mapsto A'}\mathrm{CRing}\;;\; t:\mathrm{AniInfPair}^{[n]}\xrightarrow{(A'\twoheadrightarrow A)\mapsto A}\mathrm{CRing},
\]
and for $R\in \mathrm{CRing}_{\Rat/}$, we set
\[
\mathrm{AniInfPair}^{[n]}_{R/} = \mathrm{AniInfPair}^{[n]}_{\Rat/}\times_{t,\mathrm{CRing}_{\Rat/}}\mathrm{CRing}_{R/}.
\]

Note that we also have a functor
\[
\tilde{s}:\mathrm{AniInfPair}^{[n]}_{R/}\xrightarrow{(A'\twoheadrightarrow A\leftarrow R)\mapsto \Fil^\bullet_{(A'\twoheadrightarrow A)}A'}\mathrm{FilCRing}.
\]

 So we can define two coCartesian fibrations
\begin{align*}
\mathrm{AniInfPair}^{[n]}_{R/}\times_{s,\mathrm{CRing}}\mathrm{Mod}&\to \mathrm{AniInfPair}^{[n]}_{R/}\\
\mathrm{AniInfPair}^{[n]}_{R/}\times_{\tilde{s},\mathrm{FilCRing}}\mathrm{FilMod}&\to \mathrm{AniInfPair}^{[n]}_{R/}.
\end{align*}

An \defnword{infinitesimal crystal} (resp \defnword{transversally filtered infinitesimal crystal}) over $R$ \defnword{of level $n$} is a coCartesian section $\mathcal{M}$ of the first fibration (resp. $\Fil^\bullet \mathcal{M}$ of the second), where in the second case, we require that 
\[
\mathcal{M} = \underset{n}{\colim}\Fil^n\mathcal{M}
\]
is a coCartesian section of the first fibration.

In this way, we obtain the categories $\mathrm{InfCrys}^{[n]}_{R/\Rat}$ and $\mathrm{TFilInfCrys}^{[n]}_{R/\Rat}$ of infinitesimal crystals and transversally filtered infinitesimal crystals over $R$ of level $n$, respectively.

An object $\Fil^\bullet \mathcal{M}$ in $\mathrm{TFilInfCrys}^{[n]}_{R/\Rat}$ is a \defnword{transversally filtered infinitesimal crystal of vector bundles (resp. perfect complexes)} if the filtration $\Fil^\bullet M_R$ is bounded and if the objects
\[
\mathcal{M}(R\xrightarrow{\mathrm{id}}R)\;;\; \gr^i\mathcal{M}(R\xrightarrow{\mathrm{id}}R)\in \Mod{R}
\]
are finite locally free (resp. perfect). Write $\mathrm{TFilInfCrys}^{[n],\mathrm{lf}}_{R/\Rat}$ and $\mathrm{TFilInfCrys}^{[n],\mathrm{perf}}_{R/\Rat}$ for the corresponding $\infty$-category.


Repeating all the definitions above without restricting to a particular level $n$ gives us the $\infty$-categories $\mathrm{InfCrys}_{R/\Rat}$, $\mathrm{TFilInfCrys}_{R/\Rat}$, $\mathrm{TFilInfCrys}^{\mathrm{lf}}_{R/\Rat}$ and $\mathrm{TFilInfCrys}^{\mathrm{perf}}_{R/\Rat}$. Any object in one of these categories induces objects of level $n$ for every $n\geq 1$, and we will denote this by adorning the original object with a superscript $[n]$.

We have the \defnword{structure sheaf} $\mathcal{O}$ in $\mathrm{TFilInfCrys}^{\mathrm{lf}}_{R/\Rat}$, which associates with $A'\twoheadrightarrow A$ the object $\Fil^\bullet_{(A'\twoheadrightarrow A)}A'$.

\begin{remark}
\label{rem:inf_pullback}
Any arrow $f:R\to S$ induces functors
\[
\mathrm{AniInfPair}_{S/} \to \mathrm{AniInfPair}_{R/}
\]
and hence base-change functors
\begin{align*}
f^*:\mathrm{InfCrys}_{R/\Rat}\to \mathrm{InfCrys}_{S/\Rat}\;;\; \mathrm{TFilInfCrys}_{R/\Rat}\to \mathrm{TFilInfCrys}_{S/\Rat},
\end{align*}
as well as for their level $[n]$ versions.
\end{remark}

\begin{remark}
\label{rem:inf_crys_over_stacks}
For future reference, we can use the pullback functoriality in the previous remark to define the $\infty$-categories $\mathrm{InfCrys}_{X/\Rat}$ and $\mathrm{TFilInfCrys}_{X/\Rat}$ for any prestack over $\Rat$: these will simply be the limits over the corresponding $\infty$-categories of infinitesimal crystals and transversally filtered infinitesimal crystals over $\mathrm{Aff}^{\op}_{/X}$, the $\infty$-category of pairs $(C,x)$ with $C\in \mathrm{CRing}_{\Rat/}$ and $x\in X(C)$. Of course, when $X$ is a quasi-compact derived Deligne-Mumford stack then it suffices to take this limit over the cosimplicial object in $\mathrm{Aff}^{\op}_{/X}$ determined by an \'etale cover $\Spec C\to X$,
\end{remark}

\subsection{}\label{subsec:infcoh}
Transversally filtered infinitesimal crystals are obtained from the infinitesimal cohomology of derived prestacks. Instead of doing this site-theoretically, we will give a construction following the lines of that of derived crystalline cohomology in~\cite[\S 4]{Mao2021-jt} in terms of Illusie's completed de Rham cohomology; see also~\cite[\S 4]{Bhatt2012-kp}. 

To begin, given $A'\twoheadrightarrow A$ in $\mathrm{AniPair}$ and $B\in \mathrm{CRing}_{A/}$, we define an object
\begin{align}\label{eqn:hfilt_dinf}
\Fil^\bullet_{\mathrm{Hdg}}\mathrm{InfCoh}(S/(A'\twoheadrightarrow A))\in \mathrm{FilMod}_{\Fil^\bullet_{(A'\twoheadrightarrow A)}A'},
\end{align}
the \defnword{Hodge-filtered (derived) infinitesimal cohomology of $S$ with respect to $A'\twoheadrightarrow A$}.

This is obtained as follows: We take the $\infty$-category $\Fun([1],\mathrm{AniPair})$ of arrows in $\mathrm{AniInfPair}$, which admits a set of compact projective generators 
\[
(\Int[X,Y]\twoheadrightarrow \Int[Y])\to (\Int[X,Y,X',Y']\twoheadrightarrow \Int[Y,Y'])
\]
for finite sets of variables $X,Y,X',Y'$ (see for instance the notation in the proof of Proposition~\ref{prop:inf_env_carlsson}). For each generator, we have the relative de Rham complex
\[
\Omega^\bullet_{\Int[X,Y,X',Y']/\Int[X,Y]}
\]
equipped with the filtration given by
\[
\Fil^i_{\mathrm{Hdg}}\left(\Int[X,Y]\otimes_{\Int[X,Y]}\Omega^n_{\Int[X,Y,X',Y']/\Int[X,Y]}\right) = (X)^{i-n}\Omega^n_{\Int[X,Y,X',Y']/\Int[X,Y]}
\]

Animating this assignment gives a functor
\begin{align*}
\Fun([1],\mathrm{AniPair})&\to \mathrm{FilMod}\\
\left((A'\twoheadrightarrow A)\to (S'\twoheadrightarrow S)\right)&\mapsto (\Fil^\bullet_{(A'\twoheadrightarrow A)}A',\Fil^\bullet_{\mathrm{Hdg}}\mathrm{dR}_{(S'\twoheadrightarrow S)/(A'\twoheadrightarrow A)})
\end{align*}
Completing the resulting filtered object gives us the \defnword{Hodge completed relative (derived) de Rham cohomology}
\[
\Fil^\bullet_{\mathrm{Hdg}}\widehat{\mathrm{dR}}_{(S'\twoheadrightarrow S)/(A'\twoheadrightarrow A)}.
\]
If $f$ is the arrow $A'\twoheadrightarrow A$, then this is a filtered module over the inverse limit $\varprojlim_n\mathcal{I}^{[n]}(f)$. In particular, if $A'\twoheadrightarrow A$ is in $\mathrm{AniInfPair}$, then it is a filtered module over $\Fil^\bullet_{(A'\twoheadrightarrow A)}A'$.

The Poincar\'e lemma for formal power series rings in characteristic $0$ now shows (see the proof of~\cite[Prop. 4.16]{Mao2021-jt} where the argument is presented for the divided power analogue):
\begin{lemma}
\label{lem:completed_de_rham}
Suppose that $f$ belongs to $\Fun([1],\mathrm{AniPair}_{\Rat/})$; then the natural map
\[
\Fil^\bullet_{\mathrm{Hdg}}\widehat{\mathrm{dR}}_{(S'\twoheadrightarrow S)/(A'\twoheadrightarrow A)}\to \Fil^\bullet_{\mathrm{Hdg}}\widehat{\mathrm{dR}}_{(S\xrightarrow{\mathrm{id}} S)/(A'\twoheadrightarrow A)}
\]
is a filtered equivalence.
\end{lemma}

Therefore, we define~\eqref{eqn:hfilt_dinf} to by
\[
\Fil^\bullet_{\mathrm{Hdg}}\mathrm{InfCoh}(S/(A'\twoheadrightarrow A)) = \Fil^\bullet_{\mathrm{Hdg}}\widehat{\mathrm{dR}}_{(B\xrightarrow{\mathrm{id}} B)/(A'\twoheadrightarrow A)}.
\]

This satisfies:
\begin{itemize}
  \item When $A$ and $A'$ are discrete, and $S$ is a smooth $A$-algebra lifting to a smooth $A'$-algebra $S'$, then this filtered object is equivalent to the one underlying the classical de Rham complex $\Omega^\bullet_{S'/A'}$ equipped with the tensor product of the Hodge filtration with the adic filtration:
  \[
   \Fil^i_{(A'\twoheadrightarrow A)}\Omega^n_{S'/A'} = \Fil^{i-n}_{(A'\twoheadrightarrow A)}A'\otimes_{A'}\Omega^n_{S'/A'}.
  \]

  \item For any arrow $(A'\twoheadrightarrow A)\to (B'\twoheadrightarrow B)$ in $\mathrm{AniInfPair}$, there is a canonical equivalence (see the argument for~\cite[Cor. 4.20]{Mao2021-jt})
  \[
   \Fil^\bullet_{\mathrm{Hdg}}\mathrm{InfCoh}_{S/(A'\twoheadrightarrow A)}\otimes_{\Fil^\bullet_{(A'\twoheadrightarrow A)}A}\Fil^\bullet_{(B'\twoheadrightarrow B)}\xrightarrow{\simeq}\Fil^\bullet_{\mathrm{Hdg}}\mathrm{InfCoh}_{(B\otimes_AS)/(B'\twoheadrightarrow B)}.
  \]
\end{itemize}

The second point says that, if $R\in \mathrm{CRing}_{\Rat/}$, then we can attach to $S\in \mathrm{CRing}_{R/}$, an object
\[
\Fil^\bullet_{\mathrm{Hdg}}\mathrm{InfCoh}_{S/R}\in \mathrm{TFilInfCrys}_{R/\Rat}
\]
that associates with each $(A'\twoheadrightarrow A\leftarrow R)$ in $\mathrm{AniInfPair}_{R/}$ the Hodge completed relative de Rham cohomology of $A\otimes_RS$ over $A'\twoheadrightarrow A$.

This gives us a functor
\[
\Fil^\bullet_{\mathrm{Hdg}}\mathrm{InfCoh}_{/R}:\mathrm{CRing}_{R/} = \mathrm{Aff}_R^{\mathrm{op}}\to \mathrm{TFilInfCrys}_{R/\Rat},
\]
which we can extend to a functor on prestacks via left Kan extension to obtain the \defnword{relative (derived) infinitesimal cohomology} functor
\[
\Fil^\bullet_{\mathrm{Hdg}}\mathrm{InfCoh}_{/R}:\mathrm{PStk}^{\mathrm{op}}_R\to  \mathrm{TFilInfCrys}_{R/\Rat}.
\]
Concretely, for $X\in \mathrm{PStk}^{\mathrm{op}}_R$ we have
\[
\Fil^\bullet_{\mathrm{Hdg}}\mathrm{InfCoh}_{X/R} = \varprojlim_{\Spec S \to X}\Fil^\bullet_{\mathrm{Hdg}}\mathrm{CrysCoh}_{S/R}.
\]
When $X$ is a derived Deligne-Mumford stack over $R$ with an affine \'etale cover $U = \Spec S \to X$, with associated cosimplicial object $S^{(\bullet)}\in \mathrm{Fun}(\bm{\Delta},\mathrm{CRing}_{R/})$ where
\[
\Spec S^{(n)} = \underbrace{U\times_X\cdots\times_XU}_n,
\]
then we have
\[
\Fil^\bullet_{\mathrm{Hdg}}\mathrm{InfCoh}_{X/R} = \mathrm{Tot}\left(\Fil^\bullet_{\mathrm{Hdg}}\mathrm{InfCoh}_{S^{(\bullet)}/R}\right)
\]

\subsection{}\label{subsec:inf_crys_envelopes}
We can describe infinitesimal crystals in terms of (derived) infinitesimal envelopes.  $f:S\twoheadrightarrow R$ in $\mathrm{AniPair}$, giving us in turn a cosimplicial object $f^{(\bullet)}$ in $\mathrm{CRing}_{/R}$; taking its animated infinitesimal envelope of level $n$ gives us a cosimplicial object $(\mathcal{I}^{[n]}(f^{(\bullet)})\twoheadrightarrow R)$ in $\mathrm{AniInfPair}^{[n]}_{R/}$.

Let $\mathrm{Mod}^{\Delta}_n(f)$ be the $\infty$-category of coCartesian lifts
\[
M^{(\bullet)}:\bm{\Delta}\to \mathrm{Mod}
\]
of the cosimplicial animated commutative ring $\mathcal{I}^{[n]}(f^{(\bullet)})$. 

There is a filtered variant of this involving coCartesian lifts 
\[
\Fil^\bullet M^{(\bullet)}:\bm{\Delta}\to \mathrm{FilMod}
\]
of the cosimplicial filtered animated commmutative ring 
\[
\Fil^\bullet_{(\mathcal{I}^{[n]}(f^{(\bullet)})\twoheadrightarrow R)}\mathcal{I}^{[n]}(f^{(\bullet)}).
\]
There organize into an $\infty$-category $\mathrm{FilMod}_n^\Delta(f)$.

We omit the proofs of the next three results, which are obtained by combining Proposition~\ref{prop:cocartesian_sections} with Lemma~\ref{lem:inf_env_discrete} and Proposition~\ref{prop:inf_env_carlsson}. Some details will be given when we consider their crystalline analogues in Proposition~\ref{prop:crystals_explicit} and Corollary~\ref{cor:crystals_explicit_classical}.

\begin{proposition}
\label{prop:crystals_explicit_inf}
Suppose that $S$ is formally smooth over $\Rat$. Then there exist natural equivalences
\begin{align*}
\mathrm{InfCrys}^{[n]}_{R/\Rat}&\xrightarrow{\mathcal{M}\mapsto \mathcal{M}(f_n^{(\bullet)})} \mathrm{Mod}_n^\Delta(f); \\
\mathrm{TFilInfCrys}^{[n]}_{R/\Rat}&\xrightarrow{\Fil^\bullet\mathcal{M}\mapsto \Fil^\bullet\mathcal{M}(f_n^{(\bullet)})} \mathrm{FilMod}_n^\Delta(f);
\end{align*}
\end{proposition}

\begin{corollary}
\label{cor:crystals_explicit_classical_inf}
Suppose that $S\twoheadrightarrow R$ is as in Lemma~\ref{lem:inf_env_discrete} with $S$ formally smooth over $\Rat$. Suppose that we have $g:T\twoheadrightarrow R$ in $\mathrm{AniPair}$ with $T$ a polynomial algebra over $\Rat$ (possibly in infinitely many variables). Then giving an object in $\mathrm{TFilInfCrys}^{[n],\mathrm{lf}}_{R/\Rat}$ is equivalent to giving a tuple $(\bm{M}_n,\bm{\nabla}_n,\Fil^\bullet\bm{M})$, where:
\begin{enumerate}
  \item $\bm{M}_n$ is a vector bundle over $T/(\ker g)^n$;
  \item $\bm{\nabla}_n:\bm{M}_n \to \bm{M}_n\otimes_T\Omega^1_{T/\Rat}$ is an integrable connection;
  \item $\Fil^\bullet\bm{M}_n$ is a filtration by $T/(\ker g)^n$-submodules whose image in $M = R\otimes_{T/(\ker g)^n}\bm{M}_n$ is a filtration by vector sub-bundles, and which satisfies Griffiths transversality with respect to $\bm{\nabla}_n$:
  \[
   \bm{\nabla}_n(\Fil^i\bm{M}_n)\subset \Fil^{i-1}\bm{M}_n\otimes_T\Omega^1_{T/\Rat}.
  \]
\end{enumerate}
\end{corollary}

\begin{corollary}
\label{cor:crystals_explicit_classical_inf_limit}
Suppose that $S\twoheadrightarrow R = S/J$ is as in Proposition~\ref{prop:inf_env_carlsson} with $S$ a polynomial algebra in finitely many variables over $\Rat$. Then giving an object in $\mathrm{TFilInfCrys}^{\mathrm{lf}}_{R/\Rat}$ is equivalent to giving a tuple $(\widehat{\bm{M}},\widehat{\bm{\nabla}},\Fil^\bullet\widehat{\bm{M}}_n)$, where:
\begin{enumerate}
 \item $\widehat{\bm{M}}$ is a finite projective module over $\widehat{S}_J = \varprojlim_nS/J^n$, the classical completion of $S$ along $J$.
  \item $\widehat{\bm{\nabla}}:\widehat{\bm{M}} \to \widehat{\bm{M}}\otimes_S\Omega^1_{S/\Rat}$ is an integrable connection;
  \item $\Fil^\bullet\widehat{\bm{M}}$ is a filtration by $\widehat{S}_J$-submodules whose image in $R\otimes_{\widehat{S}_I}\widehat{M}$ is a filtration by vector sub-bundles, and which satisfies Griffiths transversality with respect to $\widehat{\bm{\nabla}}$:
  \[
   \widehat{\bm{\nabla}}(\Fil^i\widehat{\bm{M}})\subset \Fil^{i-1}\widehat{\bm{M}}\otimes_S\Omega^1_{S/\Rat}.
  \]
\end{enumerate}
\end{corollary}

\begin{remark}
\label{rem:finite_extensions}
Suppose that $k=R$ is a finite extension of $\Rat$. Then we can take $S = k$ in Proposition~\ref{prop:crystals_explicit_inf} and $S\twoheadrightarrow k$ to be the identity. Now, the cosimplicial object $\mathcal{I}^{[n]}(f^{(\bullet)})$ is the \emph{constant} cosimplicial $\Rat$-algebra attached to $k$. So we find that $\mathrm{InfCrys}^{[n]}_{k/\Rat}$ is independent of $n$ and is equivalent to the $\infty$-category $\Mod{k}$.
\end{remark}


\subsection{}
Suppose that we have $\Fil^\bullet_{\mathrm{Hdg}}\mathcal{M}\in \mathrm{TFilInfCrys}^{\mathrm{perf}}_{R/\Rat}$ with $\gr^i_{\mathrm{Hdg}}\mathcal{M}\simeq 0$ for $i<-1$. Let $\mathcal{Z}(\Fil^\bullet_{\mathrm{Hdg}}\mathcal{M})$ be the $\Mod[\mathrm{cn}]{\Rat}$-valued prestack over $R$ such that, for each $(R\xrightarrow{f}C)\in \mathrm{CRing}_{R/}$, we have
\begin{align*}
 \mathcal{Z}(\Fil^\bullet_{\mathrm{Hdg}}\mathcal{M})(C) &= \varprojlim_n\Map_{\mathrm{TFilInfCrys}^{[n]}_{C/\Rat}}(\Fil^\bullet\mathcal{O}^{[n]},f^*\Fil^\bullet_{\mathrm{Hdg}}\mathcal{M}^{[n]})
\end{align*}


Set
\[
\Fil^\bullet_{\mathrm{Hdg}}M_R = \Fil^\bullet_{\mathrm{Hdg}} \mathcal{M}(R\xrightarrow{\mathrm{id}}R).
\]

\begin{proposition}
\label{prop:FIC_functorial_props}
\begin{enumerate}
  \item $\mathcal{Z}(\Fil^\bullet_{\mathrm{Hdg}}\mathcal{M})$ is an infinitesimally cohesive and nilcomplete  \'etale sheaf.
  \item $\mathcal{Z}(\Fil^\bullet_{\mathrm{Hdg}}\mathcal{M})$ admits a cotangent complex equivalent to $\Reg{\mathcal{Z}(\Fil^\bullet_{\mathrm{Hdg}}\mathcal{M})}\otimes_R(\gr^{-1}_{\mathrm{Hdg}}M_R)^\vee[1]$.
\end{enumerate}
\end{proposition}

\begin{proof}
Let $\mathcal{M}^{\nabla}$ be the $\Mod[\mathrm{cn}]{\Rat}$-valued prestack over $R$, given, for any $C\in \mathrm{CRing}_{R/}$, by
\[
\mathcal{M}^{\nabla}(C) =\varprojlim_n \Map_{\mathrm{InfCrys}^{[n]}_{C/\Rat}}(\mathcal{O}^{[n]},f^*\mathcal{M}^{[n]}).
\]

\begin{lemma}
\label{lem:M'toM}
There exists a fiber sequence 
\[
\mathcal{Z}(\Fil^\bullet_{\mathrm{Hdg}}\mathcal{M})\to \mathcal{M}^{\nabla}\to \mathbf{V}((\gr^{-1}_{\mathrm{Hdg}}M_R)^\vee)
\]
\end{lemma}
\begin{proof}
For any $A'\twoheadrightarrow A\leftarrow R$ in $\mathrm{AniInfPair}_{R/\Rat}$, Lemma~\ref{lem:graded_weight_filtration} shows that we have
\[
\gr^{-1}_{\mathrm{Hdg}}\mathcal{M}((A'\twoheadrightarrow A\leftarrow R))\simeq \gr^{-1}_{\mathrm{Hdg}}\mathcal{M}((A\xrightarrow{\mathrm{id}} A\leftarrow R))\simeq A\otimes_R \gr^{-1}_{\mathrm{Hdg}}M_R
\]

In particular, we obtain canonical arrows:
\[
\mathcal{M}^{\nabla}(C) \to \mathcal{M}((C\xrightarrow{\mathrm{id}}C\leftarrow R))\to \gr^{-1}_{\mathrm{Hdg}}\mathcal{M}((A'\twoheadrightarrow A\leftarrow R))\xrightarrow{\simeq}C\otimes_R\gr^{-1}_{\mathrm{Hdg}}M_R,
\]
whose composition gives us a map
\[
\mathcal{M}^{\nabla}(C)\to C\otimes_R\gr^{-1}_{\mathrm{Hdg}}M_R.
\]

We now claim that the natural map $\mathcal{Z}(\Fil^\bullet_{\mathrm{Hdg}}\mathcal{M})(C)\to \mathcal{M}^{\nabla}(C)$ identifies the source with the connective fiber of the above map. To check this, choose $g:S\twoheadrightarrow C$ with $S$ formally smooth (say a polynomial algebra) over $\Rat$, and note that, by Proposition~\ref{prop:crystals_explicit_inf}, we have 
\begin{align}\label{eqn:M'_expl_form}
\mathcal{M}^{\nabla}(C) &\simeq \tau^{\leq 0}\varprojlim_n\mathrm{Tot}\left(\mathcal{M}((\mathcal{I}^{[n]}(g^{(\bullet)})\twoheadrightarrow C))\right)
\end{align}
\begin{align*}
\mathcal{Z}(\Fil^\bullet_{\mathrm{Hdg}}\mathcal{M})(C) &\simeq  \tau^{\leq 0}\varprojlim_n\mathrm{Tot}\left(\Fil^0_{\mathrm{Hdg}}\mathcal{M}((\mathcal{I}^{[n]}(g^{(\bullet)})\twoheadrightarrow C))\right),
\end{align*}
and that, by the first paragaraph, we have cofiber sequences
\[
\Fil^0_{\mathrm{Hdg}}\mathcal{M}((\mathcal{I}^{[n]}(g^{(\bullet)})\twoheadrightarrow C))\to \mathcal{M}((\mathcal{I}^{[n]}(g^{(\bullet)})\twoheadrightarrow C))\to C\otimes_R\gr^{-1}_{\mathrm{Hdg}}M_R,
\]
where we view the last object here as a constant cosimplicial module. Totalizing and then taking the limit over $n$ now proves the lemma.
\end{proof}

\begin{lemma}
\label{lem:M'_cohesive}
$\mathcal{M}^{\nabla}$ is an infinitesimally cohesive and nilcomplete \'etale sheaf. In fact, it is formally \'etale over $R$ (equivalently, it admits $0$ as a cotangent complex over $R$).
\end{lemma}
\begin{proof}
Putting Proposition~\ref{prop:inf_env_trivial_quillen} and~\eqref{eqn:M'_expl_form} together shows that in fact we have
\[
\mathcal{M}^{\nabla}(A) \xrightarrow{\simeq}\mathcal{M}^{\nabla}(\pi_0(A))
\]
for any $A\in \mathrm{CRing}_{R/}$, and also that
\[
\mathcal{M}^{\nabla}(B)\xrightarrow{\simeq}\mathcal{M}^{\nabla}(A)
\]
whenever we have $B\twoheadrightarrow A$ in $\mathrm{AniInfPair}_{(R\xrightarrow{\mathrm{id}} R)/}$.

So we see that $\mathcal{M}^{\nabla}$ is formally \'etale and hence trivially nilcomplete and infinitesimally cohesive. 

To see that $\mathcal{M}^{\nabla}$ is also an \'etale sheaf, observe that, given an  \'etale map $f:C\to C'$, we can choose $g:S\twoheadrightarrow C$ such that we have an \'etale map $S\to S'$ lifting $f$ with $C' \simeq S'\otimes_SC$. Moreover, if $g':S'\twoheadrightarrow C'$ is the associated map in $\mathrm{AniPair}$, we have
\[
\mathcal{I}^{[n]}(g^{',(\bullet)})\simeq S^{',(\bullet)}\otimes_{S^{(\bullet)}}\mathcal{I}^{[n]}(g^{(\bullet)}).
\]
\end{proof}

Lemma~\ref{lem:M'toM} above shows that $\mathcal{Z}(\Fil^\bullet_{\mathrm{Hdg}}\mathcal{M})$ is the pullback of the zero section under a map of prestacks $\mathcal{M}^{\nabla}\to \mathbf{V}\left((\gr^{-1}_{\mathrm{Hdg}}M_R)^\vee\right)$ over $R$. Combined with Lemma~\ref{lem:M'_cohesive}, this proves the proposition.
\end{proof}

Our source of applications of the previous theorem will be the following lemma, which is immediate from Lemma~\ref{lem:fil_int_hom}. We will be applying it with $\mathcal{N}_1$ and $\mathcal{N}_2$ direct summands of the infinitesimal cohomology of abelian schemes.
\begin{lemma}
\label{lem:filinfcrys_internal_hom}
Suppose that we have $\Fil^\bullet\mathcal{N}_1,\Fil^\bullet\mathcal{N}_2\in \mathrm{TFilInfCrys}_{R/\Rat}$, and suppose that there exists $a\in \Int$ such that $\gr^k\mathcal{N}_1 \simeq \gr^k\mathcal{N}_2 \simeq 0$ for $k<a$. Suppose in addition that $\mathcal{N}_1$ satisfies the following additional conditions:
\begin{enumerate}
   \item $\Fil^\bullet\mathcal{N}_1$ is in $\mathrm{TFilCrys}^{\mathrm{lf}}_{R/\Rat}$;
   \item For $i\neq a,a+1$, we have $\gr^i_{\mathrm{Hdg}}N_{1,R} \simeq 0$
 \end{enumerate}
Then there exists an `internal Hom' object 
\[
\Fil^\bullet\mathcal{H} = \Fil^\bullet\mathcal{H}(\Fil^\bullet\mathcal{N}_1,\Fil^\bullet\mathcal{N}_2) \in \mathrm{TFilFCrys}_{R/\Rat}
\]
such that $\gr^k\mathcal{H}\simeq 0$ for $k<-1$, and such that, for all $(R\xrightarrow{f}C)\in \mathrm{CRing}_{R/}$, we have a canonical equivalence
\[
\Map_{\mathrm{TFilFCrys}_{C/\Rat}}(\Fil^\bullet\mathcal{O},f^*\Fil^\bullet\mathcal{H}) \simeq \Map_{\mathrm{TFilFCrys}_{C/\Rat}}(f^*\Fil^\bullet\mathcal{N}_1,f^*\Fil^\bullet\mathcal{N}_2).
\]
\end{lemma}

\begin{proposition}
\label{prop:FIC_discrete} 
Suppose that in Proposition~\ref{prop:FIC_functorial_props} $\Fil^\bullet_{\mathrm{Hdg}}\mathcal{M}$ is a filtered crystal of \emph{vector bundles}, that $\pi_0(R)$ is a finitely generated $\Rat$-algebra, and that $C\in \mathrm{CRing}_{R/}$ is such that $\pi_0(C)$ is also a finitely generated $\Rat$-algebra. Then:
\begin{enumerate}
\item\label{fic:discrete} $\mathcal{Z}(\Fil^\bullet_{\mathrm{Hdg}}\mathcal{M})(C)$ is discrete.

  \item\label{fic:separated}For $y\in \mathcal{Z}(\Fil^\bullet_{\mathrm{Hdg}}\mathcal{M})(C)$, the map
  \[
   \Spec C\times_{y,\mathcal{Z}(\Fil^\bullet_{\mathrm{Hdg}}\mathcal{M}),0}\Spec R\to \Spec C
  \]
  is a closed immersion that induces an open immersion on the underlying classical truncations.
\end{enumerate}
\end{proposition}
\begin{proof}
Fix $g:S\twoheadrightarrow C$ as above with $S$ a polynomial algebra in finitely many variables over $\Rat$. Let $g^{(\bullet)}:S^{(\bullet)}\twoheadrightarrow C$ be the associated cosimplicial object, and set $J^{(\bullet)} = \ker g^{(\bullet)}$. From~\eqref{eqn:M'_expl_form} and Proposition~\ref{prop:inf_env_carlsson}, we see that we have
\[
\mathcal{M}^{\nabla}(C) \simeq \tau^{\leq 0}\varprojlim_n \mathrm{Tot}\left(\mathcal{M}(S^{(\bullet)}/(J^{(\bullet)})^n)\right),
\]
which, as the connective limit of $0$-truncated objects, is itself $0$-truncated, and is thus discrete. Now, Lemma~\ref{lem:M'toM} implies that $\mathcal{Z}(\Fil^\bullet_{\mathrm{Hdg}}\mathcal{M})(C)$ is also discrete. 

Proposition~\ref{prop:FIC_functorial_props} implies that the fiber product in (2) is infinitesimally cohesive, nilcomplete and admits a cotangent complex over $C$. Therefore, by `easy' Lurie representability, Theorem~\ref{thm:lurie_representability}, is is enough to show that the underlying classical truncation is represented by a closed immersion, which would of course follow if we knew that it is an open and closed immersion.

Unraveling the definitions, this last point comes down to the following lemma:
\begin{lemma}
\label{lem:constancy_M}
For any map $f:C\to C'$in $\mathrm{CRing}_{\heartsuit,R/}$ with $C'$ also finitely presented over $\pi_0(R)$, the map $\mathcal{Z}(\Fil^\bullet_{\mathrm{Hdg}}\mathcal{M})(C)\to \mathcal{Z}(\Fil^\bullet_{\mathrm{Hdg}}\mathcal{M})(C')$ is injective.
\end{lemma}
\begin{proof}
By the Nullstellensatz, it is enough to prove the lemma when $C' = k$ is a finite extension of $\Rat$. 

First, suppose that $C$ is an integral domain. Let $\mathfrak{m}\subset C$ be the maximal ideal such that $C\to C'$ factors through $C/\mathfrak{m}$. Then the map $\mathcal{Z}(\Fil^\bullet_{\mathrm{Hdg}}\mathcal{M})(C/\mathfrak{m})\to \mathcal{Z}(\Fil^\bullet_{\mathrm{Hdg}}\mathcal{M})(k)$ is injective by Remark~\ref{rem:finite_extensions}. So we can further assume that $k=C/\mathfrak{m}$. 

Now, choose as usual $g:S\twoheadrightarrow C$ with $S$ a polynomial algebra over $\Rat$; then $\mathcal{Z}(\Fil^\bullet_{\mathrm{Hdg}}\mathcal{M})(C)$ maps injectively into
\[
\varprojlim_n\mathcal{M}(\mathcal{I}^{[n]}(g)\twoheadrightarrow C) \simeq \varprojlim_n\mathcal{M}(S/(\ker g)^n\twoheadrightarrow C),
\]
where we have used Lemma~\ref{prop:inf_env_carlsson}. 

Similarly, $\mathcal{Z}(\Fil^\bullet_{\mathrm{Hdg}}\mathcal{M})(C/\mathfrak{m})$ maps injectively into
\[
\varprojlim_n\mathcal{M}(S/(\ker \bar{g})^n\twoheadrightarrow C/\mathfrak{m}),
\]
where $\bar{g}:S\twoheadrightarrow C/\mathfrak{m}$ is the surjection induced from $g$.

Set $J = \ker g$ and $\mathfrak{n} = \ker\overline{g}$. We are now reduced to showing that the map $\widehat{S}_J \to \widehat{S}_{\mathfrak{n}}$ is injective; or equivalently, that $\widehat{S}_J$ is separated in the $\mathfrak{n}$-adic topology. By the Artin-Rees lemma~\cite[\href{https://stacks.math.columbia.edu/tag/00IN}{Tag 00IN}]{stacks-project} and Nakayama's lemma\cite[\href{https://stacks.math.columbia.edu/tag/00DV}{Tag 00DV}]{stacks-project}, we find that there exists $x\in \mathfrak{n}\widehat{S}_J$ such that
\[
(1+x)\cdot \bigcap_{k=1}^\infty \mathfrak{n}^k\widehat{S}_J = 0.
\] 
But $\widehat{S}_J$ is an integral domain, since $C = S/J$ is one, so we are now done under the assumption that $C$ is a domain.

For the general situation, let $P_1,\ldots,P_r$ be the minimal primes in $C$; choosing maximal ideals $\mathfrak{m}_i\subset C$ above each $P_i$, we get maps
\[
\mathcal{Z}(\Fil^\bullet_{\mathrm{Hdg}}\mathcal{M})(C)\to \prod_{i=1}^r\mathcal{Z}(\Fil^\bullet_{\mathrm{Hdg}}\mathcal{M})(C/P_i)\to \prod_{i=1}^r\mathcal{Z}(\Fil^\bullet_{\mathrm{Hdg}}\mathcal{M})(C/\mathfrak{m}_i),
\]
where the second map is injective by what we have just seen. In fact, the first map is also injective: If in the above argument, we choose $g:S\twoheadrightarrow C$ and take $J_i\subset S$ to be the pre-image of $P_i$, then this comes down to the fact that
\[
\widehat{S}_J \to \prod_{i=1}^r\widehat{S}_{J_i}
\]
is injective.

To finish it is now enough to see that the kernel of $\mathcal{Z}(\Fil^\bullet_{\mathrm{Hdg}}\mathcal{M})(C)\to \mathcal{Z}(\Fil^\bullet_{\mathrm{Hdg}}\mathcal{M})(C/\mathfrak{m}_i)$ is independent of $i$. For this, we will use the connectedness of $\Spec C$, which guarantees that for any $i\neq j$ we can find a map $C\to D$ to a finitely generated domain over $\Rat$, and maximal ideals $\mathfrak{n}_i$ and $\mathfrak{n}_j$ in $D$ inducing the maximal ideals $\mathfrak{m}_i,\mathfrak{m}_j$ of $C$. By the known case of integral domains (and by Remark~\ref{rem:finite_extensions}), we then see that the kernels of the maps associated with both $i$ and $j$ are both equal to that of $\mathcal{Z}(\Fil^\bullet_{\mathrm{Hdg}}\mathcal{M})(C)\to \mathcal{Z}(\Fil^\bullet_{\mathrm{Hdg}}\mathcal{M})(D)$.

\end{proof}

\end{proof}

\begin{remark}
\label{rem:derham_stack}
Take the example where $R = \Rat$: then giving $\Fil^\bullet\mathcal{M}$ amounts to simply specifying a filtered vector space $\Fil^\bullet \bm{M}$ over $\Rat$, and the prestack $\mathcal{M}^{\nabla}$ attaches to every $C\in \mathrm{CRing}_{\Rat/}$ the space $\pi_0(C)_{\mathrm{red}}\otimes_{\Rat}\bm{M}$. One cannot therefore expect this to be representable by a scheme. However, if one works over $\Comp$ instead, we could refine this into a representable object by realizing $\Fil^\bullet\bm{M}$ as an object underlying a $\Rat$-Hodge structure, and defining a prestack that accounts for this additional structure.
\end{remark}

\subsection{}\label{subsec:W0_condition_inf}
We will have use for the following condition on $\Fil^\bullet_{\mathrm{Hdg}}\mathcal{M}$ in Proposition~\ref{prop:FIC_functorial_props}: There is a direct sum decomposition 
\[
\Fil^\bullet_{\mathrm{Hdg}} \mathcal{M}\simeq \bigoplus_{i=0}^d\Fil^\bullet_{\mathrm{Hdg}}\mathcal{M}^{(i)},
\]
where, for each $i$, $\Fil^\bullet_{\mathrm{Hdg}}\mathcal{M}^{(i)}[i]$ is an object in $\mathrm{TFilInfCrys}^{\mathrm{lf}}_{R/\Rat}$. The main example of such an object for us will be obtained by combining the infinitesimal cohomology of abelian schemes with Lemma~\ref{lem:filinfcrys_internal_hom}; see Section~\ref{sec:abvar_crystals}.

Associated with this is a direct sum decomposition of prestacks $\mathcal{Z}(\Fil^\bullet_{\mathrm{Hdg}}\mathcal{M}) \simeq \bigoplus_{i=0}^d\mathcal{Z}(\Fil^\bullet_{\mathrm{Hdg}}\mathcal{M}^{(i)})$.

\begin{proposition}
\label{prop:W0_discrete_FIC}
 The map $\mathcal{Z}(\Fil^\bullet_{\mathrm{Hdg}}\mathcal{M}^{(0)})\to \mathcal{Z}(\Fil^\bullet_{\mathrm{Hdg}}\mathcal{M})$ is a closed immersion of prestacks that is an equivalence on the underlying classical prestacks. More precisely, for every $C\in \mathrm{CRing}_{R/}$, and every map $\Spec C\to \mathcal{Z}(\Fil^\bullet_{\mathrm{Hdg}}\mathcal{M})$ of prestacks over $R$, the map
 \[
\Spec C\times_{\mathcal{Z}(\Fil^\bullet_{\mathrm{Hdg}}\mathcal{M})}\mathcal{Z}(\Fil^\bullet_{\mathrm{Hdg}}\mathcal{M}^{(0)})\to \Spec C
 \]
 is represented by a derived closed subscheme, and is an isomorphism of the underlying classical schemes.
\end{proposition}
\begin{proof}
That the map is an equivalence on classical points amounts to saying that
\[
\mathcal{Z}(\Fil^\bullet_{\mathrm{Hdg}}\mathcal{M}^{(i)})(C) \simeq 0
\]
for $i>0$ and $C\in \mathrm{CRing}_{\heartsuit,R/}$. To see this, let $\mathcal{Z}(\Fil^\bullet_{\mathrm{Hdg}}\mathcal{M}^{(i)})[i]$ be the prestack associated with $\Fil^\bullet_{\mathrm{Hdg}}\mathcal{M}^{(i)}$; then we have
\[
\mathcal{Z}(\Fil^\bullet_{\mathrm{Hdg}}\mathcal{M}^{(i)})(C) \simeq \tau^{\leq 0}(\mathcal{Z}(\Fil^\bullet_{\mathrm{Hdg}}\mathcal{M}^{(i)})[i](C)[-i])\simeq 0
\]
since $\mathcal{Z}(\Fil^\bullet_{\mathrm{Hdg}}\mathcal{M}^{(i)})[i](C)$ is discrete by Proposition~\ref{prop:FIC_discrete}.

That the map is a closed immersion now follows from Theorem~\ref{thm:lurie_representability_easy}: For any  $\Spec C\to \mathcal{Z}(\Fil^\bullet_{\mathrm{Hdg}}\mathcal{M})$ with $C\in \mathrm{CRing}_{R/}$, the base-change
\[
\Spec C\times_{\mathcal{Z}(\Fil^\bullet_{\mathrm{Hdg}}\mathcal{M})}\mathcal{Z}(\Fil^\bullet_{\mathrm{Hdg}}\mathcal{M}^{(0)})\to \Spec C
\]
is an isomorphism on classical points (and hence has representable underlying classical prestack), is infinitesimally cohesive and nilcomplete, and admits a cotangent complex.
\end{proof}

\section{Spaces of sections of derived filtered crystals: the $p$-adic context}\label{sec:fcrys}

The goal here is to formulate a version of Theorem~\ref{introthm:filcrys}. Proofs will appear in the next section. We will be using the setup from~\cite{Mao2021-jt} and repeating some of the constructions from the previous section in this context. The main complication here is Frobenius and how best to keep track of it. I do so in a somewhat \emph{ad hoc} fashion. The advantage over the infinitesimal context is that---in contrast to the situation in Remark~\ref{rem:derham_stack}---Frobenius now imposes enough constraints that we can obtain a formal derived \emph{scheme} out of the spaces of sections of filtered $F$-crystals

A remark on notation: towards the latter part of this section, for any $N\in \Mod{\Int}$, we will often write $\overline{N}$ for the mod-$p$ fiber $N/{}^{\mathbb{L}}p$.


\subsection{}\label{subsec:anipdpair}
We will begin with a variant of~\eqref{subsec:anipair} that involves divided powers. Let $\mathrm{PDPair}$ be the category of tuples $(R\twoheadrightarrow S,\gamma)$, where $R\twoheadrightarrow S$ is a surjection of classical commutative rings whose kernel is equipped with a divided power structure $\gamma$; the morphisms in this category respect divided powers in the obvious sense.

Within $\mathrm{PDPair}$ we have the subcategory $\mathcal{E}^0$ obtained by taking the divided power envelopes\footnote{See for instance~\cite[Theorem 3.19]{berthelot_ogus:notes}.} of the arrows in $\mathcal{D}^0$: this is once again a category with a set of compact projective generators, and animating it gives us the $\infty$-category $\mathrm{AniPDPair}$, whose objects can be represented as pairs $(A'\twoheadrightarrow A,\gamma)$, where $f:A'\twoheadrightarrow A$ is in $\mathrm{AniPair}$ and $\gamma$ is viewed as a derived divided power structure on the homotopy kernel of $f$.

\subsection{}\label{subsec:pd_filtration}
For every $(R\to S,\gamma)$ in $\mathrm{PDPair}$, we obtain an object in $(R,\Fil^\bullet_{\mathrm{PD}}R)$ in $\mathrm{FilMod}$, where
\[
\Fil^i_{\mathrm{PD}}R = \mathrm{hker}(R\to S)^{[i]}\subset R
\]
is the ideal generated by the $i$-th divided powers of the elements of $\mathrm{hker}(R\to S)$. In fact, this makes $R$ a filtered commutative ring, which in somewhat more exalted terms gives us an $\mathcal{Z}_K(E_{\Int})_{\infty}\Int$-algebra in the symmetric monoidal $\infty$-category $\mathrm{FilMod}_{\Int}$. In the terminology of Section~\ref{sec:filtrations}, it gives a filtered animated commutative ring.

Animating the restriction of this functor to $\mathcal{E}^0$ gives us a functor
\[
\mathrm{AniPDPair}\xrightarrow{(A'\twoheadrightarrow A,\gamma)\mapsto \Fil^\bullet_{(A'\twoheadrightarrow A,\gamma)}A'}\mathrm{FilCRing}.
\]

\subsection{}\label{subsec:ani_pd_envelope}
The forgetful functor
\[
\mathrm{AniPDPair}\xrightarrow{(A'\twoheadrightarrow A,\gamma)\mapsto (A'\twoheadrightarrow A)}\mathrm{AniPair}
\]
has a left adjoint 
\[
\mathrm{AniPair}\xrightarrow{(f:A'\twoheadrightarrow A)\mapsto (D_{\mathrm{abs}}(f)\twoheadrightarrow A,\gamma_{\mathrm{abs}}(f))}\mathrm{AniPDPair},
\]
which we call the \defnword{absolute (derived) divided power (or PD) envelope}. This is obtained by animating the composition $\mathcal{D}^0\to \mathcal{E}^0\to \mathrm{AniPDPair}$; see~\cite[Def. 3.15]{Mao2021-jt}.

For any $(A'\twoheadrightarrow A,\gamma)\in \mathrm{AniPDPair}$, we also have a \defnword{relative (derived) PD envelope functor}
\begin{align*}
\mathrm{AniPair}_{(A'\twoheadrightarrow A)}&\to \mathrm{AniPDPair}_{(A'\twoheadrightarrow A,\gamma)/}\\
\left(\begin{diagram}A'&\rTo&B'\\\dTo^f&&\dOnto_g\\A&\rTo&B\end{diagram}\right)&\mapsto (D_{\mathrm{rel}(f)}(g)\twoheadrightarrow B,\gamma_{\mathrm{rel}(f)}(g))
\end{align*}
described in~\cite[Def. 4.49]{Mao2021-jt}. The underlying object $D_{\mathrm{rel}(f)}(g)$ is given explicitly by
\[
D_{\mathrm{rel}(f)}(g) \simeq D_{\mathrm{abs}}(g)\otimes_{D_{\mathrm{abs}}(f)}A'.
\]

The next result is immediate from~\cite[Prop. 3.72]{Mao2021-jt}.
\begin{lemma}
\label{lem:divided_power_classical}
Suppose that we have $f:S\twoheadrightarrow R$ in $\mathrm{AniPair}$ with $S$ and $R$ discrete, flat over $\Int$, and with $\ker f\subset S$ locally generated by a regular sequence. Then the divided power envelope $D_{\mathrm{abs}}(f)$ along with its divided power filtration, is also flat over $\Int$ and agrees with the classical divided power envelope.
\end{lemma}
\qed

\begin{lemma}
\label{lem:relative_to_p}
Consider $(\Int_{(p)}\overset{f}{\twoheadrightarrow}\Field_p,\delta_p)$, where $\delta_p$ is the standard divided power structure on $p\Int_{(p)}$; then for any $g:B'\twoheadrightarrow B$ with $B'\in \mathrm{CRing}_{\Int_{(p)}/}$ with $\overline{g}:B'\twoheadrightarrow B/{}^{\mathbb{L}}p$ the induced map, there is a natural equivalence
\[
D_{\mathrm{abs}}(g)\xrightarrow{\simeq}D_{\mathrm{rel}(f)}(\overline{g}).
\]
\end{lemma} 
\begin{proof}
Arguing via animation, we can reduce to considering $g$ of the form $\Int_{(p)}[X,Y]\overset{Y\mapsto 0}{\twoheadrightarrow}\Int_{(p)}[X]$ for sets of variables $X$ and $Y$. In this case, by Lemma~\ref{lem:divided_power_classical}, we have
\[
D_{\mathrm{abs}}(g) \simeq \Int_{(p)}[X]\langle Y\rangle\;;\;D_{\mathrm{abs}}(\overline{g})\simeq \Int_{(p)}[X]\langle Y,Z\rangle/(Z-p),
\]
where $\langle W\rangle$ denotes the free divided power algebra generated by the set of variables $W$. 

Since $D_{\mathrm{abs}}(f)\simeq \Int_{(p)}\langle Z\rangle/(Z-p)$ mapping to $\Int_{(p)}$ via $Z^{[i]} \mapsto p^{[i]}$, we now conclude that 
\begin{align*}
D_{\mathrm{rel}(f)}(\overline{g})&\simeq D_{\mathrm{abs}}(\overline{g})\otimes_{D_{\mathrm{abs}}(f)}\Int_{(p)}\\
&\simeq  \Int_{(p)}[X]\langle Y,Z\rangle/(Z-p)\otimes_{\Int_{(p)}\langle Z\rangle/(Z-p)}\Int_{(p)}\\
&\simeq \Int_{(p)}[X]\langle Y\rangle \simeq D_{\mathrm{abs}}(g).
\end{align*}
\end{proof}

\subsection{}\label{subsec:square_zero_pd}
Square-zero thickenings are a useful source of objects in $\mathrm{AniPDPair}$. Suppose that we have such a thickening $\tilde{C}\twoheadrightarrow C$ in $\mathrm{AniPair}$; the adic filtration $\Fil_{\tilde{C}\twoheadrightarrow C}\tilde{C}$ is given explicitly by:
\begin{align}
\label{eqn:fil_trivial_sqzero}
\Fil^i_{\tilde{C}\twoheadrightarrow C}\tilde{C} & \simeq \begin{cases}
\tilde{C}&\text{if $i\leq 0$}\\
\hker(\tilde{C}\twoheadrightarrow C)&\text{if $i = 1$}\\
0&\text{otherwise}.
\end{cases}
\end{align}

\begin{lemma}
\label{lem:square_zero_pd}
There is a canonical lift of $\tilde{C}\to C$ to an object 
\[
(\tilde{C}\to C,\gamma_{\mathrm{triv}})\in \mathrm{AniPDPair}
\]
such that the associated PD filtration on $\tilde{C}$ agrees with the adic filtration defined above.
\end{lemma}
\begin{proof}
This is a consequence of the fact that, for any $g:S\twoheadrightarrow R$ in $\mathrm{AniPair}$, the natural map
\[
S/\Fil^2_{S\twoheadrightarrow R}S \to D_{\mathrm{abs}}(g)/\Fil^2_{(D_{\mathrm{abs}}(g)\twoheadrightarrow R,\gamma_{\mathrm{abs}}(g))}D_{\mathrm{abs}}(g)
\]
is an equivalence in $\mathrm{CRing}$.
\end{proof}

\subsection{}\label{subsec:crys}
For $R\in \mathrm{CRing}$, we set
\begin{align*}
\mathrm{AniPDPair}_{R/} = \mathrm{AniPDPair}\times_{t,\mathrm{CRing}}\mathrm{CRing}_{R/}.
\end{align*}
Informally, this is the category of pairs $(A'\twoheadrightarrow A\leftarrow R,\gamma)$, where  $(A'\twoheadrightarrow A,\gamma)$ is an object in $\mathrm{AniPDPair}$. 

For any prime $p$, let $\mathrm{AniPDPair}_{p,R/}$ be the $\infty$-subcategory of $\mathrm{AniPDPair}_{R/}$ spanned by the objects $(A'\twoheadrightarrow A\leftarrow R,\gamma)$ where $p$ is nilpotent in $\pi_0(A')$. For a fixed $n\geq 1$, we let $\mathrm{AniPDPair}_{p,n,R/}$ be the $\infty$-subcategory spanned by those objects where $p^n = 0\in \pi_0(A')$.

We now consider the coCartesian fibration
\begin{equation}\label{eqn:cocart crys fib}
\mathrm{AniPDPair}_{p,n,R/}\times_{s,\mathrm{CRing}}\mathrm{Mod} \to \mathrm{AniPDPair}_{p,n,R/}
\end{equation}

A \defnword{$\Int/p^n\Int$-crystal} over $R$ is a coCartesian section $\mathcal{M}$ of this fibration. In particular, given such a crystal $\mathcal{M}$, we obtain, for every object $(A'\twoheadrightarrow A\leftarrow R,\gamma)$, $\mathcal{M}((A'\twoheadrightarrow A,\gamma))\in \Mod{A'}$, and if we have an arrow
\[
(A'\twoheadrightarrow A \leftarrow R,\gamma)\to (B'\twoheadrightarrow B\leftarrow R,\delta),
\]
then there is an equivalence
\[
B'\otimes_{A'}\mathcal{M}((A'\twoheadrightarrow A,\gamma)) \simeq \mathcal{M}((B'\twoheadrightarrow B,\delta))
\]
in $\Mod{B'}$. Write $\mathrm{Crys}_{R/(\Int/p^n\Int)}$ for the $\infty$-category of $\Int/p^n\Int$-crystals over $R$. We have the $\infty$-subcategories
\[
\mathrm{Crys}^{\mathrm{lf}}_{R/(\Int/p^n\Int)}\;;\;\mathrm{Crys}^{\mathrm{perf}}_{R/(\Int/p^n\Int)}
\]
spanned by those sections that take values in $\Mod[\mathrm{lf}]{}$ and $\Mod[\mathrm{perf}]{}$, respectively.

In particular, we have the \defnword{structure sheaf} crystal $\mathcal{O}_n\in \mathrm{Crys}^{\mathrm{lf}}_{R/(\Int/p^n\Int)}$ corresponding to the coCartesian section
\[
(A'\twoheadrightarrow A\leftarrow R,\gamma)\mapsto (A',A')\in \Mod.
\]

Repeating the constructions by using $\mathrm{AniPDPair}_{p,R/}$ instead (which we can view as letting $n$ go to infinity), we obtain the $\infty$-category $\mathrm{Crys}_{R/\Int_p}$ (resp. $\mathrm{Crys}^{\mathrm{lf}}_{R/\Int_p}$, $\mathrm{Crys}^{\mathrm{perf}}_{R/\Int_p}$) of $p$-adic crystals (resp. crystals of vector bundles, crystals of perfect complexes) over $R$, as well as the structure sheaf $\mathcal{O}\in \mathrm{Crys}^{\mathrm{lf}}_{R/\Int_p}$.


\subsection{}\label{subsec:filtered_pd}
From~\eqref{subsec:pd_filtration}, we find that there is a functor
\[
\mathrm{AniPDPair}_{p,n} \xrightarrow{(A'\twoheadrightarrow A,\gamma)\mapsto \Fil^\bullet_{(A'\twoheadrightarrow A,\gamma)}A'}\mathrm{FilCRing}
\]
endowing every object on the left with its divided power or PD filtration.

Pulling the coCartesian fibration $\mathrm{FilMod}\to \mathrm{FilCRing}$ back along the functor gives us a coCartesian fibration
\[
\mathrm{FilMod}_{p,n,\mathrm{PD}}\to \mathrm{AniPDPair}_{p,n},
\]
inducing for each $R\in \mathrm{CRing}_{{\Field_p}/}$ a coCartesian fibration
\begin{align}
\label{eqn:filmod_pd_fibration}
\mathrm{FilMod}_{p,n,\mathrm{PD}}\times_{\mathrm{AniPDPair}_{p,n}}\mathrm{AniPDPair}_{p,n,R/}\to \mathrm{AniPDPair}_{p,n,R/}.
\end{align}


\subsection{}\label{subsec:transversal_pd}
We write $\mathrm{TFilCrys}_{R/(\Int/p^n\Int)}$ for the $\infty$-category of \emph{coCartesian} sections of~\eqref{eqn:filmod_pd_fibration}. Informally, an object in this category is a filtered $\Int/p^n\Int$-crystal $\Fil^\bullet \mathcal{M}$ over $R$ such that:
\begin{itemize}
  \item For each $(A'\twoheadrightarrow A\leftarrow R,\gamma)$ in $\mathrm{AniPDPair}_{p,n,R/}$, $\Fil^\bullet \mathcal{M}((A'\twoheadrightarrow A,\gamma))$ is a filtered module over $\Fil^\bullet_{(A'\twoheadrightarrow A,\gamma)}A'$.
  \item Given an arrow
  \[
(A'\twoheadrightarrow A \leftarrow R,\gamma)\to (B'\twoheadrightarrow B\leftarrow R,\delta),
\]
the natural arrow
\[
\Fil^\bullet_{(B'\twoheadrightarrow B,\delta)}B'\otimes_{\Fil^\bullet_{(A'\twoheadrightarrow A,\gamma)}A'}\Fil^\bullet \mathcal{M}((A'\twoheadrightarrow A,\gamma))\xrightarrow{\simeq}\Fil^\bullet \mathcal{M}((B'\twoheadrightarrow B,\delta))
\]
is an equivalence of filtered modules over $\Fil^\bullet_{(B'\twoheadrightarrow B,\delta)}B$
\end{itemize}

Write $\mathrm{TFilCrys}_{R/(\Int/p^n\Int)}^{\mathrm{lf}}$ (resp. $\mathrm{TFilCrys}^{\mathrm{perf}}_{R/(\Int/p^n\Int)}$) for the subcategory of $\mathrm{TFilCrys}_{R/(\Int/p^n\Int)}$ spanned by objects $\Fil^\bullet \mathcal{M}$ such that
\[
\mathcal{M}(R\xrightarrow{\mathrm{id}}R)\;;\;\gr^i\mathcal{M}(R\xrightarrow{\mathrm{id}}R)\in \Mod{R}
\]
are all finite locally free (resp. perfect) and are nullhomotopic for all but finitely many $i$.

Just as in~\eqref{subsec:crys}, we also obtain $\infty$-categories
\[
\mathrm{TFilCrys}_{R/\Int_p}\;;\;\mathrm{TFilCrys}^{\mathrm{lf}}_{R/\Int_p}\;;\; \mathrm{TFilCrys}^{\mathrm{perf}}_{R/\Int_p}
\]
that we can view as obtained by letting $n$ go to infinity. We will call the objects in any of these categories \defnword{transversally filtered} $p$-\defnword{adic crystals} over $R$.

Note that the structure sheaf $\mathcal{O}$ lifts canonically to an object $\Fil^\bullet\mathcal{O}$ in $\mathrm{TFilCrys}^{\mathrm{lf}}_{R/\Int_p}$ with $\gr^i\mathcal{O}\simeq 0$ for $i<0$.

\begin{remark}
\label{rem:filcrys_over_stacks}
Just as Remark~\ref{rem:inf_crys_over_stacks}, there is a natural pullback functoriality 
\[
f^*:\mathrm{FilCrys}_{R/\Int_p}\to \mathrm{FilCrys}_{C/\Int_p}\;;\; \mathrm{TFilCrys}_{R/\Int_p}\to \mathrm{TFilCrys}_{R/\Int_p}
\]
and this lets us speak about the $\infty$-category of $p$-adic crystals $\mathrm{FilCrys}_{X/\Int_p}$ as well as that of transversal $p$-adic crystals $\mathrm{TFilCrys}_{X/\Int_p}$ for any prestack $X$.
\end{remark}

\subsection{}
\label{subsec:ffc_tensor_category}
The $\infty$-categories $\mathrm{Crys}_{R/\Int_p}$ and $\mathrm{TFilCrys}_{R/\Int_p}$ are symmetric monoidal stable $\infty$-categories with the additional structure being induced by that on $\mathrm{Mod}$ and $\mathrm{FilMod}$, respectively. We will denote the corresponding tensor products by $\otimes_{\mathcal{O}}\;;\; \otimes_{\Fil^\bullet\mathcal{O}}$.

If $\mathcal{N}\in \mathrm{Crys}^{\mathrm{perf}}_{R/\Int_p}$, then $\mathcal{N}$ is dualizable with dual $\mathcal{N}^\vee$ (determined by taking the dual pointwise over $\mathrm{AniPDPair}_{p,R/}$), and for any other $\mathcal{M}\in \mathrm{Crys}_{R/\Int_p}$, the tensor product $\mathcal{N}^\vee\otimes_{\mathcal{O}} \mathcal{M}$ is an \defnword{internal Hom} object in the sense that we have
\[
\Map_{\mathrm{Crys}_{R/\Int_p}}\left(\mathcal{P},\mathcal{N}^\vee\otimes_{\mathcal{O}}\mathcal{M}\right)\simeq \Map_{\mathrm{Crys}_{R/\Int_p}}\left(\mathcal{P}\otimes_{\mathcal{O}}\mathcal{N},\mathcal{M}\right)
\]
for any $\mathcal{P}\in \mathrm{Crys}_{R/\Int_p}$. 

There is of course also a compatible symmetric monoidal structure $\otimes_{\mathcal{O}_1}$ on $\mathrm{Crys}_{R/\Field_p}$ with similar assertions holding for the existence of internal Hom objects $\mathcal{N}_1^\vee\otimes_{\mathcal{O}_1}\mathcal{M}_1$ for $\mathcal{N}_1$ an $\Field_p$-crystal of perfect complexes over $R$. 

Similarly, if we have $\Fil^\bullet \mathcal{N}\in \mathrm{TFilCrys}^{\mathrm{perf}}_{R/\Int_p}$, then Lemma~\ref{lem:fil_int_hom} shows that we have a dual object $\Fil^\bullet \mathcal{N}^\vee$ such that, for all $\Fil^\bullet \mathcal{P}$ and $\Fil^\bullet \mathcal{M}$ in $\mathrm{TFilCrys}_{R/\Int_p}$, we have
\[
\Map_{\mathrm{TFilCrys}_{R/\Int_p}}\left(\Fil^\bullet\mathcal{P},\Fil^\bullet\mathcal{N}^\vee\otimes_{\Fil^\bullet\mathcal{O}}\Fil^\bullet\mathcal{M}\right)\simeq \Map_{\mathrm{TFilCrys}_{R/\Int_p}}\left(\Fil^\bullet\mathcal{P}\otimes_{\Fil^\bullet\mathcal{O}}\Fil^\bullet\mathcal{N},\Fil^\bullet\mathcal{M}\right)
\]


\subsection{}\label{subsec:pullback_functors}
Any arrow $f:R\to C$ in $\mathrm{CRing}$ induces a functor
\[
\mathrm{AniPDPair}_{p,C/}\xrightarrow{(A'\twoheadrightarrow A\xleftarrow{g}C,\gamma)\mapsto (A'\twoheadrightarrow A\xleftarrow{g\circ f}R )} \mathrm{AniPDPair}_{p,R/},
\]
and pre-composition with it gives us symmetric monoidal base-change functors
\begin{align*}{}
\mathrm{Crys}_{R/\Int_p}&\xrightarrow{\mathcal{M}\mapsto f^*\mathcal{M}} \mathrm{Crys}_{C/\Int_p}\\
\mathrm{TFilCrys}_{R/\Int_p}&\xrightarrow{\Fil^\bullet\mathcal{M}\mapsto \Fil^\bullet f^*\mathcal{M}} \mathrm{TFilCrys}_{C/\Int_p}.
\end{align*}

\subsection{}\label{subsec:cryscoh}
A natural source of transversally filtered $p$-adic crystals is derived crystalline cohomology of derived prestacks. Just as we did in~\eqref{subsec:infcoh}, we choose to present this in terms of the derived de Rham complex as in~\cite{Mao2021-jt}. 

Suppose that we have $(A'\twoheadrightarrow A,\gamma)$ in $\mathrm{AniPDPair}$ and $C\in \mathrm{CRing}_{A/}$. Then associated with this is a filtered object
\[
\Fil^\bullet_{\mathrm{Hdg}}\mathrm{CrysCoh}_{C/(A'\twoheadrightarrow A,\gamma)}\in \mathrm{FilMod}_{\Fil^\bullet_{(A'\twoheadrightarrow A,\gamma)}A'},
\]
the \defnword{Hodge filtered (derived) crystalline cohomology} of $C$ with respect to $(A'\twoheadrightarrow A,\gamma)$; see~\cite[Defn. 4.33]{Mao2021-jt}. It is constructed by animating the divided power de Rham complex, and has the following properties: 
\begin{itemize}
  \item When $A$ and $A'$ are discrete and $C$ is a smooth $A$-algebra lifting to a smooth $A'$-algebra $C'$, then this filtered object is equivalent to the one underlying the classical de Rham complex $\Omega^\bullet_{C'/A'}$ equipped with the tensor product of the Hodge filtration with the PD filtration:
  \[
   \Fil^i_{(A'\twoheadrightarrow A,\gamma)}\Omega^n_{C'/A'} = \Fil^{i-n}_{(A'\twoheadrightarrow A,\gamma)}A'\otimes_{A'}\Omega^n_{C'/A'}.
  \]

  \item If $(A'\twoheadrightarrow A)\simeq (A\xrightarrow{\mathrm{id}}A)$ then there is an equivalence
  \[  
    \Fil^\bullet_{\mathrm{Hdg}}\mathrm{CrysCoh}_{C/(A'\twoheadrightarrow A,\gamma)}\simeq \Fil^\bullet_{\mathrm{Hdg}}\mathrm{dR}_{C/A},
  \]
  where the right hand side is the \defnword{derived de Rham cohomology} of $C$ over $A$ as in~\cite{Bhatt2012-kp} (see also~\eqref{subsec:infcoh}: in the notation there, we have $\mathrm{dR}_{C/A}\simeq \mathrm{dR}_{(C\xrightarrow{\mathrm{id}}C)/(A\xrightarrow{\mathrm{id}}A)}$, and we have canonical equivalences
  \[
    \gr^i_{\mathrm{Hdg}}\mathrm{dR}_{C/A}\simeq (\wedge^i\mathbb{L}_{C/A})[-i]
  \]
  for every $i\in \Int$.

  \item Suppose that the arrows are $f:A'\twoheadrightarrow A$ and $h:A\twoheadrightarrow S$. Set $g=h\circ f$; then if $\pi_0(h)$ is surjective, we have a natural equivalence
  \[
   \Fil^\bullet_{\mathrm{Hdg}}\mathrm{CrysCoh}_{C/(A'\twoheadrightarrow A,\gamma)}\xrightarrow{\simeq}\Fil^\bullet_{(D_{\mathrm{rel}(f)}(g)\twoheadrightarrow C,\gamma_{\mathrm{rel}}(f)(g))}D_{\mathrm{rel}(f)}(g).
  \]

  \item For any arrow $(A'\twoheadrightarrow A,\gamma)\to (B'\twoheadrightarrow B,\delta)$ in $\mathrm{AniPDPair}$, there is a canonical equivalence~\cite[Corollary 4.31]{Mao2021-jt}
  \[
   \Fil^\bullet_{\mathrm{Hdg}}\mathrm{CrysCoh}_{S/(A'\twoheadrightarrow A,\gamma)}\otimes_{\Fil^\bullet_{(A'\twoheadrightarrow A,\gamma)}A'}\Fil^\bullet_{(B'\twoheadrightarrow B,\delta)}B'\xrightarrow{\simeq}\Fil^\bullet_{\mathrm{Hdg}}\mathrm{CrysCoh}_{(B\otimes_AS)/(B'\twoheadrightarrow B,\delta)}.
  \]
\end{itemize}

The last point says that, if $R\in \mathrm{CRing}$, then we can attach to $S\in \mathrm{CRing}_{R/}$, an object
\[
\Fil^\bullet_{\mathrm{Hdg}}\mathrm{CrysCoh}_{S/R}\in \mathrm{TFilCrys}_{R/\Int_p}
\]
that associates with each $(A'\twoheadrightarrow A\leftarrow R,\gamma)$ in $\mathrm{AniPDPair}_{p,R/}$ the Hodge filtered relative crystalline cohomology of $A\otimes_RS$ over $(A'\twoheadrightarrow A,\gamma)$.

This gives us a functor (depending of course on the choice of the prime $p$)
\[
\Fil^\bullet_{\mathrm{Hdg}}\mathrm{CrysCoh}_{/R}:\mathrm{CRing}_{R/} = \mathrm{Aff}_R^{\mathrm{op}}\to \mathrm{TFilCrys}_{R/\Int_p},
\]
which we can extend to a functor on prestacks via left Kan extension to obtain the \defnword{relative (derived) $p$-adic crystalline cohomology} functor
\[
\Fil^\bullet_{\mathrm{Hdg}}\mathrm{CrysCoh}_{/R}:\mathrm{PStk}^{\mathrm{op}}_R\to  \mathrm{TFilCrys}_{R/\Int_p}.
\]
Concretely, for $X\in \mathrm{PStk}^{\mathrm{op}}_R$ we have
\[
\Fil^\bullet_{\mathrm{Hdg}}\mathrm{CrysCoh}_{X/R} = \varprojlim_{\Spec S \to X}\Fil^\bullet_{\mathrm{Hdg}}\mathrm{CrysCoh}_{S/R}.
\]
When $X$ is a derived Deligne-Mumford stack over $R$ with an affine \'etale cover $U = \Spec S \to X$, with associated cosimplicial object $S^{(\bullet)}\in \mathrm{Fun}(\bm{\Delta},\mathrm{CRing}_{R/})$ with 
\[
\Spec S^{(n)} = \underbrace{U\times_X\cdots\times_XU}_n,
\]
then we have
\[
\Fil^\bullet_{\mathrm{Hdg}}\mathrm{CrysCoh}_{X/R} = \mathrm{Tot}\left(\Fil^\bullet_{\mathrm{Hdg}}\mathrm{CrysCoh}_{S^{(\bullet)}/R}\right)
\]

\subsection{}\label{subsec:crystals_explicit}
The $\infty$-categories $\mathrm{Crys}_{R/(\Int/p^n\Int)}$ and $\mathrm{TFilCrys}_{R/(\Int/p^n\Int)}$ above have a description in terms of modules over divided power envelopes.

To see this, suppose that we have an object $f:S\twoheadrightarrow R$ in $\mathrm{AniPair}$, giving us in turn a cosimplicial object $f^{(\bullet)}$ in $\mathrm{CRing}_{/R}$; taking its animated PD envelope as in~\eqref{subsec:ani_pd_envelope} and taking the derived quotient by $p^n$ gives us a cosimplicial object $(D_n(f^{(\bullet)})\twoheadrightarrow R/{}^{\mathbb{L}}p^n,\gamma_n^{(\bullet)}(f))$ in $\mathrm{AniPDPair}_{p,n,R/}$.

Let $\mathrm{Mod}^{\Delta}_n(f)$ be the $\infty$-category of coCartesian lifts
\[
M^{(\bullet)}_n:\bm{\Delta}\to \mathrm{Mod}
\]
of the cosimplicial animated commutative ring $D_n(f^{(\bullet)})$. 

There is a filtered variant of this involving coCartesian lifts 
\[
\Fil^\bullet M^{(\bullet)}_n:\bm{\Delta}\to \mathrm{FilMod}
\]
of the cosimplicial filtered animated commmutative ring 
\[
\Fil^\bullet_{(D_n(f^{(\bullet)})\twoheadrightarrow R/{}^{\mathbb{L}}p^n,\gamma_n^{(\bullet)}(f))}D_n(f^{(\bullet)}).
\]
There organize into an $\infty$-category $\mathrm{FilMod}_n^\Delta(f)$.

Given an object $\mathcal{M}_n$ in $\mathrm{Crys}_{R/(\Int/p^n\Int)}$, for each $m\in \Int_{\geq 0}$, we obtain 
\[
\mathcal{M}_n(f_n^{(m)}) \define \mathcal{M}_n((D_n(f^{(m)})\twoheadrightarrow R/{}^{\mathbb{L}}p^n,\gamma_n^{(m)}(f)))\in \Mod{D_n(f^{(m)})}.
\]
These organize into a cosimplicial object 
\[
\mathcal{M}_n(f_n^{(\bullet)}):\bm{\Delta}\to \mathrm{Mod}
\]
lifting $D_n(f^{(\bullet)})$, and the structure of a crystal on $\mathcal{M}_n$ implies that this is an object in the $\infty$-category $\mathrm{Mod}^\Delta_n(f)$.

Similarly, if $\Fil^\bullet\mathcal{M}_n$ is in $\mathrm{TFilCrys}_{R/(\Int/p^n\Int)}$, then we obtain a functorial assignment to $\Fil^\bullet\mathcal{M}_n(f_n^{(\bullet)})\in \mathrm{FilMod}^\Delta_n(f)$.

\begin{proposition}
\label{prop:crystals_explicit}
With the notation as above, suppose that one of the following holds:
\begin{enumerate}[label=(\roman*)]
	\item $S$ is a polynomial algebra over $\Int_{(p)}$ (in possibly infinitely many variables);
	\item $S$ is the $p$-completion of a smooth $\Int_{(p)}$-algebra $T$;
	\item $R$ is a perfect $\Field_p$-algebra, and $f:S=W(R)\twoheadrightarrow R$ the natural quotient map.
\end{enumerate}
Then the functors
\begin{align*}
\mathrm{Crys}_{R/(\Int/p^n\Int)}&\xrightarrow{\mathcal{M}_n\mapsto \mathcal{M}_n(f_n^{(\bullet)})} \mathrm{Mod}_n^\Delta(f); \\
\mathrm{TFilCrys}_{R/(\Int/p^n\Int)}&\xrightarrow{\Fil^\bullet\mathcal{M}_n\mapsto \Fil^\bullet\mathcal{M}_n(f_n^{(\bullet)})} \mathrm{FilMod}_n^\Delta(f);
\end{align*}
constructed above are equivalences.

\end{proposition}
\begin{proof}
The point is that, under each assumption, $(D_n(f)\twoheadrightarrow R/{}^{\mathbb{L}}p^n,\gamma_n(f))$ is a weakly initial object of $\mathrm{AniPDPair}_{p,n,R/}$ in the terminology of Definition~\ref{defn:quasi_final}. Indeed, since the animated divided power envelope functor preserves colimits, it is compatible with the formation of \u{C}ech conerves, and we so we only need to verify that, for any $(A'\twoheadrightarrow A\leftarrow R,\gamma)$ in $\mathrm{AniPDPair}_{p,n,R/}$, there is an arrow $(S\twoheadrightarrow R)\to (A'\twoheadrightarrow A)$ in $\mathrm{AniPair}$ lifting $R\to A$. 

When $S$ is a polynomial algebra over $\Int_{(p)}$, this is clear: It comes down to $A^{',I}\to A^I$ having non-empty fibers for any set $I$. 

When $S$ is a smooth $\Int_{(p)}$-algebra and $A',A$ are discrete, $A'\twoheadrightarrow A$ is a filtered colimit of divided power thickenings of finitely generated $\Int_{(p)}$-algebras, and so we reduce---by the finite presentation of $T$ over $\Int_{(p)}$---to the case where the kernel of $A'\twoheadrightarrow A$ is finitely generated and hence nilpotent, where we conclude by the formal smoothness of $S$ over $\Int_{(p)}$. For the general case, it suffices to inductively produce compatible liftings along $\tau_{\leq n}A'\to \tau_{\leq n}A$ for every $n\geq 0$. We have already done so for $n = 0$, and for the inductive step, assuming we have a lift for $n$, note that we have a factoring
\[
\tau_{\leq n+1}A'\twoheadrightarrow \tau_{\leq n}A\times_{\tau_{\leq n}A}\tau_{\leq n+1}A\twoheadrightarrow \tau_{\leq n+1}A,
\]
where the first arrow is an isomorphism after applying $\tau_{\leq n}$. In particular, it is a square-zero thickening, and so, by the formal smoothness of $S$ over $\Int_{(p)}$, we can once again produce the desired lifting to $\tau_{\leq n+1}A'$.

When we have $f:W(R)\twoheadrightarrow R$ with $R$ perfect, I claim that there is in fact a \emph{unique} arrow $W(R)\to A'$ lifting $R\to A$. When $A'$ and $A$ are discrete, this is~\cite[Lemma 1.4]{lau:displays}. The general assertion follows as in the previous case, now using the formal \'etaleness of $W(R)$ over $\Int_{(p)}$.

With the three previous paragraphs in hand, the proposition is now a special case of Proposition~\ref{prop:cocartesian_sections}.
\end{proof}

Under some conditions, we recover a classical notion:
\begin{corollary}
\label{cor:crystals_explicit_classical}
Suppose that $f:S\twoheadrightarrow R$ satisfies the hypotheses of Lemma~\ref{lem:divided_power_classical} and Proposition~\ref{prop:crystals_explicit}. Suppose also that we have $g:\widehat{T}\twoheadrightarrow R$ with $\widehat{T}$ the $p$-completion of a filtered colimit $T$ of smooth $\Int_{(p)}$-algebras (for instance, $T$ could be a polynomial algebra over $\Int_{(p)}$ for a set of variables). Let $D_n^{\mathrm{cl}}(g)\twoheadrightarrow R/{}^{\mathbb{L}}p^n \simeq R/p^nR$ be the \emph{classical} divided power envelope of $g$.\footnote{I don't know if this agrees with the derived envelope in this setting.} Then giving an object in $\mathrm{TFilCrys}^{\mathrm{lf}}_{R/(\Int/p^n\Int)}$ is equivalent to giving a tuple $(\bm{M}_n,\bm{\nabla}_n,\Fil^\bullet\bm{M}_n)$, where:
\begin{enumerate}
  \item $\bm{M}_n$ is a vector bundle over $D^{\mathrm{cl}}_n(g)$;
  \item $\bm{\nabla}_n:\bm{M}_n \to \bm{M}_n\otimes_S\Omega^1_{T/\Int_p}$ is a quasi-nilpotent integrable\footnote{Recall that this means that there are only finitely many differential operators $D$ for $T$ over $\Int_{(p)}$ such that $\bm{\nabla}_n(D)$ is a non-zero operator on $\bm{M}_n$.} connection compatible with the natural one on $D_n^{\mathrm{cl}}(g)$;
  \item $\Fil^\bullet\bm{M}_n$ is a filtration by $D_n^{\mathrm{cl}}(g)$-submodules whose image in $M = R\otimes_{D_n^{\mathrm{cl}}(g)}\bm{M}_n$ is a filtration by vector sub-bundles, and which satisfies Griffiths transversality with respect to $\bm{\nabla}_n$:
  \[
   \bm{\nabla}_n(\Fil^i\bm{M}_n)\subset \Fil^{i-1}\bm{M}_n\otimes_T\Omega^1_{T/\Int_p}.
  \]
\end{enumerate}
\end{corollary}
\begin{proof}
Let $f:S\twoheadrightarrow R$ be as in Lemma~\ref{lem:divided_power_classical}. Then that result combined with Proposition~\ref{prop:crystals_explicit} tells us that the objects of the $\infty$-category $\mathrm{Mod}^{\Delta}(f)$ are classical in nature, and are in fact objects in the derived category of modules over the discrete ring $D_n(f^{(0)})$ equipped with an HPD stratification as in~\cite[Def. 4.3H]{berthelot_ogus:notes}. Restricting to vector bundles of finite rank and using~\cite[Thm. 6.6]{berthelot_ogus:notes} shows that the objects of $\mathrm{Crys}^{\mathrm{lf}}_{R/(\Int/p^n\Int)}$ are the same as crystals of vector bundles over the \emph{classical} mod-$p^n$ crystalline site for $R$.

The corollary is now a consequence of the argument in~\cite[Theorem 3.1.2]{Ogus_Trans}. 

\end{proof}

\subsection{}\label{subsec:f-zips}
We now introduce Frobenius into the picture. To begin, for any object $(A'\twoheadrightarrow A,\gamma)$ with $A'\in \mathrm{CRing}_{\Field_p/}$, the $p$-power Frobenius endomorphism $\varphi_{A'}$ of $A'$ factors through a map $\varphi_{(A'\twoheadrightarrow A,\gamma)}:A\to A'$; see~\cite[Lemma 4.36]{Mao2021-jt}. This gives us a functor
\begin{align*}
\Mod{R}&\to \mathrm{Crys}_{R/\Field_p}\\
M&\mapsto \mathcal{O}_1\otimes_{\varphi}M
\end{align*}
where the right hand side is the coCartesian section of~\eqref{eqn:cocart crys fib} (with $n=1$) given by
\[
(A'\twoheadrightarrow A\xleftarrow{f}R) \mapsto A'\otimes_{\varphi\circ f,R}M,
\]
where we have abbreviated $\varphi_{(A'\twoheadrightarrow A,\gamma)}$ to $\varphi$. The objects obtained in this way are precisely those with \defnword{$p$-curvature zero} in the language of~\cite[\S 5]{Katz1970-qp}.

Suppose that we are given $\Fil^\bullet \mathcal{M}$ in $\mathrm{TFilCrys}_{R/\Int_p}$. Then in particular evaluating on the object $(R\xrightarrow{\mathrm{id}}R,\gamma_{\mathrm{triv}})$, we obtain a filtered object  $\Fil^\bullet M_R\in \mathrm{FilMod}_R$.

A \defnword{crystalline $F$-zip} over $R$ is a tuple $(\Fil^\bullet_{\mathrm{Hdg}}\mathcal{M},\Fil^\bullet_{\mathrm{hor}}\mathcal{M}_1,\Phi)$ where:
\begin{itemize}
  \item $\Fil^\bullet_{\mathrm{Hdg}}\mathcal{M}$ is an object in $\mathrm{TFilCrys}_{R/\Int_p}$ with underlying object $\mathcal{M}_1\in \mathrm{Crys}_{R/\Field_p}$;
  \item $\Fil^\bullet_{\mathrm{hor}}\mathcal{M}_1$ is an exhaustive filtration of $\mathcal{M}_1$ in $\mathrm{Crys}_{R/\Field_p}$\footnote{The subscript `hor' is short for `horizontal' and is supposed to harken to the fact that this filtration would correspond to one by horizontal vector bundles when the crystals in question arise from vector bundles equipped with connection.};
  \item and
  \[
    \Phi:\mathcal{O}_1\otimes_{\varphi}\gr^{-\bullet}_{\mathrm{Hdg}}M_R\xrightarrow{\simeq}\gr^\bullet_{\mathrm{hor}}\mathcal{M}_1\footnote{The indexing on the left hand side is the \emph{negative} of that on the right.}
  \]
  is an equivalence of graded objects in $\mathrm{Crys}_{R/\Field_p}$.
\end{itemize}

These objects organize into an $\infty$-category $\Fzip_{R/\Int_p}$.

The filtered structure sheaf $\Fil^\bullet\mathcal{O}$ has a canonical lift to $\Fzip_{R/\Int_p}$, with $\Fil^\bullet_{\mathrm{hor}}\mathcal{O}_1$ the trivial filtration concentrated in degree $0$ and with $\Phi$ given by the identity.

As before, for every $n\ge 1$, we can define an analogous $\infty$-category $\Fzip_{R/(\Int/p^n\Int)}$ where we start with $\Fil^\bullet_{\mathrm{Hdg}}\mathcal{M}_n$ in $\mathrm{TFilCrys}_{R/(\Int/p^n\Int)}$.

\subsection{}\label{subsec:fzips_tensor_category}
The $\infty$-category $\Fzip_{R/\Int_p}$ is a symmetric monoidal stable $\infty$-category with the structure arising from that on $\mathrm{TFilCrys}_{R/\Int_p}$ from~\eqref{subsec:ffc_tensor_category}, and for $\Fil^\bullet_{\mathrm{hor}}$ given pointwise via the filtered tensor product in $\mathrm{FilMod}_{A'}$ for every $(A'\twoheadrightarrow A,\gamma)$ in $\mathrm{AniPDPair}_{p,1,R/}$. The tensor product on the equivalence $\Phi$ is obtained by noting that the associated graded functor with respect to both $\Fil_{\mathrm{Hdg}}$ and $\Fil_{\mathrm{hor}}$ is symmetric monoidal.

Suppose in addition that $(\Fil^\bullet_{\mathrm{Hdg}}\mathcal{N},\Fil^\bullet_{\mathrm{hor}}\mathcal{N}_1,\Phi)$ in $\Fzip_{R/\Int_p}$ is such that $\Fil^\bullet_{\mathrm{Hdg}}\mathcal{N}$ belongs to $\mathrm{TFilCrys}^{\mathrm{perf}}_{R/\Int_p}$; then we obtain a dual object
\[
(\Fil^\bullet_{\mathrm{Hdg}}\mathcal{N}^\vee,\Fil^\bullet_{\mathrm{hor}}\mathcal{N}^\vee_1,(\Phi^\vee)^{-1})
\]
Here, the first object is the transversally filtered dual from~\eqref{subsec:ffc_tensor_category}, the second is the filtered dual defined---using Lemma~\ref{lem:fil_int_hom}---pointwise for $(A'\twoheadrightarrow A,\gamma)$ in $\mathrm{AniPDPair}_{p,1,R/}$ by
\[
(\Fil^\bullet_{\mathrm{hor}}\mathcal{N}^\vee_1)((A'\twoheadrightarrow A,\gamma)) = \Fil^\bullet_{\mathrm{hor}}\mathcal{N}_1((A'\twoheadrightarrow A,\gamma))^\vee\in \mathrm{FilMod}_{A'}.
\]
Note that this is sensible since the equivalence $\Phi$ shows that we are dealing with a dualizable object in $\mathrm{FilMod}_{A'}$.

Finally, $(\Phi^\vee)^{-1}$ is the inverse of the equivalence
\[
\Phi^\vee:\gr^\bullet_{\mathrm{hor}}\mathcal{M}_1^\vee\xrightarrow{\simeq}\mathcal{O}_1\otimes_{\varphi}\gr^{-\bullet}_{\mathrm{Hdg}}M_R^\vee
\]
obtained by dualizing $\Phi$.

One now sees that tensor product with this dual object gives rise to an internal Hom object much as we saw already in~\eqref{subsec:ffc_tensor_category}.

\subsection{}\label{subsec:f-zip_splitting}
Suppose that we have a crystalline $F$-zip as above, and suppose that there is an integer $a\in \Int$, $\gr^i_{\mathrm{Hdg}}\mathcal{M}\simeq 0$ for $i<a$. We will be interested in the object $\Fil^{-(a+1)}_{\mathrm{hor}}\mathcal{M}_1$ sitting in the fiber sequence
\[
(\gr^{-a}_{\mathrm{hor}}\mathcal{M}_1\simeq\Fil^{-a}_{\mathrm{hor}}\mathcal{M}_1)\to \Fil^{-(a+1)}_{\mathrm{hor}}\mathcal{M}_1 \to \gr^{-(a+1)}_{\mathrm{hor}}\mathcal{M}_1.
\]

For the additional structure that concerns us, we will need to make some choices. Let $g:S\twoheadrightarrow R$ in $\mathrm{AniPair}$ be a map satisfying the conditions of Proposition~\ref{prop:crystals_explicit}. For any $n\geq 1$, write $\Fil^\bullet D_n^{(\bullet)}$ for the cosimplicial filtered animated commutative ring
\[
\Fil^\bullet_{(D_n(g^{(\bullet)})\twoheadrightarrow R/{}^{\mathbb{L}}p^n,\gamma_n(g^{(\bullet)}))}D_n(g^{(\bullet)}).
\]
Evaluating $\Fil^\bullet_{\mathrm{Hdg}}\mathcal{M}$ on $(D_n(g^{(\bullet)})\twoheadrightarrow R/{}^{\mathbb{L}}p^n,\gamma_n(g^{(\bullet)}))$ gives us a cosimplicial object
\[
\Fil^\bullet_{\mathrm{Hdg}} M_n^{(\bullet)}\in \mathrm{FilMod}_{\Fil^\bullet D_n^{(\bullet)}}.
\]

Write $\overline{B} = B/{}^{\mathbb{L}}p$ for any $B\in \mathrm{CRing}$. By Lemma~\ref{lem:associated_graded}, we have
\[
\gr^a_{\mathrm{Hdg}}M_1^{(\bullet)}\xrightarrow{\simeq}\gr^a_{\mathrm{Hdg}}M_{\overline{R}},
\]
and we also have a fiber sequence
\begin{align*}
\gr^1D_1^{(\bullet)}\otimes_{\overline{R}}\gr^a_{\mathrm{Hdg}}M_{\overline{R}}\to \gr^{a+1}_{\mathrm{Hdg}}M_1^{(\bullet)}\to \gr^{a+1}_{\mathrm{Hdg}}M_{\overline{R}},
\end{align*}
which rotates to an arrow
\begin{align}
\label{eqn:fib_sequence_gr0_Hdg}
\gr^{a+1}_{\mathrm{Hdg}}M_{\overline{R}}\to \gr^1D_1^{(\bullet)}\otimes_{\overline{R}}\gr^a_{\mathrm{Hdg}}M_{\overline{R}}[1]
\end{align}
 
\begin{lemma}
\label{lem:frob_lift_splitting}
Fix a Frobenius lift $\varphi_S:S\to S$. This induces a Frobenius lift $\varphi_{S^{(\bullet)}}$ on $S^{(\bullet)}$, and, associated with this choice, we obtain maps 
\[
  \varphi_0:D^{(\bullet)}\to D^{(\bullet)}\;;\; \varphi_1:\Fil^1D^{(\bullet)}\to D^{(\bullet)}
\]
such that $\varphi_0\vert_{\Fil^1D^{(\bullet)}} \simeq \varphi_1\circ p$, and such that $\varphi_0$ reduces mod $p$ to the $p$-power Frobenius, and hence factors through the map
\begin{align}\label{eqn:frob_factoring}
\varphi_{(D_1^{(\bullet)}\to \overline{R},\gamma_1(g^{(\bullet)}))}:\overline{R}\to D_1^{(\bullet)}.
\end{align}
Moreover, $\varphi_1$ induces an $\overline{R}$-linear map $\gr^1D_1^{(\bullet)}\to D_1^{(\bullet)}$ which we once again denote by $\varphi_1$. Here the $\overline{R}$-module structure on the right hand side is via the map~\eqref{eqn:frob_factoring}.
\end{lemma} 
\begin{proof}
The existence of $\varphi_0$ is a consequence of the fact that we can identify $D^{(\bullet)}$ with the animated divided power envelope of the composition $S^{(\bullet)}\twoheadrightarrow R\twoheadrightarrow \overline{R}$ relative to $(\Int_{(p)}\twoheadrightarrow \Field_p,\delta_p)$; see Lemma~\ref{lem:relative_to_p}.

As for $\varphi_1$, its existence amounts to saying that the composition $\Fil^1D^{(\bullet)}\to D^{(\bullet)}\xrightarrow{\varphi_0}D^{(\bullet)}\to D_1^{(\bullet)}$ is nullhomotopic. This is of course because the Frobenius endomorphism of $D_1^{(\bullet)}$ factors through $\overline{R}$ as noted at the beginning of~\eqref{subsec:f-zips}.

The last thing to show is that the induced map $\Fil^1D_1^{(\bullet)}\to D_1^{(\bullet)}$
factors through $\gr^1D_1^{(\bullet)}$. Since we have an equivalence
\[
\gr^1_{(\overline{S}^{(\bullet)}\twoheadrightarrow \overline{R})}\overline{S}^{(\bullet)}\xrightarrow{\simeq}\gr^1D_1^{(\bullet)},
\]
it suffices to know that the restriction of $\varphi_1$ to $\Fil^2_{(S^{(\bullet)}\twoheadrightarrow R)}S^{(\bullet)}$ is divisible by $p$. For this, note that the composition
\[
S^{(\bullet)}\to D^{(\bullet)}\xrightarrow{\varphi_0}D^{(\bullet)}
\]
lifts to a map in $\mathrm{FilCRing}$ if we equip the left hand side with the adic filtration $\Fil^\bullet_{(S^{(\bullet)}\twoheadrightarrow R)}S$ and the right with the $p$-adic filtration.
\end{proof}

\subsection{}\label{subsec:partial_splitting}
Using Lemma~\ref{lem:frob_lift_splitting}, the arrow~\eqref{eqn:fib_sequence_gr0_Hdg} leads to an $\overline{R}$-linear map
\[
\gr^{a+1}_{\mathrm{Hdg}}M_{\overline{R}}\to D_1^{(\bullet)}\otimes_{\varphi,\overline{R}}\gr^a_{\mathrm{Hdg}}M_{\overline{R}}[1],
\]
which classifies an object $\widetilde{M}_1^{(\bullet)}$ sitting in a fiber sequence
\[
D_1^{(\bullet)}\otimes_{\varphi,\overline{R}}\gr^a_{\mathrm{Hdg}}M_{\overline{R}}\to \widetilde{M}_1^{(\bullet)}\to D_1^{(\bullet)}\otimes_{\varphi,\overline{R}}\gr^{a+1}_{\mathrm{Hdg}}M_{\overline{R}}.
\]
Here, we have abbreviated the arrow~\eqref{eqn:frob_factoring} to $\varphi:\overline{R}\to D_1^{(\bullet)}$.

By Proposition~\ref{prop:crystals_explicit}, this corresponds to an object $\widetilde{\mathcal{M}}_1(g,\varphi_S)$ in $\mathrm{Crys}_{R/\Field_p}$ that is an extension
\[
\mathcal{O}_1\otimes_{\varphi}\gr^a_{\mathrm{Hdg}}M_{R}\to \widetilde{\mathcal{M}}_1(g,\varphi_S)\to \mathcal{O}_1\otimes_{\varphi}\gr^{a+1}_{\mathrm{Hdg}}M_{R}
\]
Note that, by definition, this extension depends only on $\Fil^\bullet_{\mathrm{Hdg}}\mathcal{M}_1$, along with the pair $(g,\varphi_S)$. In fact, this last choice doesn't matter:

\begin{lemma}
\label{lem:tilde_M_1_ind}
Suppose that we have a map $f:R\to R'$, and $g':S'\twoheadrightarrow R'$ in $\mathrm{AniPair}$ with $S'$ discrete and formally smooth over $\Int_{(p)}$, as well as a Frobenius lift $\varphi_{S'}:S'\to S'$. Then there is a canonical equivalence
\[
f^*\left(\widetilde{\mathcal{M}}_1(g,\varphi_S)\right)\xrightarrow{\simeq}\widetilde{f^*\mathcal{M}}_1(g',\varphi_{S'})
\]
in $\mathrm{Crys}_{R'/\Field_p}$. In particular, the $\Field_p$-crystal $\widetilde{\mathcal{M}}_1 = \widetilde{\mathcal{M}}_1(g,\varphi_S)$ is independent of the choice of $g$ and $\varphi_S$.
\end{lemma}
\begin{proof}
Both the objects involved in the asserted equivalence are extensions of $\mathcal{O}_1\otimes_\varphi\gr^{a+1}_{\mathrm{Hdg}}M_{R'}$ by $\mathcal{O}_1\otimes_\varphi\gr^{a}_{\mathrm{Hdg}}M_{R'}$ in $\mathrm{Crys}_{R'/\Field_p}$. So we can proceed by trying to construct a natural homotopy between the two classifying maps
\[
\beta_1,\beta_2:\mathcal{O}_1\otimes_\varphi\gr^{a+1}_{\mathrm{Hdg}}M_{R'}\to \mathcal{O}_1\otimes_\varphi\gr^{a}_{\mathrm{Hdg}}M_{R'}[1]
\]
in $\mathrm{Crys}_{R'/\Field_p}$, the first for $f^*\left(\widetilde{\mathcal{M}}_1(g,\varphi_S)\right)$ and the second for $\widetilde{f^*\mathcal{M}}_1(g',\varphi_{S'})$.

Let
\[
\varphi'_0:D_2(g^{',(\bullet)})\to D_2(g^{',(\bullet)})
\]
and
\[
\varphi_1:\gr^1_{(\overline{S}^{(\bullet)}\twoheadrightarrow R)}\overline{S}^{(\bullet)}\to D_1(g^{(\bullet)})\;;\; \varphi'_1:\gr^1_{(\overline{S}^{',(\bullet)}\twoheadrightarrow R)}\overline{S}^{,(\bullet)}\simeq \gr^1D_1(g^{',(\bullet)})\to D_1(g^{',(\bullet)})
\]
be the maps from Lemma~\ref{lem:frob_lift_splitting} associated with $\varphi_S$ and $\varphi_{S'}$, respectively.

Choose a lift $F:S\to D_2(g)$ of the map $R\to R'\to R'/{}^{\mathbb{L}}p^2$, and let $\overline{F}:\overline{S}\to D_1(g)$ be the map induced from $F$. This lift gives us a  map of cosimplicial objects
\[
F^{(\bullet)}:S^{(\bullet)}\to D_2(g^{',(\bullet)}),
\]
which in turn induces a map
\[
\gr^1\overline{F}^{(\bullet)}:\gr^1_{(\overline{S}^{(\bullet)}\twoheadrightarrow R)}\overline{S}^{(\bullet)}\to \gr^1D_1(g^{',(\bullet)})
\]

Therefore, we get a pair of maps $F^{(\bullet)}\circ \varphi_S,\varphi'_0\circ F^{(\bullet)}:S^{(\bullet)}\to D_2(g^{(\bullet)})$ as well as another pair of maps
\begin{align*}
\overline{F}^{(\bullet)}\circ \varphi_1,\varphi'_1\circ \gr^1\overline{F}^{(\bullet)}:\gr^1_{(\overline{S}^{(\bullet)}\twoheadrightarrow R)}\overline{S}^{(\bullet)}\to D_1(g^{',(\bullet)}),
\end{align*}
which are linear over the first pair. The lemma now comes down to the fact that both these pairs of maps are compatibly homotopic as maps between cosimplicial objects.
\end{proof}

\begin{definition}
\label{defn:partial_splitting}
A \defnword{partial splitting} of the crystalline $F$-zip is now an equivalence
\[
\alpha:\widetilde{\mathcal{M}}_1\xrightarrow{\simeq}\Fil^{-(a+1)}_{\mathrm{hor}}\mathcal{M}_1
\]
of extensions of $\mathcal{O}_1\otimes_{\varphi}\gr^{a+1}_{\mathrm{Hdg}}M_{R}$ by $\mathcal{O}_1\otimes_{\varphi}\gr^a_{\mathrm{Hdg}}M_{R}$. 
\end{definition}

\begin{remark}
\label{rem:bounded_partial_splitting}
In the situation where $\gr^i_{\mathrm{Hdg}}M_R\simeq 0$ for $i\neq a,a+1$, the notion of a crystalline $F$-zip with a partial splitting becomes equivalent to that of a \defnword{Fontaine module} as in~\cite{MR2373230} or~\cite{MR4013068}: In this case, $\Fil^{-(a+1)}_{\mathrm{hor}}\mathcal{M}_1 \simeq \mathcal{M}_1$, and the datum of the partially split crystalline $F$-zip structure on $\Fil^\bullet_{\mathrm{Hdg}}\mathcal{M}$ becomes equivalent to giving the equivalence $\alpha:\widetilde{\mathcal{M}}_1\xrightarrow{\simeq}\mathcal{M}$. 
\end{remark}

\begin{lemma}
\label{lem:partial_splitting}
Write $\varphi_a$ for the composition
\[
M_1^{(\bullet)}\to \gr^a_{\mathrm{Hdg}}M_1^{(\bullet)}\simeq \gr^a_{\mathrm{Hdg}}M_{\overline{R}}\xrightarrow{1\otimes \mathrm{id}} D_1^{(\bullet)}\otimes_{\varphi,\overline{R}}\gr^a_{\mathrm{Hdg}}M_{\overline{R}}\xrightarrow{\simeq}\Fil^{-a}_{\mathrm{hor}}M_1^{(\bullet)}.
\]

Also, let $\overline{\Phi}_{a+1}$ be the composition
\[
\gr^{a+1}_{\mathrm{Hdg}}M_{\overline{R}}\xrightarrow{1\otimes \mathrm{id}} D_1^{(\bullet)}\otimes_{\varphi,\overline{R}}\gr^{a+1}_{\mathrm{Hdg}}M_{\overline{R}}\xrightarrow[\Phi]{\simeq}\gr^{-(a+1)}_{\mathrm{hor}}M_1^{(\bullet)}.
\]

Then giving a partial splitting $\alpha$ as above is equivalent to specifying a map
\[
\varphi_{a+1}:\Fil^{a+1}_{\mathrm{Hdg}}M_1^{(\bullet)}\to \Fil^{-(a+1)}_{\mathrm{hor}} M_1^{(\bullet)}
\]
with the following properties:
\begin{enumerate}
  \item $\varphi_{a+1}$ is $\varphi_{S^{(\bullet)}}$-semilinear;
  \item $\varphi_{a+1}\vert_{\Fil^{a+2}_{\mathrm{Hdg}}M_1^{(\bullet)}}\simeq 0$;
  \item $\varphi_{a+1}\vert_{\Fil^1D_1^{(\bullet)}\otimes_{D_1^{(\bullet)}}M_1^{(\bullet)}} \simeq \varphi_1\otimes\varphi_a$.
  \item The induced map
  \[
    \gr^{a+1}_{\mathrm{Hdg}}M_{\overline{R}}\to \gr^{-(a+1)}_{\mathrm{hor}}M_1^{(\bullet)}
  \]
  obtained from properties (2) and (3) is equivalent to $\overline{\Phi}_{a+1}$.
\end{enumerate}
\end{lemma}
\begin{proof}
Unraveling the definitions, one sees that giving the partial splitting $\alpha$ is equivalent to giving an $\overline{R}$-linear map
\[
\gr^{a+1}_{\mathrm{Hdg}}M_1^{(\bullet)}\to \Fil^{-(a+1)}_{\mathrm{hor}}M_1^{(\bullet)}
\]
such that it lifts the natural map
\[
\gr^{a+1}_{\mathrm{Hdg}}M_{\overline{R}}\to D_1^{(\bullet)}\otimes_{\varphi,\overline{R}}\gr^{a+1}_{\mathrm{Hdg}}M_{\overline{R}}
\]
and such that its restriction to $\gr^1D_1^{(\bullet)}\otimes_{\overline{R}}\gr^a_{\mathrm{Hdg}}M_{\overline{R}}$ is the composition
\[
\gr^1D_1^{(\bullet)}\otimes_{\overline{R}}\gr^a_{\mathrm{Hdg}}M_{\overline{R}}\xrightarrow{\varphi_1\otimes\mathrm{id}}D_1^{(\bullet)}\otimes_{\varphi,\overline{R}}\gr^a_{\mathrm{Hdg}}M_1^{(\bullet)}\to \Fil^{-(a+1)}_{\mathrm{hor}}M_1^{(\bullet)}.
\]

Considering the induced map $\Fil^{a+1}_{\mathrm{Hdg}}M_1^{(\bullet)}\to \Fil^{-(a+1)}_{\mathrm{hor}}M_1^{(\bullet)}$ now verifes the lemma.
\end{proof}

\subsection{}\label{subsec:fzips_partial_tensor}
Let $\Fzip^{\ge a}_{R/\Int_p}$ be the $\infty$-category of crystalline $F$-zips over $R$ with $\gr^{-i}_{\mathrm{Hdg}}\mathcal{M}\simeq 0$ for $i<a$, equipped with a partial splitting. For brevity we will often refer to the whole apparatus of an object here by the symbol $\mathcal{M}$. 

The structure sheaf $\Fil^\bullet\mathcal{O}$ with its canonical $F$-zip structure also admits a lift to $\Fzip^{\ge -1}_{R/\Int_p}$, obtained trivially since $\gr^{-1}\mathcal{O}_1\simeq 0$, so that $\widetilde{\mathcal{O}}_1\simeq \mathcal{O}_1\simeq \Fil^0_{\mathrm{hor}}\mathcal{O}_1$.

The symmetric monoidal structure on $\Fzip_{R/\Int_p}$ from~\eqref{subsec:fzips_tensor_category} lifts to one on the $\infty$-categories $\Fzip^{\ge a}_{R/\Int_p}$ for varying $a$, in the following sense: Given $\mathcal{M}$ in $\Fzip^{\ge a}_{R/\Int_p}$ and $\mathcal{M}'\in \Fzip^{\ge a'}_{R/\Int_p}$ the tensor product of the underlying objects in $\Fzip_{R/\Int_p}$ lifts naturally to an object $\mathcal{M}\otimes_{\mathcal{O}}\mathcal{M}'$ in $\Fzip^{\geq (a+a')}_{R/\Int_p}$. 

The additional datum of the partial splitting is provided by noting that we have
\begin{align*}
\gr^{-(a+a'+1)}_{\mathrm{hor}}(\mathcal{M}_1\otimes_{\mathcal{O}_1}\mathcal{M}'_1) &\simeq (\gr^{-(a+1)}_{\mathrm{hor}}\mathcal{M}_1\otimes_{\mathcal{O}_1}\gr^{-a'}_{\mathrm{hor}}\mathcal{M}'_1)\oplus (\gr^{-a}_{\mathrm{hor}}\mathcal{M}_1\otimes_{\mathcal{O}_1}\gr^{-(a'+1)}_{\mathrm{hor}}\mathcal{M}'_1)\\
\gr^{-(a+a')}_{\mathrm{hor}}(\mathcal{M}_1\otimes_{\mathcal{O}_1}\mathcal{M}'_1)  &\simeq \gr^{-a}_{\mathrm{hor}}\mathcal{M}_1\otimes_{\mathcal{O}_1}\gr^{-a'}_{\mathrm{hor}}\mathcal{M}'_1
\end{align*}
and the extensions
\[
\widetilde{\mathcal{M}_1\otimes_{\mathcal{O}_1}\mathcal{M}'_1}\;;\; \Fil^{-(a+a'+1)}_{\mathrm{hor}}(\mathcal{M}_1\otimes_{\mathcal{O}_1}\mathcal{M}'_1)
\]
are both constructed via the same process: Starting with extensions classified by arrows
\[
\lambda:\gr^{-(a+1)}_{\mathrm{hor}}\mathcal{M}_1\to \gr^{-a}_{\mathrm{hor}}\mathcal{M}_1[1]\;;\; \lambda:\gr^{-(a'+1)}_{\mathrm{hor}}\mathcal{M}'_1\to \gr^{-a'}_{\mathrm{hor}}\mathcal{M}'_1[1],
\]
we now obtain another extension classified by the arrow
\[
\lambda\otimes\mathrm{id} + \mathrm{id}\otimes\lambda':\gr^{-(a+a'+1)}_{\mathrm{hor}}(\mathcal{M}_1\otimes_{\mathcal{O}_1}\mathcal{M}'_1)\to \gr^{-(a+a')}_{\mathrm{hor}}(\mathcal{M}_1\otimes_{\mathcal{O}_1}\mathcal{M}'_1)[1].
\]

\subsection{}\label{subsec:fzips_partial_dual}
Dualizability here is somewhat restrictive. Let $\mathcal{N}$ be as in Remark~\ref{rem:bounded_partial_splitting} so that $\gr^iN_R\simeq 0$ if $i\neq a,a+1$, and suppose also that $\Fil^\bullet_{\mathrm{Hdg}}\mathcal{N}$ belongs to $\mathrm{TFilCrys}_{R/\Int_p}^{\mathrm{perf}}$, so that the underlying object in $\Fzip_{R/\Int_p}$ is dualizable with dual $\mathcal{N}^\vee$. Then taking the inverse dual of the equivalence $\alpha:\widetilde{\mathcal{N}}_1\xrightarrow{\simeq}\mathcal{N}_1$ gives a partial splitting on $\mathcal{N}^\vee$, endowing it with the structure of an object in $\Fzip^{\ge -(a+1)}_{R/\Int_p}$. We now find that, for any $\mathcal{M}\in \Fzip^{\ge a}_{R/\Int_p}$, there is a canonical equivalence
\begin{align*}
\Map_{\Fzip^{\ge -1}_{R/\Int_p}}(\mathcal{O},\mathcal{N}^\vee\otimes_{\mathcal{O}}\mathcal{M})\xrightarrow{\simeq}\Map_{\Fzip^{\ge a}_{R/\Int_p}}(\mathcal{N},\mathcal{M}).
\end{align*}

\subsection{}\label{subsec:ffilcrys_defn}
We now come to the objects of interest to us. Suppose that we have
\[
(\Fil^\bullet_{\mathrm{Hdg}}\mathcal{M},\Fil^\bullet_{\mathrm{hor}}\mathcal{M}_1,\Phi,\alpha)\in \Fzip^{\ge a}_{R/\Int_p}. 
\]
For any $n\geq 1$ set 
\[
\Fil^\bullet_{\mathrm{hor}}M_n^{(\bullet)} = \Fil^\bullet_{\mathrm{hor}}M_1^{(\bullet)}\times_{M_1^{(\bullet)}}M_n^{(\bullet)}\to M_n^{(\bullet)}.
\]

For any integer $a\in\Int$, we define an $\infty$-category $\mathrm{FilFCrys}^{\geq a}_{R/\Int_p}$ to be that of tuples
\[
((\Fil^\bullet_{\mathrm{Hdg}}\mathcal{M},\Fil^\bullet_{\mathrm{hor}}\mathcal{M}_1,\Phi,\alpha),\varphi_a,\varphi_{a+1})
\]
where the first sub-tuple is an object in $\Fzip^{\ge a}_{R/\Int_p}$. The remaining data $\varphi_a,\varphi_{a+1}$ classifies compatible sequences (with respect to $n$) of $\varphi_{S^{(\bullet)}}$-semilinear maps
\[
\varphi_a:M_n^{(\bullet)}\to \Fil^{-a}_{\mathrm{hor}}M_n^{(\bullet)}\;;\;\varphi_{a+1}:\Fil^{a+1}_{\mathrm{Hdg}}M^{(\bullet)}_n\to \Fil^{-(a+1)}_{\mathrm{hor}}M^{(\bullet)}_n
\]
such that $\varphi_a\vert_{\Fil^{a+1}_{\mathrm{Hdg}}M^{(\bullet)}_n}\simeq \varphi_{a+1}\circ p$
and such that, for $n=1$, the maps correspond to the partial splitting $\alpha$ as in Lemma~\ref{lem:partial_splitting}.   

Once again, for brevity we will often refer to this whole apparatus by the symbol $\mathcal{M}$. 

For any $n\geq 1$, starting with $\Fil^\bullet_{\mathrm{Hdg}}\mathcal{M}_n$ in $\mathrm{TFilCrys}_{R/(\Int/p^n\Int)}$ instead gives us an analogous $\infty$-category $\mathrm{FilFCrys}^{\ge a}_{R/(\Int/p^n\Int)}$, and we can informally view an object of $\mathrm{FilFCrys}^{\ge a}_{R/\Int_p}$ as being a compatible sequence of objects $\mathcal{M}_n\in \mathrm{FilFCrys}^{\ge a}_{R/(\Int/p^n\Int)}$.

\subsection{}\label{subsec:ffilcrys_ind}
To see that the $\infty$-category $\mathrm{FilFCrys}^{\geq a}_{R/\Int_p}$ is independent of the choice of $(g,\varphi_S)$, we will need to reframe the definition we just gave. For any $\mathcal{M}\in \Fzip^{\ge a}_{R/\Int_p}$, let
\[
\lambda_1:\gr^{-(a+1)}_{\mathrm{hor}}M_1^{(\bullet)} \simeq D_1^{(\bullet)}\otimes_{\varphi,\overline{R}}\gr^{a+1}_{\mathrm{Hdg}}M_{\overline{R}} \to D_1^{(\bullet)}\otimes_{\varphi,\overline{R}}\gr^a_{\mathrm{Hdg}}M_{\overline{R}}[1] \simeq \Fil^{-a}_{\mathrm{hor}}M^{(\bullet)}_1
\]
be the map classifying the extension $\widetilde{M}_1^{(\bullet)}$. The existence of the partial splitting $\alpha$ gives a nullhomotopy for the induced map to $M^{(\bullet)}_1[1]$. Therefore, we can consider the map
\[
\lambda_n = (\lambda_1,0):\gr^{-(a+1)}_{\mathrm{hor}}M_1^{(\bullet)} \to (D_1^{(\bullet)}\otimes_{\varphi,\overline{R}}\gr^a_{\mathrm{Hdg}}M_{\overline{R}}[1])\times_{M_1^{(\bullet)}}M_n^{(\bullet)} \simeq \Fil^{-a}_{\mathrm{hor}}M^{(\bullet)}_n,
\]
which classifies an extension
\[
\Fil^{-a}_{\mathrm{hor}}M^{(\bullet)}_n\to \widetilde{M}_n^{(\bullet)}\to \gr^{-(a+1)}_{\mathrm{hor}}M_1^{(\bullet)}.
\]

\begin{lemma}
\label{lem:filfcrys_ind}
\begin{enumerate}
  \item Giving a lift of $\mathcal{M}$ to an object in $\mathrm{FilFCrys}^{\ge a}_{R/\Int_p}$ now amounts to specifying a compatible sequence of equivalences $\{\alpha_n:\widetilde{M}^{(\bullet)}_n\xrightarrow{\simeq}\Fil^{-(a+1)}_{\mathrm{hor}}M^{(\bullet)}_n\}$ of extensions of $\gr^{-(a+1)}_{\mathrm{hor}}\mathcal{M}_1^{(\bullet)}$ by $\Fil^{-a}_{\mathrm{hor}}M^{(\bullet)}_n$, subject to the condition that $\alpha_1$ recovers the original partial splitting.

  \item The $\infty$-category $\mathrm{FilFCrys}^{\geq a}_{R/\Int_p}$ is independent of the choice of $g$ and the Frobenius lift $\varphi_S$.
\end{enumerate}
\end{lemma} 
\begin{proof}
The first part is a direct check, and the second is now shown as in Lemma~\ref{lem:tilde_M_1_ind}.
\end{proof}

The next lemma is shown just as in the case of $\Fzip^{\ge a}_{R/\Int_p}$, considered in~\eqref{subsec:fzips_partial_tensor} and~\eqref{subsec:fzips_partial_dual}.
\begin{lemma}
\label{lem:filfcrys_internal_hom}
Suppose that we have $\mathcal{N},\mathcal{M}\in \mathrm{FilFCrys}^{\geq a}_{R/\Int_p}$, and suppose in addition that $\mathcal{N}$ satisfies the following additional conditions:
\begin{enumerate}
   \item The underlying transversally filtered $p$-adic crystal $\Fil^\bullet_{\mathrm{Hdg}}\mathcal{N}$ is in $\mathrm{TFilCrys}^{\mathrm{lf}}_{R/\Int_p}$;
   \item For $i\neq a,a+1$, we have $\gr^i_{\mathrm{Hdg}}N_{R} \simeq 0$
 \end{enumerate}
Then there exists an `internal Hom' object 
\[
\mathcal{H} = \mathcal{H}(\mathcal{N},\mathcal{M}) \in \mathrm{FilFCrys}^{\geq -1}_{R/\Int_p}
\]
such that, for all $(R\xrightarrow{f}C)\in \mathrm{CRing}_{R/}$, we have a canonical equivalence
\[
\Map_{\mathrm{FilFCrys}^{\geq -1}_{C/\Int_p}}(\mathcal{O},f^*\mathcal{H}) \simeq \Map_{\mathrm{FilFCrys}^{\geq a}_{C/\Int_p}}(f^*\mathcal{N},f^*\mathcal{M}).
\]
\end{lemma}

\subsection{}\label{subsec:smooth_filfcrys}
Suppose that $\hat{R}$ is the $p$-completion of a smooth $\Int_p$-algebra $R$, and that it is equipped with a Frobenius lift $\varphi:\hat{R}\to \hat{R}$. Now, Corollary~\ref{cor:crystals_explicit_classical} (or even Proposition~\ref{prop:crystals_explicit}) shows that giving an object $\Fil^\bullet_{\mathrm{Hdg}}\mathcal{M}$ in $\mathrm{TFilCrys}^{\mathrm{lf}}_{\hat{R}/\Int_p}$ is equivalent to giving a tuple $(\hat{\bm{M}},\hat{\bm{\nabla}},\Fil^\bullet\hat{\bm{M}})$, where $\hat{\bm{M}}$ is a locally free $\hat{R}$-module, $\hat{\bm{\nabla}}$ is a compatible system of quasi-nilpotent integrable connections $\nabla_n$ on $\bm{M}_n = \hat{\bm{M}}/p^n\hat{\bm{M}}$ for every $n$, and $\Fil^\bullet\hat{\bm{M}}$ is a filtration by vector sub-bundles satisfying Griffiths transversality with respect to $\hat{\bm{\nabla}}$. In fact, this can be upgraded to an equivalence of the underlying classical $1$-categories.

Following Faltings~\cite{Faltings1989-af} (see also~\cite[(2.4.6)]{Lovering2017-fy}), we define a \defnword{strongly divisible $F$-crystal structure} on $\Fil^\bullet_{\mathrm{Hdg}}\mathcal{M}$ (with respect to the lift $\varphi$) to be a collection of $\hat{R}$-linear maps $\bm{F}_i:\varphi^*\Fil^i\hat{\bm{M}}\to \bm{M}$ satisfying the following conditions:
\begin{enumerate}
	\item For all $i$, $\bm{F}_i\vert_{\Fil^{i+1}\hat{\bm{M}}} = p\bm{F}_{i+1}$.
	\item The map
	\begin{align}\label{eqn:strong_div_map}
     \varphi^*\hat{\bm{M}}[1/p]\supset\sum_{i\in\Int}p^{-i}\varphi^*\Fil^i\hat{\bm{M}} \to \hat{\bm{M}}
	\end{align}
	given on $p^{-i}\Fil^i\hat{\bm{M}}$ by $p^i\bm{F}_i$ is an isomorphism that is compatible with the integrable connections on both sides. Here, the left hand side is equipped with the integrable connection induced by $p\varphi^*\hat{\bm{\nabla}}$.
\end{enumerate}

\begin{proposition}
\label{prop:strong_div}
Let $a\in \Int$ be the minimal integer such that $\gr^a\hat{\bm{M}} \neq 0$. Then every strongly divisible $F$-crystal structure on $\Fil^\bullet_{\mathrm{Hdg}}\mathcal{M}$ canonically gives rise to a lift of $\Fil^\bullet_{\mathrm{Hdg}}\mathcal{M}$ to $\mathrm{FilFCrys}^{\ge a}_{\hat{R}/\Int_p}$.
\end{proposition}
\begin{proof}
In the situation of~\eqref{subsec:f-zip_splitting}, we can take $g$ to be the identity $\hat{R}\twoheadrightarrow \hat{R}$. In particular, $\Fil^\bullet D_n^{(\bullet)}$ is just the constant cosimplicial object $R/p^nR$ with the trivial filtration.

Reducing the isomorphism~\eqref{eqn:strong_div_map} mod $p$ gives us a horizontal isomorphism
\begin{align}\label{eqn:strong_div_mod_p}
\bigoplus_{i\in\Int}\varphi^*\gr^i\bm{M}_1\xrightarrow{\simeq}\bm{M}_1,
\end{align}
where the connection on the left hand side is described as follows: By Griffith's transversality, we have an $R/p^2R$-linear map
\[
\gr^i\bm{M}_2 \xrightarrow{\nabla_2} \gr^{i-1}\bm{M}_2\otimes_R\Omega^1_{R/\Int_{(p)}},
\]
and the pullback of this map along $\varphi$ factors through a map
\[
\Theta_i:\varphi^*\gr^i\bm{M}_1\to \varphi^*\gr^{i-1}\bm{M}_1\otimes_{R}\Omega^1_{R/\Int_{(p)}}\xrightarrow{\simeq}p\varphi^*\gr^{i-1}\bm{M}_2\otimes_R\Omega^1_{R/\Int_{(p)}}.
\]
Classical Cartier theory endows $\varphi^*\gr^i\bm{M}_1$ with a canonical integrable connection: this corresponds to the object $\mathcal{O}_1\otimes_{\varphi}\gr^i\bm{M}$ in $\mathrm{Crys}_{R/\Field_p}$. The difference between the connection on the left hand side of~\eqref{eqn:strong_div_mod_p} and the one arising from Cartier theory is now given by $\oplus_i\Theta_i$.

Note in particular that, for every $i$, the submodule
\[
\bigoplus_{j\leq i}\varphi^*\gr^{-j}\bm{M}_1\subset \bigoplus_{j\in\Int}\varphi^*\gr^{-j}\bm{M}_1
\]
is stable for the connection, and so maps to a horizontal direct summand $\Fil^i_{\mathrm{hor}}\bm{M}_1$ of $\bm{M}_1$ under~\eqref{eqn:strong_div_mod_p}. It is now clear that this endows $\Fil^\bullet\mathcal{M}$ with the structure of an $F$-zip, and it is not so difficult now to also obtain a structure of an object of $\mathrm{FilFCrys}^{\geq a}_{R/\Int_p}$ using~\eqref{eqn:strong_div_map}.
\end{proof}

\subsection{}\label{subsec:ffilcrys_funct}
The functoriality of $\mathrm{FilFCrys}^{\geq a}_{R/\Int_p}$ under pullbacks along maps $f:R\to C$ is a little complicated because of the choices made in the course of its definition. However, Lemma~\ref{lem:filfcrys_ind} shows that these choices are inessential, and so one could with a bit more care obtain good pullback functoriality that accounts for this flexibility in choice.

We will instead assume that the choices of $g$ and $\varphi_{S}$ have also been made functorially for all objects in $\mathrm{CRing}$, following a construction from~\cite[\S 4]{De_Jong1999-yz}. For any $C\in \mathrm{CRing}$, we first consider the natural map
\[
g'_0:\Int_{(p)}[\pi_0(C)]\twoheadrightarrow \pi_0(C)
\]
from the free $\Int_{(p)}$-algebra in the set $\pi_0(C)$. This corresponds to a canonical arrow (up to homotopy)
\[
g':S' = \Int_{(p)}[\pi_0(C)]\twoheadrightarrow C,
\]
and we equip $S'$ with the Frobenius lift $\varphi_{S'}$ that acts as the $p$-power map on the generating set $\pi_0(C)$. With this convention, the functoriality of $\mathrm{FilFCrys}^{\ge a}_{R/\Int_p}$ is immediate on the nose.

In particular, for any prestack $X$, we have a good notion of the $\infty$-category $\mathrm{FilFCrys}^{\ge a}_{X/\Int_p}$, following the reasoning in Remarks~\ref{rem:inf_crys_over_stacks} and~\ref{rem:filcrys_over_stacks}.

\subsection{}\label{subsec:crys_coH_f-zip}
For any prestack $X$ over $R$, its Hodge filtered derived crystalline cohomology $\Fil^\bullet_{\mathrm{Hdg}}\mathrm{CrysCoh}_{X/R}$ underlies a canonical object in $\mathrm{FilFCrys}^{\geq 0}_{R/\Int_p}$. It suffices to show this when $X = \Spec A$ for $A\in \mathrm{CRing}_{R/}$.

In this case, evaluation on $(R\xrightarrow{\mathrm{id}}R,\gamma_{\mathrm{triv}})$ is the derived de Rham cohomology $\Fil^\bullet_{\mathrm{Hdg}}\mathrm{dR}_{A/R}\in \mathrm{FilMod}_{R}$ and we have $\gr^i_{\mathrm{Hdg}}\mathrm{dR}_{A/R}\simeq (\wedge^i\mathbb{L}_{A/R})[-i]$.

Now, the underlying object in $\mathrm{Crys}_{R/\Field_p}$ is the derived crystalline cohomology $\mathrm{CrysCoh}_{\overline{A}/\overline{R}}$,where we write $\overline{B} = B/{}^{\mathbb{L}}p$ for any $B\in \mathrm{CRing}$. By~\cite[Def. 4.36]{Mao2021-jt}, this is equipped with a canonical \defnword{conjugate filtration} $\Fil^\bullet_{\mathrm{conj}}\mathrm{CrysCoh}_{\overline{A}/\overline{R}}$ in $\mathrm{Crys}_{R/\Field_p}$ obtained by animating the canonical (or tautological) filtration on the divided powers de Rham complex. Moreover, by~\cite[Prop. 4.46]{Mao2021-jt}, there is a canonical equivalence
\[
\mathcal{O}_1\otimes_{\varphi}\gr^{-\bullet}_{\mathrm{Hdg}}\mathrm{dR}_{A/R}\xrightarrow{\simeq}\gr^{\bullet}_{\mathrm{conj}}\mathrm{CrysCoh}_{\overline{A}/\overline{R}}
\]
of graded objects in $\mathrm{Crys}_{R/\Field_p}$.

Thus, we have shown that $\Fil^\bullet_{\mathrm{Hdg}}\mathrm{CrysCoh}_{A/R}$ underlies a crystalline $F$-zip. 

\begin{lemma}
\label{lem:crys_coh_partially_split}
In the setting of~\eqref{subsec:partial_splitting}, the crystalline $F$-zip defined above lifts to an object in $\mathrm{FilFCrys}^{\ge 0}_{R/\Int_p}$.
\end{lemma}
\begin{proof}
If $\Fil^\bullet_{\mathrm{Hdg}}\mathcal{M} = \Fil^\bullet_{\mathrm{Hdg}}\mathrm{CrysCoh}_{A/R}$, then its associated object in $\mathrm{FilMod}_{\Fil^\bullet D_n^{(\bullet)}}$ is simply
\[
\Fil^\bullet_{\mathrm{Hdg}}M^{(\bullet)}_n = \Fil^\bullet_{\mathrm{Hdg}}\mathrm{CrysCoh}_{A/(D_n^{(\bullet)}\twoheadrightarrow R,\gamma_n(g^{(\bullet)}))}
\]
The right hand side can be computed as follows: Lift $R\to A$ to a smooth map $S\to S'$ of polynomial rings compatible with $g':S'\twoheadrightarrow A$, and consider the relative (crystalline) derived de Rham cohomology
\[
 \Fil^\bullet_{(D_n(g^{',(\bullet)})\twoheadrightarrow A,\gamma_n(g^{',(\bullet)}))}D_n(g^{',(\bullet)})\otimes_{S^{',(\bullet)}}\Fil^\bullet_{\mathrm{Hdg}}\Omega^\bullet_{S^{',(\bullet)}/S^{(\bullet)}}.
\]
Now, suppose that we have a Frobenius lift $\varphi_{S'}:S'\to S'$ compatible with $\varphi_S$. As in Lemma~\ref{lem:partial_splitting}, this induces maps
\[
\varphi'_0:D_n(g^{',(\bullet)})\to \Fil^0_{\mathrm{conj}}D_n(g^{',(\bullet)})\;;\;\varphi'_1:\Fil^1D_n(g^{',(\bullet)})\to \Fil^{-1}_{\mathrm{conj}}D_n(g^{',(\bullet)})
\]
Here, $\Fil^\bullet_{\mathrm{conj}}D_n(g^{',(\bullet)})$ is the pullback of the conjugate filtration on $D_1(g^{',(\bullet)})$; see~\cite[Defn. 4.58]{Mao2021-jt}.

The Frobenius lift also induces a map $\varphi'_0 = \varphi_{S'}^*:\Omega^i_{S^{',(\bullet)}/S^{(\bullet)}}\to \Omega^i_{S^{',(\bullet)}/S^{(\bullet)}}$ that is zero mod $p$ for $i>0$, and so dividing it by $p$ we obtain a map
\[
\varphi'_1:\Fil^1_{\mathrm{Hdg}}\Omega^\bullet_{S^{',(\bullet)}/S^{(\bullet)}}\to \Fil^{-1}_{\mathrm{conj}}\Omega^\bullet_{S^{',(\bullet)}/S^{(\bullet)}} = \tau_{\leq 1}\Omega^\bullet_{S^{',(\bullet)}/S^{(\bullet)}} + p\Omega^\bullet_{S^{',(\bullet)}/S^{(\bullet)}}\subset \Omega^\bullet_{S^{',(\bullet)}/S^{(\bullet)}}.
\]
that factors through $\gr^1_{\mathrm{Hdg}}\Omega^\bullet_{S^{',(\bullet)}/S^{(\bullet)}}$ mod $p$. 

Combining these maps in the obvious way now gives the desired lift to $\mathrm{FilFCrys}^{\ge 0}_{R/\Int_p}$.
\end{proof}

\subsection{}\label{subsec:formal_stacks}
Let $\mathrm{CRing}^{p-\text{nilp}}_{R/}$ be the $\infty$-subcategory of $\mathrm{CRing}_{R/}$ spanned by the objects $R\to C$ with $p$ nilpotent in $\pi_0(C)$.

A \defnword{$p$-adic formal prestack} over $R$ valued in an $\infty$-category $\mathcal{C}$ is a functor
\[
\mathrm{CRing}^{p-\text{nilp}}_{R/}\to \mathcal{C}.
\]

Any prestack $X$ over $R$ can be viewed as a \defnword{$p$-adic formal prestack} over $R$ by restricting it to $\mathrm{CRing}^{p-\text{nilp}}_{R/}$. We will denote this restriction by $X^{\form}_p$. 

We will say that a $p$-adic formal prestack $X^{\form}_p$ over $R$ is \defnword{representable by a $p$-adic formal derived scheme} over $R$ if its restriction to $\mathrm{CRing}_{(R/{}^{\mathbb{L}}p^n)/}$ is represented by a derived scheme over $R/{}^{\mathbb{L}}p^n$ for every $n\geq 1$. We will say that such a formal prestack is \defnword{locally quasi-smooth} if the corresponding derived scheme over $R/{}^{\mathbb{L}}p^n$ is locally quasi-smooth for every $n\geq 1$. 


\subsection{}
Note that the structure sheaf $\Fil^\bullet\mathcal{O}$ underlies an object $\mathcal{O}$ in $\mathrm{FilFCrys}^{\geq -1}_{R/\Int_p}$ with $\varphi_0 = \mathrm{id}$ and $\varphi_{-1} = p\cdot \mathrm{id}$. The underlying partially split crystalline $F$-zip is the canonical lift of $\Fil^\bullet\mathcal{O}$ to $\Fzip^{\geq -1}_{R/\Int_p}$ considered in~\eqref{subsec:fzips_partial_tensor}.

Suppose that we have an object $\mathcal{M}$ in $\mathrm{FilFCrys}^{\geq -1}_{R/\Int_p}$. Consider the prestacks
\[
\mathcal{Z}(\mathcal{M}),\mathcal{Z}(\mathcal{M}_n):\mathrm{CRing}_{R/} \to \Mod[\mathrm{cn}]{\Int}
\]
over $R$ given for each $(R\xrightarrow{f}C)\in \mathrm{CRing}_{R/}$ by
\begin{align*}
\mathcal{Z}(\mathcal{M}_n)(C) = \Map_{\mathrm{FilFCrys}^{\geq -1}_{C/(\Int/p^n\Int)}}(\mathcal{O}_n,f^*\mathcal{M}_n)\;;&\; \mathcal{Z}(\mathcal{M})(C) = \varprojlim_n\mathcal{Z}(\mathcal{M}_n)(C).
\end{align*}

\begin{theorem}
\label{thm:FilFCrys_functorial_props}
Suppose that $\gr^i_{\mathrm{Hdg}}M_R$ is \emph{perfect} for $i = 0,-1$. Then:
\begin{enumerate}
  \item\label{ffc:cohesive} $\mathcal{Z}(\mathcal{M})$ and $\mathcal{Z}(\mathcal{M}_n)$ are infinitesimally cohesive, nilcomplete and integrable \'etale sheaves.
  \item\label{ffc:cotangent} For every $C\in \mathrm{CRing}_{R/}$ and $P\in \Mod[\mathrm{cn}]{C}$, there is a natural equivalence
  \[
    \hker(\mathcal{Z}(\mathcal{M}_n)(C\oplus P)\to \mathcal{Z}(\mathcal{M}_n)(C))\xrightarrow{\simeq}\tau^{\leq 0}\left(P[-1]\otimes_R(\gr^{-1}_{\mathrm{Hdg}}M_R/{}^{\mathbb{L}}p^n)\right).
  \]
\end{enumerate}
\end{theorem}

\subsection{}\label{subsec:perfect_fields}
Suppose that $C\in \mathrm{CRing}_{\heartsuit,\overline{R}/}$ is a perfect ring in characteristic $p$. Suppose that $\Fil^\bullet_{\mathrm{Hdg}}\mathcal{M}$ is actually in $\mathrm{TFilCrys}^{\mathrm{lf}}_{R/\Int_p}$. We will now see that $\mathcal{Z}(\mathcal{M})(C)$ can be understood in more familiar terms: namely as Frobenius-invariant elements of certain $W(C)$-modules. 

Set $\widehat{D}(C) = W(C)\langle Z \rangle^{\wedge}_p/(Z-p)$, where $W(C)\langle Z\rangle^{\wedge}_p$ is the $p$-adic completion of the classical divided power algebra over $W(C)$ in the variable $Z$. We equip $\widehat{D}(C)$ with the divided power filtration $\Fil^\bullet\widehat{D}(C)$, where $\Fil^i\widehat{D}(C)\subset \widehat{D}(C)$ is the closure of the ideal generated by the images of $Z^{[k]}$ for $k\geq i$. It also has a Frobenius lift $\widehat{\varphi}$ agreeing with the canonical Frobenius lift on $W(C)$ and satisfying $Z^{[i]}\mapsto \frac{(pi)!}{i!}Z^{[pi]}$.

\begin{lemma}
\label{lem:W_PD_thickening}
Consider $g:W(C)\twoheadrightarrow C$. For every $n\geq 1$, there is a canonical equivalence of $\Fil^\bullet D_n(g^{(\bullet)})$ with the constant cosimplicial object $\Fil^\bullet \widehat{D}(C)/{}^{\mathbb{L}}p^n$. Moreover, the  Frobenius lift on $D_n(g^{(\bullet)})$ associated with the canonical Frobenius lift on $W(C)$ corresponds under this equivalence to the map induced from $\widehat{\varphi}$.
\end{lemma}
\begin{proof}
Note that $g$ is the base-change of $g_0:\Int_{(p)}[Z]\overset{Z\mapsto 0}{\twoheadrightarrow} \Int_{(p)}$ along $\Int_{(p)}[Z]\xrightarrow{Z\mapsto p}W(C)$. Therefore, the statements for $\Fil^\bullet D_n(g)$ are obtained via base-change from the corresponding assertions for $\Fil^\bullet D_n(g_0)$, and these are easy to show directly.

To see that $\Fil^\bullet D_n(g^{(\bullet)})$ is constant, now simply note that $W(C)$ is formally \'etale over $\Int_{(p)}$, and so any arrow of the form
\[
(W(C)^{\otimes_{\Int_{(p)}}i}\overset{g^{(i)}}{\twoheadrightarrow}R)\to (A'\twoheadrightarrow A\leftarrow C,\gamma)
\]
with the target in $\mathrm{AniPDPair}_{p,n,C/}$ has to factor through $W(C)\overset{g}{\twoheadrightarrow}C$.
\end{proof}

\subsection{}\label{subsec:M_on_perfect}
The reason for the intervention of the persnickety $Z$ here is because we are working throughout this section with absolute divided power envelopes and not the ones relative to $(\Int_{(p)}\twoheadrightarrow \Field_p,\delta_p)$, which are more commonly encountered in studies of crystalline cohomology. This, however, turns out to be a red herring in this particular situation; see assertion~\eqref{ffc:perfect} in Theorem~\ref{thm:FilFCrys_discrete} below.

Evaluating $\Fil^\bullet_{\mathrm{Hdg}}\mathcal{M}$ on $(\widehat{D}(C)/{}^{\mathbb{L}}p^n\twoheadrightarrow C,\gamma_n(g))$ and taking the limit over $n$ gives us a filtered module $\Fil^\bullet_{\mathrm{Hdg}}\mathcal{M}(\widehat{D}(C))$ over $\Fil^\bullet\widehat{D}(C)$ with $\mathcal{M}(\widehat{D}(C))$ a finite locally free $\widehat{D}(C)$-module. 

Write 
\[
\Fil^\bullet_{\mathrm{Hdg}}\mathcal{M}(W(C)) = \Fil^\bullet W(C)\otimes_{\Fil^\bullet\widehat{D}}\Fil^\bullet_{\mathrm{Hdg}}\mathcal{M}(\widehat{D}(C)), 
\]
where $\Fil^i W(C)\subset W(C)$ is the ideal generated by $\frac{p^k}{k!}$ for $k\geq i$, and $W(C)$ is being viewed as a $\widehat{D}(C)$-algebra via $Z^{[i]}\mapsto \frac{p^i}{i!}$. This is of course the evaluation of $\Fil^\bullet_{\mathrm{Hdg}}\mathcal{M}$ on $(W(C)\twoheadrightarrow C,\delta_p)$.

Set $\Fil^\bullet_{\mathrm{Hdg}}M_C = \Fil^\bullet_{\mathrm{Hdg}}\mathcal{M}(W(C))/{}^{\mathbb{L}}p$. The lift of $\Fil^\bullet_{\mathrm{Hdg}}\mathcal{M}$ to $\mathrm{FilFCrys}^{\geq -1}_{R/\Int_p}$ gives us a filtration $\Fil^\bullet_{\mathrm{hor}}M_C$ on $M_C$, and if we set
\[
\Fil^\bullet_{\mathrm{hor}}\mathcal{M}(W(C)) = \Fil^\bullet_{\mathrm{hor}}M_C\times_{M_C}\mathcal{M}(W(C)),
\]
then we also have a $W(C)$-Frobenius semi-linear map $\varphi_0:\Fil^0_{\mathrm{Hdg}}\mathcal{M}(W(C))\to\Fil^0_{\mathrm{hor}}\mathcal{M}(W(C))$, which in turn induces a map $\varphi_0:\Fil^0_{\mathrm{Hdg}}\mathcal{M}(W(C)) \to \mathcal{M}(W(C))$.

\begin{theorem}\label{thm:FilFCrys_discrete}
With the notation and hypotheses of Theorem~\ref{thm:FilFCrys_functorial_props}, suppose that $\Fil^\bullet_{\mathrm{Hdg}}\mathcal{M}$ is actually in $\mathrm{TFilCrys}^{\mathrm{lf}}_{R/\Int_p}$ and that $R\in \mathrm{CRing}_{\heartsuit}$ is a flat $p$-completely finitely generated $\Int_p$-algebra---that is, $R/{}^{\mathbb{L}}p$ is a finitely generated $\Field_p$-algebra. Then:
\begin{enumerate}
\item\label{ffc:form_etale} For every $n\ge 1$, $\mathcal{Z}(\mathcal{M}_n)^{\form}_p$ admits a cotangent complex over $R$, and we in fact have
\[
\mathbb{L}_{\mathcal{Z}(\mathcal{M}_n)^{\form}_p/R} \simeq \Reg{\mathcal{Z}(\mathcal{M}_n)^{\form}_p}\otimes_R(\gr^{-1}_{\mathrm{Hdg}}M_R/{}^{\mathbb{L}}p^n)[1].
\]
\item\label{ffc:perfect} For $C\in \mathrm{CRing}_{\heartsuit,\overline{R}/}$ perfect, we have
\[
\mathcal{Z}(\mathcal{M})^{\form}_p(C) = \Fil^0_{\mathrm{Hdg}}\mathcal{M}(W(C))^{\varphi_0 = \mathrm{id}}.
\]
\item\label{ffc:rep} For every $n\in \Int_{\ge 1}$, $\mathcal{Z}(\mathcal{M}_n)^{\form}_p$ is represented by a $p$-adic locally quasi-smooth formal scheme over $R$. Therefore, $\mathcal{Z}(\mathcal{M})^{\form}_p$ is represented by a pro-$p$-adic formal derived scheme over $R$.

  \item\label{ffc:open_closed} For $C\in \mathrm{CRing}^{p-\text{nilp}}_{R/}$ and $y\in \mathcal{Z}(\mathcal{M})^{\form}_p(C)$, the map
  \[
   \Spec C\times_{y,\mathcal{Z}(\mathcal{M})^{\form}_p,0}\Spec C\to \Spec C
  \]
  is a closed immersion that induces an open immersion on the underlying classical truncations when $\pi_0(C)$ is a finitely generated $\Int_p$-algebra.
 \end{enumerate}
\end{theorem}



\begin{remark}
Suppose that $R$ is discrete and that we have a $p$-divisible groups $G_1,G_2$ over $R$. Then the associated (covariant) Dieudonn\'e $F$-crystals $\mathcal{M}(G_1)$ and $\mathcal{M}(G_2)$ as defined for instance in~\cite{bbm:cris_ii} can be upgraded to objects in $\mathrm{FilFCrys}^{\geq 0}_{R/\Int_p}$. Using Lemma~\ref{lem:filfcrys_internal_hom}, their `internal Hom' in the category of filtered $F$-crystals can be upgraded to an object $\mathcal{H}(G_1,G_2)$ in $\mathrm{FilFCrys}^{\geq -1}_{R/\Int_p}$, with which Theorems~\ref{thm:FilFCrys_functorial_props} and~\ref{thm:FilFCrys_discrete} allow us to associate the pro-formal $p$-adic derived algebraic space
\[
\mathbb{H}(G_1,G_2)^{\form}_p \overset{\mathrm{defn}}{=}\mathcal{Z}(\mathcal{H}(G_1,G_2))^{\form}_p:\mathrm{CRing}^{p-\text{nilp}}_{R/}\to \Mod[\mathrm{cn},p\text{-comp}]{\Int}
\]
such that for any $(R\xrightarrow{f}C)\in \mathrm{CRing}^{p-\text{nilp}}_{R/}$, we have
\[
 \mathbb{H}(G_1,G_2)^{\form}_p(C)\simeq \Map_{\mathrm{FilFCrys}^{\geq 0}_{C/\Int_p}}(f^*\mathcal{M}(G_1),f^*\mathcal{M}(G_2)).
\]
A natural question now arises regarding the relationship between this and the formal scheme of homomorphisms from $G_1$ to $G_2$. I can't say anything definitive about this, but when $p>2$ and when $C$ is a discrete, excellent, flat $\Int_p$-algebra, all of whose local rings are complete intersections, one can use Grothendieck-Messing theory and a theorem of de Jong-Messing~\cite[Theorem 4.6]{De_Jong1999-yz} to show that there is a natural isomorphism
\[
\Hom(f^*G_1,f^*G_2)\xrightarrow{\simeq}\mathbb{H}(G_1,G_2)^{\form}_p(C).
\]
\end{remark}

\subsection{}\label{subsec:W0_condition_crys}
We will have use for the following extra structure: Suppose that we have a direct sum decomposition $\mathcal{M} \simeq \bigoplus_{i=0}^d\mathcal{M}^{(i)}$ in $\mathrm{FilFCrys}^{\geq -1}_{R/\Int_p}$, where $\Fil^\bullet_{\mathrm{Hdg}}\mathcal{M}^{(i)}[i]$ belongs to $\mathrm{TFilCrys}^{\mathrm{lf}}_{R/\Int_p}$.

Then we have a corresponding direct sum decomposition of $\Mod[\mathrm{cn},p\text{-comp}]{\Int}$-valued of $p$-adic formal prestacks over $R$: $\mathcal{Z}(\mathcal{M})^{\form}_p \simeq \bigoplus_{i=0}^d\mathcal{Z}(\mathcal{M}^{(i)})_p^{\form}$.

We now have the following result which is shown exactly as its infinitesimal analogue, Proposition~\ref{prop:W0_discrete_FIC}, by using Theorem~\ref{thm:FilFCrys_discrete}.
\begin{proposition}
\label{prop:W0_discrete}
 The map $\mathcal{Z}(\mathcal{M}^{(0)})_p^{\form}\to \mathcal{Z}(\mathcal{M})^{\form}_p$ is a closed immersion of $p$-adic formal prestacks that is an equivalence on the underlying classical $p$-adic formal prestacks. More precisely, for every $C\in \mathrm{CRing}_{R/}^{p-\text{nilp}}$, and every map $\Spec C\to \mathcal{Z}(\mathcal{M})^{\form}_p$ of $p$-adic formal prestacks over $R$, the map
 \[
\Spec C\times_{\mathcal{Z}(\mathcal{M})^{\form}_p}\mathcal{Z}(\mathcal{M}^{(0)})_p^{\form}\to \Spec C
 \]
 is represented by a derived closed subscheme, and is an isomorphism of the underlying classical schemes.
\end{proposition}

\section{Spaces of sections of derived filtered $p$-adic $F$-crystals: proofs}\label{sec:proofs}

We now begin the proofs of Theorems~\ref{thm:FilFCrys_functorial_props} and~\ref{thm:FilFCrys_discrete}. They are a bit involved, but the basic idea is to work with the mod-$p$ functor $\mathcal{Z}(\mathcal{M})/{}^{\mathbb{L}}p$, and to compare it with a different prestack $\mathcal{Z}^{\heartsuit}$, which admits a simple description as the homotopy fiber of a Frobenius linear map of perfect complexes, and for which the analogues of the properties in the theorems are quite immediate. We then deduce the corresponding properties for $\mathcal{Z}(\mathcal{M})/{}^{\mathbb{L}}p$ essentially by finding that the `difference' between this prestack and $\mathcal{Z}^{\heartsuit}$ is representable and \'etale (see Lemma~\ref{lem:exts_equivalent_restr} for the precise technical manifestation of this assertion). 

\subsection{}\label{subsec:global_sections}
Suppose that we have $f:R\to C$ with functorially chosen $\tilde{f}:S\to S'$, $g:S\twoheadrightarrow R$, $g':S'\twoheadrightarrow C$, and Frobenius lifts $\varphi_S$ and $\varphi_{S'}$ as in~\eqref{subsec:ffilcrys_funct}.

In this situation, note that we obtain linearized maps
\begin{align*}
S^{',(\bullet)}\otimes_{\varphi_{S^{',(\bullet)}},S^{',(\bullet)}}\gr^{-1}_{\mathrm{Hdg}}M_{\overline{C}}&\xrightarrow{1\otimes\varphi_{-1}}M_1^{',(\bullet)};\\
(S^{',(\bullet)}\otimes_{\varphi_{S^{',(\bullet)}},S^{',(\bullet)}}\gr^{0}_{\mathrm{Hdg}}M_1^{',(\bullet)})\oplus (S^{',(\bullet)}\otimes_{\varphi_{S^{',(\bullet)}},S^{',(\bullet)}}\gr^{-1}_{\mathrm{Hdg}}M_{\overline{C}})&\xrightarrow{(1\otimes\varphi_{0},1\otimes\varphi_{-1})}M_1^{',(\bullet)}.
\end{align*}

We will write $\Fil^{1}_{\mathrm{con}}M_1^{',(\bullet)}\to M_1^{',(\bullet)}$ and $\Fil^{0}_{\mathrm{con}}M_1^{',(\bullet)}\to  M_1^{',(\bullet)}$\footnote{`con' stands for `conjugate'.} for these maps. Note that the first map factors through the second, so that we can define
\[
\gr^{0}_{\mathrm{con}}M_1^{',(\bullet)} = \Fil^{0}_{\mathrm{con}}M_1^{',(\bullet)}/\Fil^{1}_{\mathrm{con}}M_1^{',(\bullet)}\;;\; \gr^{1}_{\mathrm{con}}M_1^{',(\bullet)} = \Fil^{1}_{\mathrm{con}}M_1^{',(\bullet)}
\]

\begin{lemma}
\label{lem:global_sections_crys}
Set
\begin{align*}
R\Gamma_{\mathrm{crys}}(C,f^*\mathcal{M}_n) = \mathrm{Tot}M^{',(\bullet)}_n\;&;\;\Fil^\bullet_{\mathrm{Hdg}}R\Gamma_{\mathrm{crys}}(C,f^*\mathcal{M}_n) = \mathrm{Tot}\Fil^\bullet_{\mathrm{Hdg}}M^{',(\bullet)}_n\\
\Fil^i_{\mathrm{con}}R\Gamma_{\mathrm{crys}}(\overline{C},f^*\mathcal{M}_1) =\mathrm{Tot}\Fil^i_{\mathrm{con}}M^{',(\bullet)}_1\;&;\;\gr^{i}_{\mathrm{con}}R\Gamma_{\mathrm{crys}}(\overline{C},f^*\mathcal{M}_1) =\mathrm{Tot}\gr^{i}_{\mathrm{con}}M^{',(\bullet)}_1,
\end{align*}
where $i=-1,0$.

Then for $i=-1,0$, we have maps
\[
\varphi_i:\Fil^i_{\mathrm{Hdg}}R\Gamma_{\mathrm{crys}}(C,f^*\mathcal{M}_n) \to R\Gamma_{\mathrm{crys}}(C,f^*\mathcal{M}_n)
\]
and they induce maps
\[
\varphi_i:\gr^i_{\mathrm{Hdg}}R\Gamma_{\mathrm{crys}}(\overline{C},f^*\mathcal{M}_1) \to \Fil^{-i}_{\mathrm{con}}R\Gamma_{\mathrm{crys}}(\overline{C},f^*\mathcal{M}_1)
\]
that in turn produce equivalences
\[
\varphi_i:\gr^i_{\mathrm{Hdg}}R\Gamma_{\mathrm{crys}}(\overline{C},f^*\mathcal{M}_1)\xrightarrow{\simeq}\gr^{-i}_{\mathrm{con}}R\Gamma_{\mathrm{crys}}(\overline{C},f^*\mathcal{M}_1).
\]
\end{lemma}
\begin{proof}
Only the last point requires proof. Unraveling definitions, this comes down to the assertion that the natural map
\[
\mathrm{Tot}(\gr^i_{\mathrm{Hdg}}M_1^{',(\bullet)}
)\to \mathrm{Tot}(S^{',(\bullet)}\otimes_{\varphi_{S^{',(\bullet)}},S^{',(\bullet)}}\gr^i_{\mathrm{Hdg}}M_1^{',(\bullet)}
)
\]
is an equivalence for $i=-1,0$.

This in turn reduces to showing the same assertion but with $\gr^i_{\mathrm{Hdg}}M_1^{',(\bullet)}$ replaced with one of the following cosimplicial objects (where the first two are constant):
\[
\gr^{-1}_{\mathrm{Hdg}}M_{\overline{C}}\;;\;\gr^{0}_{\mathrm{Hdg}}M_{\overline{C}}\;;\; \gr^1D_1(g^{',(\bullet)})\otimes_{\overline{C}}\gr^{-1}_{\mathrm{Hdg}}M_{\overline{C}}.
\]

For simplicity, set $\overline{C}_\varphi^{(\bullet)} = S^{',(\bullet)}\otimes_{\varphi_{S^{',(\bullet)}},S^{',(\bullet)}}\overline{C}$. Then the assertion for the first two cosimplicial objects reduces to showing that $\overline{C} \to \mathrm{Tot}(\overline{C}^{(\bullet)}_\varphi)$ is an equivalence, where it is an immediate consequence of faithfully flat descent with respect to the finite flat map $\overline{C}\to \overline{C}^{(0)}_{\varphi}$.

For the third, note that we have $\gr^1D_1(g^{',(\bullet)})\simeq \mathbb{L}_{\overline{C}/\overline{S}^{',(\bullet)}}[-1]$, so that it suffices to show that the map
\[
\mathrm{Tot}(\mathbb{L}_{\overline{C}/\overline{S}^{',(\bullet)}})\to \mathrm{Tot}(\mathbb{L}_{\overline{C}^{(\bullet)}_\varphi/\overline{S}^{',(\bullet)}})
\]
is an equivalence. 

For this we use the canonical fiber sequences
\begin{align*}
\overline{C}\otimes_{\overline{S}^{',(\bullet)}}\mathbb{L}_{\overline{S}^{',(\bullet)}/\Field_p} &\to \mathbb{L}_{\overline{C}/\Field_p}\to \mathbb{L}_{\overline{C}/S^{',(\bullet)}};\\
\overline{C}^{(\bullet)}_\varphi\otimes_{\overline{S}^{',(\bullet)}}\mathbb{L}_{\overline{S}^{',(\bullet)}/\Field_p} &\to \mathbb{L}_{\overline{C}^{(\bullet)}_\varphi/\Field_p}\to \mathbb{L}_{\overline{C}^{(\bullet)}_\varphi/S^{',(\bullet)}}.
\end{align*}

The first cosimplicial object in each of the rows is homotopy equivalent to the zero object by~\cite[Lemma 2.6]{Bhatt2012-kp}, which means that we have finally reduced the lemma down to showing that the natural map $\mathbb{L}_{\overline{C}/\Field_p}\to \mathrm{Tot}(\mathbb{L}_{\overline{C}^{(\bullet)}_\varphi/\Field_p})$ is an equivalence, which is a consequence of the fact that the cotangent complex satisfies fppf descent; see~\cite[Remark 2.8]{Bhatt2012-kp}.
\end{proof}

\begin{lemma}
\label{lem:FilFCrys_global_secs_cosimp}
In the notation of Lemma~\ref{lem:global_sections_crys}, set
\[
\mathcal{Z}^+_n(C) = \hker\left(\Fil^0_{\mathrm{Hdg}} R\Gamma_{\mathrm{crys}}(C,f^*\mathcal{M}_n)\xrightarrow{\varphi_0- \mathrm{id}} R\Gamma_{\mathrm{crys}}(C,f^*\mathcal{M}_n)\right).
\]
Then we have 
\[
\mathcal{Z}(\mathcal{M}_n)(C) \simeq \tau^{\le 0}\mathcal{Z}^+_n(C).
\]
\end{lemma}
\begin{proof}
Set
\begin{align*}
\Fil^0_{\mathrm{hor}}R\Gamma_{\mathrm{crys}}(C,f^*\mathcal{M}_n) &= \mathrm{Tot}\Fil^\bullet_{\mathrm{hor}}M_n^{',(\bullet)}\\ 
\Fil^0_{\mathrm{Hdg}}\Fil^0_{\mathrm{hor}}R\Gamma_{\mathrm{crys}}(C,f^*\mathcal{M}_n) & = \Fil^0_{\mathrm{Hdg}}R\Gamma_{\mathrm{crys}}(C,f^*\mathcal{M}_n)\times_{R\Gamma_{\mathrm{crys}}(C,f^*\mathcal{M}_n)}\Fil^0_{\mathrm{hor}}R\Gamma_{\mathrm{crys}}(C,f^*\mathcal{M}_n).
\end{align*}
Then it follows quite directly from the definitions and Propositions~\ref{prop:crystals_explicit} and~\ref{prop:cocartesian_sections} that we have
\[
\mathcal{Z}(\mathcal{M}_n)(C) = \tau^{\le 0}\left(\Fil^0_{\mathrm{Hdg}}\Fil^0_{\mathrm{hor}}R\Gamma_{\mathrm{crys}}(C,f^*\mathcal{M}_n)\xrightarrow{\varphi_0- \mathrm{id}}\Fil^0_{\mathrm{hor}}R\Gamma_{\mathrm{crys}}(C,f^*\mathcal{M}_n)\right).
\]

It is immediate from this that $\mathcal{Z}(\mathcal{M}_n)(C) \simeq \tau^{\leq 0}\mathcal{Z}^+_n(C)$.
\end{proof}

\subsection{}
Now, note that we have
\[
\mathcal{Z}^+_1(C) \simeq \hker\left(\Fil^0_{\mathrm{Hdg}} R\Gamma_{\mathrm{crys}}(\overline{C},f^*\mathcal{M}_1)\xrightarrow{\varphi_0 - \mathrm{id}} R\Gamma_{\mathrm{crys}}(\overline{C},f^*\mathcal{M}_1)\right).
\]

Set
\[
\mathcal{Z}^{\diamond}_1(C) = \hker\left(\Fil^0_{\mathrm{Hdg}} R\Gamma_{\mathrm{crys}}(\overline{C},f^*\mathcal{M}_1)\xrightarrow{\varphi_0- \mathrm{id}} R\Gamma_{\mathrm{crys}}(\overline{C},f^*\mathcal{M}_1)/\Fil^1_{\mathrm{con}}R\Gamma_{\mathrm{crys}}(\overline{C},f^*\mathcal{M}_1)\right),
\]
so that we have a fiber sequence
\begin{align}\label{eqn:Mmodp_M1}
\mathcal{Z}^+_1(C) \to \mathcal{Z}^{\diamond}(C)\to \Fil^1_{\mathrm{con}}R\Gamma_{\mathrm{crys}}(\overline{C},f^*\mathcal{M}_1)\xleftarrow[\simeq]{\varphi_1}\gr^{-1}_{\mathrm{Hdg}}M_{\overline{C}}.
\end{align}

\subsection{}
The following abstract setup will be useful in studying $\mathcal{Z}^{\diamond}(C)$. Suppose that we have two arrows $f_1,f_2:L\to L'$ in $\Mod{\Field_p}$ with the same source and target. Set $f_3 = f_1 - f_2$, and $K_i = \hker(f_i)$ for $i=1,2,3$, and, for $i < j$, set $K_{i,j} = \hker(f_i)\times_L\hker(f_j)$. Note that for each $i<j$, we obtain a diagram
\[
\begin{diagram}
K_{i,j}&\rTo&K_i&\rTo^{f_j\vert_{K_i}}&L'\\
\dTo&&\dTo&&\dEquals\\
K_j&\rTo&L&\rTo_{f_j}&L'\\
\dTo^{f_i\vert_{K_j}}&&\dTo_{f_i}\\
L'&\rEquals&L'
\end{diagram}
\]
where the top two rows and the two columns on the left are fiber sequences. Rotating the fiber sequences on top as well as the on the left gives us arrows
\[
\eta_{i,j}:L'\to K_{i,j}[1]\;;\; \eta_{j,i}:L'\to K_{i,j}[1].
\]

Since $f_1$ and $f_3 = f_1-f_2$ have homotopic restrictions to $K_2$, there is a canonical equivalence
\[
\alpha:K_{2,3}\simeq \hker(f_3\vert_{K_2}) \simeq \hker(f_1\vert_{K_2}) \simeq K_{1,2}.
\]
 
\begin{lemma}
\label{lem:common_kernel}
We have 
\[
\alpha[1]\circ \eta_{1,2}\simeq \eta_{3,2}-\eta_{2,3}:L'\to K_{2,3}[1].
\]
\end{lemma}
\begin{proof}
There is probably a more elegant way to do this, but here is a proof that uses explicit representatives for $L$ and $L'$ as chain complexes of $\Field_p$-vector spaces, chosen so that the maps $f_1,f_2$ also lift to maps of complexes. We denote these lifts by the same letters. 

In this case, $K_i$ is given by the shifted cone with degree $n$ entry $L^{',n-1}\oplus L^{n}$ and differentials given by $(v',v)\mapsto (-d(v')+f_i(v),d(v))$, and $K_{i,j}$ for $i<j$ is represented by the complex with degree $n$ entry $L^{',n-1}\oplus L^{',n-1}\oplus L^{n}$ and differentials given by $(v'_1,v'_2,v)\mapsto (-d(v'_1)+f_j(v),-d(v'_2)+f_i(v),d(v))$. 

The classifying maps $\eta_{i,j}$ and $\eta_{j,i}$ lifting to $v'\mapsto (v',0,0)$ and $v'\mapsto (0,v',0)$, respectively.

The equivalence $\alpha$ is represented explicitly by the isomorphism
\begin{align*}
K_{1,2} &\xrightarrow{\simeq} K_{2,3}\\
(v'_1,v'_2,v)&\mapsto (v'_2-v'_1,v'_1,v)
\end{align*}
and it is now clear that we have $\alpha\circ \eta_{1,2} = \eta_{3,2}-\eta_{2,3f}$ on the nose: both are equal to the map
\begin{align*}
L'&\to K_{2,3}\\
v'&\mapsto (-v',v',0).
\end{align*}
\end{proof}

\subsection{}
For any $i,j\in \Int$, set
\[
\Fil^j_{\mathrm{Hdg}}\Fil^i_{\mathrm{con}}R\Gamma_{\mathrm{crys}}(\overline{C},f^*\mathcal{M}_1) = \Fil^j_{\mathrm{Hdg}}R\Gamma_{\mathrm{crys}}(\overline{C},f^*\mathcal{M}_1)\times_{R\Gamma_{\mathrm{crys}}(\overline{C},f^*\mathcal{M}_1) }\Fil^i_{\mathrm{con}}R\Gamma_{\mathrm{crys}}(\overline{C},f^*\mathcal{M}_1).
\]

Then noting that $\varphi_0$ factors through $\Fil^0_{\mathrm{con}}R\Gamma_{\mathrm{crys}}(\overline{C},f^*\mathcal{M}_1)$, we find
\[
\mathcal{Z}^{\diamond}(C)\simeq \hker(\Fil^0_{\mathrm{Hdg}}\Fil^0_{\mathrm{con}} R\Gamma_{\mathrm{crys}}(\overline{C},f^*\mathcal{M}_1)\xrightarrow{\varphi_0- \iota} \gr^0_{\mathrm{con}}R\Gamma_{\mathrm{crys}}(\overline{C},f^*\mathcal{M}_1)),
\]
where we write $\iota$ for the natural map
\begin{align}\label{eqn:id_map_fil0con}
\Fil^0_{\mathrm{Hdg}}\Fil^0_{\mathrm{con}} R\Gamma_{\mathrm{crys}}(\overline{C},f^*\mathcal{M}_1)\to \Fil^0_{\mathrm{con}} R\Gamma_{\mathrm{crys}}(\overline{C},f^*\mathcal{M}_1)\to \gr^0_{\mathrm{con}}R\Gamma_{\mathrm{crys}}(\overline{C},f^*\mathcal{M}_1).
\end{align}

We now apply the discussion above Lemma~\ref{lem:common_kernel} with 
\[
L = \Fil^0_{\mathrm{Hdg}}\Fil^0_{\mathrm{con}} R\Gamma_{\mathrm{crys}}(\overline{C},f^*\mathcal{M}_1); L' = \gr^0_{\mathrm{con}}R\Gamma_{\mathrm{crys}}(\overline{C},f^*\mathcal{M}_1); f_1 = \varphi_0;f_2 = \iota.
\]
Then in the notation there, we have $\mathcal{Z}^{\diamond}(C) = K_3$. Set
\[
\mathcal{Z}^{\spadesuit}(C) = K_1 = \hker(\varphi_0)\;;\; \mathcal{Z}^{\heartsuit}(C) = K_2 = \hker(\iota).
\]
Also, set
\[
\Fil^1_{\mathrm{con}}\mathcal{Z}^{\diamond}(C) = K_{2,3} = \hker(\iota\vert_{\mathcal{Z}^+(C)})\simeq \hker((\varphi_0-\iota)\vert_{\mathcal{Z}^{\heartsuit}(C)})\simeq \hker(\varphi_0\vert_{\mathcal{Z}^{\heartsuit}(C)})\simeq \hker(\iota\vert_{\mathcal{Z}^{\spadesuit}(C)}),
\]
and write 
\[
\eta^\diamond_C,\eta^\heartsuit_C:\gr^0_{\mathrm{Hdg}}R\Gamma_{\mathrm{crys}}(\overline{C},f^*\mathcal{M}_1)\to \Fil^1_{\mathrm{con}}\mathcal{Z}^{\diamond}(C)[1].
\]
for the maps $\eta_{3,2},\eta_{2,3}$, respectively. 

\subsection{}
Our goal now is to describe the difference $\eta^\diamond_C-\eta^{\heartsuit}_C$ explicitly. First, note that we have
\[
\Fil^1_{\mathrm{con}}\mathcal{Z}^{\diamond}(C) \simeq \Fil^1_{\mathrm{Hdg}}\Fil^1_{\mathrm{con}}R\Gamma_{\mathrm{crys}}(\overline{C},f^*\mathcal{M}_1),
\]
so that we have a fiber sequence
\begin{align}\label{eqn:cool_fiber_sequence}
\Fil^1_{\mathrm{con}}\mathcal{Z}^{\diamond}(C)\to \Fil^1_{\mathrm{Hdg}}R\Gamma_{\mathrm{crys}}(\overline{C},f^*\mathcal{M}_1)\to R\Gamma_{\mathrm{crys}}(\overline{C},f^*\mathcal{M}_1)/\Fil^1_{\mathrm{con}}R\Gamma_{\mathrm{crys}}(\overline{C},f^*\mathcal{M}_1),
\end{align}
which rotates to yield a map
\begin{align}
\label{eqn:fil1con_desc}
R\Gamma_{\mathrm{crys}}(\overline{C},f^*\mathcal{M}_1)\to \Fil^1_{\mathrm{con}}\mathcal{Z}^{\diamond}(C)[1]
\end{align}

Using the equivalence
\begin{align}
\label{eqn:varphi0equiv}
\varphi_0:\gr^0_{\mathrm{Hdg}}R\Gamma_{\mathrm{crys}}(\overline{C},f^*\mathcal{M}_1)\xrightarrow{\simeq}\gr^0_{\mathrm{con}}R\Gamma_{\mathrm{crys}}(\overline{C},f^*\mathcal{M}_1)
\end{align}
we can view $\eta^\diamond_C,\eta^{\heartsuit}_C$ as maps out of $\gr^0_{\mathrm{Hdg}}R\Gamma_{\mathrm{crys}}(\overline{C},f^*\mathcal{M}_1)$, which we denote by the same symbols.

We will also need the fiber sequence
\begin{align*}
\gr^1_{\mathrm{Hdg}}R\Gamma_{\mathrm{crys}}(\overline{C},\mathcal{O}_1)\otimes_C\gr^{-1}_{\mathrm{Hdg}}M_{\overline{C}} \to \gr^0_{\mathrm{Hdg}}R\Gamma_{\mathrm{crys}}(\overline{C},f^*\mathcal{M}_1)\to \gr^0_{\mathrm{Hdg}}M_{\overline{C}}.
\end{align*}

Here is the key result:
\begin{lemma}
\label{lem:exts_equivalent_restr}
The difference $\eta^\diamond_C-\eta^\heartsuit_C$ is naturally equivalent to the composition
\begin{align}\label{eqn:diff_composition}
\gr^0_{\mathrm{Hdg}}R\Gamma_{\mathrm{crys}}(\overline{C},f^*\mathcal{M}_1)\xrightarrow{\varphi_0}R\Gamma_{\mathrm{crys}}(\overline{C},f^*\mathcal{M}_1)\xrightarrow{\eqref{eqn:fil1con_desc}}\Fil^1_{\mathrm{con}}\mathcal{Z}^{\diamond}(C)[1],
\end{align}
which factors through an arrow
\begin{align}\label{eqn:difference_arrow}
\gr^0_{\mathrm{Hdg}}M_{\overline{C}}\to \Fil^1_{\mathrm{con}}\mathcal{Z}^{\diamond}(C)[1].
\end{align}
Moreover, the arrow
\[
\hker(\varphi_{\overline{C}})\otimes_{\overline{C}}\gr^0_{\mathrm{Hdg}}M_{\overline{C}}\to \Fil^1_{\mathrm{con}}\mathcal{Z}^{\diamond}(C)[1]
\]
induced from~\eqref{eqn:difference_arrow} is nullhomotopic.
\end{lemma}
\begin{proof}
We first consider the map
\[
\eta_{1,2}:\gr^0_{\mathrm{con}}R\Gamma_{\mathrm{crys}}(\overline{C},f^*\mathcal{M}_1)\to \Fil^1_{\mathrm{con}}\mathcal{Z}^{\diamond}(C)[1]
\]
corresponding to the extension $\mathcal{Z}^{\spadesuit}(C)$. For this, note that this sits in the fiber sequence\footnote{We omit $R\Gamma_{\mathrm{crys}}(\overline{C},f^*\mathcal{M}_1)$ from the notation for the sake of clarity.}
\[
(\Fil^1_{\mathrm{con}}\mathcal{Z}^{\diamond}(C)\simeq\Fil^1_{\mathrm{Hdg}}\Fil^1_{\mathrm{con}}) \to (\mathcal{Z}^{\spadesuit}(C) \simeq \Fil^1_{\mathrm{Hdg}}\Fil^0_{\mathrm{con}})\to (\gr^0_{\mathrm{con}}\xleftarrow[\simeq]{\varphi_0}\gr^0_{\mathrm{Hdg}}).
\]

This is simply the pullback of the sequence~\eqref{eqn:cool_fiber_sequence} along the map
\[
\gr^0_{\mathrm{Hdg}}R\Gamma_{\mathrm{crys}}(\overline{C},f^*\mathcal{M}_1)\xrightarrow{\varphi_0}R\Gamma_{\mathrm{crys}}(\overline{C},f^*\mathcal{M}_1)\to R\Gamma_{\mathrm{crys}}(\overline{C},f^*\mathcal{M}_1)/\Fil^1_{\mathrm{con}}R\Gamma_{\mathrm{crys}}(\overline{C},f^*\mathcal{M}_1).
\]

That the difference $\eta^\diamond_C-\eta^\heartsuit_C$ is as stated now follows from Lemma~\ref{lem:common_kernel}.

Consider the restriction of $\varphi_0$ along the natural map
\[
\gr^1_{\mathrm{Hdg}}R\Gamma_{\mathrm{crys}}(\overline{C},\mathcal{O}_1)\otimes_C\gr^{-1}_{\mathrm{Hdg}}M_{\overline{C}} \to \gr^0_{\mathrm{Hdg}}R\Gamma_{\mathrm{crys}}(\overline{C},f^*\mathcal{M}_1):
\]
This is of course the tensor product $\varphi_1\otimes\varphi_{-1}$. I claim that this factors through
\[
\Fil^1_{\mathrm{Hdg}}R\Gamma_{\mathrm{crys}}(\overline{C},f^*\mathcal{M}_1)\to R\Gamma_{\mathrm{crys}}(\overline{C},f^*\mathcal{M}_1),
\]
and hence that the map~\eqref{eqn:diff_composition} factors through $\gr^0_{\mathrm{Hdg}}M_{\overline{C}}$

To see this, by the multiplicativity of the filtrations and the fact that 
\[
\gr^i_{\mathrm{Hdg}}R\Gamma_{\mathrm{crys}}(\overline{C},f^*\mathcal{M}_1)\simeq 0
\]
for $i<-1$, it is enough to know that the map
\[
\varphi_1:\gr^1_{\mathrm{Hdg}}R\Gamma_{\mathrm{crys}}(\overline{C},\mathcal{O}_1)\to R\Gamma_{\mathrm{crys}}(\overline{C},\mathcal{O}_1)
\]
factors through $\Fil^2_{\mathrm{Hdg}}R\Gamma_{\mathrm{crys}}(\overline{C},\mathcal{O}_1)$. This just amounts to the fact that the map
\[
\varphi_1:\gr^1D_1(g^{',(\bullet)})\to D_1(g^{',(\bullet)})
\]
factors through $\Fil^2D_1(g^{',(\bullet)})$. In fact, it follows from the construction of the map that it factors through $\Fil^pD_1(g^{',(\bullet)})$.

Now, note that the map $\varphi_0$ is $\overline{C}$-linear where $R\Gamma_{\mathrm{crys}}(\overline{C},f^*\mathcal{M}_1)$ is being viewed as a $\overline{C}$-module via the equivalence
\[
\overline{C}\xrightarrow[\simeq]{\varphi_0}\Fil^0_{\mathrm{con}}R\Gamma_{\mathrm{crys}}(\overline{C},\mathcal{O}_1).
\]
The composition of this equivalence with
\[
\Fil^0_{\mathrm{con}}R\Gamma_{\mathrm{crys}}(\overline{C},\mathcal{O}_1)\to R\Gamma_{\mathrm{crys}}(\overline{C},\mathcal{O}_1)\to \gr^0_{\mathrm{Hdg}}R\Gamma_{\mathrm{crys}}(\overline{C},\mathcal{O}_1)\simeq \overline{C}
\]
is $\varphi_{\overline{C}}$. In particular, the composition
\[
\varphi_0:\hker(\varphi_{\overline{C}})\to \overline{C} \to R\Gamma_{\mathrm{crys}}(\overline{C},\mathcal{O}_1) 
\]
factors through $\Fil^1_{\mathrm{Hdg}}R\Gamma_{\mathrm{crys}}(\overline{C},\mathcal{O}_1)$. This shows that~\eqref{eqn:difference_arrow} restricts to a nullhomotopic map on $\hker(\varphi_{\overline{C}})\otimes_{\overline{C}}\gr^0_{\mathrm{Hdg}}M_{\overline{C}}$. 
\end{proof}

\begin{proposition}
\label{prop:m1_expl}
Suppose that $\gr^i_{\mathrm{Hdg}}M_R$ is \emph{perfect} in $\Mod{R}$ for $i = -1,0$. Then the $\Mod{\Field_p}$-valued prestack $\mathcal{Z}^{\diamond}$ has the following properties:
\begin{enumerate}
  \item It is an infinitesimally cohesive, nilcomplete, integrable and locally finitely presented \'etale sheaf.
  \item If $P\in \Mod[\mathrm{cn}]{C}$, there is a canonical equivalence
\[
\hker\left(\mathcal{Z}^{\diamond}(C\oplus P)\to \mathcal{Z}^{\diamond}(C)\right)\simeq (\overline{P}[-1]\otimes_C\gr^{-1}_{\mathrm{Hdg}}M_C)\oplus (\overline{P}\otimes_C\gr^{-1}_{\mathrm{Hdg}}M_C).
\]
\end{enumerate}
\end{proposition}
\begin{proof}
We first claim that the proposition holds with $\mathcal{Z}^{\diamond}$ replaced by $\mathcal{Z}^{\heartsuit}$. By construction, we have a fiber sequence
\[
\mathcal{Z}^{\heartsuit}(C)=\Fil^0_{\mathrm{Hdg}}\Fil^1_{\mathrm{con}}R\Gamma_{\mathrm{crys}}(\overline{C},f^*\mathcal{M}_1)\to \Fil^1_{\mathrm{con}}R\Gamma_{\mathrm{crys}}(\overline{C},f^*\mathcal{M}_1)\to \gr^{-1}_{\mathrm{Hdg}}M_{\overline{C}},
\]
where the object in the middle admits an equivalence
\[
\gr^{-1}_{\mathrm{Hdg}}M_{\overline{C}}\xrightarrow{\simeq}\Fil^1_{\mathrm{con}}R\Gamma_{\mathrm{crys}}(\overline{C},f^*\mathcal{M}_1).
\]

Therefore, $\mathcal{Z}^{\heartsuit}(C)$ is equivalent to the fiber of a functorial-in-$C$ $\varphi_{\overline{C}}$-semilinear map
\begin{align}\label{eqn:f_C_funct}
f_C:\gr^{-1}_{\mathrm{Hdg}}M_{\overline{C}} \to \gr^{-1}_{\mathrm{Hdg}}M_{\overline{C}}.
\end{align}

From this, and the perfectness of $\gr^{-1}_{\mathrm{Hdg}}M_{R}$, we see that $\mathcal{Z}^{\heartsuit}$ enjoys property (1). Indeed, it is clear that it preseves filtered colimits in $C$, and our hypothesis implies that $\gr^{-1}_{\mathrm{Hdg}}M_{\overline{R}}$ is also perfect in $\Mod{R}$ and hence dualizable. Write $N\in \Mod{R}$ for its dual. Then the functor $C\mapsto \gr^{-1}_{\mathrm{Hdg}}M_{\overline{C}}$ is equivalent to the limit preserving functor $C\mapsto \underline{\Map}_{\Mod{R}}(N,C)$. 

For the analogue of property (2), note that we can identify the kernel in question with that of the map
\[
\overline{P}\otimes_C\gr^{-1}_{\mathrm{Hdg}}M_C\to \overline{P}\otimes_C\gr^{-1}_{\mathrm{Hdg}}M_C
\]
induced by the $\varphi$-semilinear map $f_{C\oplus P}$. Since $\varphi_{\overline{C}\oplus \overline{P}}$ restricts to a nullhomotopic map on $\overline{P}$, the map above is also nullhomotopic, and so its kernel is canonically identified with
\[
(\overline{P}[-1]\otimes_C\gr^{-1}_{\mathrm{Hdg}}M_C) \oplus (\overline{P}\otimes_C\gr^{-1}_{\mathrm{Hdg}}M_C),
\] 
as desired.

To go from $\mathcal{Z}^{\heartsuit}$ to $\mathcal{Z}^{\diamond}$, we will use Lemma~\ref{lem:exts_equivalent_restr}. Consider the prestack $\mathbb{K}$ over $R$ given by
\[
\mathbb{K}(C) = \hker\left(\mathcal{Z}^{\heartsuit}(C)\to \gr^0_{\mathrm{Hdg}}M_{\overline{C}}\right)
\]
where the map in question is the composition
\[
\mathcal{Z}^{\heartsuit}(C)\to \gr^0_{\mathrm{con}}R\Gamma_{\mathrm{crys}}(\overline{C},f^*\mathcal{M}_1)\to \gr^0_{\mathrm{Hdg}}M_{\overline{C}},
\]
where the second arrow is obtained via the inverse to the equivalence~\eqref{eqn:varphi0equiv}. Note that there is a natural map
\[
j_C:\Fil^1_{\mathrm{con}}\mathcal{Z}^{\diamond}(C)\to \mathbb{K}(C)
\]
such that we have a diagram of fiber sequences
\[
\begin{diagram}
\Fil^1_{\mathrm{con}}\mathcal{Z}^{\diamond}(C)&\rTo&\mathcal{Z}^{\heartsuit}(C)&\rTo&\gr^0_{\mathrm{con}}R\Gamma_{\mathrm{crys}}(\overline{C},f^*\mathcal{M}_1)\\
\dTo^{j_C}&&\dEquals&&\dTo\\
\mathbb{K}(C)&\rTo&\mathcal{Z}^{\heartsuit}(C)&\rTo& \gr^0_{\mathrm{Hdg}}M_{\overline{C}}
\end{diagram}
\]

Set $\overline{\eta}^\heartsuit_C = j_C[1]\circ \eta^\heartsuit_C$, and consider the composition
\[
\delta_C:\gr^0_{\mathrm{Hdg}}M_{\overline{C}}\xrightarrow{\eqref{eqn:difference_arrow}}\Fil^1_{\mathrm{con}}\mathcal{Z}^{\diamond}(C)\to \mathbb{K}(C)[1].
\]
Then Lemma~\ref{lem:exts_equivalent_restr} shows that we have a natural equivalence
\begin{equation}\label{eqn:M1_K_fiber}
\mathcal{Z}^{\diamond}(C)\simeq \hker(\overline{\eta}^\heartsuit_C+\delta_C:\gr^0_{\mathrm{Hdg}}M_{\overline{C}}\to \mathbb{K}(C)).
\end{equation}

Given that $\mathcal{Z}^{\heartsuit}$ enjoys the properties in (1) and that $\gr^0_{\mathrm{Hdg}}M_R$ is perfect, one finds that the properties in (1) are also true for $\mathbb{K}$. Combining this with~\eqref{eqn:M1_K_fiber} and using the perfectness of $\gr^0_{\mathrm{Hdg}}M_R$ once again then implies the same for $\mathcal{Z}^{\diamond}$.

As for (2), for any prestack $X$ over $R$ valued in $\Mod{\Int}$, set
\[
F_X(C,P) = \hker(X(C\oplus P)\to X(C)).
\]
Then by the last assertion of Lemma~\ref{lem:exts_equivalent_restr} the map
\[
\overline{P}\otimes_{\overline{C}}\gr^0_{\mathrm{Hdg}}M_{\overline{C}} \to F_{\mathbb{K}}(C,P)
\]
induced from $\delta_{C\oplus P}$ is nullhomotopic: indeed, $\overline{P}\to \overline{C}\oplus \overline{P}$ factors through $\hker(\varphi_{\overline{C}\oplus \overline{P}})$.

This gives us natural equivalences
\begin{align*}
F_{\mathcal{Z}^{\diamond}}(C,P)&\simeq \hker(\overline{\eta}^\heartsuit_{C\oplus P} + \delta_{C\oplus P}:\overline{P}\otimes_{\overline{C}}\gr^0_{\mathrm{Hdg}}M_{\overline{C}} \to F_{\mathbb{K}}(C,P))\\
&\simeq\hker(\overline{\eta}^\heartsuit_{C\oplus P}:\overline{P}\otimes_{\overline{C}}\gr^0_{\mathrm{Hdg}}M_{\overline{C}} \to F_{\mathbb{K}}(C,P))\\
&\simeq F_{\mathcal{Z}^{\heartsuit}}(C,P).
\end{align*}
This reduces assertion (2) to its analogue for $\mathcal{Z}^{\heartsuit}$, which we have already verified.
\end{proof}


\begin{proof}[Proof of Theorem~\ref{thm:FilFCrys_functorial_props}]
Assertion~\eqref{ffc:cohesive} amounts to saying that, for every $n$, $\mathcal{Z}(\mathcal{M}_n)$ respects certain limit diagrams in $\mathrm{CRing}_{R/}$.  Since $\mathcal{Z}(\mathcal{M}_n) \simeq\tau^{\leq 0}\mathcal{Z}^+_n$, this would follow if we knew that $\mathcal{Z}^+_n$ respected such diagrams for each $n$ (in the $\infty$-category $\Mod{\Int/p^n\Int}$). Using d\'evissage and fiber sequences of the form $\mathcal{Z}^+_{n-1}\to \mathcal{Z}^+_n\to \mathcal{Z}^+_1$, one reduces to the case $n=1$. Using the fiber sequence~\eqref{eqn:Mmodp_M1}, it is sufficient to check that $\mathcal{Z}^{\diamond}$ and $C\mapsto \gr^{-1}_{\mathrm{Hdg}}M_{\overline{C}}$ each respects the limit diagrams in question. For the former, this is (1) of Proposition~\ref{prop:m1_expl}, and for the latter, it follows from the perfectness of $\gr^{-1}_{\mathrm{Hdg}}M_R$.

We still have to show assertion~\eqref{ffc:cotangent}. For this, first note that it suffices to construct, for every $n\geq 1$, a canonical equivalence
\[
F_{\mathcal{Z}^+_n}(C,P)\xrightarrow{\simeq}P[-1]\otimes_R(\gr^{-1}_{\mathrm{Hdg}}M_R/{}^{\mathbb{L}}p^n).
\]

Now, note that we have a canonical equivalence
\begin{align}\label{eqn:Mmodp_tangent_space}
F_{\mathcal{Z}^+_1}(C,P)\xrightarrow{\simeq} \overline{P}[-1]\otimes_R\gr^{-1}_{\mathrm{Hdg}}M_R.
\end{align}
This follows from assertion (2) of Proposition~\ref{prop:m1_expl} and the fact that the projection
\[
\hker(\mathcal{Z}^{\diamond}(C\oplus P)\to \mathcal{Z}^{\diamond}(C))\to \overline{P}\otimes_C\gr^{-1}_{\mathrm{Hdg}}M_C
\]
given by that assertion is precisely the induced map on tangent spaces associated with the map
\[
\mathcal{Z}^{\diamond}(C)\to\gr^{-1}_{\mathrm{Hdg}}M_{\overline{C}} 
\]
involved in~\eqref{eqn:Mmodp_M1}.

Set
\[
\tilde{C}_n = C/{}^{\mathbb{L}}p^n \oplus (P/{}^{\mathbb{L}}p^n )\;;\; C_n = C/{}^{\mathbb{L}}p^n\;;\;
\]

By Lemma~\ref{lem:square_zero_pd}, for every $n\geq 1$, the arrow $\tilde{C}_n\twoheadrightarrow C_n$, admits a canonical lift to $\mathrm{AniPDPair}_{p,n,R/}$, and so we can evaluate $\Fil^0_{\mathrm{Hdg}}\mathcal{M}$ on it to obtain an object
\[
\widetilde{\Fil}^0\tilde{\mathsf{M}}_n = \Fil^0_{\mathrm{Hdg}}\mathcal{M}(\tilde{C}_n\twoheadrightarrow C_n)\in \Mod{\tilde{C}_n}\in \Mod{\tilde{C}_n}.
\]
On the other hand we can also evaluate it on the trivial thickening $\tilde{C}_n\xrightarrow{\mathrm{id}}\tilde{C}_n$ and obtain another object
\[
\Fil^0\tilde{\mathsf{M}}_n = \Fil^0_{\mathrm{Hdg}}\mathcal{M}(\tilde{C}_n\xrightarrow{\mathrm{id}}\tilde{C}_n)\in \Mod{\tilde{C}_n}\in \Mod{\tilde{C}_n}.
\]

Note that we have a canonical map $\Fil^0\tilde{\mathsf{M}}_n\to \widetilde{\Fil}^0\tilde{\mathsf{M}}_n$ obtained from the map of pairs $(\tilde{C}_n\xrightarrow{\mathrm{id}}\tilde{C}_n)\to (\tilde{C}_n\twoheadrightarrow C_n)$. We claim that we have a fiber sequence
\[
(P/{}^\mathbb{L}p^n)[-1]\otimes_R\gr^{-1}_{\mathrm{Hdg}}M_R \to \Fil^0\tilde{\mathsf{M}}_n\to \widetilde{\Fil}^0\tilde{\mathsf{M}}_n.
\]
Indeed, the middle object is just $\tilde{C}_n\otimes_R\Fil^0M_R$, while the third is equivalent to $(C_n\otimes_R\Fil^0M_R)\oplus (P/{}^{\mathbb{L}}p^n\otimes_RM_R)$. Here, of course we are viewing $\tilde{C}$ as an $R$-algebra via the trivial lifting $C\to \tilde{C}$. Denote the corresponding arrow by $\tilde{f}:R\to \tilde{C}$.

The proof of the theorem is now completed by the following lemma.
\begin{lemma}
\label{lem:canonical_arrow_P_kernel}
There exists a canonical commuting square 
\[
\begin{diagram}
\mathcal{Z}^+_n(\tilde{C})&\rTo&\mathcal{Z}^+_n(C)\\
\dTo&&\dTo\\
\Fil^0\widetilde{\mathsf{M}}_n&\rTo&\widetilde{\Fil}^0\tilde{\mathsf{M}}_n
\end{diagram}
\]
compatible for varying $n\geq 1$ such that the ensuing arrow
\[
F_{\mathcal{Z}^+_n}(C,P)\to \hker(\Fil^0\widetilde{\mathsf{M}}_n\to\widetilde{\Fil}^0\tilde{\mathsf{M}}_n)\simeq (P/{}^{\mathbb{L}}p^n)[-1]\otimes_R\gr^1_{\mathrm{Hdg}}M_R
\]
is equivalent to~\eqref{eqn:Mmodp_tangent_space} when $n=1$.
\end{lemma}
\begin{proof}
We of course proceed by looking at $\tilde{g}':\tilde{S}\twoheadrightarrow \tilde{C}$ with $\tilde{S} = \Int[\pi_0(C)\oplus \pi_0(P)]$. Write $\tilde{g}'_1:\tilde{S}\twoheadrightarrow C$ for the composition of this map with the projection onto $C$. Then by Proposition~\ref{prop:crystals_explicit}, the canonical map $D_n(\tilde{g}_1^{',(\bullet)})\to D_n(g^{',(\bullet)})$ induces an equivalence
\begin{align}\label{eqn:RGamma_crys_alt}
\mathrm{Tot}\Fil^0_{\mathrm{Hdg}}\mathcal{M}\left((D_n(\tilde{g}_1^{',(\bullet)})\twoheadrightarrow C_n,\gamma_n(\tilde{g}_1^{',(\bullet)}))\right)\xrightarrow{\simeq} \Fil^0_{\mathrm{Hdg}}R\Gamma_{\mathrm{crys}}(C,f^*\mathcal{M}_n).
\end{align}

Now, via $\tilde{g}'$, we obtain canonical maps
\begin{align*}
\Fil^0_{\mathrm{Hdg}}\mathcal{M}\left((D_n(\tilde{g}^{',(\bullet)})\twoheadrightarrow \tilde{C}_n,\gamma_n(\tilde{g}^{',(\bullet)}))\right)&\to \Fil^0\tilde{\mathsf{M}}_n\\
\Fil^0_{\mathrm{Hdg}}\mathcal{M}\left((D_n(\tilde{g}_1^{',(\bullet)})\twoheadrightarrow C,\gamma_n(\tilde{g}_1^{',(\bullet)}))\right)&\to \widetilde{\Fil}^0\tilde{\mathsf{M}}_n.
\end{align*}

Combining these with with the natural maps
\begin{align*}
\mathcal{Z}^+_n(\tilde{C})\to \Fil^0_{\mathrm{Hdg}}R\Gamma_{\mathrm{crys}}(\tilde{C},\tilde{f}^*\mathcal{M}_n)&\simeq\mathrm{Tot}\Fil^0_{\mathrm{Hdg}}\mathcal{M}\left((D_n(\tilde{g}^{',(\bullet)})\twoheadrightarrow \tilde{C}_n,\gamma_n(\tilde{g}^{',(\bullet)}))\right)\\
\mathcal{Z}^+_n(C)\to \Fil^0_{\mathrm{Hdg}}R\Gamma_{\mathrm{crys}}(C,f^*\mathcal{M}_n)&\xrightarrow[\eqref{eqn:RGamma_crys_alt}]{\simeq}\mathrm{Tot}\Fil^0_{\mathrm{Hdg}}\mathcal{M}\left((D_n(\tilde{g}_1^{',(\bullet)})\twoheadrightarrow C_n,\gamma_n(\tilde{g}_1^{',(\bullet)}))\right)
\end{align*}
now gives us the desired commuting square. That the ensuing arrow for $n=1$ is equivalent to~\eqref{eqn:Mmodp_tangent_space} is a matter of unraveling the definitions.
\end{proof}

\end{proof}

The rest of this section will be devoted to the proof of Theorem~\ref{thm:FilFCrys_discrete}, so we will assume from now on that $\mathcal{M}$ satisfies the hypotheses of the theorem.

\begin{proposition}
\label{prop:m1_repble}
The $p$-adic formal prestack $\mathcal{Z}(\mathcal{M}_1)^{\form}_p$ underlying $\mathcal{Z}(\mathcal{M}_1)$ is represented by a locally quasi-smooth $p$-adic formal derived scheme over $R$. 
\end{proposition}
\begin{proof}
The fiber sequence~\eqref{eqn:Mmodp_M1} combined with assertion (1) of Proposition~\ref{prop:m1_expl} shows that $\mathcal{Z}^+_1$, as an $\Mod{\Field_p}$-valued prestack, is a locally finitely presented \'etale sheaf. By exactness of filtered colimits, $\mathcal{Z}(\mathcal{M}_1)^{\form}_p$ is also locally finitely presented.

Next, assertion (2) of Theorem~\ref{thm:FilFCrys_functorial_props} shows that we have
\[
F_{\mathcal{Z}(\mathcal{M}_1)^{\form}_p}(C,P)\simeq \tau^{\leq 0}(\overline{P}[-1]\otimes_C\gr^{-1}_{\mathrm{Hdg}}M_C).
\]
Since $\gr^{-1}_{\mathrm{Hdg}}M_R$ is a finite projective object in $\Mod{R}$ under our hypothese, this shows that $\mathcal{Z}(\mathcal{M}_1)^{\form}_p$ admits a cotangent complex over $R$ of Tor amplitude $[-1,0]$; in fact, we have shown
\[
\mathbb{L}_{\mathcal{Z}(\mathcal{M}_1)^{\form}_p/R}\simeq \Reg{\mathcal{Z}(\mathcal{M}_1)^{\form}_p}\otimes_R(\gr^{-1}_{\mathrm{Hdg}}M_{\overline{R}})^\vee[1].
\]

Given the two previous paragraphs, assertion (1) of Theorem~\ref{thm:FilFCrys_functorial_props}, Lurie's representability criterion Theorem~\ref{thm:lurie_representability}, and~\cite[\href{https://stacks.math.columbia.edu/tag/03XX}{Tag 03XX}]{stacks-project}, knowing that $\mathcal{Z}(\mathcal{M}_1)^{\form}_p$ is representable by a locally quasi-smooth $p$-adic formal derived scheme over $R$ amounts to knowing that $\mathcal{Z}(\mathcal{M}_1)^{\form}_p(C)$ is discrete for all $C\in \mathrm{CRing}_{\heartsuit,R/}$; or, equivalently, that
\[
H^i(\mathcal{Z}^+_1(C)) = 0.
\] 
for $i\leq -1$.

We will repeatedly use the fact that, for $C\in \mathrm{CRing}_{\heartsuit}$, $\overline{C} = C/{}^{\mathbb{L}}p$ is a square-zero thickening of the discrete ring $C/pC$ by $(C[p])[1]$, where $C[p]\subset C$ is the ideal of $p$-torsion elements. In particular, for every locally free module $N$ over $C$, we have
\[
H^i(N_{\overline{C}}) = \begin{cases}
C[p]\otimes_CN&\text{if $i=-1$};\\
N_{C/pC}&\text{if $i=0$};\\
0&\text{otherwise}.
\end{cases}
\]

The fiber sequence~\eqref{eqn:Mmodp_M1} reduces us to showing that the homotopy kernel of the associated map
\begin{align}\label{eqn:M1_to_gr-1}
\mathcal{Z}^{\diamond}(C)\to \gr^{-1}_{\mathrm{Hdg}}M_{\overline{C}}.
\end{align}
is $0$-coconnective. This amounts to saying that $H^i(\mathcal{Z}^{\diamond}(C)) = 0$ for $i<-1$ and that the map
\begin{align}\label{eqn:H-1_map_inj}
H^{-1}(\mathcal{Z}^{\diamond}(C)) \to H^{-1}(\gr^{-1}_{\mathrm{Hdg}}M_{\overline{C}}) \simeq C[p]\otimes_C\gr^{-1}_{\mathrm{Hdg}}M_C
\end{align}
is injective. 

Just as in Proposition~\ref{prop:m1_expl}, we will first consider the analogues of these claims for $\mathcal{Z}^{\heartsuit}(C)$. Here, the counterpart to~\eqref{eqn:M1_to_gr-1} is the first map in the fiber sequence
\[
\mathcal{Z}^{\heartsuit}(C)\to \gr^{-1}_{\mathrm{Hdg}}M_{\overline{C}}\xrightarrow{f_C}\gr^{-1}_{\mathrm{Hdg}}M_{\overline{C}},
\]
Since $H^i(\gr^{-1}_{\mathrm{Hdg}}M_{\overline{C}}) = 0$ for $i<-1$, we see that $H^i(\mathcal{Z}^{\heartsuit}(C)) = 0$ for $i<-1$. Furthermore, the map
\[
H^1(f_C):C[p]\otimes_C\gr^{-1}_{\mathrm{Hdg}}M_{C}\to C[p]\otimes_C\gr^{-1}_{\mathrm{Hdg}}M_{C}
\]
is $0$ by virtue of the fact that $f_C$ is $\varphi_{\overline{C}}$-semilinear. This shows that the map
\[
H^{-1}(\mathcal{Z}^{\heartsuit}(C)) \to H^{-1}(\gr^{-1}_{\mathrm{Hdg}}M_{\overline{C}}) \simeq C[p]\otimes_C\gr^{-1}_{\mathrm{Hdg}}M_C
\]
is in fact an isomorphism.

Next, consideration of the fiber sequence
\begin{align}\label{eqn:mheart_K_seq}
\mathcal{Z}^{\heartsuit}(C)\to \gr^0_{\mathrm{Hdg}}M_{\overline{C}}\xrightarrow{\overline{\eta}_C^\heartsuit}\mathbb{K}(C)[1]
\end{align}
now shows that we have $H^i(\mathbb{K}(C)) = 0$ for $i<-1$ and that the first map in the sequence
\[
H^{-1}(\mathbb{K}(C))\to H^{-1}(\mathcal{Z}^{\heartsuit}(C))\to H^{-1}(\gr^{0}_{\mathrm{Hdg}}M_{\overline{C}}) \xrightarrow{H^{-1}(\overline{\eta}^\heartsuit_C)} H^0(\mathbb{K}(C))
\]
is injective. We claim that the second arrow in the sequence is zero, which of course implies that the first arrow is actually an isomorphism, and that the last arrow is injective. To prove the claim, consider the diagram
\[
\begin{diagram}
H^{-1}(\mathcal{Z}^{\heartsuit}(C))&\rTo& H^{-1}(\gr^{-1}_{\mathrm{Hdg}}M_{\overline{C}}) = C[p]\otimes_C\gr^{-1}_{\mathrm{Hdg}}M_C\\
\dTo&&\dTo_{\varphi_{-1}}\\
H^{-1}(\gr^0_{\mathrm{Hdg}}M_{\overline{C}})&\rTo&H^{-1}(M_{\overline{C}}/\Fil^1_{\mathrm{Hdg}}M_{\overline{C}})
\end{diagram}
\]
where the top and bottom arrows are both injective and the right arrow---being $\varphi_{\overline{C}}$-semilinear---is zero.

We can now use the fiber sequence (see the proof of Proposition~\ref{prop:m1_expl})
\begin{align}\label{eqn:m+1_K_seq}
\mathcal{Z}^{\diamond}(C)\to \gr^0_{\mathrm{Hdg}}M_{\overline{C}}\xrightarrow{\overline{\eta}_C^\heartsuit+\delta_C}\mathbb{K}(C)[1]
\end{align}
to conclude that $H^i(\mathcal{Z}^{\diamond}(C)) = 0$ for $i<-1$, and that we have an exact sequence
\[
0\to H^{-1}(\mathbb{K}(C))\to H^{-1}(\mathcal{Z}^{\diamond}(C))\to H^{-1}(\gr^{0}_{\mathrm{Hdg}}M_{\overline{C}}) \xrightarrow{H^{-1}(\overline{\eta}^\heartsuit_C+\delta_C)} H^0(\mathbb{K}(C)).
\]

Since $\delta_C$ restricts to a nullhomotopic map on $\hker(\varphi_{\overline{C}})\otimes_C\gr^0_{\mathrm{Hdg}}M_C$, and since 
\[
H^{-1}(\gr^0_{\mathrm{Hdg}}M_{\overline{C}}) = C[p]\otimes_C\gr^0_{\mathrm{Hdg}}M_C,
\]
we see that the last map in the exact sequence above is equal to $H^{-1}(\overline{\eta}^\heartsuit_C)$, and is therefore injective. This shows that we have isomorphisms
\[
C[p]\otimes_C\gr^{-1}_{\mathrm{Hdg}}M_C\xleftarrow\simeq H^{-1}(\mathcal{Z}^{\heartsuit}(C))\xleftarrow{\simeq}H^{-1}(\mathbb{K}(C))\xrightarrow{\simeq}H^{-1}(\mathcal{Z}^{\diamond}(C)).
\]

To finish the proof of the discreteness of $\mathcal{Z}(\mathcal{M}_1)^+(C)$, note now that the composition
\[
C[p]\otimes_C\gr^{-1}_{\mathrm{Hdg}}M_C\xrightarrow{\simeq}H^{-1}(\mathcal{Z}^{\diamond}(C))\xrightarrow{\eqref{eqn:H-1_map_inj}}C[p]\otimes_C\gr^{-1}_{\mathrm{Hdg}}M_C
\]
is the identity.
\end{proof}

\begin{lemma}
\label{lem:M(k)_is_easy}
Suppose that $C\in \mathrm{CRing}_{\heartsuit,\overline{R}/}$ is perfect. Then with the notation as in~\eqref{subsec:M_on_perfect}, we have
\[ 
\mathcal{Z}(\mathcal{M}_n)(C) \simeq \left(\Fil^0_{\mathrm{Hdg}}\mathcal{M}(W(C))/p^n\right)^{\varphi_0 = \mathrm{id}}\;;\; \mathcal{Z}(\mathcal{M})(C) \simeq \Fil^0_{\mathrm{Hdg}}\mathcal{M}(W(C))^{\varphi_0 = \mathrm{id}}.
\]
\end{lemma}
\begin{proof}
First, note that both $\Fil^0_{\mathrm{Hdg}}\mathcal{M}(W(C))$ and $\mathcal{M}(W(C))$ are $p$-torsion free, so that the derived quotient mod-$p^n$ agrees with the na\"ive one for them both.

Let $\Fil^\bullet_{\mathrm{Hdg}}M_C$ be the filtered base change of $\Fil^\bullet_{\mathrm{Hdg}}\mathcal{M}(\widehat{D}(C))$ over $k$. Then by Lemma~\ref{lem:associated_graded} we have
\[
\gr^{-1}_{\mathrm{Hdg}}\mathcal{M}(W(C)) \simeq \gr^{-1}_{\mathrm{Hdg}}\mathcal{M}(\widehat{D}(C))\simeq \gr^{-1}_{\mathrm{Hdg}}M_C.
\]

Let $J = \ker(\widehat{D}(C)\to W(C))$; then we have short exact sequences
\begin{align*}
0\to J\otimes_{\widehat{D}(C)}\mathcal{M}(\widehat{D}(C))&\to \Fil^0_{\mathrm{Hdg}}\mathcal{M}(\widehat{D}(C))\to \Fil^0_{\mathrm{Hdg}}\mathcal{M}(W(C))\to 0\\
0\to J\otimes_{\widehat{D}(C)}\mathcal{M}(\widehat{D}(C))&\to \mathcal{M}(\widehat{D}(C))\to \mathcal{M}(W(C))\to 0.
\end{align*}

Taking the derived quotient mod $p^n$ gives us fiber sequences
\begin{align*}
(J\otimes_{\widehat{D}(C)}\mathcal{M}(\widehat{D}(C)))/{}^{\mathbb{L}}p^n&\to (\Fil^0_{\mathrm{Hdg}}\mathcal{M}(\widehat{D}(C)))/{}^{\mathbb{L}}p^n\to \Fil^0_{\mathrm{Hdg}}\mathcal{M}(W(C))/p^n\\
(J\otimes_{\widehat{D}(C)}\mathcal{M}(\widehat{D}(C)))/{}^{\mathbb{L}}p^n&\to \mathcal{M}(\widehat{D}(C))/{}^{\mathbb{L}}p^n\to \mathcal{M}(W(C))/p^n.
\end{align*}

As in the proof of Lemma~\ref{lem:FilFCrys_global_secs_cosimp}, we have
\[
\mathcal{Z}(\mathcal{M}_n)(C) \simeq \hker^{\mathrm{cn}}\left((\Fil^0_{\mathrm{Hdg}}\mathcal{M}(\widehat{D}(C)))/{}^{\mathbb{L}}p^n\xrightarrow{\varphi_0-\mathrm{id}}\mathcal{M}(\widehat{D}(C))/{}^{\mathbb{L}}p^n\right).
\]

So to prove the first identity, it is enough to know that $\varphi_0- \mathrm{id}$ induces an isomorphism from $J\otimes_{\widehat{D}(C)}\mathcal{M}(\widehat{D}(C))$ to itself. This can be checked after taking the derived quotient mod $p$, where in fact it is equal to $-\mathrm{id}$ and hence is an isomorphism.

For the second identity, set 
\[
\mathcal{Z}^+(C) = \hker(\Fil^0_{\mathrm{Hdg}}\mathcal{M}(W(C))\xrightarrow{\varphi_0-\mathrm{id}}\mathcal{M}(W(C)))\in \Mod[p\text{-comp}]{\Int}.
\]
Note that we have
\[
H^0(\mathcal{Z}^+(C))\simeq \Fil^0_{\mathrm{Hdg}}\mathcal{M}(W(C))^{\varphi_0 = \mathrm{id}}
\]

Then $\mathcal{Z}(\mathcal{M}_n)(C) \simeq H^0(\mathcal{Z}^+(C)/{}^{\mathbb{L}}p^n)$, and we have a fiber sequence
\[
H^0(\mathcal{Z}^+(C))/{}^{\mathbb{L}}p^n\to \mathcal{Z}(\mathcal{M}_n)(C)\to H^1(\mathcal{Z}^+(C))[p^n],
\]
where on the right, we have the $p^n$-torsion in $H^1(\mathcal{Z}^+(C))$ placed in degree $0$. 

Taking the connective limit over $n$ gives us a fiber sequence
\[
H^0(\mathcal{Z}^+(C))\to \mathcal{Z}(\mathcal{M})(C)\to \tau^{\leq 0}\varprojlim_n H^1(\mathcal{Z}^+(C))[p^n],
\]
where the transition maps on the right are given by multiplication by $p$. 

To finish the proof, we need to know that the limit on the right is nullhomotopic, which follows from the derived $p$-completeness of $H^1(\mathcal{Z}^+(C))$. In more detail: the $p$-completeness implies that we have a fiber sequence
\[
\varprojlim_n (H^1(\mathcal{Z}^+(C))[p^n])[1]\to \bigl(H^1(\mathcal{Z}^+(C))\simeq\varprojlim_n H^1(\mathcal{Z}^+(C))/{}^{\mathbb{L}}p^n\bigr)\to \varprojlim_n H^1(\mathcal{Z}^+(C))/p^nH^1(\mathcal{Z}^+(C)).
\]
The middle term here is discrete and the right-most term is coconnective, which shows that we have
\[
H^i\left(\varprojlim_n H^1(\mathcal{Z}^+(C))[p^n]\right) \simeq H^{i-1}\left(\varprojlim_n (H^1(\mathcal{Z}^+(C))[p^n])[1]\right) = 0
\]
for $i-1\leq -1\Leftrightarrow i\leq 0$.
\end{proof}

\begin{proof}[Proof of Theorem~\ref{thm:FilFCrys_discrete}]
For simplicity, set $\mathcal{Z} = \mathcal{Z}(\mathcal{M})$.

Assertion~\eqref{ffc:form_etale} was shown during the proof of assertion~\eqref{ffc:cotangent} of Theorem~\ref{thm:FilFCrys_functorial_props}.

Assertion~\eqref{ffc:perfect} is immediate from Lemma~\ref{lem:M(k)_is_easy}.

Assertion~\eqref{ffc:rep} for the case $n=1$ is Proposition~\ref{prop:m1_repble}. The general case can be deduced via the same argument: The only thing to check is the discreteness of $\mathcal{Z}(\mathcal{M}_n)(C)$ for $C\in \mathrm{CRing}_{\heartsuit,R/}$, and this can be reduced to the $n=1$ case by d\'evissage via the canonical fiber sequence
\[
\mathcal{Z}^+_1(C)\xrightarrow{p^{n-1}} \mathcal{Z}^+_n(C)\to \mathcal{Z}^+_{n-1}(C).
\]

Before we go further, we need the following lemma:
\begin{lemma}
\label{lem:mformp_discrete}
Suppose that $C\in \mathrm{CRing}_{\heartsuit,R/}$ is a finitely generated discrete $\Int_p$-algebra. Then $\mathcal{Z}(\mathcal{M})^{\form}_p(C)$ is flat over $\Int_p$---and in particular discrete.
\end{lemma}
\begin{proof}
First, note that, for any surjection $C'\twoheadrightarrow C$ in $\mathrm{CRing}^{p\text{-nilp}}_{\heartsuit,R/}$ with square-zero kernel $I$, assertion~\eqref{ffc:form_etale} shows that we have a fiber sequence
\[
\mathcal{Z}(C')\to \mathcal{Z}(C)\to I\otimes_R\gr^{-1}_{\mathrm{Hdg}}M_R.
\]
Therefore, $\mathcal{Z}(C')$ is flat over $\Int_p$ if and only if $\mathcal{Z}(C)$ is so. Using this observation, we can reduce the proof of the lemma to the case where $C\in \mathrm{CRing}_{\heartsuit,\overline{R}/}$ is a finitely generated $\Field_p$-algebra.

By fpqc descent for $\mathcal{Z}$---which is a consequence of~\eqref{ffc:rep}, but can also be deduced as in the proof of assertion~\eqref{ffc:cohesive} of Theorem~\ref{thm:FilFCrys_functorial_props}---we can reduce the proof of this assertion to the case where $C$ is the quotient of the perfection $C'$ of a finitely generated polynomial algebra by a finitely generated ideal $I$. Let $\tilde{C}$ be the completion of $C'$ along $I$. The proof of assertion~\eqref{ffc:cohesive} of Theorem~\ref{thm:FilFCrys_functorial_props} shows that we have
\[
\mathcal{Z}(\tilde{C}) \simeq \varprojlim_{n}\mathcal{Z}(\tilde{C}/I^n\tilde{C}),
\]
and the argument from the first paragraph shows that the map
\[
H^i(\mathcal{Z}(\tilde{C}/I^nC))\to H^i(\mathcal{Z}(C'/I))
\]
is an isomorphism for all $i<0$. Therefore, it is enough to show that $\mathcal{Z}(\tilde{C})$ is flat over $\Int_p$, which reduces us to the case of perfect $\Field_p$-algebras, where we are done by assertion~\eqref{ffc:perfect}.
\end{proof}

We now move on to~\eqref{ffc:open_closed}: Just as in the proof of~\eqref{fic:separated} of Proposition~\ref{prop:FIC_discrete}, we reduce to the claim that the map in question is an open and closed immersion on classical truncations. This amounts to showing that $\mathcal{Z}^{\form}_p(C)\to \mathcal{Z}^{\form}_p(C')$ is injective for all maps $C\to C'$ in $\mathrm{CRing}^{p-\text{nilp}}_{\heartsuit,R}$ with $\Spec C$ connected and $C$ finitely generated over $\Int_p$.

Using Lemma~\ref{lem:mformp_discrete} and the observation from the first paragraph of its proof, we see that $\mathcal{Z}^{\form}_p(C)\to \mathcal{Z}^{\form}_p(C/pC)$ is an injective map of abelian groups for any $C\in \mathrm{CRing}^{p-\text{nilp}}_{\heartsuit,R/}$. Therefore, we can first reduce the desired injectivity statement to the case where $C$ is an $\Field_p$-algebra. 

Here, we will actually show that
\[
   V(y) = (\Spec C\times_{y,\mathcal{Z}^{\form}_p,0}\Spec C)^{\mathrm{cl}}
\]
maps isomorphically onto $\Spec C$. 

We first claim that it is represented by a closed subscheme of $\Spec C$ and is hence equal to $\Spec C/I_y$ for some ideal $I_y\subset R$.  For this, we will begin by showing that, for any $C\in \mathrm{CRing}_{\heartsuit,\overline{R}/}$, and $y\in \mathcal{Z}^{\form}_p(C)$ inducing $\overline{y}\in \mathcal{Z}(\mathcal{M}_1)^{\form}_p(C)$, the locus
\[
V_1(y) = (\Spec C\times_{\overline{y},\mathcal{Z}(\mathcal{M}_1)^{\form}_p,0}\Spec C)^{\mathrm{cl}}
\]
is represented by a closed subscheme of $\Spec C$. Indeed, iterating this claim, we find a sequence of closed subschemes
\[
\cdots \subset V_n(y)\subset V_{n-1}(y)\subset \cdots \subset V_1(y)\subset \Spec C
\] 
such that $V_n(y)$ parameterizes those $D\in \mathrm{CRing}_{\heartsuit,C/}$ where $y$ maps to $p^n\mathcal{Z}^{\form}_p(D)$.

We have $V(y) = \bigcap_{n\ge 1}V_n(y)$, and the noetherianity of $C$ ensures that $V(y) = V_n(y)$ for $n$ sufficiently large. 

The claim that $V_1(y)$ is represented by a closed subscheme of $\Spec C$ comes down to knowing that the underlying classical algebraic space of $\mathcal{Z}(\mathcal{M}_1)^{\form}_p$---that is, its restriction to $\mathrm{CRing}_{\heartsuit,\overline{R}/}$---is \emph{separated}.

Since everything here is locally Noetherian, we can check this using the valuative criterion for separatedness~\cite[\href{https://stacks.math.columbia.edu/tag/03KV}{Tag 03KV}]{stacks-project}. That is, we have to check that, for any valuation ring (or actually a \emph{discrete} valuation ring) $C$ in $\mathrm{CRing}_{\heartsuit,\overline{R}/}$ with fraction field $F$, the map
\[
\mathcal{Z}(\mathcal{M}_1)^{\form}_p(C)\to \mathcal{Z}(\mathcal{M}_1)^{\form}_p(F)
\]
is injective. We can reduce to the case where $C$ is perfect, where this follows immediately from Lemma~\ref{lem:M(k)_is_easy}. 

To finish, we need to show that $I_y = 0$. For this, it is enough to show that $I_y\otimes_C\widehat{C}_{\mathfrak{m}} = 0$ for all maximal ideals $\mathfrak{m}\subset C$. 

Now, for any maximal ideal $\mathfrak{m}\subset C$, the natural map (of abelian groups) $\mathcal{Z}^{\form}_p(\widehat{C}_{\mathfrak{m}})\to \mathcal{Z}^{\form}_p(C/\mathfrak{m})$ is injective. This shows that the map $\widehat{C}_{\mathfrak{m}}\to (C/I_y)^{\wedge}_{\mathfrak{m}}$ is an isomorphism. By the Noetherianity of $C$ and Nakayama's lemma, this shows that $I_y\otimes_C\widehat{C}_{\mathfrak{m}} = 0$, as desired.
\end{proof}

\section{Abelian schemes, crystals and derived homomorphisms}\label{sec:abvar_crystals}

Suppose that we are given two abelian schemes $X,Y$ over $R\in \mathrm{CRing}_{\heartsuit}$. The goal of this section is to prove Theorem~\ref{introthm:end} about the existence of a locally quasi-smooth derived scheme of maps $\mathbb{H}(X,Y)$ between them. Along the way we will prove a useful technical result on integrating quotients of cotangent complexes on certain derived schemes to closed derived subschemes via a Beauville-Laszlo gluing argument.

\subsection{}\label{subsec:H0defn}
We can think of $X$ and $Y$ as abelian group objects in $\mathrm{PStk}_R$, that is, as presheaves
\[
\mathrm{CRing}_{R/}\to \Mod[\mathrm{cn}]{\Int}.
\]
See Appendix~\ref{sec:modules} for an explanation of this. 

Associated with them is the $\Mod[\mathrm{cn}]{\Int}$-valued prestack
\begin{align*}
\widetilde{\mathbb{H}}(X,Y):\mathrm{Aff}_{R}^{\mathrm{op}}&\to \Mod[\mathrm{cn}]{\Int}\\
C&\mapsto \mathrm{hker}^{\mathrm{cn}}\left(\Map(X_C,Y_C)\xrightarrow{\circ [0]} Y(C)\right),
\end{align*}
where the map here is composition with the zero section of $X$. Note that we are actually only using the abelian group object structure on $Y$ here; see~\eqref{subsec:module_objects_hom_spaces}.

Write $\omega_Y\in \Mod{R}$ for the module of invariant differential forms on $Y$.

\begin{proposition}
\label{prop:hom_stack}
The prestack $\widetilde{\mathbb{H}}(X,Y)$ is a derived scheme over $R$ with the following properties:
\begin{enumerate}
  \item Its underlying classical scheme
\[
\widetilde{\mathbb{H}}(X,Y)^{\mathrm{cl}} = \underline{\Hom}(X,Y)
\]
is the scheme of homomorphisms from $X$ to $Y$.
\item We have a canonical equivalence
\[
\mathbb{L}_{\widetilde{\mathbb{H}}(X,Y)/R}\simeq \Reg{\widetilde{\mathbb{H}}(X,Y)}\otimes_R(\omega_Y\otimes_R(\tau^{\ge 1}R\Gamma(X,\Reg{X}))^\vee).
\]
\end{enumerate}
\end{proposition}
\begin{proof}
We see from Proposition~\ref{prop:deform_maps} and Remark~\ref{rem:mapping_stack_criteria} that the mapping prestack $\shvMap(X,Y)$ is represented by a derived scheme over $R$. Now, by construction $\widetilde{\mathbb{H}}(X,Y)$ is a closed derived subscheme of $\shvMap(X,Y)$.

That
\[
\widetilde{\mathbb{H}}(X,Y)(C) \simeq \Hom(X_C,Y_C)
\]
when $C$ is a discrete ring is a consequence of  the rigidity property for abelian schemes; see~\cite[Corollary 6.4]{mumford:git}.

To check that it admits the cotangent complex as stated, set, for any $C\in \mathrm{CRing}_{R/}$ and $P\in\Mod[\mathrm{cn}]{C}$,
\begin{align*}
H(C,P) &= \mathrm{hker}\left(\widetilde{\mathbb{H}}(X,Y)(C\oplus P)\to \widetilde{\mathbb{H}}(X,Y)(C)\right);\\
M(C,P) &= \mathrm{hker}\left(\shvMap(X_{C\oplus P},Y_{C\oplus P[1]})\to \shvMap(X_{C},Y_{C})\right).
\end{align*}
Then, by Proposition~\ref{prop:deform_maps} once again, and Remark~\ref{rem:abelian_scheme_dualizing}, we have
\begin{align*}
H(C,P)&\simeq \mathrm{hker}\left(M(C,P)\xrightarrow{\circ [0]}\mathrm{hker}(Y(C\oplus P)\to Y(C))\right)\\
&\simeq \mathrm{hker}\left(\Map_{\Mod{R}}(\omega_Y,P\otimes_CR\Gamma(X,\Reg{X}))\to \Map_{\Mod{R}}(\omega_Y,P)\right)\\
&\simeq \Map_{\Mod{R}}(\omega_Y,P\otimes_C\tau^{\geq 1}R\Gamma(X,\Reg{X})).
\end{align*}
Here, in the last step, we have used the fact that the induced map $[0]^*:H^0(X,\Reg{X})\to R$ is an isomorphism.
\end{proof}

\begin{corollary}
\label{cor:elliptic_curve_hom}
Suppose that $X$ is an elliptic curve. Then $\mathbb{H}(X,Y) = \widetilde{\mathbb{H}}(X,Y)$ is a locally quasi-smooth derived scheme over $R$ with cotangent complex 
\[
\mathbb{L}_{\mathbb{H}(X,Y)/R}\simeq \Reg{\mathbb{H}(X,Y)}\otimes_R(\omega_Y\otimes_R\omega_{X^\vee})[1],
\]
where $X^\vee$ is the dual elliptic curve.\footnote{Of course, we have a canonical isomorphism $X\simeq X^\vee$, but we maintain the distinction here because first, it emphasizes the covariance of $\omega_{X^\vee}$, and second, because this is the version that we will look to generalize in Section~\ref{sec:abvar_crystals}.}
\end{corollary}
\begin{proof}
Note that $\tau^{\geq 1}R\Gamma(X,\Reg{X}) = H^1(X,\Reg{X})[-1]$ in this case, and that we have a canonical isomorphism
\[
\omega_{X^\vee} \simeq H^1(X,\Reg{X})^\vee.
\]
Therefore, the corollary is immediate from Proposition~\ref{prop:hom_stack}.
\end{proof}

\subsection{}\label{subsec:absch_cohomology}
Let $\Fil^\bullet R\Gamma_{\mathrm{dR}}(X/R)\in \mathrm{FilMod}_R$ be the Hodge-filtered derived de Rham cohomology of $X$: This is defined for any prestack over $R$ via right Kan extension from $\mathrm{Aff}_R$ of the Hodge-filtered derived de Rham cohomology functor considered in~\eqref{subsec:infcoh}, but in this case can be obtained quite easily. Choose a smooth $\Int$-algebra $R'$ and an abelian scheme $X'$ over $R'$ such that 
\[
X \simeq X'\times_{\Spec R'}\Spec R
\]
for some map $R'\to R$; moduli considerations show that this is always possible. Now take
\[
\Fil^\bullet R\Gamma_{\mathrm{dR}}(X/R) \simeq R'\otimes_R\Fil^\bullet_{\mathrm{Hdg}}R\Gamma_{\mathrm{dR}}(X'/R'),
\]
where 
\[
\Fil^\bullet_{\mathrm{Hdg}}R\Gamma_{\mathrm{dR}}(X'/R') = R\Gamma(X',\Fil^\bullet_{\mathrm{Hdg}}\Omega^\bullet_{X'/R'}).
\]

The right hand side is a filtered commutative differential graded algebra over $R'$ via cup product, and pullback along the identity section of $X'$ gives rise to a natural map $R'\to R\Gamma(X',\Fil^\bullet_{\mathrm{Hdg}}\Omega^\bullet_{X'/R'})$. 

If $\Fil^\bullet_{\mathrm{Hdg}}H^1_{dR}(X'R')$ is the first cohomology of $R\Gamma(X',\Fil^\bullet_{\mathrm{Hdg}}\Omega^\bullet_{X'/R'})$, cup product now leads to a canonical quasi-isomorphism
\begin{align}\label{eqn:formality_derham}
\bigoplus_{i=0}^{\dim_{R'}X'}(\wedge^i_{R'}H^1_{dR}(X'/R'))[-i]\xrightarrow{\simeq}R\Gamma(X',\Fil^\bullet_{\mathrm{Hdg}}\Omega^\bullet_{X'/R'})
\end{align}
of filtered commutative differential graded algebras over $R'$.

Note that the right hand side is equipped with a canonical integrable connection over $\Int$, giving rise, after base-change along $R'\to R$, to the Hodge-filtered (derived) infinitesimal cohomology 
\[
\Fil^\bullet_{\mathrm{Hdg}}\mathrm{InfCoh}_{X_{\Rat}/R_{\Rat}}\in \mathrm{TFilInfCrys}_{R_{\Rat}/\Rat},
\]
of~\eqref{subsec:infcoh}, as well as for every prime $p$, the Hodge-filtered (derived) $p$-adic crystalline cohomology
\[
\mathrm{CrysCoh}_{p,X/R}\in \mathrm{FilFCrys}^{\geq 0}_{R/\Int_p}
\]
from~\eqref{subsec:crys_coH_f-zip}.

\begin{lemma}
\label{lem:ab_coh_split}
Let $g$ be the relative dimension of $X$ over $R$. Then there are canonical direct sum decompositions, functorial in $X$ and $R$,
\begin{align*}
\Fil^\bullet_{\mathrm{Hdg}}\mathrm{InfCoh}_{X_{\Rat}/R_{\Rat}}&\simeq \bigoplus_{i=0}^g\Fil^\bullet_{\mathrm{Hdg}}\mathrm{InfCoh}^{(i)}_{X_{\Rat}/R_{\Rat}}[-i];\\
\mathrm{CrysCoh}_{p,X/R}&\simeq \bigoplus_{i=0}^g\mathrm{CrysCoh}^{(i)}_{p,X/R}[-i]
\end{align*}
in $\mathrm{TFilInfCrys}_{R_{\Rat}/\Rat}$ and $\mathrm{FilFCrys}^{\geq 0}_{R/\Int_p}$, respectively, characterized by the fact that, for $R$ smooth over $\Int$, they arise from the decomposition on the left hand side of the isomorphism~\eqref{eqn:formality_derham}.
\end{lemma}

\subsection{}
Set 
\[
\Fil^\bullet_{\mathrm{Hdg}}\mathrm{InfCoh}^{\ge 1}_{X_{\Rat}/R_{\Rat}} = \bigoplus_{i=1}^g\Fil^\bullet_{\mathrm{Hdg}}\mathrm{InfCoh}^{(i)}_{X_{\Rat}/R_{\Rat}}[-i]\;;\; \mathrm{CrysCoh}^{\ge 1}_{p,X/R} = \bigoplus_{i=1}^g\mathrm{CrysCoh}^{(i)}_{p,X/R}[-i].
\]

Then applying Lemma~\ref{lem:filinfcrys_internal_hom} with $a = 0$ gives us `internal Hom' objects
\begin{align*}
\Fil^\bullet_{\mathrm{Hdg}}\widetilde{\mathcal{H}}_\infty(X,Y) & = \mathcal{H}\left(\Fil^\bullet_{\mathrm{Hdg}}\mathrm{InfCoh}^{(1)}_{Y_{\Rat}/R_{\Rat}},\Fil^\bullet_{\mathrm{Hdg}}\mathrm{InfCoh}^{\ge 1}_{X_{\Rat}/R_{\Rat}}[1]\right);\\
\mathcal{H}_\infty(X,Y) & = \mathcal{H}\left(\Fil^\bullet_{\mathrm{Hdg}}\mathrm{InfCoh}^{(1)}_{X_{\Rat}/R_{\Rat}},\Fil^\bullet_{\mathrm{Hdg}}\mathrm{InfCoh}^{(1)}_{Y_{\Rat}/R_{\Rat}}\right)
\end{align*}
in $\mathrm{TFilInfCrys}_{R_{\Rat}/\Rat}$ satisfying the hypotheses of Proposition~\ref{prop:FIC_functorial_props}. Moreover, the

Similarly, applying Lemma~\ref{lem:filfcrys_internal_hom} with $a=0$ gives us `internal Hom' objects
\begin{align*}
\widetilde{\mathcal{H}}_{p,crys}(X,Y) & = \mathcal{H}\left(\mathrm{CrysCoh}^{(1)}_{p,Y/R},\mathrm{CrysCoh}^{\ge 1}_{p,X/R}[1]\right);\\
\mathcal{H}_{p,crys}(X,Y) & = \mathcal{H}\left(\mathrm{CrysCoh}^{(1)}_{p,Y/R},\mathrm{CrysCoh}^{{(1)}}_{p,X/R}\right)
\end{align*}
in $\mathrm{FilFCrys}^{\geq -1}_{R/\Int_p}$ satisfying the hypotheses of Theorem~\ref{thm:FilFCrys_functorial_props}. Moreover, the second object also satisfies the hypotheses of Theorem~\ref{thm:FilFCrys_discrete}.

From Proposition~\ref{prop:FIC_functorial_props}, we now obtain a infinitesimally cohesive, nilcomplete prestacks $\widetilde{\mathbb{H}}_\infty(X,Y)$ and $\mathbb{H}_\infty(X,Y)$ over $R_\Rat$ given by
\[
\widetilde{\mathbb{H}}_\infty(X,Y) \overset{\mathrm{defn}}{=}\mathcal{Z}(\Fil^\bullet_{\mathrm{Hdg}}\widetilde{\mathcal{H}}_\infty(X,Y))\;;\; \mathbb{H}_\infty(X,Y) \overset{\mathrm{defn}}{=}\mathcal{Z}(\Fil^\bullet_{\mathrm{Hdg}}\mathcal{H}_\infty(X,Y))
\]

Similarly, from Theorem~\ref{thm:FilFCrys_functorial_props}, we obtain $p$-adic \emph{formal} prestacks $\widetilde{\mathbb{H}}^{\form}_{p,crys}(X,Y)$, $\mathbb{H}^{\form}_{p,crys}(X,Y)$
over $R$ given by
\[
\widetilde{\mathbb{H}}^{\form}_{p,crys}(X,Y) \overset{\mathrm{defn}}{=}\mathcal{Z}(\widetilde{\mathcal{H}}_{p,crys}(X,Y))^{\form}_p\;;\; \mathbb{H}^{\form}_{p,crys}(X,Y) \overset{\mathrm{defn}}{=}\mathcal{Z}(\mathcal{H}_{p,crys}(X,Y))^{\form}_p
\]

\begin{lemma}
\label{lem:tilde classical same}
The maps
\[
\mathbb{H}^{\form}_{p,crys}(X,Y)\to\widetilde{\mathbb{H}}^{\form}_{p,crys}(X,Y)\;;\;\mathbb{H}_\infty(X,Y)\to\widetilde{\mathbb{H}}_\infty(X,Y)
\]
are closed immersions of (formal) prestacks that are equivalences on classical points.
\end{lemma}
\begin{proof}
Use Propositions~\ref{prop:W0_discrete} and~\ref{prop:W0_discrete_FIC}.
\end{proof}

\subsection{}
Before presenting our next result, we need a little bit of a detour into derived de Rham cohomology. Suppose that we have $A,B\in \mathrm{CRing}_{C/}$ and $P\in \Mod[\mathrm{cn}]{C}$. Set $\tilde{C} = C\oplus P$ and $\tilde{B} = \tilde{C}\otimes_CB$. View $\tilde{C}$ as an object in $\mathrm{CRing}_{C/}$ via the trivial section. Derived de Rham cohomology (see~\eqref{subsec:infcoh}) now gives us a commuting diagram of maps
\begin{align*}
\begin{diagram}
\Map_{\mathrm{CRing}_{C/}}(A,\tilde{B})&\rTo& \Map_{\mathrm{FilMod}_C}(\Fil^\bullet_{\mathrm{Hdg}}\mathrm{dR}_{A/C},\Fil^\bullet_{\mathrm{Hdg}}\mathrm{dR}_{\tilde{B}/\tilde{C}});\\
\dTo&&\dTo\\
\Map_{\mathrm{CRing}_{C/}}(A,B)&\rTo &\Map_{\mathrm{FilMod}_C}(\Fil^\bullet_{\mathrm{Hdg}}\mathrm{dR}_{A/C},\Fil^\bullet_{\mathrm{Hdg}}\mathrm{dR}_{(\tilde{B}\twoheadrightarrow B)/(\tilde{C}\to C)})
\end{diagram}
\end{align*}
Here, for $T\in \mathrm{CRing}_{U/}$, we write $\mathrm{dR}_{T/U}$ instead of $\mathrm{dR}_{(T\xrightarrow{\mathrm{id}}T)/(U\xrightarrow{id}U)}$. The vertical map on the left is composition with the projection $\tilde{B}\to B$ and the right vertical one from the natural map of pairs
\[
(\tilde{B}\xrightarrow{\mathrm{id}}\tilde{B})\to (\tilde{B}\twoheadrightarrow B).
\]
The top arrow arises from the obvious functoriality of $\mathrm{dR}$, and the bottom one from the `crystalline' property of derived de Rham cohomology, which gives an equivalence
\[
\Fil^\bullet_{\mathrm{Hdg}}\mathrm{dR}_{(\tilde{B}\twoheadrightarrow B)/(\tilde{C}\to C)}\xrightarrow{\simeq}\Fil^\bullet_{\mathrm{Hdg}}\mathrm{dR}_{(B\xrightarrow{\mathrm{id}}B)/(\tilde{C}\to C)}
\]
For Hodge-completed derived de Rham cohomology, this was Lemma~\ref{lem:completed_de_rham}; but an argument with divided powers (see~\cite[Prop. 4.16]{Mao2021-jt}), combined with Lemma~\ref{lem:square_zero_pd}, shows one doesn't need the Hodge completion here.

\begin{lemma}
We have a canonical equivalence
\[
\hker(\Fil^\bullet_{\mathrm{Hdg}}\mathrm{dR}_{\tilde{B}/\tilde{C}}\to \Fil^\bullet_{\mathrm{Hdg}}\mathrm{dR}_{(\tilde{B}\twoheadrightarrow B)/(\tilde{C}\to C)} )\simeq P[-1]\otimes_C\gr^{\bullet - 1}_{\mathrm{Hdg}}\mathrm{dR}_{B/C}
\]
in $\mathrm{FilMod}_C$.
\end{lemma}
\begin{proof}
By the crystalline property of derived de Rham cohomology, the source of the map in question is equivalent to $\tilde{C}\otimes_C\Fil^\bullet_{\mathrm{Hdg}}\mathrm{dR}_{B/C}$ while the target is equivalent to the filtered base change
\[
\Fil^\bullet_{\tilde{C}\twoheadrightarrow C}\tilde{C}\otimes_C\Fil^\bullet_{\mathrm{Hdg}}\mathrm{dR}_{B/C}.
\]
So we can identify the homotopy kernel here with 
\[
\hker(\tilde{C}\to \Fil^\bullet_{\tilde{C}\twoheadrightarrow C}\tilde{C})\otimes_C\Fil^\bullet_{\mathrm{Hdg}}\mathrm{dR}_{B/C}.
\]
The first factor here, as an object of $\mathrm{FilMod}_C$ reduces to the filtered module associated with the graded module $P[-1]\{-1\}$ showing that what we have is simply the graded tensor product
\[
P[-1]\{-1\}\otimes_C\gr^\bullet_{\mathrm{Hdg}}\mathrm{dR}_{B/C}\simeq P[-1]\otimes_C\gr^{\bullet - 1}_{\mathrm{Hdg}}\mathrm{dR}_{B/C}.
\]
\end{proof}

Suppose now that we have $f:A\to B$, inducing the trivial lift $\tilde{f}_0:A\to \tilde{B}$. Then subtracting the derived de Rham realization of $\tilde{f}_0$, and using the lemma and the commuting diagram above gives us a a map
\begin{align*}
\Map_{\mathrm{CRing}_{C/\backslash B}}(A,\tilde{B})&\to \Map_{\mathrm{FilMod}_C}(\Fil^\bullet_{\mathrm{Hdg}}\mathrm{dR}_{A/C},P[-1]\otimes_C\gr^{\bullet -1}_{\mathrm{Hdg}}\mathrm{dR}_{B/C})\\
&\to \Map_{\mathrm{GrMod}_C}(\gr^\bullet_{\mathrm{Hdg}}\mathrm{dR}_{A/C},P[-1]\otimes_C\gr^{\bullet -1}_{\mathrm{Hdg}}\mathrm{dR}_{B/C})\\
&\to \Map_{\Mod{C}}(\gr^1_{\mathrm{Hdg}}\mathrm{dR}_{A/C},P[-1]\otimes_C\gr^0_{\mathrm{Hdg}}\mathrm{dR}_{B/C})\\
&\xrightarrow{\simeq} \Map_{\Mod{C}}(\mathbb{L}_{A/C}[-1],P[-1]\otimes_CB)\\
&\xrightarrow{\simeq} \Map_{\Mod{C}}(\mathbb{L}_{A/C},P\otimes_CB).
\end{align*}
Here, we have written $\Map_{\mathrm{CRing}_{C/\backslash B}}(A,\tilde{B})$ for the fiber of $\Map_{\mathrm{CRing}_{C/}}(A,\tilde{B})$ over $f$.

Therefore, we have constructed an arrow
\begin{align}\label{eqn:alpha_dR}
\mathrm{fib}_f(\Map_{\mathrm{CRing}_{C/}}(A,\tilde{B})\to\Map_{\mathrm{CRing}_{C/}}(A,B))&\to \Map_{\Mod{C}}(\mathbb{L}_{A/C},P\otimes_CB).
\end{align}

\begin{lemma}\label{lem:cohom_compat_local}
The map~\eqref{eqn:alpha_dR} factors through
\[
\Map_{\Mod{A}}(\mathbb{L}_{A/C},P\otimes_CB)
\]
and this factoring yields an inverse to the canonical equivalence
\begin{align*}
 \Map_{\Mod{A}}(\mathbb{L}_{A/C},P\otimes_CB)&\xrightarrow{\simeq}\mathrm{fib}_f(\Map_{\mathrm{CRing}_{C/}}(A,\tilde{B})\to\Map_{\mathrm{CRing}_{C/}}(A,B))
\end{align*}
arising from the definition of the cotangent complex $\mathbb{L}_{A/C}$.
\end{lemma}
\begin{proof}
Let $A'\twoheadrightarrow A$ be the trivial square-zero thickening with fiber $\mathbb{L}_{A/C}$; then we have
\begin{align}\label{eqn:fib_f_alt_descp}
\mathrm{fib}_f(\Map_{\mathrm{CRing}_{C/}}(A,\tilde{B})\to\Map_{\mathrm{CRing}_{C/}}(A,B))&\xrightarrow{\simeq}\mathrm{Map}_{\mathrm{AniPair}_{(A\xrightarrow{\mathrm{id}}A)/}}(A'\twoheadrightarrow A,\tilde{B}\twoheadrightarrow B).
\end{align}
Now, we have canonical equivalences
\begin{align*}
\gr^\bullet_{\mathrm{Hdg}}\mathrm{dR}_{(A'\twoheadrightarrow A)/(C\xrightarrow{\mathrm{id}}C)}&\simeq \gr^\bullet_{(A'\twoheadrightarrow A)}A'\otimes_{A'}\gr^\bullet_{\mathrm{Hdg}}\mathrm{dR}_{A'/C}\\
\gr^\bullet_{\mathrm{Hdg}}\mathrm{dR}_{(\tilde{B}\twoheadrightarrow B)/(\tilde{C}\twoheadrightarrow C)}&\simeq \gr^\bullet_{(\tilde{B}\twoheadrightarrow B)}\tilde{B}\otimes_{\tilde{B}}\gr^\bullet_{\mathrm{Hdg}}\mathrm{dR}_{\tilde{B}/\tilde{C}},
\end{align*}
which in particular give us:
\begin{align*}
\gr^1_{\mathrm{Hdg}}\mathrm{dR}_{(A'\twoheadrightarrow A)/(C\xrightarrow{\mathrm{id}}C)}&\xrightarrow{\simeq} (\mathbb{L}_{A/C}\otimes_{A'}A')\oplus (A\otimes_{A'}\mathbb{L}_{A'/C}[-1])\\
\gr^1_{\mathrm{Hdg}}\mathrm{dR}_{(\tilde{B}\twoheadrightarrow A)/(\tilde{C}\twoheadrightarrow C)}&\xrightarrow{\simeq} ((P\otimes_CB)\otimes_{\tilde{B}}\tilde{B})\oplus (B\otimes_{\tilde{B}}\mathbb{L}_{\tilde{B}/\tilde{C}}[-1]).
\end{align*}

Combining these with the equivalence~\eqref{eqn:fib_f_alt_descp} now gives us a map
\begin{align*}
\mathrm{fib}_f(\Map_{\mathrm{CRing}_{C/}}(A,\tilde{B})\to\Map_{\mathrm{CRing}_{C/}}(A,B))&\to \Map_{\Mod{C}}(\gr^1_{\mathrm{Hdg}}\mathrm{dR}_{(A'\twoheadrightarrow A)/(C\xrightarrow{\mathrm{id}}C)},\gr^1_{\mathrm{Hdg}}\mathrm{dR}_{(\tilde{B}\twoheadrightarrow B)/(\tilde{C}\twoheadrightarrow C)})\\
&\to \Map_{\Mod{C}}(\mathbb{L}_{A/C},P\otimes_CB)
\end{align*}
which evidently has the properties asserted in the lemma, but is also equivalent to the arrow~\eqref{eqn:alpha_dR}.
\end{proof}

\begin{remark}
The argument used above is simply a translation to our context of that used in~\cite[Theorem (3.21)]{berthelot_ogus:f_isoc}.
\end{remark}

\begin{proposition}
\label{prop:hom_stack_tangent_space}
Consider the derived scheme $\widetilde{\mathbb{H}}(X,Y)$ over $R$ from Proposition~\ref{prop:hom_stack}. Then cohomological realization gives rise to natural formally \'etale maps
\[
\widetilde{\mathbb{H}}(X,Y)_{\Rat}\to \widetilde{\mathbb{H}}_\infty(X,Y)\;;\; \widetilde{\mathbb{H}}(X,Y)^{\form}_p\to \widetilde{\mathbb{H}}^{\form}_{p,crys}(X,Y)
\]
More precisely, we have such maps, and, for any $C\in \mathrm{CRing}_{R_{\Rat}/}$ (resp. $C\in \mathrm{CRing}_{R/}^{p-\text{nilp}}$) and $P\in \Mod[\mathrm{cn}]{C}$, the maps
\begin{align*}
\hker\left(\widetilde{\mathbb{H}}(X,Y)(C\oplus P)\to \widetilde{\mathbb{H}}(X,Y)(C)\right)&\to\hker\left(\widetilde{\mathbb{H}}_{\infty}(X,Y)(C\oplus P)\to \widetilde{\mathbb{H}}_{\infty}(X,Y)(C)\right)\\
(\text{resp. }\hker\left(\widetilde{\mathbb{H}}(X,Y)(C\oplus P)\to \widetilde{\mathbb{H}}(X,Y)(C)\right)&\to\hker\left(\widetilde{\mathbb{H}}^{\form}_{p,crys}(X,Y)(C\oplus P)\to \widetilde{\mathbb{H}}^{\form}_{p,crys}(X,Y)(C)\right))
\end{align*}
are equivalences.
\end{proposition}
\begin{proof}
For brevity, we will omit the $(X,Y)$ from our notation in what follows.

There are functorial-in-$C$ cohomological realization maps
\begin{align*}
\widetilde{\mathbb{H}}(C)&\to \Map_{\mathrm{TFilInfCrys}_{C_{\Rat}/\Rat}}\left(\Fil^\bullet_{\mathrm{Hdg}}\mathrm{InfCoh}^{\ge 1}_{Y_{\Rat,C}/C_{\Rat}},\Fil^\bullet_{\mathrm{Hdg}}\mathrm{InfCoh}^{\ge 1}_{X_{\Rat,C}/C_{\Rat}}\right)\\
\widetilde{\mathbb{H}}(C)&\to \Map_{\mathrm{FilFCrys}^{\geq 0}_{C/Int_p}}\left(\mathrm{CrysCoh}^{\ge 1}_{p,Y/C},\mathrm{CrysCoh}^{\ge 1}_{p,X/C}\right).
\end{align*}

Restricting along the natural maps
\[
\Fil^\bullet_{\mathrm{Hdg}}\mathrm{InfCoh}^{(1)}_{Y_{\Rat,C}/C_{\Rat}}\to \Fil^\bullet_{\mathrm{Hdg}}\mathrm{InfCoh}^{\ge 1}_{Y_{\Rat,C}/C_{\Rat}}\;;\; \mathrm{CrysCoh}^{(1)}_{p,Y/C}\to \mathrm{CrysCoh}^{\ge 1}_{p,Y/C}
\]
now gives us the maps that we need.

We need to show that these map are formally \'etale. Write
\[
F_{\widetilde{\mathbb{H}}}(C,P)\;;\; F_{\widetilde{\mathbb{H}}_\infty}(C,P)\;;\; F_{\widetilde{\mathbb{H}}^{\form}_{p,crys}}(C,P)
\]
for each of the homotopy kernels involved here. 

First, note that all of them are canonically equivalent to $\Map_{\Mod{R}}\left(\omega_Y,P\otimes_R\tau^{\ge 1}R\Gamma(X,\Reg{X})\right)$. For $\widetilde{\mathbb{H}}$, this follows from (2) of Proposition~\ref{prop:hom_stack}. For $\widetilde{\mathbb{H}}_\infty$ (resp. $\widetilde{\mathbb{H}}^{\form}_{p,crys}$), this follows from (2) of Proposition~\ref{prop:FIC_functorial_props} (resp. (2) of Theorem~\ref{thm:FilFCrys_functorial_props}).



Let us briefly recall the construction of the map inducing the last two equivalnces.

Suppose that we have $C\in \mathrm{CRing}_{R/}$ and $P\in \Mod[\mathrm{cn}]{C}$. Write $\tilde{C}\twoheadrightarrow C$ for the trivial square-zero thickening with kernel $P$. The kernel of the map
\begin{align}\label{eqn:square_zero_kernel}
\begin{diagram}
\underline{\Map}_{\mathrm{FilMod}_{R}}\left(\Fil^\bullet_{\mathrm{Hdg}}H^1_{dR}(Y/R),\tilde{C}\otimes_R\Fil^\bullet_{\mathrm{Hdg}}\tau^{\ge 1}R\Gamma_{dR}(X/R)\right)\\
\dTo\\ 
\underline{\Map}_{\mathrm{FilMod}_{R}}\left(\Fil^\bullet_{\mathrm{Hdg}}H^1_{\dR}(Y/R),\Fil^\bullet_{\tilde{C}\twoheadrightarrow C}\tilde{C}\otimes_R\Fil^\bullet_{\mathrm{Hdg}}\tau^{\ge 1}R\Gamma_{dR}(X/R)\right)
\end{diagram}
\end{align}
can be identified with
\[
\underline{\Map}_{\mathrm{FilMod}_R}\left(\Fil^\bullet_{\mathrm{Hdg}}H^1_{dR}(Y/R),P[-1]\otimes_R\gr^{\bullet-1}_{\mathrm{Hdg}}\tau^{\ge 1}R\Gamma_{\dR}(X/R)\right),
\]
and thus with
\[
\underline{\Map}_{\Mod{R}}\left(\omega_Y,P\otimes_R\tau^{\ge 1}R\Gamma(X,\Reg{X})\right).
\] 

Note that when $C$ is in $\mathrm{CRing}_{R_{\Rat}/}$ the map~\eqref{eqn:square_zero_kernel} can be identified with
\[
\Fil^0\mathcal{H}_\infty(\tilde{C}\xrightarrow{\mathrm{id}}\tilde{C})\to \Fil^0\mathcal{H}_\infty(\tilde{C}\twoheadrightarrow C).
\]
We now use the natural maps 
\[
\widetilde{\mathbb{H}}_\infty(\tilde{C})\to \Fil^0\mathcal{H}_\infty(\tilde{C}\xrightarrow{\mathrm{id}}\tilde{C})\;;\; \widetilde{\mathbb{H}}_\infty(C)\to \Fil^0\mathcal{H}_\infty(\tilde{C}\twoheadrightarrow C)
\]
to obtain maps
\begin{align}\label{eqn:FHinfty}
F_{\widetilde{\mathbb{H}}_\infty}(C,P)&\to\hker(\Fil^0\mathcal{H}_\infty(\tilde{C}\xrightarrow{\mathrm{id}}\tilde{C})\to \Fil^0\mathcal{H}_\infty(\tilde{C}\twoheadrightarrow C))\\
&\xrightarrow{\simeq} \Map_{\Mod{R}}\left(\omega_Y,P\otimes_R\tau^{\ge 1}R\Gamma(X,\Reg{X})\right)\nonumber
\end{align}

Similarly, when $C$ is in $\mathrm{CRing}^{p-\text{nilp}}_{R/}$, the map~\eqref{eqn:square_zero_kernel} can be identified with
\[
 \Fil^0\mathcal{H}_{p,crys}(\tilde{C}\xrightarrow{\mathrm{id}}\tilde{C})\to \Fil^0\mathcal{H}_{p,crys}(\tilde{C}\twoheadrightarrow C)
\]
and the same process gives a map
\begin{align}\label{eqn:FHpcris}
F_{\widetilde{\mathbb{H}}^{\form}_{p,crys}}(C,P)&\to \Map_{\Mod{R}}\left(\omega_Y,P\otimes_R\tau^{\ge 1}R\Gamma(X,\Reg{X})\right).
\end{align}

That~\eqref{eqn:FHinfty} is an equivalence can be extracted from the proof of Proposition~\ref{prop:FIC_functorial_props}, while the fact that~\eqref{eqn:FHpcris} is one was explicitly stated in the proof of Theorem~\ref{thm:FilFCrys_functorial_props}; see Lemma~\ref{lem:canonical_arrow_P_kernel}.

To finish it suffices to show that the maps
\[
F_{\widetilde{\mathbb{H}}}(C,P)\to \Map_{\Mod{R}}\left(\omega_Y,P\otimes_R\tau^{\ge 1}R\Gamma(X,\Reg{X})\right)
\]
induced by cohomological realization is equivalent to the one given by (2) of Proposition~\ref{prop:hom_stack}. This is simply a globalization of Lemma~\ref{lem:cohom_compat_local}.
\end{proof}

\begin{theorem}\label{thm:deform_end}
There exists $\mathbb{H}(X,Y)\in \mathrm{Ab}(\mathrm{PStk}_R)$ represented by a locally finite unramified and quasi-smooth derived scheme over $R$, and equipped with a closed immersion $\mathbb{H}(X,Y)\hookrightarrow \widetilde{\mathbb{H}}(X,Y)$ in $\mathrm{Ab}(\mathrm{PStk}_R)$, satisying the following properties:
\begin{enumerate}
  \item \label{hom:classical} If $C\in \mathrm{CRing}_{\heartsuit,R/}$ is discrete, then we have
  \[
    \mathbb{H}(X,Y)(C)\simeq \Hom(X_C,Y_C),
  \]
  where the right hand side is the module of homomorphisms of classical abelian schemes.

 \item \label{hom:cotangent} We have
 \[
   \mathbb{L}_{\mathbb{H}(X,Y)/R}\simeq \Reg{\mathbb{H}(X,Y)}\otimes_R(\omega_Y\otimes\omega_{X^\vee})[1].
 \]

 \item \label{hom:generic} We have
 \[
   \mathbb{H}(X,Y)_{\Rat} \simeq \widetilde{\mathbb{H}}(X,Y)_{\Rat}\times_{\widetilde{\mathbb{H}}_{\infty}(X,Y)}\mathbb{H}_\infty(X,Y).
 \]

 \item \label{hom:p-adic} For any prime $p$, we have
 \[
  \mathbb{H}(X,Y)^{\form}_p \simeq \widetilde{\mathbb{H}}(X,Y)^{\form}_p\times_{\widetilde{\mathbb{H}}^{\form}_{p,crys}(X,Y)}\mathbb{H}^{\form}_{p,crys}(X,Y).
 \]

\end{enumerate}
\end{theorem}

The proof of this requires a little argument for gluing the various local subschemes of $\widetilde{\mathbb{H}}(X,Y)$ together, which we will consider in~\eqref{subsec:closed_immersions} below.

\begin{remark}
As mentioned in the introduction, I don't know if the derived scheme $\mathbb{H}(X,Y)$ admits a direct moduli theoretic construction in general, but we have the following observations:
\begin{itemize}
  \item When $X$ is an elliptic curve, then it is immediate that $\mathbb{H}(X,Y)\simeq \widetilde{\mathbb{H}}(X,Y)$. In higher dimensions, because of the non-vanishing of $H^i(X,\Reg{X})$ for $i>1$, this will no longer hold, and one needs to find some additional conditions to cut out the right space.

  \item In characteristic $0$, one can show that $\mathbb{H}(X,Y)$ parameterizes maps of abelian group objects from $X$ to $Y$. 

  \item The prestack of maps of abelian group objects is too large in finite characteristic: its cotangent complex is unbounded. This is related to the fact that $\mathrm{RHom}(\mathbb{G}_a,\mathbb{G}_a)$ has non-zero cohomology in arbitrarily large degrees in characteristic $p$---see~\cite{Breen1978-pq}---which in turn is related to the non-vanishing of the cohomology of Eilenberg-MacLane spaces. 
\end{itemize} 
\end{remark}

\begin{remark}
Even if we could somehow find a direct construction of $\mathbb{H}(X,Y)$, we will still need the gluing techniques used here in Section~\ref{sec:shimura} when we will prove Theorem~\ref{introthm:W_scheme}.
\end{remark}

\subsection{}\label{subsec:closed_immersions}
Suppose that we have $C\in \mathrm{CRing}$, and an abelian group object $X:\mathrm{CRing}_{C/}\to \Mod[\mathrm{cn}]{\Int}$ in $\mathrm{PStk}_C$. Suppose in addition that the underlying prestack of $X$ is represented by a locally finite unramified derived scheme over $C$, and that there exists $\kappa_X$ in $\Mod[\mathrm{perf},\mathrm{cn}]{C}$ along with an equivalence
\[
\iota:\Reg{X}\otimes_C\kappa_X\xrightarrow{\simeq}\mathbb{L}_{X/C}[-1].
\]

Choose a map $\kappa_Y^\vee\to \kappa_X^\vee$ in $\Mod[\mathrm{perf}]{C}$ with dual $u:\kappa_X\to \kappa_Y$, and suppose that $\kappa_Y$ is \emph{locally free} over $C$. For any $C'\in \mathrm{CRing}_{C/}$, a \defnword{$u$-marked scheme over $X_{C'}$} is a finite unramified map $j:Y\to X_{C'}$ equipped with the following:
\begin{enumerate}
  \item A structure of a group scheme over $\pi_0(C')$ on the classical truncation $Y^{\mathrm{cl}}$ such that the map $Y^{\mathrm{cl}}\to X^{\mathrm{cl}}_{C'}$ is a homomorphism of group schemes over $\pi_0(C')$;
  \item A factoring of the zero section $[0]:\Spec C'\to X_{C'}$ through $Y$;
  \item An equivalence
  \[
   (\Reg{Y}\otimes_C\kappa_X \xrightarrow{1\otimes u} \Reg{Y}\otimes_C\kappa_Y)\xrightarrow{\simeq}(j^*\mathbb{L}_{X_C'/C'}\to \mathbb{L}_{Y/C'})[-1]
  \]
  in $\mathrm{Ab}(\mathrm{Fun}([1],\mathrm{QCoh}_{Y}))$ extending the equivalence
  \[
   \Reg{Y}\otimes_C\kappa_X\xrightarrow{\simeq}\mathbb{L}_{X_{C'}/C'}[-1]
  \]
  induced by $\iota$.
\end{enumerate}
Note that the condition on $\mathbb{L}_{Y/C'}$---along with our hypothesis that $\kappa_Y$ is locally free---ensures that $Y$ is actually locally \emph{quasi-smooth} (and of course locally finite and unramified) over $C'$.

\begin{proposition}
\label{prop:unique_Y}
\begin{enumerate}
  \item Suppose that $j:Y\to X_{C'}$ is a $u$-marked scheme over $X_{C'}$. Then $Y$ underlies a canonical abelian group object in $\mathrm{PStk}_{C'}$ such that $j$ is a map of derived group schemes.

  \item Suppose that $j_1:Y_1\to X_{C'}$ and $j_2:Y_2\to X_{C'}$ are two $u$-marked schemes over $X_{C'}$ equipped with an isomorphism between their classical truncations. Then there exists a canonical equivalence $Y_1\xrightarrow{\simeq}Y_2$ in $\mathrm{Ab}(\mathrm{PStk})_{/X_{C'}}$.
\end{enumerate}
\end{proposition} 

\begin{proof}
For (1), we will show by induction on $n$ that the restriction of $Y$ to $\mathrm{CRing}_{\leq n,C'/}$ admits a canonical lift to a $\Mod[\mathrm{cn}]{\Int}$-valued functor such that the map to $X$ is a map of derived group schemes.

For $n=1$ this is part of our hypothesis. So assume the inductive claim known for $n\geq 1$: Any $D\in \mathrm{CRing}_{\leq n+1,C'/}$ sits in a canonical Cartesian square
\[
\begin{diagram}
D&\rTo&\tau_{\leq n}D\\
\dTo&&\dTo\\
\tau_{\leq n}D&\rTo&\tau_{\leq n}D\oplus \pi_{n+1}(D)[n+2]
\end{diagram}
\]
which in turn gives us a Cartesian square of maps
\[
\begin{diagram}
(Y(D)\to X(D))&\rTo&(Y(\tau_{\leq n}D)\to X(\tau_{\leq n}D))\\\
\dTo&&\dTo\\
\left(Y(\tau_{\leq n}D)\to X(\tau_{\leq n}D)\right)&\rTo&\left(Y(\tau_{\leq n}D\oplus \pi_{n+1}(D)[n+2])\to X(\tau_{\leq n}D\oplus \pi_{n+1}(D)[n+2])\right)
\end{diagram}
\]

Therefore, we see that it is enough to know that the restriction of $Y$ along the functor
\[
\mathrm{CRing}_{\leq n,C'/}\times_{\mathrm{CRing}}\Mod[\mathrm{cn}]{}\xrightarrow{(A,M)\mapsto A\oplus M[1]}\mathrm{CRing}_{C'/}
\]
has a canonical lift to a $\Mod[\mathrm{cn}]{\Int}$-valued one. For this, note that for $(A,M)$ in the source of the above functor, we have a canonical equivalence
\[
X(A\oplus M[1])\simeq X(A)\oplus \Map_{\mathrm{Mod}_C}(\kappa_X,M)
\]
in $\Mod[\mathrm{cn}]{D}$. We claim that the map
\[
Y(A\oplus M[1])\to X(A\oplus M[1])\simeq X(A)\oplus \Map_{\mathrm{Mod}_C}(\kappa_X,M)
\]
factors through an equivalence
\[
Y(A\oplus M[1]) \simeq Y(A)\oplus \Map_{\mathrm{Mod}_C}(\kappa_Y,M).
\]
Indeed, the action of $\Map_{\mathrm{Mod}_C}(\kappa_X,M)$ on $X(A\oplus M[1])$ presents it as a (trivial) torsor over $X(A)$, and our conditions ensure that the inherited action of $\Map_{\mathrm{Mod}_C}(\kappa_Y,M)$ lifts to one on $Y(A\oplus M[1])$, and presents it as a trivial torsor over $Y(A)$.

This finishes the inductive step, since the right hand side has a canonical lift to a $\Mod[\mathrm{cn}]{\Int}$-valued functor compatibly with its counterpart for $X$.

The proof of (1) shows that $Y$, along with its lift to a $\Mod[\mathrm{cn}]{\Int}$-valued prestack, is uniquely determined by its classical truncation and its structure as a $(u$-marked scheme over $X$. This of course proves (2).
\end{proof}

\begin{corollary}
\label{cor:Y_determined_by_points}
Suppose that $X^{\mathrm{cl}}$ is locally of finite type over $\Int$, and suppose that we have two $u$-marked schemes $f:Y\to X$ and $\tilde{f}:\tilde{Y}\to X$ satisfying the following conditions: 
\begin{enumerate}
	\item If $k$ is $\Comp$ or a finite field, then the map $\tilde{Y}(k)\to X(k)$ factors through an isomorphism of groups $\tilde{Y}(k)\xrightarrow{\simeq} Y(k)$. 
	\item $Y^{\mathrm{cl}}\hookrightarrow X^{\mathrm{cl}}$ is a closed immersion.
\end{enumerate}
Then $\tilde{f}$ factors through a canonical equivalence $Y\xrightarrow{\simeq}\tilde{Y}$ in $\mathrm{Ab}(\mathrm{PStk})_{/X}$.
\end{corollary}
\begin{proof}
Given (2) of Proposition~\ref{prop:unique_Y}, it suffices to show that the classical truncation of $\tilde{f}$ factors canonically through an isomorphism of group schemes $\tilde{Y}^{\mathrm{cl}}\xrightarrow{\simeq} Y^{\mathrm{cl}}$. 

This question is local on $X$, so we can assume that $X^{\mathrm{cl}} = \Spec A$, $Y^{\mathrm{cl}} = \Spec A/I$ for some ideal $I\subset A$, and $\tilde{Y}^{\mathrm{cl}} = \Spec B$ for a finite unramified $A$-algebra $B$. 

Our hypotheses, along with the argument used in the proof of (1) of that proposition shows that, for any complete local ring $R$ with residue field $k$ with $k = \Comp$ or $k$ a finite field, the map $\tilde{Y}(R)\to X(R)$ factors through a bijection of groups $\tilde{Y}(R)\to Y(R)$. This means that, for any maximal ideal $\mathfrak{m}\subset A$, there is at most one maximal ideal $\mathfrak{n}\subset B$ lying over $\mathfrak{m}$, and, when the latter exists, the map $A\to \widehat{B}_{\mathfrak{n}}$ factors through $A/I$. Therefore, $\tilde{Y}^{\mathrm{cl}}\to X^{\mathrm{cl}}$ factors through a finite \'etale map $\tilde{Y}^{\mathrm{cl}}\to Y^{\mathrm{cl}}$ that is a bijection on closed points, and is hence an isomorphism.
\end{proof}

\begin{lemma}
\label{lem:Y_filt_colimit}
Suppose that $C'$ is a filtered colimit $C' = \colim_{i\in I}C_i$ in $\mathrm{CRing}_{C/}$. If there is a $u$-marked scheme over $X_{C'}$, then there is a $u$-marked scheme over $X_{C_i}$ for some $i\in I$.
\end{lemma}  
\begin{proof}
Suppose that $j:Y\to X_{C'}$ is a $u$-marked scheme. To begin with, there is some $i\in I$ and a quasi-smooth derived scheme $Y_i$ over $C_i$ with underlying classical scheme $Y^{\mathrm{cl}}_{C_i}$, equipped with an equivalence $C'\otimes_{C_i}Y_i\xrightarrow{\simeq}Y$ that lifts the identity on $Y^{\mathrm{cl}}_{C'}$. This follows from~\cite[Proposition 1.6]{Toen2012-yy}.\footnote{Thanks to A. Khan for pointing me to this reference.}

With this in hand, it is not difficult to see that the rest of the required data in conditions (1), (2), (3) of the definition of a $u$-marked scheme also has to be defined over $C_j$ for some $j\in I$.
\end{proof}

\begin{corollary}
\label{cor:exist_Y}
Suppose that we have a finite map of classical group schemes $f:Y^{\mathrm{cl}}\to X^{\mathrm{cl}}$ such that the following hold:
\begin{enumerate}
  \item There exists a $u$-marked scheme over $X_{C\otimes_{\Int}\Rat}$ with underlying classical truncation $f_{\Rat}$.
  \item For every prime $p$ and $n\geq 1$, there exists a $u$-marked scheme over $X_{C/{}^{\mathbb{L}}p^n}$ with underlying classical scheme $f_{\Int/p^n\Int}$.
\end{enumerate}
Then there exists a $u$-marked scheme over $X$ with underlying classical truncation $f$.
\end{corollary}

\begin{proof}
The first assumption  implies that there exists an integer $N\geq 1$ and a $u$-marked scheme over $X_{C\otimes_{\Int}\Int[1/N]}$ with classical truncation $f_{\Int[1/N]}$: this is immediate from Lemma~\ref{lem:Y_filt_colimit} above.

We now claim that, for every prime $p$, there exists a $u$-marked scheme over $X_{C\otimes_{\Int}\Int_{(p)}}$ with classical truncation $f_{\Int_{(p)}}$: This is a consequence of part (2) of Proposition~\ref{prop:unique_Y}, our assumptions (1) and (2) here, and the derived Beauville-Laszlo gluing of~\cite[Prop. 5.6]{Bhatt2016-ky} in the form of Proposition~\ref{prop:beauville_laszlo} below applied to an affine open cover of $X$.

Now, the proof of the corollary is completed by using (2) of Proposition~\ref{prop:unique_Y} and applying the following result to an affine open cover of $X$, and successively to each prime $p$ dividing $N$.
\begin{lemma}
\label{lem:gluing_p_pinv}
Let $p$ be a prime. For any $A\in \mathrm{CRing}$, set $A_{(p)} = A\otimes_{\Int}\Int_{(p)}$; also set $A_{\Rat} = A\otimes_{\Int}\Rat$. Then the natural functors
\[
\Mod{A}\to \Mod{A\otimes_{\Int}\Int[1/p]}\times_{\Mod{A_{\Rat}}}\Mod{A_{(p)}};
\]
\[
\mathrm{CRing}_{A/} \to \mathrm{CRing}_{A\otimes_{\Int}\Int[1/p]/}\times_{\mathrm{CRing}_{A_{\Rat}/}}\mathrm{CRing}_{A_{(p)}/}
\]
are equivalences.
\end{lemma}
\begin{proof}
This is basically faithfully flat descent, but we can argue directly as well. Denote both functors involved by $\Phi$.

 If we have $M_{(p)}\in \Mod{A_{(p)}}$, $M_{\Rat}\in \Mod{A_{\Rat}}$, $M[1/p]\in \Mod{A\otimes_{\Int}\Int[1/p]}$ equipped with equivalences
\[
\alpha:\Rat\otimes_{\Int_{(p)}}M_{(p)}\xrightarrow{\simeq}M_{\Rat}\;;\; \beta:\Rat\otimes_{\Int[1/p]}M[1/p]\xrightarrow{\simeq}M_{\Rat},
\]
and if we take $M\in \Mod{A}$ so that we have a Cartesian diagram
\begin{align}\label{diag:M_(p)_square}
\begin{diagram}
M&\rTo&M_{(p)}\\
\dTo&&\dTo_{\alpha}\\
M[1/p]&\rTo_{\beta}&M_{\Rat}
\end{diagram}
\end{align}
then the natural maps $\Int[1/p]\otimes_{\Int}M\to M[1/p]$ and $\Int_{(p)}\otimes_{\Int}M\to M_{(p)}$ are equivalences. Indeed, this comes down to the fact that the vertical arrow on the right (resp. the bottom arrow) is an equivalence after tensoring with $\Int[1/p]$ (resp. with $\Int_{(p)}$).

This construction gives us a functor
\[
\Psi:\Mod{A\otimes_{\Int}\Int[1/p]}\times_{\Mod{A_{\Rat}}}\Mod{A_{(p)}}\to \mathrm{CRing}_{A/} 
\]
which is an inverse to $\Phi$.

The forgetful functor $\mathrm{CRing}_{B/}\to \Mod{B}$ is conservative, and respects limits for any $B\in \mathrm{CRing}$, as well as base-change along any map $B\to B'$, so the argument for the invertibility of the first functor applies to show that of the second as well.
\end{proof}

\end{proof}

\begin{proposition}\label{prop:beauville_laszlo}
For $A\in \mathrm{CRing}$, the natural functors
\[
\Mod{A_{(p)}}\to \Mod{\widehat{A}_p}\times_{\Mod{\widehat{A}_p\otimes_{\Int}\Rat}}\Mod{A\otimes_{\Int}\Rat};
\]
\[
\mathrm{CRing}_{A_{(p)}/}\to \mathrm{CRing}_{\widehat{A}_p/}\times_{\mathrm{CRing}_{(\widehat{A}_p\otimes_{\Int}\Rat)/}}\mathrm{CRing}_{(A\otimes_{\Int}{\Rat})/}
\]
are equivalences. Here $\widehat{A}_p$ is the $p$-completion of $A$.
\end{proposition}
\begin{proof}
Just as in Lemma~\ref{lem:gluing_p_pinv}, it suffices to show that the first functor is an equivalence. This follows from (the arguments in)~\cite[Prop. 5.6, Lemma 5.12(2)]{Bhatt2016-ky} applied with $X = \Spec \widehat{A}_p$, $Y = \Spec A_{(p)}$, $Z = \Spec A/{}^{\mathbb{L}}p$ and $\pi:X\to Y$ the natural map.
\end{proof}

\begin{proof}
[Proof of Theorem~\ref{thm:deform_end}]
Consider the natural map
\[
u:\omega_Y\otimes_R(\tau^{\ge 1}R\Gamma(X,\Reg{X}))^\vee[-1]\to \omega_Y\otimes_RH^1(X,\Reg{X})^\vee\simeq \omega_Y\otimes_R\omega_{X^\vee}.
\]

With $\tilde{\mathbb{H}} = \widetilde{\mathbb{H}}(X,Y)$, the theorem, restated in the terminology introduced above, amounts to asking for a $u$-marked scheme over $\tilde{\mathbb{H}}$ with the identity as its underlying classical truncation. The existence of such an object over $R_{\Rat}$ and over $R/{}^{\mathbb{L}}p^n$ for any prime $p$ and $n\ge 1$ is verified by considering the objects on the right hand side of the purported equivalences in assertions~\eqref{hom:generic} and~\eqref{hom:p-adic} of the theorem, respectively, and using Lemma~\ref{lem:tilde classical same} and Proposition~\ref{prop:hom_stack_tangent_space}. We now conclude using Corollary~\ref{cor:exist_Y}.
\end{proof}

The next result shows that $\mathbb{H}(\_\_,\_\_)$ behaves like a space of homomorphisms in an important way: It respects products in the first variable.
\begin{corollary}
\label{cor:hom_additive}
Suppose that $X'$ is another abelian scheme over $R$. Write $X\times X'$ for the fiber product over $\Spec R$. Then restriction along the maps
\[
X\xrightarrow{(\mathrm{id},0)}X\times X'\;;\; X'\xrightarrow{(0,\mathrm{id})}X\times X'
\]
produces a canonical equivalence
\[
\mathbb{H}(X\times X',Y) \xrightarrow{\simeq}\mathbb{H}(X,Y)\times\mathbb{H}(X',Y).
\]
\end{corollary}
\begin{proof}
Restriction along the given maps gives an arrow
\[
\widetilde{\mathbb{H}}(X\times X',Y) \to \widetilde{\mathbb{H}}(X,Y)\times\widetilde{\mathbb{H}}(X',Y).
\]
The corollary now amounts to saying that the restriction of this arrow to $\mathbb{H}(X\times X',Y)$ factors through an equivalence
\[
\mathbb{H}(X\times X',Y) \xrightarrow{\simeq}\mathbb{H}(X,Y)\times\mathbb{H}(X',Y).
\] 

Consider the natural map
\[
u:\omega_Y\otimes_R\left[(\tau^{\ge 1}R\Gamma(X,\Reg{X}))^\vee\oplus(\tau^{\ge 1}R\Gamma(X',\Reg{X'}))^\vee\right]\to \omega_Y\otimes_R(\omega_{X^\vee}\oplus\omega_{X^{',\vee}}).
\]
Then the desired result follows by noting that both $\mathbb{H}(X\times X',Y)$ and $\mathbb{H}(X,Y)\times\mathbb{H}(X',Y)$ are exhibited over $\tilde{\mathbb{H}} = \widetilde{\mathbb{H}}(X,Y)\times \widetilde{\mathbb{H}}(X',Y)$ as $u$-marked schemes of $\tilde{\mathbb{H}}$ with underlying classical truncations both isomorphic to the identity.
\end{proof}

\begin{remark}
\label{rem:hom_dual}
An analogous construction allows for the definition of a derived scheme $\mathbb{H}(X,B\Gm)$ with underlying classical truncation $X^\vee$, the dual abelian scheme, with cotangent complex the pullback of $\omega_{X^\vee}$. This is of course just the smooth $R$-scheme $X^\vee$, now viewed as a derived scheme over $R$. 

This mode of construction of $X^\vee$ makes evident the double-duality isomorphism $X\xrightarrow{\simeq}\mathbb{H}(X^\vee,B\Gm)$ and also provides a canonical equivalence $\mathbb{H}(X,Y)\xrightarrow{\simeq}\mathbb{H}(Y^\vee,X^\vee)$ of derived schemes over $R$.
\end{remark}

\section{Derived special cycles}\label{sec:shimura}

In this section, I prove Theorem~\ref{introthm:W_scheme}, following the strategy sketched in the introduction. The main result is Theorem~\ref{thm:W_representability}.

\subsection{}
Let $(G,X)$ be a Shimura datum with reflex field $E$ with a character $\nu:G\to \Gm$. For any $x\in X$, we have the Deligne cocharacter $h_x:\mathbb{S}\to G_{\Real}$ and the Shimura cocharacter
\[
\mu_x: \Gmh{\Comp}\xrightarrow{z\mapsto (z,1)}\Gmh{\Comp}\times \Gmh{\Comp}\simeq \mathbb{S}_{\Comp}\xrightarrow{h_x}G_{\Comp}.
\]

A \defnword{representation of Siegel type} for $(G,X)$ is an algebraic representation $H$ for $G$ that can be equipped with a symplectic form $\psi$ such that the action of $G$ on $H$ is via a map $G\to \GSp(H,\psi)$ such that the similitude character of $\GSp(H,\psi)$ restricts to $\nu$ on $G$ and such that it extends to a map of Shimura data
\[
(G,X)\to (\GSp(H,\psi),S^{\pm}(H,\psi)),
\]
where on the right hand side we are considering the Siegel Shimura datum associated with $(H,\psi)$. We will call $\psi$ a $G$-\defnword{polarization} of $H$.

Recall that $(G,X)$ is of Hodge type if it admits a faithful representation of Siegel type. We will only be concerned with such Shimura data from here on.\footnote{Everything that is said here applies also to Shimura data of abelian type by allowing representations of Siegel type for central covers of $G$, but the arguments get even more technical, so I've stuck with the Hodge type case here.} 

Fix a compact open subgroup $K\subset G(\Adele_f)$. By~\cites{kisin:abelian,Kim2016-fb}, the Shimura stack $\Sh_K(G,X)$ over $E$ with $\Comp$-points
\[
\Sh_K(G,X)(\Comp) = G(\Rat)\backslash (X\times G(\Adele_f)/K) 
\]
admits a smooth integral canonical model $\Ss_K$ over $\Reg{E}[N_K^{-1}]$, where, the primes $p\nmid N_K$ are precisely those such that the $p$-primary part $K_p\subset G(\Rat_p)$ of $K$ is obtained as the group of $\Int_p$-points of a reductive model for $G$ over $\Int_p$.

\subsection{}
\label{subsec:automorphic_sheaves}
The integral model $\Ss_K$ carries various `automorphic' sheaves functorially associated with a pair $(W,W_{\Int})$, where $W$ is a finite dimensional $\Rat$-vector space equipped with an algebraic representation of $G$, and $W_{\Int}\subset W$ is a $K$-\defnword{stable} lattice, so that $W_{\Int}(\widehat{\Int})\subset W(\Adele_f)$ is stabilized by $K$. What we need is summarized in~\cite{Lovering2017-me}, for instance.

First, we have the \defnword{de Rham realization} $\Fil^\bullet_{\mathrm{Hdg}}\mathbf{dR}(W_{\Int})$, which is a filtered vector bundle over $\Ss_K$ with an integrable connection satisfying Griffiths transversality. The restriction of this object over $\Sh_K$ depends only on $W$, and so we will denote it by $\Fil^\bullet_{\mathrm{Hdg}}\mathbf{dR}(W)$. By Corollary~\ref{cor:crystals_explicit_classical_inf}, the above filtered vector bundle with connection gives an object in $\mathrm{TFilInfCrys}_{\Ss_K/\Rat}$, which we will conflate with $\Fil^\bullet_{\mathrm{Hdg}}\mathbf{dR}(W)$ without further comment. In other words, we have a canonical functor
\begin{align}
\label{eqn:tfilinfcrys_funct}
\Rep(G)&\to \mathrm{TFilInfCrys}_{\Sh_K/\Rat}\\
W&\mapsto \Fil^\bullet_{\mathrm{Hdg}}\mathbf{dR}(W).\nonumber
\end{align}

Over the complex analytic fiber, $\Fil^\bullet_{\mathrm{Hdg}}\mathbf{dR}(W)_{\Comp}$ underlies a variation of $\Int$-Hodge structures, which we denote by $\mathbf{HS}(W_{\Int})$. Note that for any point $z = [(x,g)]\in \Sh_K(\Comp)$, the underlying $\Int$-module of the fiber $\mathbf{HS}_z(W_{\Int})$ is identified with $W\cap g\cdot W_{\Int}(\widehat{\Int})$, and the fiber $\Fil^\bullet_{\mathrm{Hdg}}\mathbf{dR}_z(W)$ is identified with $W(\Comp)$ equipped with the filtration split by $\mu_x^{-1}$. We will write $\mathbf{HS}^{(p,q)}_z(W_{\Int})$ for the space of $(p.q)$-tensors in the $\Int$-module underlying $\mathbf{H}_z(W_{\Int})$.

For every prime $\ell$, the $\Int$-local system underlying $\mathbf{HS}(W_{\Int})$ gives rise to a $\Int_{\ell}$-local system, which descends to a $\Int_{\ell}$-local system $\mathbf{Et}_{\ell}(W_{\Int})$ over $\Sh_K$: this can be described explicitly at any geometric point $z$ of the base as the local system attached to the representation
\[
\pi_1(\Sh_K,z)\to K_{\ell}\to \GL(W_{\Int})(\Int_\ell),
\]
where $K_{\ell}\subset G(\Rat_{\ell})$ is the image of $K$, and the first map classifies the pro-covering
\[
\Sh_{K,\ell} = \varprojlim_{K'\subset K}\Sh_{K'}\to \Sh_K
\]
where $K'$ ranges over the finite index subgroups of $K$ whose image in $G(\Adele_f^{\ell})$ agrees with that of $K$. In fact, the pro-covering and hence the local system extend canonically over $\Ss_K[\ell^{-1}]$, and we will denote this extension by the same symbol.

\subsection{}\label{subsec:realization_props}
For every prime $p$, the restriction of $\mathbf{dR}(W_{\Int})$ over the $p$-completion $\Ss^{\form}_{K,p}$ underlies a crystal $\mathbf{Crys}_p(W_{\Int})$ over the absolute crystalline site of $\Ss^{\form}_{K,p}$, which, via Corollary~\ref{cor:crystals_explicit_classical} gives an object in $\mathrm{TFilCrys}_{\Ss^{\form}_{K,p}/\Int_p}$. As we did over the generic fiber, we will conflate $\mathbf{Crys}_p(W_{\Int})$ with this object. Let $\mathrm{Latt}_{G,K}$ be the category of pairs $(W,W_{\Int})$ where $W\in \Rep(G)$ and $W_{\Int}\subset W$ is a $K$-stable lattice. Then we have a functor
\begin{align}
\label{eqn:tfilcrys_funct}
\mathrm{Latt}_{G,K}&\to \mathrm{TFilCrys}_{\Ss^{\form}_{K,p}/\Int_p}\\
(W,W_{\Int})&\mapsto \mathbf{Crys}_p(W_{\Int})\nonumber.
\end{align}

These realizations have various useful properties:

\begin{enumerate}
	\item\label{real:abelian} (Realizations of abelian schemes)
      If we have $(H,H_{\Int})$ in $\mathrm{Latt}_{G,K}$ with $H$ of Siegel type, then there exists an abelian scheme $\mathcal{A}_H\to \Ss_K$ essentially characterized by the fact that its de Rham \emph{homology} is identified with $\Fil^\bullet_{\mathrm{Hdg}}\mathbf{dR}(H_{\Int})$, and the associated variation of Hodge structures over $\Sh_K(\Comp)$ is $\mathbf{HS}(H_{\Int})$. Its $\ell$-adic Tate module over $\Ss_K[\ell^{-1}]$ is identified with $\mathbf{Et}_\ell(H_{\Int})$. Moreover, $\mathbf{Crys}_p(H_{\Int})$ is simply the dual of the first relative crystalline cohomology sheaf of $\mathcal{A}_H$ over $\Ss^{\form}_{K,p}$. In particular, by Lemma~\ref{lem:ab_coh_split}, it lifts canonically to an object in $\mathrm{FilFCrys}^{\ge 0}_{\Ss^{\form}_{K,p}}$, which we denote by $\mathbf{FFCrys}_p(H_{\Int})$.

   \item\label{real:berthelot-ogus} (Berthelot-Ogus type isomorphism)
   Suppose that $\mathcal{O}$ is a complete DVR of mixed characteristic $(0,p)$ and residue field $k$. For $z\in \Ss_K(\mathcal{O})$, let $z_0\in \Ss_K(k)$ be the special fiber. Then, for any $(W,W_{\Int})$ in $\mathrm{Latt}_{G,K}$, there is a canonical isomorphism
   \begin{align}
    \label{eqn:de rham cris}
    \mathcal{O}[p^{-1}]\otimes_{W(k)}\mathbf{Crys}_{p,z_0}(W_{\Int})\xrightarrow{\simeq}\mathbf{dR}_z(W_{\Int})[p^{-1}]
    \end{align}
      
   \item\label{real:integral_p-adic} (Integral $p$-adic comparison)
   With the notation as above, fix an algebraic closure of $\mathcal{O}[p^{-1}]$ with associated absolute Galois group $\Gamma$, and view $\mathbf{Et}_{p,z}(W_{\Int})$ as a continuous representation of $\Gamma$. Then $\mathbf{Et}_{p,z}(W_{\Int})$ is a lattice in a crystalline Galois representation of $\Gamma$. 

   We will now need some integral $p$-adic Hodge theory from~\cite[(1.3)]{kisin:abelian}. Fix a uniformizer $\pi\in \mathcal{O}$, and let $\mathcal{E}(u)\in W(k)[u]$ be the associated Eisenstein polynoimial. Let $\mathrm{Mod}^\varphi_{/W(k)\pow{u}}$ be the category of finite free $W(k)\pow{u}$-modules $\mathfrak{M}$ equipped with an isomorphism $\varphi^* \mathfrak{M}[\mathcal{E}(u)^{-1}]\xrightarrow{\simeq}\mathfrak{M}[\mathcal{E}(u)^{-1}]$. 

   Let $\Rep^{\mathrm{cris}\circ}_{\Gamma}$ be the category of $\Gamma$-stable lattices in crystalline Galois representations of $\Gamma$. Then there is a fully faithful tensor functor $\mathfrak{M}:\Rep^{\mathrm{cris}\circ}_{\Gamma}\to \mathrm{Mod}^\varphi_{/W(k)\pow{u}}$.

   Now, we have a canonical isomorphism of $W(k)$-modules
   \[
    W(k)\otimes_{u\mapsto 0,W(k)\pow{u}}\varphi^*\mathfrak{M}(\mathbf{Et}_{p,z}(W_{\Int}))\xrightarrow{\simeq}\mathbf{Crys}_{p,z_0}(W_{\Int})
   \]
   characterized by the fact that when $(W,W_{\Int}) = (H,H_{\Int})$ is of Siegel type, it is the dual of the isomorphism arising in~\cite[Theorem (1.1.6)]{Kisin2017-qa}.
\end{enumerate}

Consider the following assumption on $W\in \Rep(G)$:
\begin{assumption}
\label{assump:siegel_type}
There exist representations $H_1$ and $H_2$ of Siegel type such that $W$ embeds into $\Hom(H_1,H_2)$ as a representation of $G$.
\end{assumption}

\begin{lemma}
\label{lem:rep_0_props}
The subcategory $\Rep(G)^0_{\mu}\subset\Rep(G)$ of representations satisfying the above assumption has the following properties:
\begin{enumerate}
	\item It contains the trivial $1$-dimensional representation of $G$.

	\item It is closed under formation of subquotients, duals and direct sums.

	\item For every $W\in \Rep(G)^0_\mu$, we have
     \[
	\gr^i_{\mathrm{Hdg}}\mathbf{dR}(W) \neq 0 \Rightarrow\text{ $i=-1,0,1$}.
	\]
\end{enumerate}
\end{lemma}
\begin{proof}
For any Siegel type representation $H$, the scalars in $\End(H)$ span the trivial representation of $G$: this shows (1).

In (2), closure under sub-quotients is clear from the semi-simplicity of $\Rep(G)$. Now observe that, if $W\subset \Hom(H_1,H_2)$, then $W^\vee$ is a quotient of $\Hom(H_1,H_2)^\vee \simeq \Hom(H_2,H_1)$: this gives the assertion about duals. The statement about direct sums is also easy: If $H'_1,H'_2$ are two other representations of Siegel type, then $\Hom(H_1,H_2)\oplus \Hom(H'_1,H'_2)$ is a subrepresentation of $\Hom(H_1\oplus H'_1,H_2\oplus H'_2)$.

The last assertion is immediate from the fact that it is true for representations of the form $\Hom(H_1,H_2)$.
\end{proof}


As in the introduction, let $\mathrm{Latt}^0_{G,\mu,K}\subset \mathrm{Latt}_{G,K}$ be the subcategory of pairs $(W,W_{\Int})$ where $W$ is in $\Rep(G)^0_\mu$. This is an exact category in a natural way.

\begin{proposition}
\label{prop:cohom realizations}
The restriction of the functor~\eqref{eqn:tfilcrys_funct} to $\mathrm{Latt}^0_{G,\mu,K}$ lifts canonically to a functor
\begin{align*}
\mathrm{Latt}^0_{G,\mu,K}&\to \mathrm{FilFCrys}^{\ge -1}_{\Ss^{\form}_{K,p}/\Int_p}\\
(W,W_{\Int})&\mapsto \mathbf{FFCrys}_p(W_{\Int})
\end{align*}
with the following property: If $(H_1,H_{1,\Int})$ and $(H_2,H_{2,\Int})$ are two pairs with $H_1,H_2$ of Siegel type and if
  \[
    (E,E_{\Int}) = (\Hom(H_1,H_2),\Hom(H_{1,\Int},H_{2,\Int}),
  \]
then $\mathbf{FFCrys}_p(E_{\Int})$ is the internal Hom object between $\mathbf{FFCrys}_p(H_{1,\Int})$ and $\mathbf{FFCrys}_p(H_{2,\Int})$ given by Lemma~\ref{lem:filfcrys_internal_hom}.
 \end{proposition}
\begin{proof}
Given Proposition~\ref{prop:strong_div}, this follows from~\cite[Theorem 3.3.3]{Lovering2017-fy}. Note that I'm implicitly using here (and also above in the construction of $\mathbf{dR}(W_{\Int})$) the fact that, if $K_p = G_{\Int_p}(\Int_p)$ for a reductive model $G_{\Int_p}$ of $G_{\Rat_p}$, then any $K_p$-stable lattice of a representation of $G_{\Rat_p}$ underlies an algebraic representation of $G_{\Int_p}$: this is a consequence of the Zariski density of $K_p = G_{\Int_p}(\Int_p)$ in $G_{\Int_p}$.
\end{proof}

For any ($p$-adic formal) stack $X$, write $\mathrm{qSmAb}_{/X}$ for the $\infty$-category of locally (on the source) quasi-smooth and finite unramified morphisms $Y\to X$ lifting to abelian group objects in the $\infty$-category $\mathrm{PStk}_{/X}$: The following corollary is immediate from the results of Sections~\ref{sec:infcrys} and~\ref{sec:fcrys}.
\begin{corollary}
\label{cor:funct_local_spaces}
There are canonical functors
\begin{align*}
\Rep(G)&\to \mathrm{Ab}(\mathrm{PStk}_{/\Sh_K})\\
W&\mapsto \mathcal{Z}_{K,\infty}(W) \overset{\mathrm{defn}}{=}\mathcal{Z}(\Fil^\bullet_{\mathrm{Hdg}}\mathbf{dR}(W));
\end{align*}
\begin{align*}
\mathrm{Latt}^0_{G,\mu,K}&\to \mathrm{qSmAb}_{/\Ss^{\form}_{K,p}}\\
(W,W_{\Int})&\mapsto \mathcal{Z}_{K,crys}(W_{\Int})^{\form}_p\overset{\mathrm{defn}}{=}\mathcal{Z}(\mathbf{FFCrys}_p(W_{\Int}))^{\form}_p,
\end{align*}
where we are using notation from Proposition~\ref{prop:FIC_functorial_props} and Theorem~\ref{thm:FilFCrys_functorial_props}. In particular, we have canonical equivalences
\begin{align*}
\mathbb{L}_{\mathcal{Z}_{K,\infty}(W)/\Sh_K}&\xrightarrow{\simeq}\Reg{\mathcal{Z}_{K,\infty}(W)}\otimes_{\Reg{\Ss_K}}\bm{co}(W_{\Int})[1]\\
\mathbb{L}_{\mathcal{Z}_{K,crys}(W_{\Int})^{\form}_p/\Ss^{\form}_{K,p}}&\xrightarrow{\simeq}\Reg{\mathcal{Z}_{K,crys}(W_{\Int})^{\form}_p}\otimes_{\Reg{\Ss_K}}\bm{co}(W_{\Int})[1],
\end{align*}
where
\[
\bm{co}(W_{\Int}) \overset{\mathrm{defn}}{=}(\gr^{-1}_{\mathrm{Hdg}}\mathbf{dR}(W_{\Int}))^\vee.
\]
\end{corollary}

\begin{remark}
\label{rem:points_of_spaces_over_fields}
The points of the above prestacks over points $z\in \Ss_K(k)$ with $k$ a perfect field can be described quite explicitly:
\begin{itemize}
\item When $k$ is a field of characteristic $0$, $\Fil^\bullet_{\mathrm{Hdg}}\mathbf{dR}_z(W)$ is a filtered vector space over $k$ equipped with an integrable connection $\nabla$ with respect to $\Omega^1_{k/\Rat}$, and we have
\[
\mathcal{Z}_{K,\infty}(W)(z) \simeq \Fil^0_{\mathrm{Hdg}}\mathbf{dR}_z(W)^{\nabla = 0}.
\]

\item When $k$ is a (perfect) field of characteristic $p$, $\mathbf{FFCrys}_{p,z}(W_{\Int})$ can be evaluated on $W(k)\twoheadrightarrow k$ to give a filtered $F$-crystal over $W(k)$, which we denote once again by the same symbol (see~\eqref{subsec:M_on_perfect}). By Theorem~\ref{thm:FilFCrys_discrete}~\eqref{ffc:perfect}, we have
\[
\mathcal{Z}_{K,crys}(W_{\Int})^{\form}_p(z) \simeq \mathbf{FFCrys}_{p,z}(W_{\Int})^{\varphi_0 = \mathrm{id}}.
\]
\end{itemize}
\end{remark}

The functors above satisfy some useful exactness properties which are immediate from the definitions.
\begin{proposition}
\label{prop:spaces_exactness_properties}
Suppose that we have a short exact sequence 
\[
0\to (W',W'_{\Int})\to (W,W_{\Int})\to (W'',W''_{\Int})\to 0
\]
in $\mathrm{Latt}^0_{G,\mu,K}$. Then we have
\begin{align*}
    \mathcal{Z}_{K,\infty}(W')&\simeq \hker^{\mathrm{cn}}(\mathcal{Z}_{K,\infty}(W)\to \mathcal{Z}_{K,\infty}(W''))\\
	\mathcal{Z}_{K,crys}(W'_{\Int})^{\form}_p&\simeq \hker^{\mathrm{cn}}(\mathcal{Z}_{K,crys}(W_{\Int})^{\form}_p\to \mathcal{Z}_{K,crys}(W''_{\Int})^{\form}_p).
\end{align*}

\end{proposition}

\subsection{}
\label{subsec:homomorphism_spaces}
Let $H_1,H_2$ be a representations of Siegel type, and let $H_{i,\Int}\subset H_i$ be a $K$-stable lattice for $i=1,2$. Let $\mathcal{A}_1,\mathcal{A}_2\to \Ss_K$ be the two abelian schemes over $\Ss_K$ associated with $(H_1,H_{1,\Int})$ and $(H_2,H_{2,\Int})$. 

By Theorem~\ref{thm:deform_end}, we have a locally finite quasi-smooth unramified stack 
\[
\mathbb{H}(\mathcal{A}_1,\mathcal{A}_2)\to \Ss_K
\]
of derived homomorphisms from $\mathcal{A}_1$ to $\mathcal{A}_2$, equipped with formally \'etale maps
\[
\mathbb{H}(\mathcal{A}_1,\mathcal{A}_2)_{\Rat}\to \mathcal{Z}_{K,\infty}(E)\;;\; \mathbb{H}(\mathcal{A}_1,\mathcal{A}_2)^{\form}_p \to \mathcal{Z}_{K,p}(E_{\Int})^{\form}_p.
\]
Note that we can identify the cotangent complex of $\mathbb{H}(\mathcal{A}_1,\mathcal{A}_2)$ with the pullback of the shifted vector bundle $\bm{co}(E_{\Int})[1]$.

We can now state the main result of this section, which will imply Theorem~\ref{introthm:W_scheme}.
\begin{theorem}
\label{thm:W_representability}
There is a canonical functor
\begin{align*}
\mathrm{Latt}^0_{G,\mu,K}&\to \mathrm{qSmAb}_{/\Ss_K}\\
(W,W_{\Int})&\mapsto \mathcal{Z}_K(W_{\Int})
\end{align*}
determined by the following properties:
\begin{enumerate}
\item (Cotangent complex) There is a functorial equivalence
\[
\mathbb{L}_{\mathcal{Z}_K(W_{\Int})/\Ss_K}\xrightarrow{\simeq}\Reg{\mathcal{Z}_K(W_{\Int})}\otimes_{\Reg{\Ss_K}}\bm{co}(W_{\Int})[1].
\]
\item (Local realizations) There exist natural formally \'etale maps
\[
\mathcal{Z}_K(W_{\Int})_{\Rat}\to \mathcal{Z}_{K,\infty}(W)\;;\;\mathcal{Z}_{K}(W_{\Int})^{\form}_p\to \mathcal{Z}_{K,crys}(W_{\Int})^{\form}_p
\]
compatible with the obvious identifications of cotangent complexes.
\item (Complex points) There is a functorial equivalence
\[
\mathcal{Z}_K(W_{\Int})(\Comp)\xrightarrow{\simeq}G(\Rat)\backslash X(W_{\Int})/K,
\]
where
\[
X(W_{\Int}) = \{(w,x,g)\in W(\Adele_f)\times X\times G(\Adele_f):\; w\in gW_{\Int}(\widehat{\Int})\;;h_x(i)\cdot w = w\}.
\]

\item (Mod $p$ points) For any finite field $k$ of characteristic $p$ and any $z\in \Ss_K(k)$, there is a functorial isomorphism
\[
\Int_p\otimes_{\Int}\mathcal{Z}_K(W_{\Int})(z)\xrightarrow{\simeq}\mathcal{Z}_{K,crys}(W_{\Int})^{\form}_p(z).
\]

\item (Homomorphisms) If $(E,E_{\Int})\in \mathrm{Latt}^0_{G,\mu,K}$ is as in~\eqref{subsec:homomorphism_spaces} above, then there is a canonical equivalence
\[
\mathbb{H}(\mathcal{A}_1,\mathcal{A}_2)\xrightarrow{\simeq}\mathcal{Z}_K(E_{\Int})
\]
compatible with the natural identifications of cotangent complexes.
\end{enumerate}
\end{theorem}

The proof will be completed after Proposition~\ref{prop:W_special_l-adic} below.

\begin{remark}
\label{rem:canonical_cycle}
For \emph{any} Shimura datum of Hodge type, we have a \emph{canonical} instance of a representation $W$ as above: the adjoint representation on the Lie algebra $\Lie G$! This is a subrepresentation of $\End(H)$ for any representation $H$ of Siegel type. Therefore, we obtain a canonical family of locally quasi-smooth maps $\mathcal{Z}_K(L_{\Int})\to \Ss_K$, depending on choices of $K$-stable lattices $L_{\Int}\subset \Lie G$, of virtual codimension $\dim \Sh_K$. As indicated to me by A. Mihatsch, these appear closely related to the `big fat' CM cycles considered in~\cite{MR4250392} and~\cite{Mihatsch2021-mf}.
\end{remark}	

\subsection{}
We begin by fixing $(W,W_{\Int})$ in $\mathrm{Latt}^0_{G,\mu,K}$, as well as representations $H_1$ and $H_2$ of Siegel type such that $W\subset E = \Hom(H_1,H_2)$. Let $H_{i,\Int}\subset H_i$ for $i=1,2$ be $K$-stable lattices such that $W_{\Int}\subset \Hom(H_{1,\Int},H_{2,\Int})$ and set
\[
W'_{\Int} = W\cap \Hom(H_{1,\Int},H_{2,\Int}).
\]

Let $\mathcal{A}_1,\mathcal{A}_2\to \Ss_K$ be the abelian schemes associated with $H_{1,\Int}$ and $H_{2,\Int}$, respectively. 

Set $\mathcal{Z}_K(E_{\Int}) \overset{\mathrm{defn}}{=} \mathbb{H}(\mathcal{A}_1,\mathcal{A}_2)$. Now define:
\[
\mathcal{Z}_K(W'_{\Int})_{\Rat} \overset{\mathrm{defn}}{=} \mathcal{Z}_K(E_{\Int})_{\Rat}\times_{\mathcal{Z}_{K,\infty}(E)}\mathcal{Z}_{K,\infty}(W')\;;\; \mathcal{Z}_K(W'_{\Int})^{\form}_p \overset{\mathrm{defn}}{=}  \mathcal{Z}_K(E_{\Int})^{\form}_p\times_{\mathcal{Z}_{K,crys}(E_{\Int})^{\form}_p}\mathcal{Z}_{K,crys}(W'_{\Int})^{\form}_p,
\]

\begin{lemma}
\label{lem:Winf_Wp_connected}
We have closed immersions
\begin{align}
\label{eqn:generic fiber W to E}
\mathcal{Z}_K(W'_{\Int})_{\Rat}\hookrightarrow \mathcal{Z}_{K}(E_{\Int})_{\Rat}
\end{align}
in $\mathrm{Ab}(\mathrm{PStk}_{/\Sh_K})$; and
\begin{align}
\label{eqn:pcomp W to E}
\mathcal{Z}_K(W'_{\Int})^{\form}_p \hookrightarrow \mathcal{Z}_K(E_{\Int})^{\form}_p
\end{align}
in $\mathrm{Ab}(\mathrm{PStk}_{/\Ss^{\form}_{K,p}})$, which are both open immersions of the underlying classical (formal) stacks.
\end{lemma}

\begin{proof}
Consider the pair
\[
(U,U_{\Int}) = (E/W,E_{\Int}/W'_{\Int}).
\]

Then by Proposition~\ref{prop:spaces_exactness_properties}, we have
\[
\mathcal{Z}_{K,\infty}(W') \simeq \mathrm{hker}^{\mathrm{cn}}\left(\mathcal{Z}_{K,\infty}(E)\to \mathcal{Z}_{K,\infty}(U)\right)\text{ (resp. $\mathcal{Z}_{K,crys}(W'_{\Int})^{\form}_p\simeq \mathrm{hker}^{\mathrm{cn}}\left(\mathcal{Z}_{K,crys}(E_{\Int})^{\form}_p\to \mathcal{Z}_{K}(U_{\Int})^{\form}_{p,crys}\right) $)}
\]

Over the generic fiber, we now have
\[
\mathcal{Z}_K(W'_{\Int})_\Rat = \mathcal{Z}_K(E_{\Int})_{\Rat}\times_{\mathcal{Z}_{K,\infty}(U),0}\Sh_K\hookrightarrow \mathcal{Z}_K(E_{\Int})_{\Rat}.
\]

Since the classical truncation $\mathcal{Z}^{\mathrm{cl}}_K(E_{\Int})$ is locally of finite type over $\Int$, we can invoke (2) of Proposition~\ref{prop:FIC_discrete} to conclude.

For the $p$-adic version, the same argument applies, except that we now have to use assertion (2) of Theorem~\ref{thm:FilFCrys_discrete}.
\end{proof}

\begin{definition}
Suppose that $k$ is a field and that $s\in \mathcal{Z}_K(E_{\Int})(k)$ is a $k$-valued point. We will say that $s$ is \defnword{$W'$-special} if one of the following holds:
\begin{enumerate}
  \item $k$ has characteristic $0$ and $s\in \mathcal{Z}_K(W'_{\Int})_{\Rat}(k)$;
  \item $k$ has characteristic $p$ and $s\in \mathcal{Z}_K(W'_{\Int})^{\form}_p(k)$.
\end{enumerate}
\end{definition}

By definition, this is a geometric condition:
\begin{lemma}
\label{lem:W_special_geometric}
With the notation above, $s$ is $W'$-special if and only if for some (hence any) field extension $L/k$, the induced point of $\mathcal{Z}_K(E_{\Int})(L)$ is $W'$-special.
\end{lemma}
\begin{proof}
This is because being $W'$-special is a condition equivalent to belonging to the open and closed substack $\mathcal{Z}^{\mathrm{cl}}_K(W'_{\Int})_{\Rat}\subset \mathcal{Z}^{\mathrm{cl}}_{K}(E)_{\Rat}$ (in the characteristic $0$ case), or the open and closed formal substack $\mathcal{Z}_K^{\mathrm{cl}}(W'_{\Int})^{\form}_p\subset \mathcal{Z}_K^{\mathrm{cl}}(E_{\Int})^{\form}_p$ (in the characteristic $p$ case).
\end{proof}

\begin{lemma}
\label{lem:W_special_complex}
Suppose that $k = \Comp$; then $s\in \mathcal{Z}_K(E_{\Int})(\Comp)$ lying over $x\in \Ss_K(\Comp)$ is $W'$-special precisely when its Betti realization lies in the subspace 
\[
\mathbf{HS}^{(0,0)}_x(W'_{\Int}) \subset \mathbf{HS}^{(0,0)}_x(E_{\Int}) = \Hom(H_1(\mathcal{A}_{1,x},\Int),H_1(\mathcal{A}_{2,x},\Int)).
\]
In particular, we have an isomorphism of complex orbifolds:
\[
G(\Rat)\backslash X(W'_{\Int})/K \xrightarrow{\simeq}\mathcal{Z}_K(W'_{\Int})(\Comp),
\]
where $X(W'_{\Int})$ is as in Theorem~\ref{thm:W_representability}.
\end{lemma}
\begin{proof}
The first statement is immediate from the definition and the second is now standard. See for instance the proof of~\cite[Prop. 2.3.2]{HMP:mod_codim}.
\end{proof}

\begin{lemma}
\label{lem:W_special}
Suppose that $X$ is a connected, classical Noetherian scheme with a map $x:X\to \mathcal{Z}_K(E_{\Int})$. Then the following are equivalent:
\begin{enumerate}
  \item For some point $s:\Spec k \to X$ with $k$ a field, $x\circ s\in \mathcal{Z}_K(E_{\Int})(k)$ is $W'$-special.
  \item For \emph{every} point $s:\Spec k\to X$ with $k$ a field, $x\circ s\in \mathcal{Z}_K(E_{\Int})(k)$ is $W'$-special.
\end{enumerate}
\end{lemma} 
\begin{proof}
If $s_1,s_2$ are two points of $X$, we need to show that $x\circ s_1$ is $W'$-special if and only if $x\circ s_2$ is so. By considering a connecting sequence of genericizations and specializations, we reduce to the case where $X = \Spec \mathcal{O}$ where $\mathcal{O}$ is a complete DVR and $s_1$ and $s_2$ are its generic and special points, respectively. Using Lemma~\ref{lem:W_special_geometric}, we can also assume that the residue field is perfect.

If $\mathcal{O}$ is of equal characteristic, then we are already done by Lemma~\ref{lem:Winf_Wp_connected}.

Therefore, we have to consider the situation where it is of unequal characteristic $(0,p)$. Note that $x$ corresponds to a point $y\in \Ss_K(\mathcal{O})$ and a homomorphism $f:\mathcal{A}_{1,y}\to \mathcal{A}_{2,y}$. Such a homomorphism has realizations
\[
\bm{f}_{crys,s_2}\in \mathbf{Crys}_{p,s_2}(E_{\Int})\;;\; \bm{f}_{dR,s_1}\in \mathbf{dR}_{s_1}(E)..
\]

Now, using Remark~\ref{rem:points_of_spaces_over_fields}, one finds that $x\circ s_2$ (resp. $x\circ s_1$) is special precisely when we have $\bm{f}_{crys,s_2}\in \mathbf{Crys}_{p,s_2}(W'_{\Int})$ (resp. $\bm{f}_{dR,s_1}\in \mathbf{dR}_{s_1}(W')$). That these conditions are equivalent is a consequence of the functorial isomorphism~\eqref{eqn:de rham cris}.
\end{proof}

Let $\mathcal{Z}^{\mathrm{cl}}_K(W'_{\Int})\subset \mathcal{Z}^{\mathrm{cl}}_K(E_{\Int})$ be the union of connected components meeting either $\mathcal{Z}_K(W'_{\Int})_{\Rat}$ or $\mathcal{Z}_K(W_{\Int})^{\form}_p$ for some prime $p$. As an immediate consequence of Lemmas~\ref{lem:Winf_Wp_connected} and~\ref{lem:W_special}, we obtain:

\begin{proposition}
\label{prop:classical_W}
The generic fiber (resp.  $p$-adic formal completion) of the open and closed immersion $\mathcal{Z}^{\mathrm{cl}}_K(W'_{\Int})\hookrightarrow \mathcal{Z}^{\mathrm{cl}}_K(E_{\Int})$ is the open and closed immersion of classical stacks underlying~\eqref{eqn:generic fiber W to E} (resp. ~\eqref{eqn:pcomp W to E}). In particular, for every discrete ring $C$, we have
\[
\mathcal{Z}^{\mathrm{cl}}_K(W'_{\Int})(C) = \{x\in \mathcal{Z}_K(E_{\Int})(C):\;x\circ s\text{ is $W'$-special for all points $s$ of $\Spec C$}\}
\]
\end{proposition} 

\begin{remark}
Lemma~\ref{lem:W_special} and hence the representability of the functor given by the right hand side of the equality above can be established using only classical methods; see the argument used for the equivalence of assertions (2) and (3) in~\cite[Prop. 4.3.4]{Andreatta2018-tt}. Where the derived methods seem to be somewhat essential in this general situation is in establishing that the functor has the `correct' deformation theory as described in (2) of Corollary~\ref{cor:classical_truncation} below: In the context of the cited result from~\cite{Andreatta2018-tt}, this would essentially amount to the equivalence of assertions (1) and (3), which used strong \emph{a priori} knowledge of the deformation theory of $W'$-special points of $\mathcal{Z}_K(E_{\Int})$ in that particular context.

Also, as I noted in the introduction, I don't know how to avoid the results of Section~\ref{sec:fcrys} in the construction of the closed substack $\mathcal{Z}_K^{\mathrm{cl}}(W_{\Int})$, if we cannot choose $H_{1,\Int}$ and $H_{2,\Int}$ so that $W_{\Int} = W'_{\Int}$.
\end{remark}

The next result is a generalization of~\cite[Theorem 6.4(2)]{mp:tatek3}, and can be seen as a sort of Tate's theorem for $W'$-special endomorphisms. The proof is a little different from that in \emph{loc. cit.}, in that it uses some finer information from~\cite{Kisin2017-qa}. I should warn the reader that the ideas involved here are somewhat orthogonal to those germane to the rest of this paper.
\begin{proposition}
\label{prop:tate_conjecture}
Suppose that $k$ is a finite field of characteristic $p$; then for all $z\in \Ss_K(k)$, the natural map
\[
\Int_p\otimes_{\Int}\mathcal{Z}^{\mathrm{cl}}_K(W'_{\Int})(z)\to \mathcal{Z}_{K,crys}(W'_{\Int})^{\form}_p(z)
\]
is an isomorphism.
\end{proposition}
\begin{proof}
If we replace $W'_{\Int}$ with $E_{\Int}$, then the proposition is the crystalline avatar of Tate's theorem on homomorphisms between abelian varieties over finite fields. Given that we have
\[
\mathcal{Z}^{\mathrm{cl}}_K(W'_{\Int})(z) = \mathcal{Z}_K(E_{\Int})(z)\cap \mathcal{Z}_{K,crys}(W'_{\Int})^{\form}_p(z)\subset \mathcal{Z}_{K,crys}(E_{\Int})^{\form}_p(z),
\]
it now suffices to verify that the map in question is an isomorphism after inverting $p$.

Let $(H,\psi)$ be a faithful representation of $G$ of Siegel type: replacing it with $H\oplus H_1\oplus H_2$ if necessary, we can assume that $H_1\oplus H_2$ is a direct summand of $H$ as a symplectic representation of $G$.

Associated with $(H,\psi)$ is a family of abelian schemes up-to-isogeny $\mathcal{A}\to \Ss_K$, and by our hypothesis $\mathcal{A}_z$ admits $\mathcal{A}_{1,z}\times \mathcal{A}_{2,z}$ as an isogeny direct summand. 

For each $\ell\neq p$, the Frobenius endomorphism of $\mathcal{A}_z$ induces an element $\gamma_{z,\ell}\in \Aut(\Rat_\ell\otimes_{\Int_\ell}\mathbf{Et}_{\ell,z}(H_{\Int}))$. When $\ell = p$, if $p^r = q = |k|$, then the $r$-th iterate of $\varphi_0$ on $\mathbf{Crys}_{p,z}(H_{\Int})[p^{-1}]$ is $\Int_q$-linear, and so gives rise to an element
\[
\gamma_{z,p}\in \Aut(\Rat_q\otimes_{\Int_q}\mathbf{Crys}_{p,z}(H_{\Int})).
\]
Here, we have written $\Int_q$ for $W(k)$ and $\Rat_q$ for its field of fractions.

As in~\cite[(1.3.4)]{Kisin2021-cl}, write $V^\otimes$ to denote the direct sum of tensor powers of an object $V$ and its dual $V^\vee$ in a rigid tensor category. Then, as explained in \emph{loc. cit.}, any set of tensors $\{s_{\alpha}\}\subset H^\otimes$ whose joint stabilizer is $G$ canonically gives rise to tensors
\[
\{s_{\alpha,\ell,z}\}\subset \Rat_{\ell}\otimes_{\Int_{\ell}}\mathbf{Et}_{\ell,z}(H_{\Int})^\otimes\;;\; \{s_{\alpha,p,z}\}\subset \Rat_q\otimes_{\Int_q}\mathbf{Crys}_{p,z}(H_{\Int})^\otimes
\]
with the following property: For each $\alpha$, $s_{\alpha,\ell,z}$ is $\gamma_{\ell,z}$-invariant for $\ell\neq p$, and $s_{\alpha,p,z}$ is $\varphi_0$-invariant. Moreover, there exist canonical isomorphisms (up to the action of $G$)
\begin{align*}
\Rat_\ell\otimes_{\Rat}H&\xrightarrow{\simeq}\Rat_{\ell}\otimes_{\Int_{\ell}}\mathbf{Et}_{\ell,z}(H_{\Int})\text{ for $\ell\neq p$};\\
\Rat_q\otimes_{\Rat}H&\xrightarrow{\simeq}\Rat_q\otimes_{\Int_q}\mathbf{Crys}_{p,z}(H_{\Int})
\end{align*}
carrying $1\otimes s_{\alpha}$ to $s_{\alpha,z,\ell}$ for all primes $\ell$ and all indices $\alpha$. For $\ell\neq p$, this is immediate from the definition of the functor $W_{\Int}\mapsto \mathbf{Et}_{\ell,z}(W_{\Int})$, and when $\ell = p$, this follows from~\cite[Corollary (1.4.3)(3)]{kisin:abelian}. If one is allowed to replace $k$ with a finite extension, then this also follows from~\cite[Proposition 1.3.7]{Kisin2021-cl} and Steinberg's theorem.

In particular, we can and will view $G_{\Rat_{\ell}}$ for $\ell\neq p$ as the joint stabilizer of the subset $\{s_{\alpha,z,\ell}\}$ of $\Rat_\ell\otimes_{\Int_{\ell}}\mathbf{Et}_{\ell,z}(H_{\Int})^\otimes$, and $G_{\Rat_q}$ as the joint stabilizer of the subset $\{s_{\alpha,z,p}\}$ of $\Rat_q\otimes_{\Int_q}\mathbf{Crys}_{p,z}(H_{\Int})^\otimes$. In particular, we see that we have
\[
\gamma_{z,\ell}\in G(\Rat_\ell)\text{ for $\ell \neq p$}\;;\; \gamma_{z,p}\in G(\Rat_q).
\]

As in~\cite[2.1.3]{Kisin2021-cl}, let $I_z\subset \underline{\Aut}_{\Rat}(\mathcal{A}_z)$ be the largest reductive subgroup fixing the tensors $s_{\alpha,\ell,z}$ for all $\ell$ and all $\alpha$ via its homological realization. 

For each $\ell\neq p$, let $I_{z,\ell}$ be the joint stabilizer in $\GL(\mathbb{Et}_{\ell,z}(H_{\Int})[\ell^{-1}])$ of the tensors $s_{\alpha,\ell,z}$ as well as $\gamma_{z,\ell}$.

For $\ell = p$, let $I_{z,p}$ be as in~\cite[2.1.3]{Kisin2021-cl}: it is the reductive group of automorphsims of the $F$-isocrystal with $G$-structure $\mathbf{Crys}_{p,z}(H_{\Int})[p^{-1}]$. 

From~\cite[Corollary 2.1.9]{Kisin2021-cl}---or, in the case of good reduction considered here, from~\cite[Corollary (2.3.2)]{Kisin2017-qa}---we obtain:
\begin{proposition}
\label{prop:tates_theorem}
For every $\ell$, homological realization gives us an isomorphism of reductive groups
\[
\Rat_{\ell}\otimes_{\Rat}I_z\xrightarrow{\simeq}I_{z,\ell}.
\]
\end{proposition}

\begin{lemma}
\label{lem:Igamma0}
There exists $\gamma_0\in G(\Rat)$ which lies in the $G$-conjugacy class of $\gamma_{z,\ell}$ for every prime $\ell$. Moreover, let $I_{\gamma_0}\subset G$ be the centralizer of $\gamma_0$; then, after replacing $k$ by a finite extension if necessary, we can find $\gamma_0$ such that there exist isomorphisms
\begin{align*}
\beta_{\ell}:\Rat_\ell\otimes_{\Rat}H&\xrightarrow{\simeq}\Rat_{\ell}\otimes_{\Int_{\ell}}\mathbf{Et}_{\ell,z}(H_{\Int})\text{ for $\ell\neq p$};\\
\beta_p:\Rat_q\otimes_{\Rat}H&\xrightarrow{\simeq}\Rat_q\otimes_{\Int_q}\mathbf{Crys}_{p,z}(H_{\Int})
\end{align*}
such that: 
\begin{enumerate}[label=(\roman*)]
    \item $\beta_{\ell}(1\otimes s_{\alpha}) = s_{\alpha,\ell,z}$ for all $\alpha$;
	\item $\beta_{\ell}\circ \gamma_0 = \gamma_{z,\ell}\circ \beta_{\ell}$ for all primes $\ell$, so that conjugation by $\beta_{\ell}$ induces isomorphisms
	\begin{align*}
	j_{\ell}:\Rat_\ell\otimes_{\Rat}I_{\gamma_0}\xrightarrow{\simeq}I_{z,\ell}\xrightarrow{\simeq}\Rat_{\ell}\otimes_{\Rat}I_z\text{ if $\ell\neq p$};\\
	j_{p}:\Rat_q\otimes_{\Rat}I_{\gamma_0}\xrightarrow{\simeq}\Rat_q\otimes_{\Rat_p}I_{z,p}\xrightarrow{\simeq}\Rat_{q}\otimes_{\Rat}I_z\text{ if $\ell = p$}.
	\end{align*}
	
	\item For any pair of primes $\ell,\ell'$ and any algebraically closed field $K$ containing $\Rat_{\ell}$ (for $\ell\neq p$) and $\Rat_q$, $j_{\ell'}^{-1}\circ j_{\ell}$ is given by conjugation by an element of $I_{\gamma_0}(K)$.
\end{enumerate}
\end{lemma}
\begin{proof}
 The existence of isomorphisms $\beta_{\ell}$ satisfying condition (i) was already noted above. That we can find $\gamma_0$ such that there exists isomorphisms satisfying (ii) as well follows from~\cite[Corollary (2.3.1)]{Kisin2017-qa} and Steinberg's theorem (for $\ell =p$). This last point is the reason we might need a finite extension of $k$.

 Assertion (iii) is shown as in the proof of~\cite[Proposition 2.2.13]{Kisin2021-cl} using Weil points; see in particular the penultimate paragraph. Note that, because of~\cite[Corollary (2.3.1)]{Kisin2017-qa}, we can replace $\Aut'(G)$ in the argument from~\cite{Kisin2021-cl} with $G$ itself.
\end{proof}	

If $k'/k$ is a finite extension and $z'\in \Ss_K(k')$ is the induced point, then we have
\[
\mathcal{Z}^{\mathrm{cl}}_K(W'_{\Int})(z) = \mathcal{Z}_K(E_{\Int})(z)\cap \mathcal{Z}^{\mathrm{cl}}_K(W'_{\Int})(z')\;;\; \mathcal{Z}_{K,crys}(W'_{\Int})^{\form}_p(z) = \mathcal{Z}_{K,crys}(E_{\Int})^{\form}_p(z)\cap \mathcal{Z}_{K,crys}(W'_{\Int})^{\form}_p(z'). 
\]
Therefore, it is permissible for our purposes to replace $k$ with any finite extension. In particular, we can assume that the conclusions of the lemma above are valid.

For simplicity, set $\mathbf{E}(z) = \Rat\otimes_{\Int}\mathcal{Z}_K(E_{\Int})(z)$. The proposition would be proved if we could find a (rational) subspace $\mathbf{W}(z)\subset \mathbf{E}(z)$ such that the composition $\mathbf{W}(z)\to \mathbf{E}(z) \to \mathbf{Crys}_{p,z}(E_{\Int})$ induces an isomorphism
\begin{align}\label{eqn:desired_w_isom}
\Rat_q\otimes_{\Rat}\mathbf{W}(z)\xrightarrow{\simeq}\Rat_q\otimes_{\Int_q}\mathbf{Crys}_{p,z}(W'_{\Int})^{\gamma_{z,p} = \mathrm{id}}
\end{align}
of $\Rat_q$-vector spaces.

Tate's theorem gives us canonical isomorphisms
\[
\alpha_{\ell}:\Rat_{\ell}\otimes_{\Rat}\mathbf{E}(z)\xrightarrow{\simeq}\Rat_{\ell}\otimes_{\Int_\ell}\mathbf{Et}_{\ell,z}(E_{\Int})^{\gamma_{z,\ell} = \mathrm{id}}\;;\;\alpha_p:\Rat_q\otimes_{\Rat}\mathbf{E}(z)\xrightarrow{\simeq}\Rat_q\otimes_{\Int_q}\mathbf{Crys}_{p,z}(E_{\Int})^{\gamma_{z,p}=\mathrm{id}}.
\]
On the other hand, set $E^0 = E^{\gamma_0 = \mathrm{id}}$ and $W^0 = W^{\gamma_0 = \mathrm{id}}$; then the isomorphisms $\beta_{\ell}$ from Lemma~\ref{lem:Igamma0} induce isomorphisms
\begin{align*}
\beta_{\ell}:\Rat_{\ell}\otimes_{\Rat}E^0 &\xrightarrow{\simeq} \Rat_{\ell}\otimes_{\Int_\ell}\mathbf{Et}_{\ell,z}(E_{\Int})^{\gamma_{z,\ell} = \mathrm{id}}\text{ if $\ell\neq p$};\\
\beta_p:\Rat_q\otimes_{\Rat}E^0&\xrightarrow{\simeq}\Rat_q\otimes_{\Int_q}\mathbf{Crys}_{p,z}(E_{\Int})^{\gamma_{z,p} = \mathrm{id}}.
\end{align*}

Combining these with the $\alpha_?$ gives us isomorphisms
\[
\eta_{\ell} = \beta_{\ell}^{-1}\circ \alpha_{\ell}:\Rat_{\ell}\otimes_{\Rat}\mathbf{E}(z)\xrightarrow{\simeq}\Rat_{\ell}\otimes_{\Rat}E^{\gamma_0=\mathrm{id}}\;;\;\eta_p = \beta_p^{-1}\circ \alpha_p:\Rat_q\otimes_{\Rat}\mathbf{E}(z)\xrightarrow{\simeq}\Rat_q\otimes_{\Rat}E^0.
\]
Assertion (ii) of Lemma~\ref{lem:Igamma0} implies that, for any algebraically closed field $K$ containing $\Rat_\ell$, for $\ell\neq p$, and $\Rat_q$, the composition $\eta_{\ell'}\circ \eta_{\ell}^{-1}$ is given by an element of $\GL(E^0)(K)$ lying in the image of $I_{\gamma_0}(K)$. From this one deduces that there exists a canonical subspace $\mathbf{W}(z)\subset \mathbf{E}(z)$ such that
\[
\eta_{\ell}(\Rat_{\ell}\otimes_{\Rat}\mathbf{W}(z)) = \Rat_{\ell}\otimes_{\Rat}W^0\;;\; \eta_p(\Rat_q\otimes_{\Rat}\mathbf{W}(z)) = \Rat_q\otimes_{\Rat}W^0.
\]

Therefore we have
\[
\alpha_p(\Rat_q\otimes_{\Rat}\mathbf{W}(z)) = \Rat_q\otimes_{\Int_q}\mathbf{Crys}_{p,z}(W_{\Int})^{\gamma_{z,p} = \mathrm{id}},
\]
which gives us the desired isomorphism~\eqref{eqn:desired_w_isom}.

\end{proof}

As a corollary to the proof above, we find:
\begin{proposition}
\label{prop:W_special_l-adic}
For any prime $\ell$, the $\ell$-adic realization map $\mathcal{Z}_K(E_{\Int})[\ell^{-1}]\to \mathbf{Et}_{\ell}(E_{\Int})$ restricts to a map
\[
\mathcal{Z}_K(W'_{\Int})[\ell^{-1}]\to \mathbf{Et}_{\ell}(W'_{\Int}).
\]
\end{proposition}	
\begin{proof}
This can be checked at closed geometric points $z\in\Ss_K(k)$ of characteristic $p\neq \ell$, where it follows from the fact that, in the notation of the proof above, $\beta_\ell$ carries $\Rat_\ell\otimes\mathbf{W}(z)$ onto $\Rat_{\ell}\otimes_{\Int_{\ell}}\mathbf{Et}_{\ell,z}(W'_{\Int})$.
\end{proof}

\begin{proof}
[Proof of Theorem~\ref{thm:W_representability}]

We will employ the constructions above that depended on the fixed choice of $(H_1,H_{1,\Int})$ and $(H_2,H_{2,\Int})$, and proceed in steps:

\emph{Step 1}: \underline{Existence of $\mathcal{Z}_K(W'_{\Int})$}

The cotangent complex of $\mathcal{Z}_K(E_{\Int})$ over $\Ss_K$ can be identified with the pullback of $\bm{co}(E_{\Int})[1]$ by (4) of Theorem~\ref{thm:deform_end}, and the cotangent complexes of $\mathcal{Z}_K(W'_{\Int})_\Rat$ over $\Sh_K$ and of $\mathcal{Z}_K(W'_{\Int})^{\form}_p$ over $\Ss^{\form}_{K,p}$ can both be identified with the respective pullbacks of $\bm{co}(W'_{\Int})[1]$. 

Choose an \'etale neighborhood $\Spec R\to \Ss_K$, write 
\[
\kappa_E = H^0(\Spec R,\bm{co}(E_{\Int}))\;;\; \kappa_{W'}= H^0(\Spec R,\bm{co}(W'_{\Int})),
\]
and write $u':\kappa_E\to \kappa_{W'}$ for the natural surjection.

Then in the language of~\eqref{subsec:closed_immersions} and Proposition~\ref{prop:unique_Y}, the existence of $\mathcal{Z}_K(W'_{\Int})$ satisfying properties (1) through (5) of Theorem~\ref{thm:W_representability} now comes down to the assertion that, for any \'etale neighborhood as above, there exists a $u$-marked scheme over $\mathcal{Z}_K(E_{\Int})\times_{\Ss_K}\Spec R$ with underlying classical scheme $\mathcal{Z}^{\mathrm{cl}}(W'_{\Int})\times_{\Ss_K}\Spec R$. By Corollary~\ref{cor:exist_Y}, this amounts to showing this existence over $R_{\Rat}$ and $R/p^n$ for every prime $p$ and $n\geq 1$, which we already have, from the construction of $\mathcal{Z}_K(W'_{\Int})_{\Rat}$ and $\mathcal{Z}_K(W'_{\Int})^{\form}_p$, respectively. This completes the construction of $\mathcal{Z}_K(W'_{\Int})$.

\emph{Step 2:} \underline{Existence of $\mathcal{Z}_K(W_{\Int})$}

Set $N = [W'_{\Int}:W_{\Int}]$; then the definition over $\Int[N^{-1}]$ is easy. Proposition~\ref{prop:W_special_l-adic} gives us a map
\[
\mathcal{Z}_K(W'_{\Int})[N^{-1}] \to \prod_{\ell\mid N}\mathbf{Et}_{\ell}(W'_{\Int})/\mathbf{Et}_{\ell}(W_{\Int})\vert_{\Ss_K[N^{-1}]},
\]
and we take $\mathcal{Z}_K(W_{\Int})[N^{-1}]$ to be the fiber over $0$ under this map: this is an open and closed substack of $\mathcal{Z}_K(W'_{\Int})[N^{-1}]$.

For $p\mid N$, let $M$ be the largest prime-to-$p$ divisor of $N$ and let $\tilde{\mathcal{Z}} \to \mathcal{Z}_K(W'_{\Int})[M^{-1}]$ be the fiber over $0$ of the map
\[
\mathcal{Z}_K(W'_{\Int})[M^{-1}]\to \prod_{\ell\mid M}\mathbf{Et}_{\ell}(W'_{\Int})/\mathbf{Et}_{\ell}(W_{\Int})\vert_{\Ss_K[M^{-1}]}.
\]
This is of course an open and closed substack of $\mathcal{Z}_K(W'_{\Int})[M^{-1}]$.

The locally finite unramified map $\mathcal{Z}_{K,crys}(W_{\Int})^{\form}_p\to \mathcal{Z}_{K,crys}(W'_{\Int})^{\form}_p$ is a closed immersion of $p$-adic formal stacks since it is injective on points valued in perfect fields; see Remark~\ref{rem:points_of_spaces_over_fields}. Now, set
\[
\mathcal{Z}_K(W_{\Int})^{\form}_p \overset{defn}{=} \tilde{Z}^{\form}_p\times_{\mathcal{Z}_{K,crys}(W'_{\Int})^{\form}_p}\mathcal{Z}_{K,crys}(W_{\Int})^{\form}_p.
\]
This gives a closed formal substack of $\mathcal{Z}_K(W'_{\Int})^{\form}_p$ with cotangent complex $\Reg{\mathcal{Z}_K(W_{\Int})^{\form}_p}\otimes_{\Ss_K} \bm{co}(W_{\Int})[1]$ over $\Ss_{K,p}^{\form}$. If $\Spf R'\to \mathcal{Z}_K(W'_{\Int})$ is an \'etale chart, and $\hat{R}\in \mathrm{CRing}_{R'/}$ is such that
\[
\Spf \hat{R} \xrightarrow{\simeq}\Spf \hat{R}'_p\times_{\mathcal{Z}_K(W'_{\Int})^{\form}_p}\mathcal{Z}_K(W_{\Int})^{\form}_p,
\]
then $\hat{R}'_p[1/p]\to \hat{R}[1/p]$ is an \'etale surjection, and so exhibits $\Spec \hat{R}[1/p]$ as an open and closed subscheme of $\Spec\hat{R}'_p[1/p]$. I claim that this is identified with
\[
\Spec\hat{R}'_p[1/p]\times_{\mathcal{Z}_K(W'_{\Int})[N^{-1}]}\mathcal{Z}_K(W_{\Int})[N^{-1}]
\]
Indeed, it suffices to check this claim on $L$-points for a finite extension $L/\Rat_p$ extending to $\Reg{L}$-points of $\mathcal{Z}_K(W'_{\Int})$, where it is a consequence of the integral comparison isomorphism~\eqref{subsec:realization_props}~\eqref{real:integral_p-adic}.

Therefore, we can glue the $p$-adic formal closed substacks $\mathcal{Z}_K(W_{\Int})^{\form}_p$ of $\mathcal{Z}_K(W'_{\Int})^{\form}_p$ for $p\mid N$ with $\mathcal{Z}_K(W_{\Int})[N^{-1}]$ to get the stack $\mathcal{Z}_K(W_{\Int})$ with the desired properties. 

\emph{Step 3:} \underline{Functoriality and independence from choices}

It remains to show that this construction is functorial in $(W,W_{\Int})$. Suppose that we have two such pairs $(W,W_{\Int})$ and $(W^\flat,W^\flat_{\Int})$, as well as a $G$-equivariant map $\beta:W\to W^\flat$ carrying $W_{\Int}$ into $W^\flat_{\Int}$. The proof now follows a standard embedding trick. We will repeatedly use the functoriality from Corollary~\ref{cor:funct_local_spaces} without comment.

Suppose that $(H^\flat_1,H^\flat_{1,\Int})$ and $(H^\flat_2,H^\flat_{2,\Int})$ is a choice of pairs of Siegel type with $W^\flat\subset E^\flat = \Hom(H^\flat_1,H^\flat_2)$ and $W^\flat_{\Int} \subset \Hom(H^\flat_{1,\Int},H^\flat_{2,\Int})$. Write $\mathcal{Z}_K(W^\flat_{\Int})\hookrightarrow \mathcal{Z}_K(E^\flat_{\Int})$ for the closed immersion of locally quasi-smooth stacks over $\Ss_K$ arising from this data via Step 2.

For $i=1,2$, take $(\tilde{H}_i,\tilde{H}_{i,\Int}) = (H_i\oplus H^\flat_i,H_{i,\Int}\oplus H^\flat_{i,\Int})$; then the graph of $\beta$ gives us an embedding of direct summands
\[
W_{\Int}\hookrightarrow E_{\Int}\oplus E^\flat_{\Int}\hookrightarrow \tilde{E}_{\Int} = \Hom(\tilde{H}_{1,\Int},\tilde{H}_{2,\Int}).
\]
Repeating the construction of Step 2 with this embedding gives us a closed immersion $\widetilde{\mathcal{Z}}_K(W_{\Int})\hookrightarrow \mathcal{Z}_K(\widetilde{E}_{\Int})$.
 
Corollary~\ref{cor:hom_additive} gives us a natural projection $\mathcal{Z}_K(\widetilde{E}_{\Int})\to \mathcal{Z}_K(E_{\Int})$, which in turn gives us a map $\alpha:\widetilde{\mathcal{Z}}_K(W_{\Int})\to\mathcal{Z}_K(E_{\Int})$; I claim that this factors through an equivalence $\widetilde{\mathcal{Z}}_K(W_{\Int})\xrightarrow{\simeq}\mathcal{Z}_K(W_{\Int})$.

By construction, for any \'etale map $\Spec R\to \Ss_K$, the restriction of the maps
\[
\widetilde{\mathcal{Z}}_K(W_{\Int})\to \mathcal{Z}_K(E_{\Int})\;;\; \mathcal{Z}_K(W_{\Int})\to \mathcal{Z}_K(E_{\Int})
\]
over $\Spec R$ are both $u$-marked, where 
\[
u:H^0(\Spec R,\bm{co}(E_{\Int}))\to H^0(\Spec R,\bm{co}(W_{\Int})).
\]

One sees that Proposition~\ref{prop:tate_conjecture} is now valid with $W'_{\Int}$ replaced by $W_{\Int}$. Secondly, in the notation of Lemma~\ref{lem:W_special_complex}, both $\widetilde{\mathcal{Z}}_K(W_{\Int})(\Comp)$ and $\mathcal{Z}_K(W_{\Int})(\Comp)$ are canonically identified with $G(\Rat)\backslash X(W_{\Int})/K$. 

The claim now follows from Corollary~\ref{cor:Y_determined_by_points}.

Now, Corollary~\ref{cor:hom_additive} also gives us a projection $\mathcal{Z}_K(\widetilde{E}_{\Int})\to \mathcal{Z}_K(E^\flat_{\Int})$, and thus a map $\widetilde{\mathcal{Z}}_K(W_{\Int})\to \mathcal{Z}_K(E^\flat_{\Int})$, which factors canonically through $\mathcal{Z}_K(W^\flat_{\Int})$: check this using the definitions separately over the generic fiber and over the $p$-adic formal completions. 

Putting the two previous paragraphs together gives us the arrow $\mathcal{Z}_K(W_{\Int})\to \mathcal{Z}_K(W^\flat_{\Int})$ associated with $\beta$. The only thing that needs to be checked is that this construction is now compatible with composition of maps, and this can be verified using a similar argument. This of course also shows the independence from the choices made during the course of the construction.
\end{proof}

The next corollary yields Theorem~\ref{introthm:classical}:
\begin{corollary}
\label{cor:classical_truncation}
Let $\mathcal{Z}^{\mathrm{cl}}_K(W_{\Int})$ be the classical truncation of $\mathcal{Z}_K(W_{\Int})$
\begin{enumerate}
	\item For every square-zero thickening $\tilde{C}\twoheadrightarrow C$ in $\mathrm{CRing}_{\heartsuit}$ with kernel $I$, and every point $\tilde{x}\in \Ss_K(\tilde{C})$ inducing $x\in \Ss_K(C)$, there is an obstruction map:
\[
\mathrm{obst}_{\tilde{x}}:\mathcal{Z}^{\mathrm{cl}}_K(W_{\Int})(x)\to \gr^{-1}_{\mathrm{Hdg}}\mathbf{dR}_x(W_{\Int})\otimes_CI
\]
such that
\[
\mathcal{Z}^{\mathrm{cl}}_K(W_{\Int})(\tilde{x}) = \ker\mathrm{obst}_{\tilde{x}}.
\]
\item Let $\mathcal{Z}\subset \mathcal{Z}_K(W_{\Int})$ be a connected component, and let $\mathcal{Z}^{\mathrm{cl}}\subset \mathcal{Z}^{\mathrm{cl}}_K(W_{\Int})$ be its classical truncation. Then the following are equivalent:
\begin{enumerate}[label=(\alph*)]
\item The natural map $\mathcal{Z}\to \mathcal{Z}^{\mathrm{cl}}$ is an equivalence.
\item $\mathcal{Z}^{\mathrm{cl}}$ is an lci stack over $\Int$ of relative dimension $\dim\Sh_K - nd_+(W)$.
\item $\mathcal{Z}^{\mathrm{cl}}$ is an equidimensional stack over $\Int$ of relative dimension $\dim\Sh_K - nd_+(W)$.
\end{enumerate}
Here, $d_+(W)$ is the rank of the vector bundle $\gr^{-1}_{\mathrm{Hdg}}\mathbf{dR}(W)$.
\end{enumerate} 
\end{corollary}
\begin{proof}
The assertion about the deformation theory is a direct translation of (1) of Theorem~\ref{thm:W_representability}. 

For the second assertion about classicality, note that a finite unramified map of classical Deligne-Mumford stacks is quasi-smooth if and only if it is locally on the source given by a regular immersion; see~\cite[Propositions 2.2.4,2.3.8]{Khan2018-dk}. Since $\Ss_K$ is smooth over $\Int$, this means that the map $\mathcal{Z}^{\mathrm{cl}}\to \Ss_K$ is quasi-smooth if and only if $\mathcal{Z}^{\mathrm{cl}}$ is an lci stack over $\Int$. So the equivalence of statements (a) and (b) is implied by the observation---immediate from~\cite[(25.3.6.6)]{Lurie2018-kh}---that a quasi-smooth unramified morphism of derived Deligne-Mumford stacks is an equivalence if and only if it has virtual codimension $0$ and is an isomorphism of the underlying classical stacks. For the equivalence of (b) and (c), note that \'etale locally on the source $\mathcal{Z}^{\mathrm{cl}}\to \Ss_K$ is cut out by at most $nd_+(W)$ equations. Therefore, the desired equivalence comes down to the fact that a sequence $a_1,\ldots,a_m$ in a Cohen-Macaulay Noetherian local ring $A$ is regular if and only if 
\[
\dim\left(A/(a_1,\ldots,a_m)\right) = \dim A - m.
\]




\end{proof}

\section{Cycle classes}\label{sec:cycle_classes}

We will maintain the notation from the previous section. The results there will now be combined here with Appendix~\ref{app:cycle_classes} to yield our main theorem on the existence and properties of the classes $\mathcal{C}_{\Ss_K,W}(T,\varphi)$. We begin by establishing some addenda to Theorem~\ref{thm:W_representability}.

\begin{proposition}
\label{prop:product_formula}
Suppose that we have two pairs $(W_1,W_{1,\Int})$ and $(W_2,W_{2,\Int})$ in $\mathrm{Latt}^0_{G,\mu,K}$ with associated locally quasi-smooth maps $\mathcal{Z}_{K}(W_{1,\Int}),\mathcal{Z}_{K}(W_{2,\Int})\to \Ss_K$. Set $(U,U_{\Int}) = (W_1\oplus W_2,W_{1,\Int}\oplus W_{2,\Int})$. Then there is a canonical equivalence 
\[
\mathcal{Z}_{K}(W_{1,\Int})\times_{\Ss_K}\mathcal{Z}_{K}(W_{2,\Int}) \simeq \mathcal{Z}_{K}(U_{\Int}).
\]
\end{proposition}
\begin{proof}
The theorem gives us \'etale maps in both directions, and it remains to verify that they are inverses when evaluated over geometric points $z\in \Ss_K(k)$. This is clear when $k$ has characteristic $0$, and, when $k$ has characteristic $p$, it can be checked using (4) of Theorem~\ref{thm:W_representability}.
\end{proof}

\begin{proposition}
\label{prop:trivial_rep}
Suppose that $G$ acts trivially on $W$; then $\mathcal{Z}_{K}(W_{\Int})$ is equivalent to the locally constant \'etale sheaf $\underline{W}_{\Int}$ on $\Ss_K$ associated with $W_{\Int}$.
\end{proposition}
\begin{proof}
By proposition~\ref{prop:product_formula}, it suffices to do this for lattices in the $1$-dimensional trivial representation, which embeds as the space of scalars in $\End(H)$ for any representation $H$ of Siegel type, and the result is easy to check directly in this case.
\end{proof}

\begin{proposition}
\label{prop:W_and_W'}
Suppose that we have $W\in \mathrm{Rep}(G)^0_{\mu}$ and two $K$-stable lattices $W_{\Int}\subset W'_{\Int}\subset W$ such that $K$ acts trivially on the quotient $W'_{\Int}/W_{\Int}$. Then homotopy cokernel of the map $\mathcal{Z}_{K}(W_{\Int})\to \mathcal{Z}_K(W'_{\Int})$ of abelian group objects in $\mathrm{PStk}_{/\Ss_K}$ is canonically equivalent to the constant \'etale group scheme $\underline{W'_{\Int}/W_{\Int}}$.
\end{proposition}
\begin{proof}
Set $\mathcal{Z} = \mathcal{Z}_K(W_{\Int})$ and $\mathcal{Z}' = \mathcal{Z}_K(W'_{\Int})$, and write $i:\mathcal{Z}\to \mathcal{Z}'$ for the map between them. We now have a Cartesian diagram
\[
\begin{diagram}
\hcoker(i)&\rTo^\pi&\Ss_K\\
\dTo&&\dTo_{0}\\
\mathcal{Z}[1]&\rTo_{i[1]}&\mathcal{Z}'[1]
\end{diagram}
\]
in $\mathrm{PStk}_{/\Ss_K}$. Using this and (1) of Theorem~\ref{thm:W_representability}, we find that we have a cofiber sequence
\begin{equation}
\label{eqn:etale_cotangent_complex}
\mathbb{L}_{\hcoker(i)/\Ss_K}\to \pi^*\bm{co}(W'_{\Int})[2]\to \pi^*\bm{co}(W_{\Int})[2]
\end{equation}
in $\mathrm{QCoh}_{\hcoker(i)}$.

\begin{lemma}\label{lem:U_for_W'_and_W}
The map $\bm{co}(W'_{\Int})\to \bm{co}(W_{\Int})$ is an isomorphism of vector bundles over $\Ss_K$.
\end{lemma}
\begin{proof}
This is clear over $\Sh_K$. To see this over the integral model, choose a surjection $U_{\Int}\to W'_{\Int}/W_{\Int}$ with $U_{\Int}$ a finite free abelian group, and consider the following diagram of short exact sequences, where the top sequence is obtained from the bottom via pullback:
\[
\begin{diagram}
0&\rTo&W_{\Int}&\rTo&W^\sharp_{\Int}&\rTo& U_{\Int}&\rTo& 0\\
&&\dTo&&\dTo&&\dTo\\
0&\rTo &W_{\Int}&\rTo&W'_{\Int}&\rTo&W'_{\Int}/W_{\Int}&\rTo& 0.
\end{diagram}
\]

Then $W^\sharp_{\Int}$ is a lattice in a $\Rat$-vector space $W^\sharp$, viewed as a representation of $G$ obtained as the (necessarily trivial) extension of $U = U_{\Int}\otimes\Rat$ with the trivial action of $G$ by $W$. Now, since $\gr^i_{\mathrm{Hdg}}\mathbf{dR}(U_{\Int}) = 0$ for $i\neq 0$, we find that the natural maps
\[
\gr^{-1}_{\mathrm{Hdg}}\mathbf{dR}(W_{\Int})\to \gr^{-1}_{\mathrm{Hdg}}\mathbf{dR}(W^\sharp_{\Int})\;;\;\gr^{-1}_{\mathrm{Hdg}}\mathbf{dR}(W^\sharp_{\Int})\to \gr^{-1}_{\mathrm{Hdg}}\mathbf{dR}(W'_{\Int})
\]
are both surjective, which gives the lemma.
\end{proof}	

The sequence~\eqref{eqn:etale_cotangent_complex} and the above lemma combine to show that $\hcoker(i)$ is (locally) finite \'etale over $\Ss_K$. Therefore to prove the proposition, it suffices to establish the equivalence $\hcoker(i)\simeq \underline{W'_{\Int}/W_{\Int}}$ over the generic fiber. 

\begin{lemma}
There is a functorial \'etale realization map
\[
\mathcal{Z}_K(W_{\Int})\vert_{\Sh_K}\to \mathbf{Et}_{\widehat{\Int}}(W_{\Int}) \overset{\mathrm{defn}}{=} \prod_{\ell}\mathbf{Et}_{\ell}(W_{\Int}).
\]
\end{lemma}
\begin{proof}
This follows from the construction in the proof of Theorem~\ref{thm:W_representability}, combined with Proposition~\ref{prop:W_special_l-adic}. 
\end{proof}

The above lemma gives us a canonical map
\[
\hcoker(i)\vert_{\Sh_K}\to \mathbf{Et}_{\widehat{\Int}}(W'_{\Int})/\mathbf{Et}_{\widehat{\Int}}(W_{\Int}) \simeq \underline{W'_{\Int}/W_{\Int}},
\]
where the last isomorphism follows from the fact that $K$ acts trivially on $W'_{\Int}/W_{\Int}$. That the above composition is an isomorphism can now be checked over $\Comp$ using (2) of Theorem~\ref{thm:W_representability}.
\end{proof}

\begin{proposition}
\label{prop:W_dual}
Suppose that we have $(W,W_{\Int})$ in $\mathrm{Latt}^0_{G,\mu,K}$ and a $K$-stable lattice $\tilde{W}_{\Int}\subset W^\vee$ in the dual representation. Then, for any $s\in \Ss_K(C)$ with $C$ discrete, there exists a canonical pairing
\[
\mathcal{Z}_K(W_{\Int})(s)\times \mathcal{Z}_K(\tilde{W}_{\Int})(s)\to \Rat,
\]
characterized by the following properties: 
\begin{enumerate}
	\item When $C= \Comp$ and $s = [(x,g)]$,  the pairing is the restriction of the canonical $\Rat$-valued pairing between $W$ and $W^\vee$.
	\item When $C = k$ is a perfect field of characteristic $p$, the pairing is the restriction of the canonical pairing
	\[
      \mathbf{Crys}_{p,s}(W_{\Int})\times \mathbf{Crys}_{p,s}(\tilde{W}_{\Int})\to W(k)[1/p]
	\]
\end{enumerate}
\end{proposition}
\begin{proof}
The uniqueness of a functorial in $s$ pairing with the given properties is clear. It remains to establish existence. The assertion is insensitive to the choice of $K$-stable lattices $W_{\Int}$ and $\tilde{W}_{\Int}$, so we will choose them as necessary for the proof. 

Fix representations of Siegel type $H_1,H_2$ with $W \subset \Hom(H_1,H_2)$, as well as $G$-polarizations on $H_1$, $H_2$, giving rise to $G$-equivariant isomorphisms $j_i:H_i^\vee\xrightarrow{\simeq}H_i(\nu)$ for $i=1,2$. This equips $\Hom(H_1,H_2)$ with a $G$-equivariant map
\[
\Hom(H_1,H_2)\xrightarrow{\alpha\mapsto \alpha^*}\Hom(H_2,H_1),
\]
where
\[
\alpha^* = j_1\circ\alpha^\vee\circ j_2^{-1}\in \Hom(H_1(\nu),H_2(\nu)) \simeq \Hom(H_2,H_1).
\]
The pairing $(\alpha,\beta)\mapsto \mathrm{Tr}(\alpha^*\circ \beta)$ gives a \emph{non-degenerate} $G$-invariant bilinear form on $\Hom(H_1,H_2)$. We can therefore use this to write down an orthogonal decomposition
\[
\Hom(H_1,H_2) = W \oplus W^{\perp}.
\]
In turn, this gives us a decomposition
\[
\Hom(H_2,H_1) \simeq \Hom(H_1,H_2)^\vee = W^{\vee}\oplus W^{\perp,\vee},
\]
and allows us to view $W^{\vee}$ as a sub-representation of $\Hom(H_2,H_1)$. Given this, we reduce the proof of existence to the case where $W = \Hom(H_1,H_2)$, where we can take a $K$-stable lattices of the form 
\[
W_{\Int} = \Hom(H_{1,\Int},H_{2,\Int})\;;\;\tilde{W}_{\Int} = \Hom(H_{2,\Int},H_{1,\Int})
\]
for $K$-stable lattices $H_{i,\Int}$. In this case, we have associated abelian schemes $\mathcal{A}_1,\mathcal{A}_2$ over $\Ss_K$ along with identifications
\[
\mathcal{Z}^{\mathrm{cl}}_K(W_{\Int})\simeq \underline{\Hom}(\mathcal{A}_1,\mathcal{A}_2)\;;\; \mathcal{Z}^{\mathrm{cl}}_K(\tilde{W}_{\Int})\simeq \underline{\Hom}(\mathcal{A}_2,\mathcal{A}_1),
\]
and the canonical pairing is given by the composition
\[
\underline{\Hom}(\mathcal{A}_1,\mathcal{A}_2)\times \underline{\Hom}(\mathcal{A}_2,\mathcal{A}_1)\xrightarrow{(f,g)\mapsto g\circ f}\underline{\End}(\mathcal{A}_1)\xrightarrow{\mathrm{Tr}}\underline{\Int},
\]
where the last map is the unique one compatible with the trace on every homological realization.
\end{proof}

\begin{remark}
\label{rem:pairing_etale}
By construction, the pairing above is also compatible under the $\ell$-adic realization map with the canonical pairing
\[
\mathbf{Et}_{\ell}(W_{\Int})\times \mathbf{Et}_{\ell}(\tilde{W}_{\Int})\to \underline{\Rat}_{\ell}
\]
over $\Ss_K[\ell^{-1}]$.
\end{remark}

\begin{remark}
\label{rem:pairing_integral}
If $\tilde{W}_{\Int} = W^\vee_{\Int}$ is the $\Int$-dual of $W_{\Int}$, then the pairing constructed above is actually $\Int$-valued.
\end{remark}

\begin{corollary}
\label{cor:modular_quadratic}
With the notation as above, suppose that $q:W\to \Rat$ is a $G$-invariant quadratic form with associated bilinear form $[w_1,w_2]_q = q(w_1+w_2)-q(w_1)-q(w_2)$. Then, for every $s\in \Ss_K(C)$ with $C$ discrete, there exists a canonical quadratic form $q':\mathcal{Z}_K(W_{\Int})(s)\to \Rat$ that, when $C = \Comp$ and $s = [(x,g)]$, agrees with the form
\[
q':\mathcal{Z}_K(W_{\Int})(s) \simeq \{w\in W(\Rat)\cap gW_{\Int}(\widehat{\Int}):\;h_x(\Comp^\times)w = w\}\xrightarrow{w\mapsto q(w)}\Rat.
\]
\end{corollary}
\begin{proof}
The quadratic form gives rise to a $G$-equivariant map $W\to W^\vee$. Let $\tilde{W}_{\Int}\subset W^\vee$ be the smallest $K$-stable lattice such that $W_{\Int}$ is carried into $\tilde{W}_{\Int}$. Then we obtain a map
\[
\mathcal{Z}_K(W_{\Int})\to \mathcal{Z}_K(\tilde{W}_{\Int})
\]
canonically associated with $q$. 

Combining this with the pairing from Proposition~\ref{prop:W_dual} now proves the lemma.
\end{proof}

\begin{lemma}
\label{lem:modular_quadratic_congruence}
Suppose that in the situation of Corollary~\ref{cor:modular_quadratic}, we have an inclusion of $K$-stable lattices $W_{\Int}\subset W'_{\Int}\subset W$ such that the following properties hold:
\begin{enumerate}[label=(\roman*)]
	\item $q(W_{\Int})\subset \Int$.
	\item For all $x\in W_{\Int}$, $x'\in W'_{\Int}$, we have $[x,x']_q\in \Int$.
	\item $K$ acts trivially on $W'_{\Int}/W_{\Int}$.
\end{enumerate}{}
Then for every $s\in \Ss_K(C)$ with $C$ discrete and for all $u'\in \mathcal{Z}_K(W'_{\Int})(s)$ mapping to $\mu\in \mathcal{Z}_K(W'_{\Int})(s)/\mathcal{Z}_K(W_{\Int})(s)\simeq W'_{\Int}/W_{\Int}$, we have
\[
q'(u') \equiv q(\tilde{\mu})\pmod{\Int},
\]
where $\tilde{\mu}\in W'_{\Int}$ is any lift of $\mu$.
\end{lemma}	
\begin{proof}
One reduces immediately to the case where $C = k$ is either $\Comp$ or a perfect field of characteristic $p$. The first case is easy, and for the second, the lemma can be checked after inverting $p$ using Remark~\ref{rem:pairing_etale}. It can be checked after localizing at $p$ by looking at the crystalline realizations, and using the integral comparison from~\eqref{subsec:realization_props}~\eqref{real:integral_p-adic}.
\end{proof}

\begin{lemma}
\label{lem:quadratic_polarization}
Suppose that we have $W\in \mathrm{Rep}(G)^0_\mu$ and a $G$-invariant quadratic form $q:W\to \Rat$ such that, for some (hence any) $z=[(x,g)]\in \Sh_K(\Comp)$, the $\Rat$-Hodge structure $\mathbf{HS}_z(W)$ is polarized by $q$. Then for all $s\in \Ss_K(C)$ with $C$ discrete, and any $K$-stable lattice $W_{\Int}\subset W$, the associated quadratic form $q'$ on $\mathcal{Z}_K(W_{\Int})(s)_{\Rat}$ is positive definite.
\end{lemma}
\begin{proof}
The positive definiteness of $q'$ a consequence of the construction and the positivity of the Rosati involution associated with a polarization on an abelian variety; see~\cite[III.21, Theorem 1]{Mumford1974-bw}. 

More precisely, choose any pair $H_1,H_2$ of representations of Siegel type with a $G$-equivariant embedding $W\subset\Hom(H_1,H_2)$. Replacing both $H_1$ and $H_2$ by $H = H_1\oplus H_2$, we see that we have $W\subset\End(H)$. If $\psi$ is a $G$-polarization on $H$, then the abelian scheme up-to-isogeny  $\mathcal{A}\to \Ss_K$ determined by $H$ is equipped with a polarization associated with $\psi$, and this gives a positive Rosati involution $x\mapsto x^*$ on $\End^0(\mathcal{A}_s)$, whose restriction to
\[
\mathcal{Z}_K(W_{\Int})(s)_{\Rat}\subset \End^0(\mathcal{A}_s)
\]
gives a positive definite form $\tilde{q}':x\mapsto \mathrm{Tr}(x^*x)$.

This positive definite form arises from the construction from Corollary~\ref{cor:modular_quadratic} via a non-degenerate$G$-invariant quadratic form $\tilde{q}:W\to \Rat$ determined as follows: Let $j_{\psi}:H\xrightarrow{\simeq}H^\vee(\nu)$ be the $G$-equivariant isomorphism attached to $\psi$, then we have, for $w\in W$,
\[
\tilde{q}(w) = \mathrm{Tr}(j_{\psi}^{-1}\circ w^\vee \circ j_{\psi}\circ w).
\]

Set $B = \End_G(W)$; then we obtain an involution $b\mapsto b^*$ on $B$ determined by 
\[
[bw_1,w_2]_{\tilde{q}} = [w_1,b^*w_2]_{\tilde{q}}
\]
for all $w_1,w_2\in W$ and $b\in B$. The fact that $\tilde{q}$ polarizes $\mathbf{HS}_z(W)$ implies that this is a \emph{positive} involution.\footnote{Recall this means that that $(b_1,b_2)\mapsto \mathrm{Tr}_{B/\Rat}(b_1b_2^*)$ is a positive definite bilinear form on $B$.}

Now, there exists an element $b\in B^\times$ such that, for all $w\in W$, we have
\[
q(w) = [bw,w]_{\tilde{q}} = [w,b^*w]_{\tilde{q}}.
\]
Moreover, the fact that $q$ also polarizes $\mathbf{HS}_z(W)$ is equivalent to the assertion that $b$ is of the form $c^*c$ for some $c\in B^\times_{\Real}$; see~\cite[Lemma 2.8]{kottwitz:points}.

By the functoriality of the construction of $\mathcal{Z}_K(W_{\Int})$ in Theorem~\ref{thm:W_representability}, we have a map
\[
B\to \End(\mathcal{Z}_K(W_{\Int})(s)_{\Rat}),
\]
and we also have
\[
q'(u) = [bu,u]_{\tilde{q}'} = [cu,cu]_{\tilde{q}'}  = 2\tilde{q}'(cu)
\]
for all $u\in \mathcal{Z}_K(W_{\Int})(s)_{\Real}$, showing that $q'$ is also positive definite.
\end{proof}

\subsection{}\label{subsec:W_D_hypotheses}
For the rest of this section, we will fix a semisimple $\Rat$-algebra $D$ equipped with a positive involution $\iota$, and we will assume that $W$ is a representation of $G$ equipped with a $G$-equivariant structure of a $D$-module. We will also assume that there is a non-degenerate $G$-invariant $\iota$-Hermitian form $\Phi:W\times W\to D$. This datum is equivalent to that of the underlying $G$-invariant $\Rat$-bilinear form $\mathrm{Trd}_{D/\Rat}\circ \Phi$, which is of the form $[\cdot,\cdot]_q$ for a $G$-invariant quadratic form $q:W\to \Rat$, and satisfies 
\begin{align}\label{eqn:hermitian_q}
[dw_1,w_2]_q = [w_1,\iota(d)w_2]_q
\end{align}
for all $d\in D$ and $w_1,w_2\in W$. Suppose also that $q$ satisfies the hypotheses of Lemma~\ref{lem:quadratic_polarization}. Then the lemma (and its proof) gives:

\begin{proposition}
\label{prop:W_hermitian}
If $C$ is discrete, then, for any $s\in \Ss_K(C)$, $\mathcal{Z}_K(W_{\Int})(s)_{\Rat}$ has a canonical $D$-module structure, and is equipped with a canonical positive definite $\iota$-Hermitian form $\Phi'$ valued in $D$.
\end{proposition}



\subsection{}
Fix an integer $n\ge 1$. Suppose that  we have a $K$-invariant coset $\mu+U_{\Int}(\widehat{\Int})$ of a $K$-stable lattice $U_{\Int}\subset W^n$. We need not (and indeed should not) assume that $U_{\Int}$ is the $n$-th power of the lattice $W_{\Int}$.

We now have a locally quasi-smooth and finite unramified map $\mathcal{Z}_K(U_{\Int},\mu)\to \Ss_K$ constructed as follows: Let $U'_{\Int}\subset W^n$ be the $K$-stable lattice generated by $U_{\Int}$ and $\mu$; then in the notation of Proposition~\ref{prop:W_and_W'}, $\mathcal{Z}_K(U_{\Int},\mu)$ is the pre-image in $\mathcal{Z}_K(U'_{\Int})$ of the constant section $\mu$ of $\underline{U'_{\Int}/U_{\Int}}$.

By Proposition~\ref{prop:W_hermitian}, $\mathcal{Z}_K(U_\Int,\mu)$ is fibered over the discrete set $\mathrm{Herm}^+_n(D)$ of positive semi-definite $\iota$-Hermitian matrices $n\times n$-matrices over $D$. Indeed, it suffices to see this in the case where $\mu$ can be chosen to be $0$, where we have a canonical identification $\mathcal{Z}_{K}(U_{\Int})(s)_{\Rat} = \mathcal{Z}_K(W_{\Int})(s)_{\Rat}^n$, for any $s\in \Ss_K(C)$ with $C$ discrete. Now, taking the Gram matrix of an element of this space with respect to the canonical $\iota$-Hermitian form $\Phi'$ on $\mathcal{Z}_K(W_{\Int})(s)_{\Rat}$ gives us the desired map $\mathcal{Z}_K(U_{\Int},\mu)\to \mathrm{Herm}^+_n(D)$. More precisely, given an element $(x_1,\ldots,x_n)\in \mathcal{Z}_K(W_{\Int})(s)_{\Rat}^n$, the associated matrix is
\[
\left(\frac{1}{2}\Phi'(x_i,x_j)\right)_{1\leq i,j,\leq n}\in \mathrm{Herm}_n^+(D).
\]

For any $T\in \mathrm{Herm}^+_n(D)$, write $\mathcal{Z}_K(U_{\Int},\mu,T)$ for the fiber of this map over $T$. The positive definiteness of the Hermitian form on $\mathcal{Z}_K(W_{\Int})(s)$ ensures that this is a \emph{finite} quasi-smooth stack over $\Ss_K$.

\subsection{}
Let $\mathcal{S}(W(\Adele_f)^n)$ be the Schwartz space of locally constant $\Comp$-valued functions with compact support on $W(\Adele_f)^n$. For $g\in \GL_n(D)$ and $\varphi\in \mathcal{S}(W(\Adele_f)^n)$ write $L_g\varphi$ for the function $(L_g\varphi)(x) = \varphi(g^{-1}x)$, where the action of $\GL_n(D)$ on $W(\Adele_f)^n$ is via the identification $W(\Adele_f)^n \simeq D^n\otimes_DW(\Adele_f)$. Note that this action commutes with that of $G$ and so preserves $K$-invariant functions. 

Now, $W^n$ inherits an $\iota$-Hermitian form $\Phi^n$ from $W$. For $N$ in the space $\mathrm{Herm}_n(D)$ of $\iota$-Hermitian $n\times n$-matrices over $D$, we obtain a new $\iota$-Hermitian form $\Phi^n_N$ on $W^n$ via $\Phi^n_N(w_1,w_2) = \Phi_n(w_1,(N\otimes \mathrm{id})w_2)$. Here, we are viewing $N\otimes \mathrm{id}$ as an endomorphism of $W^n$ via the identification $W^n = D^n\otimes_DW$. Let $\psi_f:\Adele_f\to \Comp^\times$ be the standard additive character given by
\[
\psi_f((x_p)) = e^{-2\pi i\sum_p\overline{x}_p},
\]
where $\overline{x}_p$ is the image of $x_p\in \Rat_p$ in $\Rat_p/\Int_p\simeq \Int[1/p]/\Int$, and define
\begin{align*}
s_N:W(\Adele_f)^n&\to \Comp\\
w&\mapsto \psi_f\left(\frac{1}{2}\mathrm{Trd}_{D/\Rat}(\Phi^n_N(w,w))\right),
\end{align*}
where $\mathrm{Trd}_{D/\Rat}:D\to \Rat$ is the reduced trace map. 

The next theorem and its proof verify most of Theorems~\ref{introthm:derived_cycles} and~\ref{introthm:main} from the introduction.
\begin{theorem}
\label{thm:main_thm_adelic}
For each $n\ge 1$ and each $T\in \mathrm{Herm}^+_n(D)$, there is a $\Comp$-linear map
\begin{align*}
\mathcal{S}(W(\Adele_f)^n)^K&\to \mathrm{CH}^{nd_+}(\Ss_K)_{\Comp}\\
\varphi&\mapsto \mathcal{C}_{\Ss_K,W}(T,\varphi)
\end{align*}
with the following properties:
\begin{enumerate}
	\item (Characterization) If $\varphi$ is the characteristic function of a $K$-invariant coset $\mu+U_{\Int}(\widehat{\Int})$ of a $K$-stable lattice $U_{\Int}\subset W^n$, then we have
	\[
      \mathcal{C}_{\Ss_K,W}(T,\varphi) = [\mathcal{Z}_K(U_{\Int},\mu,T)/\Ss_K].
	\]
	Here we are using Definition~\ref{defn:qs_cycle_class}.
    
	\item (Linear invariance) For $g\in \GL_n(D)$, we have
	\[
     \mathcal{C}_{\Ss_K,W}({}^t\iota(g)Tg,L_g\varphi) = \mathcal{C}_{\Ss_K,W}(T,\varphi).
	\]

	\item (Translation invariance) For $N\in \mathrm{Herm}_n(D)$, we have
	\[
      \mathcal{C}_{\Ss_K,W}(T,s_N\varphi) = e^{-2\pi i\mathrm{Trd}_{D/\Rat}(\mathrm{Tr}(NT))}\mathcal{C}_{\Ss_K,W}(T,\varphi)
	\]
    Here, $\mathrm{Tr}:M_n(D)\to D$ is the usual trace map on $n\times n$-matrices over $D$.
\end{enumerate}

\end{theorem}	

\begin{proof}
Given the formula for $\mathcal{C}_{\Ss_K,W}(T,\varphi)$ in assertion (1) when $\varphi = \varphi_{U_{\Int},\mu}$ is the characteristic function of a $K$-invariant coset $\mu+U_{\Int}(\widehat{\Int})$ of a $K$-stable lattice, the proof that it defines an assignment on all of $\mathcal{S}(W(\Adele_f)^n)$ can be seen as follows: First, note that $\mathcal{S}(W(\Adele_f)^n)$ is generated as a $\Comp$-vector space by such characteristic functions subject to the relations
\[       
\varphi_{U_{\Int},\mu} = \sum_{i=1}^r\varphi_{U'_{\Int},\nu_i},
\]
whenever we have $\mu+U_{\Int}(\widehat{\Int}) = \bigsqcup_{i=1}^r\nu_i+U'_{\Int}(\widehat{\Int})$ with the right hand side a decomposition into (once again) $K$-invariant cosets of $K$-stable lattices. In this situation, $U_{\Int}/U'_{\Int}$ is generated by the images of the differences $\nu_i-\nu_j$, and in particular, inherits a \emph{trivial} $K$-action. Set:
\[
U''_{\Int} = U_{\Int} + \langle \mu\rangle = U'_{\Int} + \langle \nu_i:\;1\leq i\leq r\rangle.
\]
This is once again a $K$-stable lattice with $K$-acting trivially on $U''_{\Int}/U'_{\Int}$.

Assertion (1) comes down to the geometric statement that we have an equivalence 
\begin{align}\label{eqn:desired_disjoint_decomp}
\mathcal{Z}_K(U_{\Int},\mu,T)&\xrightarrow{\simeq}\bigsqcup_{i=1}^r\mathcal{Z}_K(U'_{\Int},\nu_i,T)
\end{align}
of quasi-smooth stacks over $\Ss_K$. To see this, note that, by Proposition~\ref{prop:W_and_W'}, we obtain a commuting diagram of $\mathrm{Ab}$-valued prestacks over $\Ss_K$:
\[
\begin{diagram}
\mathcal{Z}_{K}(U'_{\Int})&\rTo &\mathcal{Z}_{K}(U''_{\Int}) &\rTo &\underline{U''_{\Int}/U_{\Int}}\\
\dTo&&\dTo&&\dTo\\
\mathcal{Z}_{K}(U_{\Int})&\rTo &\mathcal{Z}_{K}(U''_{\Int}) &\rTo&\underline{U''_{\Int}/U'_{\Int}},
\end{diagram}
\]
where all the maps in the left square are finite \'etale, and where each row is a cofiber sequence. One now finds that both sides of~\eqref{eqn:desired_disjoint_decomp} are equivalent to the pre-images of $\mu$ under the bottom right map.

For statement (2), it suffices to know that there is a canonical equivalence of $\Ss_K$-stacks
\[
\mathcal{Z}_K(U_{\Int},\mu,T)\xrightarrow{\simeq}\mathcal{Z}_K(gU_{\Int},g\mu,{}^t\iota(g)Tg).
\]
But such an equivalence is provided by the functoriality of $(W^n,U_{\Int})\mapsto \mathcal{Z}_{K}(U_{\Int})$ and the $G$-equivariant isomorphism $(W^n,U_{\Int})\xrightarrow{\simeq}(W^n,gU_{\Int})$ given by the action of $g\in \GL_n(D)$ on $W^n = D^n\otimes_DW$.

Statement (3) is a consequence of the following claim: Suppose that $U_{\Int}\subset W^n$ is a $K$-stable lattice and that $\mu+U_{\Int}(\widehat{\Int})$ is a $K$-invariant coset on which $s_N$ takes constant value $\alpha$. If $\mathcal{Z}_K(U_{\Int},\mu,T)$ is non-empty, then we have
\[
\alpha = e^{-2\pi i\mathrm{Trd}_{D/\Rat}(\mathrm{Tr}(NT))}.
\]
To see this, first note that our hypotheses are equivalent to the following statements being true: 
\begin{itemize}
	\item $s_N$ is identically $1$ on $U_{\Int}(\widehat{\Int})$.
	\item For any lift $\tilde{\mu}\in W$ of $\mu$, we have $\alpha = e^{-\pi i\mathrm{Trd}_{D/\Rat}(\Phi^n_N(\tilde{\mu},\tilde{\mu}))}$.
	\item For all $x\in U_{\Int}(\widehat{\Int})$, we have
	\[
      \psi_f\left(\mathrm{Trd}_{D/\Rat}(\Phi^n_N(\tilde{\mu},x))\right) = 1
	\]
\end{itemize}

This can be reformulated in the language of Lemma~\ref{lem:modular_quadratic_congruence} as follows: Consider the quadratic form $q_N$ on $W^n$ given by
\[
q_N(x) = \mathrm{Trd}_{D/\Rat}\left(\frac{1}{2}\Phi^n_N(x,x)\right),
\]
so that $s_N = \psi_f\circ q_N$. Let $U'_{\Int}\subset W^n$ be the $K$-stable lattice generated by $U_{\Int}$ and any lift $\tilde{\mu}$ of $\mu$. Then $q_N$ is $\Int$-valued on $U_{\Int}$, and we have $[x,x']_{q_N}\in \Int$ for all $x\in U_{\Int}$ and $x'\in U'_{\Int}$, and also $\alpha = e^{-2\pi i q_N(\tilde{\mu})}$.

By Lemma~\ref{lem:modular_quadratic_congruence}, $q_N$ induces a canonical quadratic form $q'_N$ on $\mathcal{Z}_{K}(U'_{\Int})(s)$ for $s\in \Ss_K(C)$ with $C$ discrete, satisfying
\[
q'_N(w')\equiv q_N(\tilde{\mu})\pmod{\Int}.
\]
Here, $w'\in \mathcal{Z}_K(U_{\Int},\mu)(s)\subset \mathcal{Z}_{K}(U'_{\Int})(s)$. To finish, we now only need to observe that we have
\[
q'_N(w') = \mathrm{Trd}_{D/\Rat}(\mathrm{Tr}(NT))
\]
for any $w'\in \mathcal{Z}_K(U_{\Int},\mu,T)(s)$. Indeed, $q'_N$ is nothing but the quadratic form $q'_N(w) = \frac{1}{2}\mathrm{Trd}_{D/\Rat}(\Phi'_n(w,Nw))$, where $\Phi'_n$ is the $\iota$-Hermitian form on $\mathcal{Z}_{K}(U'_{\Int})(s)_{\Rat}\simeq \mathcal{Z}_K(W_{\Int})(s)^n_{\Rat}$ induced from $\Phi'$: Here the action of $N$ on $\mathcal{Z}_K(W_{\Int})(s)^n_{\Rat}$ is via the identification $D^n\otimes_D\mathcal{Z}_K(W_{\Int})(s)_{\Rat}\simeq \mathcal{Z}_K(W_{\Int})(s)^n_{\Rat}$.
\end{proof}	

\begin{remark}
As observed in Remark~\ref{rem:adeel_construction}, the results of~\cite{khan:virtual} give virtual fundamental classes for quasi-smooth morphisms in the motivic cohomology of any stable motivic spectrum over $\Int$. The argument used in Theorem~\ref{thm:main_thm_adelic} would give an assignment valued in such motivic cohomology groups. It would be interesting to know if there are any arithmetic consequences to be drawn from a different choice of such spectrum. For instance, using the results of~\cite[\S 3]{Holmstrom2015-to}, we find that we can associate canonical cycle classes  in Deligne-Beilinson cohomology: these might be relevant to the construction of arithmetic cycle classes.
\end{remark}

\begin{remark}
\label{rem:generic_fiber_cycles}
If one is happy to work only over the generic fiber, then a quick analysis of the construction will indicate that it works compatibly for arbitrary $K\subset G(\Adele_f)$. So it is permissible to take the inductive limit over all $K$, and therefore produce a $G(\Adele_f)$-equivariant map
\[
\mathcal{S}(W(\Adele_f)^n) \xrightarrow{\varphi\mapsto C_W(T,\varphi)}\mathrm{CH}^{nd_+}(\Sh),
\]
where $\Sh = \varprojlim_K \Sh_K$.
\end{remark}

The following proposition proves the product formula and completes the proof of Theorems~\ref{introthm:main} and~\ref{introthm:derived_cycles} from the introduction.
\begin{proposition}[Product formula]
\label{prop:product_formula_cycles}
For $n_1,n_2\ge 1$, $\varphi_1\in \mathcal{S}(W(\Adele_f)^{n_1})^K$, $\varphi_2\in \mathcal{S}(W(\Adele_f)^{n_2})^K$ and $T_1\in \mathrm{Herm}^+_{n_1}(D)$, $T_2\in \mathrm{Herm}^+_{n_2}(D)$, we have
	\[
     \mathcal{C}_{\Ss_K,W}(T_1,\varphi_1)\cdot \mathcal{C}_{\Ss_K,W}(T_2,\varphi_2) = \sum_{T = \begin{pmatrix}T_1&*\\ *&T_2\end{pmatrix}}\mathcal{C}_{\Ss_K,W}(T,\varphi_1\otimes\varphi_2)
	\]
	where the sum on the right is indexed by $\iota$-Hermitian matrices in $\mathrm{Herm}_{n_1+n_2}^+(D)$ with $T_1$ and $T_2$ appearing as block diagonals.
\end{proposition}
\begin{proof}
Since both sides are bilinear in $\varphi_1,\varphi_2$, using Remark~\ref{rem:product_classes}, it is enough to know that there is a canonical equivalence of $\Ss_K$-stacks
\[
\mathcal{Z}_K(U_{1,\Int},\mu_1,T_1)\times_{\Ss_K}\mathcal{Z}_K(U_{2,\Int},\mu_2,T_2)\xrightarrow{\simeq}\bigsqcup_{T = \begin{pmatrix}T_1&*\\ *&T_2\end{pmatrix}}\mathcal{Z}_K(U_{1,\Int}\oplus U_{2,\Int},\mu_1+\mu_2,T).
\]
This is an immediate consequence of the definitions and Proposition~\ref{prop:product_formula}.
\end{proof}

Using Proposition~\ref{prop:product_formula} and applying Remark~\ref{rem:product_classes}, we also obtain:
\begin{proposition}[Pullback formula]\label{prop:pullback_formula}
Suppose now that we have another representation $W'$ equipped with a $G$-equivariant action of $D$ satisfying the hypotheses in~\eqref{subsec:W_D_hypotheses}; then so does $\tilde{W} = W\oplus W'$. Write $d_+,d'_+$ and $\tilde{d}_+ = d'_++d_+$ for the minimal codimensions associated with $W,W'$ and $\tilde{W}$, respectively. Then, for any $\tilde{T}\in \mathrm{Herm}_n^+(D)$, we have a commuting diagram
\[
\begin{diagram}
\mathcal{S}(W(\Adele_f)^n)^K\otimes_{\Comp}\mathcal{S}(W'(\Adele_f)^n)^K&\rTo^{\underset{T+T' = \tilde{T}}{\sum}\mathcal{C}_{\Ss_K,W}(T,\_\_)\otimes \mathcal{C}_{\Ss_K,W'}(T',\_\_)}&\mathrm{CH}^{nd_+}(\Ss_K)_{\Comp}\otimes_{\Comp}\mathrm{CH}^{nd'_+}(\Ss_K)_{\Comp}\\
\dTo^{\simeq}&&\dTo_{\cdot}\\
\mathcal{S}(\tilde{W}(\Adele_f)^n)^K&\rTo_{\mathcal{C}_{\Ss_K,\tilde{W}}(\tilde{T},\_\_)}&\mathrm{CH}^{n\tilde{d}_+}(\Ss_K)_{\Comp},
\end{diagram}
\]
where the right vertical arrow is the intersection product.
\end{proposition}

\subsection{}
Write 
\[
\Phi:W^n\to \mathrm{Herm}_n(D)
\]
for the map taking $\underline{x} = (x_1,\ldots,x_n)$ to the matrix $\left(\frac{1}{2}\Phi(x_i,x_j)\right)_{1\leq i,j\leq n}$, set $q = \mathrm{Trd}_{D/\Rat}\circ \mathrm{Tr}\circ \Phi$. The next result is immediate from the definitions and Proposition~\ref{prop:trivial_rep}.

\begin{proposition}
\label{prop:trivial_cycles}
Suppose that $G$ acts trivially on $W$; then $W$ is a positive definite $\iota$-Hermitian space, and, for all $T\in \mathrm{Herm}_n^+(D)$, $\varphi\in \mathcal{S}(W(\Adele_f)^n)^K$, we have
\[
\mathcal{C}_{\Ss_K,W}(T,\varphi) = r_W(T,\varphi)\cdot \mathbf{1}\in \mathrm{CH}^0(\Ss_K)_{\Comp},
\]
where 
\[
r_W(T,\varphi) = \sum_{\underset{\Phi(\underline{x}) = T}{\underline{x}\in W(\Rat)^n}}\varphi(\underline{x})
\]
and $\mathbf{1}$ is the fundamental class of $\Ss_K$.
\end{proposition}

Combining the two previous propositions gives us:
\begin{corollary}
\label{cor:pullback_trivial}
Suppose that in the context of Proposition~\ref{prop:pullback_formula}, $G$ acts trivially on $W'$; then, for any $\tilde{T}\in \mathrm{Herm}_n^+(D)$, $\varphi\in \mathcal{S}(W(\Adele_f)^n)^K$ and $\varphi'\in \mathcal{S}(W'(\Adele_f)^n)$, we have
\[
\mathcal{C}_{\Ss_K,\tilde{W}}(\varphi\otimes\varphi',\tilde{T}) = \sum_{T+T'=\tilde{T}}r_{W'}(T',\varphi')\mathcal{C}_{\Ss_K,W}(\varphi,T)\in \mathrm{CH}^{nd_+}(\Ss_K)_{\Comp}.
\]
\end{corollary}

\begin{proposition}
[Constant term]\label{prop:constant_term}
If $T = 0\in \mathrm{Herm}^+_n(D)$ is the zero matrix, then $\mathcal{C}_{\Ss_K,W}(0,\varphi)$ is determined by the following property: For any $K$-stable lattice $U_{\Int}\subset W^n$ and any $K$-invariant element $\mu\in W^n/U_{\Int}$, let $\varphi_{U_{\Int},\mu}\in \mathcal{S}(W(\Adele_f)^n)^K$ be the characteristic function of the corresponding coset of $U_{\Int}(\widehat{\Int})$. Then we have
\[
\mathcal{C}(0,\varphi_{U_{\Int},\mu}) = \begin{cases}
c_{top}(\gr^{-1}_{\mathrm{Hdg}}\mathbf{dR}(U_{\Int}))&\text{if $\mu = 0$};\\
0&\text{otherwise.}
\end{cases}
\]
\end{proposition}
\begin{proof}
One checks that the quasi-smooth stack $\mathcal{Z}_K(U_{\Int},\mu,0)\to \Ss_K$ is empty if $\mu \neq 0$ and that its underlying classical stack maps isomorphically onto $\Ss_K$ when $\mu = 0$: This amounts to saying that the only element in $\mathcal{Z}_K(W_{\Int})(s)_{\Rat}^n$ with zero Gram matrix is the zero vector. The rest now follows from Lemma~\ref{lem:quasi-smooth_with_section} and Remark~\ref{rem:top_chern_class}, and the fact that the cotangent complex for $\mathcal{Z}_K(U_{\Int})$ over $\Ss_K$ is the pullback of $\bm{co}(W_{\Int})[1]$.
\end{proof}

\begin{remark}
\label{rem:product_with_zero}
Suppose that we have $g\in \GL_n(D)$ such that $({}^t\iota(g))^{-1}Tg^{-1}$ is a block diagonal matrix of the form $ \begin{pmatrix}
T_0&0\\
0&0
\end{pmatrix} $ with $T_0\in \mathrm{Herm}^+_{n_0}(D)$. Then combining Propositions~\ref{prop:product_formula_cycles} and~\ref{prop:constant_term} with (2) of Theorem~\ref{thm:main_thm_adelic} shows that, for $\varphi_0\in \mathcal{S}(W(\Adele_f)^{n_0})^K$ and a $K$-stable lattice $U_{\Int}\subset W(\Adele_f)^{n-n_0}$, we have
\[
\mathcal{C}_{\Ss_K,W}(T_0,\varphi_0)\cdot c_{top}(\gr^{-1}_{\mathrm{Hdg}}\mathbf{dR}(U_{\Int})) = \mathcal{C}_{\Ss_K,W}(T,L_g(\varphi_0\otimes\varphi_{U_{\Int},0})). 
\]
\end{remark}

\begin{remark}
\label{rem:cycles_over_generic_fiber}
Suppose that Assumption~\ref{assumption} holds; then using the complex uniformization from Lemma~\ref{lem:W_special_complex}, one can show that the classical truncation of $\mathcal{Z}_K(U_{\Int},\mu,T)_{\Rat}$ is a smooth stack over $\Comp$ of dimension $\dim \Sh_K - d_+\mathrm{rank}(T)$. 

Combining this with the previous remark now shows that we have
\[
\mathcal{C}_{\Ss_K,W}(T,\varphi_{U_{\Int},\mu})\vert_{\Sh_K} = [\mathcal{Z}_K(U_{\Int},\mu,T)_{\Rat}]\cdot c_{top}(\gr^{-1}_{\mathrm{Hdg}}\mathbf{dR}(W))^{n-\mathrm{rank}(T)}\in \mathrm{CH}^{nd_+}(\Sh_K)_{\Rat},
\]
where $[\mathcal{Z}_K(U_{\Int},\mu,T)_{\Rat}]$ is the na\"ive cycle class defined as in~\cite[(A.1.4)]{HMP:mod_codim}.

In particular, we find that the construction here agrees with that of Kudla as found for instance in~\cite{kudla:special_cycles} or~\cite{Kudla2019-dq} for GSpin Shimura varieties. This justifies the last sentence of Theorem~\ref{introthm:main}.
\end{remark}

\section{The modularity conjecture}\label{sec:modularity}
Here I formulate the main modularity conjecture for the cycle classes defined above, and provide some rudimentary evidence for it. The main contribution here is the unconditional definition of the classes themselves in the general context to which Kudla's conjectures should extend.

\subsection{}\label{subsec:the_group_for_modularity}
Let $\mathrm{U}_D(n,n)$ be the reductive group over $\Rat$ obtained as the unitary group associated with the $\iota$-skew Hermitian space $D^{2n} = D^n\oplus D^n$ with the form given by 
\[
\Psi((\underline{x}_1,\underline{y}_1),(\underline{x}_2,\underline{y}_2)) = \iota({}^t\underline{x}_1)\cdot\underline{y}_2 - \iota({}^t\underline{x}_2)\cdot \underline{y}_1.
\]

Set $\psi = \mathrm{Trd}_{D/\Rat}\circ \Psi$: this is a symplectic form on the $\Rat$-vector space underlying $D^{2n}$.

We have a \emph{connected} Shimura datum $(\mathrm{U}_D(n,n),\mathcal{H}_D(n,n))$, where $\mathcal{H}_D(n,n)$ is the space of maximal $\Comp\otimes_{\Rat}D$-stable, isotropic subspaces $E\subset \Comp\otimes_{\Rat}D^{2n}$ such that the Hermitian form $(\underline{z}_1,\underline{z}_2)\mapsto i\psi(\overline{\underline{z}}_1,\underline{z}_2)$ on the underlying $\Comp$-vector space of $E$ is positive definite. Here, the complex conjugation is with respect to the natural real structure $\Real\otimes_{\Rat}D^{2n}$.

Since we have a decomposition
\begin{align}
\label{eqn:polarization_D2n}
D^{2n} = (D^n\oplus \{0\})\oplus (\{0\}\oplus D^n)
\end{align}
into complementary rational maximal isotropic subspace, both of which are complementary to any $\tau\in \mathcal{H}_D(n,n)$, we can identify $\mathcal{H}_D(n,n)$ with an open subspace of $\mathrm{Herm}_n(D)$: It consists of those matrices $\tau$ such that the elements $(\tau(\underline{y}),\underline{y})$ span an isotropic subspace in $\mathcal{H}_D(n,n)$ as $\underline{y}$ ranges over $\Comp\otimes_{\Rat}D^n$. We will use this to identify $\mathcal{H}_D(n,n)$ with a subspace of $\mathrm{Herm}_n(D)$; in particular, for any $T\in \mathrm{Herm}_n^+(D)$, and $\tau\in \mathcal{H}_D(n,n)$, one can make sense of the parameter
\[
\bm{q}^T = e^{2\pi i\mathrm{Trd}_{D/\Rat}(\mathrm{Tr}(\tau T))}
\]
as a holomorphic function on $\mathcal{H}_D(n,n)$. 

In the notation of Theorem~\ref{thm:main_thm_adelic}, we now set
\begin{align}
\label{defn:geometric_theta_function}
\Theta^{\mathrm{geom}}_{\Ss_K,W}(\varphi,\tau) = \sum_{T\in \mathrm{Herm}_n^+(D)}\mathcal{C}_{\Ss_K,W}(T,\varphi)\bm{q}^T.
\end{align}
Here, we are viewing the right hand side as a \emph{formal} sum of holomorphic functions on $\mathcal{H}_D(n,n)$ valued in $\mathrm{CH}^{nd_+}(\Ss_K)_{\Comp}$. 

\subsection{}
\label{subsec:weil_representation}
Before I explain the precise way in which~\eqref{defn:geometric_theta_function} is supposed to be automorphic, I need to recall the Weil representation on the vector space $\mathcal{S}(W(\Adele)^n)$ obtained by viewing $W^n$ as the Lagrangian subspace
\[
W^n = (\{0\}\oplus D^n)\otimes_DW\subset D^{2n}\otimes_DW \overset{\mathrm{defn}}{=} \mathbf{V}(W,n),
\]
where we equip $\mathbf{V}(W,n)$---viewed as a $\Rat$-vector space---with the symplectic form 
\[
\langle v_1\otimes w_2,v_2\otimes w_2\rangle = \mathrm{Trd}_{D/\Rat}\left(\iota(\Psi(v_1,v_2))\Phi(w_1,w_2)\right).
\]
We have mutually commuting actions of $\mathrm{U}_D(n,n)$ and $G$ on $\mathbf{V}(W,n)$ giving an embedding
\[
\mathrm{U}_D(n,n)\times G \hookrightarrow \mathrm{Sp}(\mathbf{V}(W,n)).
\]

Set $\mathbf{V} = \mathbf{V}(W,n)$. Let $\psi:\Adele\to \Comp^\times$ be the lift of $\psi_f$ restricting to $\psi_\infty:x\mapsto e^{2\pi i x}$ on $\Real$. Let $\mathcal{S}(W(\Real)^n)$ be the usual space of Schwartz functions on the real vector space $W(\Real)^n$. 

For every place $v$ of $\Rat$, there is a canonical central $\mu_2$-extension $\mathrm{Mp}(\mathbf{V})(\Rat_v)$ of $\mathrm{Sp}(\mathbf{V})(\Rat_v)$ equipped with the Weil representation $\omega^\psi_v$ on $\mathcal{S}(W(\Rat_v)^n)$. These can be amalgamated into a central $\mu_2$-extension $\mathrm{Mp}(\mathbf{V})(\Adele)$ of $\mathrm{Sp}(\mathbf{V}(\Adele))$ equipped with a representation $\omega^\psi$ on $\mathcal{S}(W(\Adele)^n)$. This is \emph{automorphic} in the sense that there is a canonical lift $\mathrm{Sp}(\mathbf{V})(\Rat)$ to a subgroup of $\mathrm{Mp}(\mathbf{V})(\Adele)$, and we have $\omega^\psi(\gamma\tilde{g}) = \omega^\psi(\tilde{g})$ for all $\gamma\in \mathrm{Sp}(\mathbf{V})(\Rat)$ and $\tilde{g}\in \mathrm{Mp}(\mathbf{V})(\Adele)$.

For explicit formulas describing all this, see~\cite{Ranga_Rao1993-rc}.

There now exists a central $\mu_2$-extension $\widetilde{\mathrm{U}}_D(n,n)(\Adele)$ of $\mathrm{U}_D(n,n)(\Adele)$, independent of $W$, and a lift
\[
\widetilde{\mathrm{U}}_D(n,n)(\Adele)\times G(\Adele)\to \mathrm{Mp}(\mathbf{V})(\Adele)
\]
such that $\omega^\psi$ restricts to the natural representation $(\omega^\psi((1,g))\varphi)(\underline{w}) = \varphi((\mathrm{id}\otimes g^{-1})\underline{w})$ on $G(\Adele)$. 

If $W^n$ and $\psi$ are clear from context, we will simply write $\omega$ for the corresponding representation of $\widetilde{\mathrm{U}}_D(n,n)(\Adele)$. This is once again an automorphic representation in these sense that there exists a lift of $\mathrm{U}_D(n,n)(\Rat)$ to a subgroup of $\widetilde{\mathrm{U}}_D(n,n)(\Adele)$, compatible with the lift of $\mathrm{Sp}(\mathbf{V})(\Rat)$ to $\mathrm{Mp}(\mathbf{V}(\Adele))$, so that $\omega$ is invariant under the left action of $\mathrm{U}_D(n,n)(\Rat)$.


\subsection{}\label{subsec:half-integer_wt}
When $v = \infty$, we can describe $\widetilde{\mathrm{U}}_D(n,n)(\Real)$ in terms of the decomposition~\eqref{eqn:polarization_D2n} as the group of pairs $(g,j(g;\tau))$ in $\mathrm{U}_D(n,n)(\Real)\times \Comp^\times$, where 
\[
g = \begin{pmatrix}
A_1&A_2\\
A_3&A_4
\end{pmatrix}\in \mathrm{U}_D(n,n)(\Real)\;;\; j(g;\tau)^2 = \mathrm{Nrd}_{M_n(D)/\Rat}(A_3\tau+A_4).
\]
Here, $\mathrm{Nrd}_{M_n(D)/\Rat}$ is the reduced norm on $M_n(D)$. 

For any $k\in \Real$, we now have a right action of $\widetilde{\mathrm{U}}_D(n,n)(\Real)$ on holomorphic functions on $\mathcal{H}_D(n,n)$ given for $\tilde{g} = (g,j(g;\tau))$ by:
\[
(f\vert_k[\tilde{g}])(\tau) = j(\tau;g)^{-2k}f(g\tau).
\]
More generally, for any discrete set $S$, we can equip holomorphic functions on $S\times \mathcal{H}_D(n,n)$ with an action given by the same formula, and for which we will use the same notation.

For $k\in \frac{1}{2}\Int_{\ge 0}$, we now say that a holomorphic function
\[
f:\mathcal{S}(W(\Adele_f)^n)^K\times \mathcal{H}_D(n,n) \to \Comp
\]
is a \defnword{parallel weight $k$ automorphic form on $\mathrm{U}_D(n,n)$ with respect to the Weil representation on $\mathcal{S}(W(\Adele)^n)$} if, for all $\gamma\in \mathrm{U}_D(n,n)(\Rat)$, $\varphi\in \mathcal{S}(W(\Adele_f)^n)^K$ and $\tau\in \mathcal{H}_D(n,n)$, we have
\[
(f\vert_k[\tilde{\gamma}_\infty])(\omega(\tilde{\gamma}_f)\varphi,\tau) = f(\varphi,\tau).
\]
Here, $\tilde{\gamma}\in \tilde{\mathrm{U}}_D(n,n)(\Adele)$ is the lift of $\gamma$ determined by the fixed splitting of the double cover over $\mathrm{U}_D(n,n)(\Rat)$ mentioned at the end of~\eqref{subsec:weil_representation}.

Actually, if the adjoint group of $\mathrm{U}_D(n,n)$ has $\Rat$-rank $1$---which happens precisely when $D $ is $\Rat$ or an imaginary quadratic extension of $\Rat$ with non-trivial $\iota$, and $n = 1$---$\mathcal{H}_D(n,n)$ is (biholomorphic to) the upper half plane, and in this case, we will also require that $(f\vert_k[\tilde{\gamma}_\infty])(\varphi,\tau)$ is holomorphic at $\infty$ for all $\gamma\in \mathrm{U}_D(n,n)(\Rat)$ and all $\varphi\in \mathcal{S}(W(\Adele_f)^n)^K$.

I can finally state the conjecture:
\begin{conjecture}
\label{conj:main}
Suppose that $W$ has rank $m$ over $D$. For any linear functional $\ell:\mathrm{CH}^{nd_+}(\Ss_K)_{\Comp}\to \Comp$ and any $\varphi\in \mathcal{S}(W(\Adele_f)^n)^K$, the series 
\[
\ell\left(\Theta^{\mathrm{geom}}_{\Ss_K,W}(\varphi,\tau)\right) = \sum_{T\in \mathrm{Herm}_n^+(D)}\ell(\mathcal{C}_{\Ss_K,W}(T,\varphi))\bm{q}^T 
\]
converges absolutely to a parallel weight $\frac{m}{2}$ automorphic form on $\mathrm{U}_D(n,n)$ with respect to the Weil representation on $\mathcal{S}(W(\Adele)^n)$.
\end{conjecture}

If the above conjecture is true, we will say that it \defnword{holds for (the representation) $W$}. We will also engage in some concision of language and say simply that $\Theta^{\mathrm{geom}}_{\Ss_K,W}$ \defnword{converges absolutely to a parallel weight $\frac{m}{2}$ automorphic form on $\mathrm{U}_D(n,n)$}.

It is quite easy now to deduce automorphy with respect to the Siegel parabolic in $\mathrm{U}_D(n,n)$.
\begin{proposition}
\label{prop:easy_modularity}
Let $\mathrm{P}_D(n,n)\subset\mathrm{U}_D(n,n)$ be the maximal parabolic subgroup stabilizing $D^n\oplus \{0\}\subset D^{2n}$. For any $\gamma\in \mathrm{P}_D(n,n)(\Rat)$, we have an identity of formal series
\[
\Theta^{\mathrm{geom}}_{\Ss_K,W}\vert_{\frac{m}{2}}[\tilde{\gamma}_\infty](\omega(\tilde{\gamma}_f)\varphi,\tau) = \Theta^{\mathrm{geom}}_{\Ss_K,W}(\varphi,\tau).
\]
\end{proposition}	
\begin{proof}
Using the usual formulas determining the Weil representation, this is simply a translation of properties (2) and (3) of Theorem~\ref{thm:main_thm_adelic}.
\end{proof}

\begin{remark}[Trivial case]
\label{rem:trivial_theta}
When $G$ acts trivially on $W$, Proposition~\ref{prop:trivial_cycles} shows that we have
\[
\Theta^{\mathrm{geom}}_{\Ss_K,W}(\varphi,\tau) = \Theta_W(\varphi,\tau)\cdot \mathbf{1},
\]
where $\Theta_W(\varphi,\tau) = \sum_{T\in \mathrm{Herm}_n^+(D)}r_W(T,\varphi)\bm{q}^T$ is a classical theta function associated with the positive definite $\iota$-Hermitian space $W$. The modularity of this series is now a consequence of Poisson summation.
\end{remark}

We can actually generalize this observation, leading to what is sometimes called the `embedding trick', used in~\cite{Kudla2019-dq} and~\cite{HMP:mod_codim}.
\begin{proposition}[Embedding trick]
\label{prop:embedding_trick}
Let us return to the hypotheses of Corollary~\ref{cor:pullback_trivial}, and suppose that Conjecture~\ref{conj:main} holds for $\tilde{W}$; then it also holds for $W$.
\end{proposition}
\begin{proof}
For simplicity, for $\varphi\in \mathcal{S}(W(\Adele_f)^n))^K$ and $\varphi'\in \mathcal{S}(W'(\Adele_f)^n)$, set
\[
\widetilde{\Theta}(\varphi\otimes\varphi') = \Theta^{\mathrm{geom}}_{\Ss_K,\tilde{W}}(\varphi\otimes\varphi',\tau)\;;\;\Theta(\varphi) = \Theta^{\mathrm{geom}}_{\Ss_K,W}(\varphi,\tau)\;;\; \Theta'(\varphi') = \Theta_{W'}(\varphi',\tau).
\]

Corollary~\ref{cor:pullback_trivial} shows that we have an equality of formal series
\[
\widetilde{\Theta}(\varphi\otimes \varphi') = \Theta(\varphi)\cdot \Theta'(\varphi')
\]
with coefficients in $\mathrm{CH}^{nd_+}(\Ss_K)_{\Comp}$. 

By hypothesis, for any linear functional $\ell:\mathrm{CH}^{nd_+}(\Ss_K)_{\Comp}$, $\ell(\tilde{\Theta}(\varphi\otimes\varphi')) = \ell(\Theta(\varphi))\cdot \Theta'(\varphi')$ converges absolutely to a parallel weight $\tilde{n}/2$ automorphic form on $\mathrm{U}_D(n,n)$ with respect to the Weil representation on 
\[
\mathcal{S}(\tilde{W}(\Adele)^n) = \mathcal{S}(W(\Adele)^n)\otimes_{\Comp}\mathcal{S}(W'(\Adele)^n). 
\]
By Remark~\ref{rem:trivial_theta}, $\Theta'(\varphi')$ converges absolutely to a parallel weight $\frac{n'}{2}$ automorphic form, but now with respect to $\mathcal{S}(W'(\Adele)^n)$. 

To conclude, it is now sufficient to show that, for every $\varphi\in \mathcal{S}(W(\Adele_f)^n)^K$, $\ell(\Theta(\varphi))$ converges absolutely to a holomorphic function on $\mathcal{H}_D(n,n)$. 

The idea is explained in~\cite[\S 7,\S 8]{Kudla2019-dq} in the case where $D$ is a totally real extension of $\Rat$. The same argument applies in general. There are two key points:

\begin{enumerate}[label=(\roman*)]
	\item First, the ring of formal $\bm{q}$-series in which these generating functions live is an integral domain. This comes down (after invoking linear invariance) to the fact that the Baily-Borel compactifications of the (connected) Shimura varieties associated with $\mathrm{U}_D(n,n)$ are \emph{normal}.

	\item Second, for any $\tau_0\in \mathcal{H}_D(n,n)$, we can find $\varphi'$ such that $\Theta'(\varphi',\tau_0)\neq 0$. This follows from the argument of~\cite[Lemma 7.2]{Kudla2019-dq} (taken from~\cite{MR2551992}), using the non-vanishing of the theta lift to $\mathrm{U}_D(n',n')$ from the unitary group attached to the positive definite $\iota$-Hermitian space $W'$. The main thing here is to know that, for a given element $g_\infty\in \mathrm{U}_D(n',n')(\Real)$, the subset
\[
\mathrm{U}_D(n',n')(\Rat)g_\infty\mathrm{U}_D(n',n')(\Adele_f)\subset \mathrm{U}_D(n',n')(\Adele)
\]
is dense: this is a consequence of real approximation~\cite[Theorem 7.7]{Platonov1994-ib}.
\end{enumerate}
Given these two points, the rest of the argument is now quite formal; see the proof of~\cite[Prop. 8.2]{Kudla2019-dq}.
\end{proof}

\begin{remark}
\label{rem:pullback}
The reason this is called the `embedding trick' is that in practice the representation $\tilde{W}$ is obtained via restriction along a map of Shimura data $(G,X)\to (\tilde{G},\tilde{X})$, and the series $\Theta^{\mathrm{geom}}_{\Ss_K,\tilde{W}}$ is obtained via pullback of the corresponding one valued in the Chow group of the (integral model of a) Shimura variety associated with $\tilde{G}$. In favorable situations, one can exploit the fact that the coefficients of the series on the larger Shimura variety are cycles in relatively small codimension to obtain stronger modularity results than otherwise: this is done in different ways in~\cite{Kudla2019-dq} and~\cite{HMP:mod_codim}. The embedding trick above then allows us to conclude modularity for the series $\Theta^{\mathrm{geom}}_{\Ss_K,W}$ as well.

I wonder if it's possible to get some stronger `bootstrapping' results by using the full power of Proposition~\ref{prop:pullback_formula}.
\end{remark}	

\begin{remark}
\label{rem:tautological_series}
In the setting of the example in point~\eqref{rem:canonical_cycle} of Remark~\ref{introrem:cycles}, Conjecture~\ref{conj:main} predicts that we should obtain a whole range of classical half-integer weight modular forms indexed by Shimura data of Hodge type. It would be very interesting to work this out explicitly.
\end{remark}

\appendix

\section{Fibrations}\label{sec:fibrations}

\subsection{}\label{subsec:cartesian}
Suppose that we have a functor $\pi:\mathcal{D}\to \mathcal{C}$ of $\infty$-categories. Following~\cite{Mazel-Gee2015-gl}, we will say that an arrow $\tilde{f}:d_1\to d_2$ is a \defnword{$\pi$-Cartesian} if, with $f = \pi(\tilde{f})$, if the natural diagram
\[
\begin{diagram}
\mathcal{D}_{/d_1}&\rTo^{\circ \tilde{f}}&\mathcal{D}_{/d_2}\\
\dTo&&\dTo\\
\mathcal{C}_{/c_1}&\rTo_{\circ f}&\mathcal{C}_{/c_2}
\end{diagram}
\]
is a pullback square in $\mathrm{Cat}_\infty$. In this situation, we will say that $\tilde{f}$ is a \defnword{Cartesian lift} of $f$, or that it is \defnword{$\pi$-Cartesian}

Given the same functor $\pi$, applying these definitions to the induced functor of the opposite categories gives us the notion of a \defnword{coCartesian lift} of $f$, along with the analogous notion of a $\pi$-\defnword{coCartesian} map in $\mathcal{D}$.

\begin{remark}
If $\tilde{f}:d_1\to d_2$ is $\pi$-Cartesian, then in particular, for any object $d$ in $\mathcal{D}$ with $c = \pi(d)$, we have a pullback diagram of spaces
\[
\begin{diagram}
\Map_{\mathcal{D}}(d,d_1)&\rTo^{\circ \tilde{f}}&\Map_{\mathcal{D}}(d,d_2)\\
\dTo&&\dTo\\
\Map_{\mathcal{C}}(c,c_1)&\rTo_{\circ f}&\Map_{\mathcal{C}}(c,c_2)
\end{diagram}
\]
From this, one can deduce that for $d_2$ and $f$ fixed, the $\pi$-Cartesian lift $\tilde{f}$ of $f = \pi(\tilde{f})$ (and in particular $d_1$) is unique up to a contractible space of choices.

The same remark applies to $\pi$-coCartesian lifts.
\end{remark}

\subsection{}
We say that $\pi$ is a \defnword{Cartesian fibration} (resp. \defnword{coCartesian fibration}) if, for every object $d$ of $\mathcal{D}$ and every arrow $f:c\to \pi(d)$ (resp. $f:\pi(d)\to c$), there exists a Cartesian lift (resp. a coCartesian lift) of $f$. We say that a Cartesian (resp. coCartesian) fibration $\pi$ is a \defnword{right fibration} (resp. \defnword{left fibration}) if, for every object $c$ in $\mathcal{C}$, the pre-image $\pi^{-1}(c)$ is an $\infty$-groupoid (that is, it belongs to the essential image of the functor $\mathrm{Spc}\to \mathrm{Cat}_{\infty}$).

Cartesian fibrations (resp. right fibrations) over $\mathcal{C}$ form an $\infty$-category $\mathrm{CartFib}_{\mathcal{C}}$ (resp. $\mathrm{RFib}_{\mathcal{C}})$, where the arrows are functors of $\infty$-categories over $\mathcal{C}$ that preserve Cartesian lifts (such functors are called \defnword{Cartesian}). Similarly, coCartesian fibrations and left fibrations form $\infty$-categories $\mathrm{coCartFib}_{\mathcal{C}}$ and $\mathrm{LFib}_{\mathcal{C}}$, respectively.

Note that the identity functor $\mathcal{C}\to \mathcal{C}$ is a right and left fibration; a Cartesian (resp. coCartesian) functor between this and any Cartesian (resp. coCartesian) fibration $\pi:\mathcal{D}\to\mathcal{C}$ is called a \defnword{Cartesian section} (resp. \defnword{coCartesian section})

Via the Grothendieck construction (or `straightening-unstraightening'; see~\cite[\S 3.2]{Lurie2009-oc}), when $\pi$ is a Cartesian (resp. right) fibration, the assignment 
\[
c\mapsto \pi^{-1}(c)
\]
underlies a functor $F_\pi:\mathcal{C}^{\op}\to \mathrm{Cat}_{\infty}$ (resp. a presheaf $F_\pi\in \mathrm{PSh}(\mathcal{C})$). Heuristically, given an arrow $f:c'\to c$, the associated functor $F_\pi(f):\pi^{-1}(c)\to \pi^{-1}(c')$ is given by carrying $d$ in $\pi^{-1}(c)$ to the source of a Cartesian lift of $f$ with target $d$. We will usually refer to this functor as \defnword{base-change along $f$} and denote it by $f^*$.

In fact, there are equivalences of $\infty$-categories
\[
\mathrm{CartFib}_{\mathcal{C}}\xrightarrow{\pi\mapsto F_\pi} \mathrm{Fun}(\mathcal{C}^{\op},\mathrm{Cat}_{\infty})\;;\; \mathrm{RFib}_{\mathcal{C}}\xrightarrow{\pi\mapsto F_\pi}\mathrm{PSh}(\mathcal{C}).
\]
With this notation, will say that the Cartesian (or left) fibration $\pi$ is \defnword{classified by the functor} $F_\pi$.

\begin{example}
The basic example of a right fibration is given for $c$ in $\mathcal{C}$ by the natural functor $\mathcal{C}_{/c}\to \mathcal{C}$ that carries an arrow $c'\to c$ to its source. It is classified by the representable pre-sheaf $h_c\in \mathrm{PSh}(\mathcal{C})$.
\end{example}

\begin{example}
When $\mathcal{C}$ admits finite limits, the functor
\[
t:\mathrm{Fun}(\Delta[1],\mathcal{C})\to \mathcal{C}
\]
carrying an arrow to its target is a Cartesian fibration, classified by the functor 
\[
\mathcal{C}\xrightarrow{c\mapsto \mathcal{C}_{/c}}\mathrm{Cat}_\infty,
\]
where the base-change functor is given by fiber product.
\end{example}

\begin{example}\label{ex:presheaf_cartesian}
Applying the previous example to $\mathrm{PSh}(\mathcal{C})$ gives us a Cartesian fibration
\[
\mathrm{P}(\mathcal{C}) = \mathrm{Fun}(\Delta[1],\mathrm{PSh}(\mathcal{C})\times_{t,\mathrm{PSh}(\mathcal{C}),h}\mathcal{C}\to \mathcal{C}
\]
classified by the functor
\[
\mathcal{C}\xrightarrow{c\mapsto \mathrm{PSh}(\mathcal{C})_{/h_c}}\mathrm{Cat}_\infty.
\]
Given a presheaf $F\in \mathrm{PSh}(\mathcal{C})$, there is up to equivalence a unique Cartesian section $s_F:\mathcal{C}\to \mathrm{P}(\mathcal{C})$ such that $s_F(c) = F\times h_c \to h_c$.
\end{example}

\begin{example}
The basic example of a coCartesian fibration is the forgetful functor $\Mod{}\to \mathrm{CRing}$: here the coCartesian lifts are precisely the ones of the form $M\to S\otimes_RM$ for $R\to S$ an arrow in $\mathrm{CRing}$ and $M\in \Mod{R}$. We will see in Appendix~\ref{sec:filtrations} below another example $\mathrm{FilMod}\to \mathrm{FilCRing}$ involving filtered modules over filtered animated commutative rings.
\end{example}

\subsection{}\label{subsec:cart_sec_functor}
Suppose that we have a Cartesian fibration $\pi:\mathcal{D}\to \mathcal{C}$ equipped with two Cartesian sections $s_1,s_2:\mathcal{C}\to \mathcal{D}$. Then we obtain a presheaf
\[
\shvMap(s_1,s_2):\mathcal{C}^{\op}\to \mathrm{Spc}
\]
given on objects by $c\mapsto \Map_{\mathcal{D}}(s_1(c),s_2(c))$. More precisely, this is the pre-sheaf classifying the right fibration 
\[
\mathcal{M}(s_1,s_2)\to \mathcal{C}
\]
where the left hand side sits in a pullback diagram
\[
\begin{diagram}
\mathcal{M}(s_1,s_2)&\rTo&\mathcal{C}\\
\dTo&&\dTo_{(s_1,s_2,\mathrm{id})}\\
\Fun([1],\mathcal{D})&\rTo_{(d_1\xrightarrow{f}d_2)\mapsto (d_1,d_2,\pi(f))}&\mathcal{D}\times \mathcal{D}\times\Fun([1],\mathcal{C})
\end{diagram}
\]

\begin{remark}\label{rem:presheaf_int_hom}
Applying this in the situation of Example~\ref{ex:presheaf_cartesian} shows that we have an \defnword{internal Hom} in $\mathrm{PSh}(\mathcal{C})$: For $F_1,F_2\in \mathrm{PSh}(\mathcal{C})$, we obtain the presheaf
\[
\shvMap(F_1,F_2):\mathcal{C}^{\op}\to \mathrm{Spc}
\]
given on objects by $c\mapsto \Map_{\mathrm{PSh}(\mathcal{C})/h_c}(F_1\times h_c,F_2\times h_c)$.
\end{remark}

\begin{remark}
\label{rem:int_hom_functors}
Suppose that $\pi':\mathcal{D}'\to \mathcal{C}$ is another Cartesian fibration, and that we have Cartesian functor $F:\mathcal{D}\to \mathcal{D}'$; then there is a canonical arrow of presheaves 
\[
\shvMap(s_1,s_2)\to \shvMap(F(s_1),F(s_2))
\]
given on objects $c$ by applying $F$ to $\Map_{\mathcal{D}}(s_1(c),s_2(c))$.
\end{remark}

\begin{lemma}
\label{lem:maps_cocartesian_sections}
Let $\mathrm{Cart}(\pi)$ be the $\infty$-category of Cartesian sections of $\pi$; then for any two such sections $s_1,s_2$ we have
\[
\Map_{\mathrm{Cart}(\pi)}(s_1,s_2) \simeq \varprojlim\shvMap(s_1,s_2)\in \mathrm{Spc}
\]
\end{lemma}
\begin{proof}
The space on the left is the fiber of the functor
\[
\Fun([1],\mathrm{Cart}(\pi)) \to \Fun(\{0\},\mathrm{Cart}(\pi))\times \Fun(\{1\},\mathrm{Cart}(\pi)) \simeq \mathrm{Cart}(\pi)\times \mathrm{Cart}(\pi)
\]
over $(s_1,s_2)$. Since $\mathrm{Cart}(\pi)$ is a fully faithful subcategory of the $\infty$-category $\Fun_{\mathcal{C}}(\mathcal{C},\mathcal{D})$ of sections of $\pi$, which in turn is equivalent to the $\infty$-category 
\[
\Fun_{\pi}(\mathcal{C},\Fun([1],\mathcal{D}))
\]
of functors $\mathcal{C}\to \Fun([1],\mathcal{D})$ lying over the identity section $\mathcal{C}\to \Fun([1],\mathcal{C})$, we see that what we have is the fiber over $(s_1,s_2)$ of the functor
\[
\Fun_{\pi}(\mathcal{C},\Fun([1],\mathcal{D})) \to \Fun_{\mathcal{C}}(\mathcal{C},\mathcal{D})\times \Fun_{\mathcal{C}}(\mathcal{C},\mathcal{D}).
\]
But this is precisely the space of sections of the right fibration $\mathcal{M}(s_1,s_2)\to \mathcal{C}$ considered in~\eqref{subsec:cart_sec_functor}. 

Therefore, the lemma now follows from~\cite[Cor. 3.3.3.2]{Lurie2009-oc}.
\end{proof}

\begin{lemma}
\label{lem:cofinal_effective_epi}
The following are equivalent for an object $X\in \mathcal{C}$ in an $\infty$-category $\mathcal{C}$ with finite products, and let $X_\bullet:\bm{\Delta}^{\op}\to \mathcal{C}$ be the associated \v{C}ech nerve with $X_n = \prod_{i\in [n]}X$. Then the following are equivalent:
\begin{enumerate}
  \item For every other object $Y$, the geometric realization
  \[
    \colim_{[n]\in \bm{\Delta}^{\op}}\Map_{\mathcal{C}}(Y,X_\bullet)\in \mathrm{Spc}
  \]
  is contractible.

  \item For every object $Y$, the space $\Map_{\mathcal{C}}(Y,X)$ is non-empty.
\end{enumerate}
\end{lemma}
\begin{proof}
Clearly (1) implies (2). The other implication is because $\mathrm{Spc}$ is an $\infty$-topos; see for instance the proof of~\cite[Prop. A.3.3.1]{Lurie2018-kh}.
\end{proof}

\begin{definition}\label{defn:quasi_final}
An object $X$ in an $\infty$-category $\mathcal{C}$ with finite products is \defnword{weakly final} if it satisfies the equivalent properties of the previous lemma. We will say that it is \defnword{weakly initial} if it is weakly final in the opposite category $\mathcal{C}^{\op}$ (which is assumed to admit finite products).
\end{definition}

\begin{proposition}
\label{prop:cartesian_sections}
Suppose that $\pi:\mathcal{D}\to \mathcal{C}$ is a Cartesian fibration, that $\mathcal{C}$ has finite products and that $X\in \mathcal{C}$ is a weakly final object. Consider the associated \v{C}ech nerve
\[
f:\Delta^{\op}\xrightarrow{[n]\mapsto X_n}\mathcal{C},
\]
and write
\[
\pi':\mathcal{D}\times_{\mathcal{C}}\Delta^{\op}\to \Delta^{\op}
\]
for the pullback of $\pi$ along $f$. Then the natural restriction functor
\[
\mathrm{Cart}(\pi) \to \mathrm{Cart}(\pi')
\]
is an equivalence.
\end{proposition}
\begin{proof}
First, we claim that $f$ is \emph{cofinal} as in~\cite[(4.1.1.1)]{Lurie2009-oc}. By Theorem 4.1.3.1 and Proposition 3.3.4.5 of \emph{loc. cit.}, this is equivalent to knowing that, for every object $Y$ of $\mathcal{C}$, the left fibration  $\bm{\Delta}^{\op}\times_{\mathcal{C}}\mathcal{C}_{Y/}\to \bm{\Delta}^{\op}$ is classified by a functor $F_Y:\bm{\Delta}^{\op}\to \mathrm{Spc}$ such that $\colim F_Y$ is contractible. But $F_Y$ is just the functor $[n]\mapsto \Map_{\mathcal{C}}(C,X_n)$, whose colimit is contractible by hypothesis.

Therefore, by~\cite[4.1.1.8]{Lurie2009-oc}, we see that, for any presheaf $p:\mathcal{C}^{\op}\to \mathrm{Spc}$ inducing a functor $p\circ f:\bm{\Delta}\to \mathrm{Spc}$, the natural map $\varprojlim p \to \varprojlim p\circ f $ is an equivalence. Using Lemma~\ref{lem:maps_cocartesian_sections}, this shows that the functor $\mathrm{Cart}(\pi) \to \mathrm{Cart}(\pi')$ is fully faithful. 

To show that it is also essentially surjective, we can replace $\mathcal{D}$ by the underlying left fibration classified by the functor 
\[
\mathcal{C}\xrightarrow{c\mapsto \mathcal{D}_c}\mathrm{Spc},
\]
where the proposition is immediate from the cofinality of $f$.
\end{proof}

Dually we have:
\begin{proposition}
\label{prop:cocartesian_sections}
For any coCartesian fibration $\pi:\mathcal{D}\to \mathcal{C}$ write $\mathrm{coCart}(\pi)$ for the $\infty$-category of coCartesian sections of $\pi$. Suppose that $\mathcal{C}$ has finite coproducts and that $X\in \mathcal{C}$ is a weakly initial object. Consider the associated \v{C}ech conerve
\[
f:\Delta\xrightarrow{[n]\mapsto X^{(n)}}\mathcal{C},
\]
and write
\[
\pi':\mathcal{D}\times_{\mathcal{C}}\Delta\to \Delta
\]
for the pullback of $\pi$ along $f$. Then the natural restriction functor
\[
\mathrm{coCart}(\pi) \to \mathrm{coCart}(\pi')
\]
is an equivalence. In particular, for any two sections $s_1,s_2\in \mathrm{coCart}(\pi)$, we have
\[
\Map_{\mathrm{coCart}(\pi)}(s_1,s_2)\simeq \mathrm{Tot}\Map_{\mathcal{D}}(s_1(X^{(\bullet)}),s_2(X^{(\bullet)}))
\]
\end{proposition}

\section{Completions}\label{sec:completions}

Fix  a prime $p$. Recall that $M\in \Mod{\Int}$ is \defnword{(derived) $p$-complete} if the natural map $M\to \varprojlim_n M/{}^\mathbb{L}p^n$ is an equivalence. 

\begin{lemma}
\label{lem:torsion_complete}
Suppose that we have $M\in \Mod[\mathrm{cn}]{\Int}$ and $n\in \Int_{\geq 1}$ such that $p^n\cdot H^i(M) = 0$ for all $i$. Then $M$ is $p$-complete.
\end{lemma}
\begin{proof}
The hypothesis implies that the map $p^m:M\to M$ is nullhomotopic for $m\geq n$, and so the inverse system $\{M/{}^{\mathbb{L}}p^n\}_n$ is equivalent to to $\{M[1]\oplus M\}_n$ where the transition maps are given by multiplication-by-$p$ in the first summand and the identity on the second. Therefore, to finish, it is enough to observe that
\[
\varprojlim_{p}M \simeq \varprojlim_{p^n}M \simeq 0.
\]
\end{proof}

The next lemma is standard, and the proof is merely a derived restatement of~\cite[Theorem 2.8]{Yekutieli2018-qt}.
\begin{lemma}\label{lem:completions}
Suppose that we have $a\in \Int$ and an object
\[
\{M_n\}_{n\ge 1}\in \varprojlim_n\Mod[\leq a]{\Int/p^n\Int}.
\]
Set $M = \varprojlim_n M_n\in \Mod{\Int}$. Then $M$ is $p$-complete, and for each $n\geq 1$, the natural map $M/{}^{\mathbb{L}}p^n\to M_n$ is an equivalence.
\end{lemma}
\begin{proof}
For any $Q\in \Mod{\Int}$, the object $\widehat{Q} = \varprojlim_n Q/{}^{\mathbb{L}}p^n$ is always $p$-complete, and the natural map $Q/{}^{\mathbb{L}}p^n \to \widehat{Q}/{}^\mathbb{L}p^n$ is an equivalence. 

In particular, the canonical factorization $M\to M/{}^{\mathbb{L}}p^n\to M_n$ of the map $M\to M_n$ gives us a factorization $M\to \widehat{M}/{}^{\mathbb{L}}p^n\to M_n$, and taking inverse limits gives us a section $\widehat{M}\to M$. Thus, $M$ is a direct summand of a $p$-complete object, and must itself therefore be $p$-complete, implying of course that the natural map $M\to \widehat{M}$ is an equivalence.

Let $K_n$ be the fiber in $\Mod{\Int}$ of the natural map $M/{}^{\mathbb{L}}p^n\to M_n$; note that this underlies an object $\{K_n\}$ in $\varprojlim_n\Mod[\leq a+1]{\Int/p^n\Int}$, and also that $\varprojlim_nK_n \simeq 0$. 

We have to show that this implies $K_n\simeq 0$ for all $n$. Via descending induction on $a$ combined with shifting, it is enough to show that, when $a+1  = 0$, we have $H^0(K_n)=0$ for all $n\geq 1$. Taking the limit of the cofiber sequence $\tau^{\geq -1}K_n\to K_n\to H^0(K_n)$ shows that we have an equivalence
\[
\varprojlim_n H^0(K_n)[0] \simeq \varprojlim_n\tau^{\leq {-1}}K_n[1]
\]
Since the inverse limit functor has cohomological dimension $1$, the right hand side is in $\Mod[\geq -1]{\Int}$, while the left hand side is concentrated in degrees $0,1$; so we see that both sides have to be equivalent to $0$. Now, for $m\leq n$, we have $H^0(K_n)/p^m \simeq H^0(K_m)$, so the transition maps in the sequence $\{H^0(K_n)\}$ of abelian groups are surjective, and the vanishing of the limit now shows that $H^0(K_n) = 0$ for all $n\geq 1$.
\end{proof}

\section{Filtrations and gradings}\label{sec:filtrations}

We review the $\infty$-category theoretic perspective on filtered and graded objects. A good summary can be found in~\cite[\S 2.4]{Mao2021-jt}. We will also use the language of symmetric monoidal $\infty$-categories from~\cite{Lurie2017-oh}.

\subsection{}\label{subsec:filtered}
Given an $\infty$-category $\mathcal{C}$, the $\infty$-category of \defnword{filtered objects in $\mathcal{C}$} is the functor category 
\[
\mathrm{Fil}(\mathcal{C}) = \mathrm{Fun}((\Int,\geq),\mathcal{C}). 
\]

We will usually denote an object in $\Fil(\mathcal{C})$ by $\Fil^\bullet X$ where, for $k\in \Int$, $\Fil^kX$ is the object associated with $k\in \Int$. Usually, $\mathcal{C}$ will be admit all colimits and in this case we will say that $\Fil^\bullet X$ is an \defnword{exhaustive filtration} of $X = \colim_{k\in \Int}\Fil^k X$. The filtration is \defnword{bounded above} (resp. \defnword{bounded below}) if there exists $n\in \Int$ such that the map $\Fil^{k+1}X \to \Fil^{k}X$ is an equivalence for $k\leq n$ (resp. $k \geq n$); it is \defnword{bounded} or \defnword{finite} if it is bounded both below and above. 

For $m\in \Int$, there is a \defnword{shift functor}
\begin{align*}
\mathrm{Fil}(\mathcal{C}) &\xrightarrow{\Fil^\bullet X \mapsto \Fil^\bullet X\{m\}}\Fil(\mathcal{C}),
\end{align*}
where $\Fil^kX\{m\} = \Fil^{k+m}X$ for $k\in\Int$. 

\subsection{}
If $\mathcal{C}$ underlies a symmetric monoidal $\infty$-category, then so does $\Fil(\mathcal{C})$. The monoidal structure is given by the \defnword{Day convolution} induced by the additive monoidal structure on $(\Int,\geq)$. Prosaically, for two filtered objects $\Fil^\bullet X$, $\Fil^\bullet Y$ in $\mathcal{C}$, we have
\[
\Fil^k(X\otimes Y) \simeq \colim_{i+j\geq k}\Fil^iX\otimes\Fil^jY.
\]
Therefore, we can make sense of an $\mathcal{Z}_{K,\infty}(E)$ or commmutative algebra object in $\Fil(\mathcal{C})$. If $\mathcal{C}$ admits all (small) colimits and if the symmetric monoidal structure respects colimits in each variable, then, for any commutative algebra object $\Fil^\bullet X$ in $\Fil(\mathcal{C})$, $X$ is a commutative algebra object in $\mathcal{C}$

This will apply in particular when $\mathcal{C} = \Mod{R}$ for some $R\in \mathrm{CRing}$. We will write $\mathrm{FilMod}_R$ for $\Fil(\Mod{R})$. If $R$ is represented by a simplicial commutative ring with underlying differential graded ring $R_{\mathrm{dg}}$, then any filtered commutative differential graded algebra over $R_{\mathrm{dg}}$ will give rise to a commutative algebra object in $\mathrm{FilMod}_R$. 

A \defnword{filtered animated commutative algebra} over $R$ is a commutative algebra object $\Fil^\bullet S$ in $\mathrm{FilMod}_R$ equipped with the structure of an animated commutative $R$-algebra on $S$ underlying its $\mathcal{Z}_{K,\infty}(E)-R$-algebra structure; write $\mathrm{FilCRing}_{R/}$ for the $\infty$-category of such objects. If $R = \Int$, then we will write $\mathrm{FilCRing}$ for the corresponding category and refer to its objects as \defnword{filtered animated commutative rings}. Any $A\in \mathrm{CRing}$ can be canonically upgraded to a filtered animated commutative ring with $\Fil^i A = A$ for $i\leq 0$ and $\Fil^i A = 0$ for $i>0$.

\subsection{}\label{subsec:filmod}
By general considerations (see~\cite[\S 4.5.2]{Lurie2017-oh}), for any $\Fil^\bullet S$ in $\mathrm{FilCRing}$, there is a symmetric monoidal $\infty$-category $\mathrm{FilMod}_{\Fil^\bullet S}$ of filtered modules over $\Fil^\bullet S$ where the symmetric monoidal structure admits $\Fil^\bullet S$ as the unit object. 

To alleviate notation, we will write $\Map_{\Fil^\bullet S}(\_\_,\_\_)$ instead of $\Map_{\mathrm{FilMod}_{\Fil^\bullet S}}(\_\_,\_\_)$ for the mapping spaces here.

Given a map of filtered animated commutative rings $f:\Fil^\bullet R \to \Fil^\bullet S$, there exists a monoidal filtered base-change functor
\begin{align*}
\mathrm{FilMod}_{\Fil^\bullet R} &\to \mathrm{FilMod}_{\Fil^\bullet S}\\
\Fil^\bullet M &\mapsto \Fil^\bullet S\otimes_{\Fil^\bullet R}\Fil^\bullet M
\end{align*}
that is left adjoint to the restriction functor in the other direction.

More succinctly the symmetric monoidal $\infty$-categories $\mathrm{FilMod}_{\Fil^\bullet S}$ organize into a coCartesian fibration $\mathrm{FilMod}\to \mathrm{FilCRing}$.

\subsection{}\label{subsec:graded}
There is a completely analogous graded analogue. Given an $\infty$-category $\mathcal{C}$, the $\infty$-category of \defnword{graded objects in $\mathcal{C}$} is defined to be the functor category 
\[
\mathrm{Gr}(\mathcal{C}) = \mathrm{Fun}(\Int,\mathcal{C}),
\]
where we view $\Int$ as a discrete category. We will usually denote an object in $\mathrm{Gr}(\mathcal{C})$ by $X_\bullet$ where, for $k\in \Int$, $X_k$ is the object associated with $k\in \Int$. 

For $m\in \Int$, there is a \defnword{shift functor}
\begin{align*}
\mathrm{Gr}(\mathcal{C}) &\xrightarrow{X_\bullet \mapsto X\{m\}_\bullet}\mathrm{Gr}(\mathcal{C}),
\end{align*}
where $X\{m\}_k = X_{k+m}$ for $k\in\Int$.

\subsection{}
If $\mathcal{C}$ underlies a symmetric monoidal $\infty$-category, then so does $\Fil(\mathcal{C})$. The monoidal structure is given by the \defnword{Day convolution} induced by the additive monoidal structure on $\Int$. Prosaically, for two graded objects $X_\bullet$ and $Y_\bullet$, we have
\[
(X\otimes Y)_k \simeq \bigoplus_{i+j = k}X_i\otimes Y_j
\]
Therefore, we can make sense of an $\mathcal{Z}_{K,\infty}(E)$ or commmutative algebra object in $\mathrm{Gr}(\mathcal{C})$.

This will apply in particular when $\mathcal{C} = \Mod{R}$ for some $R\in \mathrm{CRing}$. We will write $\mathrm{GrMod}_R$ for $\mathrm{Gr}(\Mod{R})$. If $R$ is represented by a simplicial commutative ring with underlying differential graded ring $R_{\mathrm{dg}}$, then any graded commutative differential graded algebra over $R_{\mathrm{dg}}$ will give rise to a commutative algebra object in $\mathrm{GrMod}_R$. 

A \defnword{graded animated commutative algebra} over $R$ is a commutative algebra object $S_\bullet$ in $\mathrm{GrMod}_R$ equipped with the structure of an animated commutative $R$-algebra on $\bigoplus_k S_k$ underlying its $\mathcal{Z}_{K,\infty}(E)-R$-algebra structur; write $\mathrm{GrCRing}_{R/}$ for the $\infty$-category of such objects. If $R = \Int$, then we will write $\mathrm{GrCRing}$ for the corresponding category and refer to its objects as \defnword{graded animated commutative rings}. Any $A\in \mathrm{CRing}$ can be canonically upgraded to a graded animated commutative ring $A_\bullet$ with $A_0 = A$ and $A_n = 0$ for $n\neq 0$.

\subsection{}
By general considerations (see~\cite[\S 4.5.2]{Lurie2017-oh}), for any $S_\bullet$ in $\mathrm{GrCRing}$, there is a symmetric monoidal $\infty$-category $\mathrm{GrMod}_{S_\bullet}$ of graded modules over $S_\bullet$ where the symmetric monoidal structure admits $S_\bullet$ as the unit object.

To alleviate notation, we will write $\Map_{S_\bullet}(\_\_,\_\_)$ instead of $\Map_{\mathrm{GrMod}_{S_\bullet}}(\_\_,\_\_)$ for the mapping spaces here.

Given a map of graded animated commutative rings $f:R_\bullet \to S_\bullet$, there exists a graded base-change functor
\begin{align*}
\mathrm{GrMod}_{R_\bullet} &\to \mathrm{GrMod}_{S_\bullet}\\
M_\bullet &\mapsto S_\bullet\otimes_{R_\bullet}M_\bullet
\end{align*}
that is left adjoint to the restriction functor.

\subsection{}\label{subsec:filtered_to_graded}
Let $\mathcal{C}$ be a stable $\infty$-category. Then $\Fil(\mathcal{C})$ and $\mathrm{Gr}(\mathcal{C})$ are also stable $\infty$-categories, and, for each $i\in \Int$, there is a canonical functor
\begin{align*}
\mathrm{gr}^i:\mathrm{Fil}(\mathcal{C})&\to \mathcal{C}\\
\Fil^\bullet X &\mapsto \hcoker(\Fil^{i+1}X \to \Fil^iX).
\end{align*}

These organize to give a natural functor
\begin{align*}
\mathrm{gr}^\bullet:\mathrm{Fil}(\mathcal{C})&\to \mathrm{Gr}(\mathcal{C})\\
\Fil^\bullet X &\mapsto \gr^{\bullet}X.
\end{align*}

If $\mathcal{C} = \Mod{R}$, then $\gr^\bullet$ is a symmetric monoidal functor; see for instance~\cite[Prop. 2.26]{Gwilliam2018-qm}. In particular, for any filtered animated commutative ring $\Fil^\bullet R$, $\gr^\bullet R$ is a graded animated commutative ring, and we obtain a symmetric monoidal functor
\[
\gr^\bullet: \mathrm{FilMod}_{\Fil^\bullet R}\to \mathrm{GrMod}_{\gr^\bullet R}.
\]

\begin{lemma}
\label{lem:graded_weight_filtration}
Suppose that $A_\bullet$ is a graded animated commutative ring with $A_m \simeq 0$ for $m<0$, and suppose that we have $M_\bullet\in \mathrm{GrMod}_{A_\bullet}$ with $a\in \Int$ such that $M_k\simeq 0$ for $k<a$. Write 
\[
\overline{M}_\bullet = A_0\otimes_{A_\bullet}M_\bullet
\]
for the graded base change of $M_\bullet$. Then $M_\bullet$ admits a canonical filtration $\Fil^\bullet_{\mathrm{wt}}M_\bullet$ in $\mathrm{GrMod}_{A_\bullet}$ with
\[
\gr^{-i}_{\mathrm{wt}}M_\bullet \simeq A_\bullet\otimes_{A_0}\gr^i\overline{M}_\bullet\left\{-i\right\}.
\]
\end{lemma}
\begin{proof}
The construction of the filtration is by descending induction on $a$.  

We first claim that the natural map $M_a \to \overline{M}_a$ is an equivalence. To see this, observe that we have a fiber sequence
\[
A_{\geq 1}\otimes_{A_\bullet}M_\bullet \to M_\bullet \to \overline{M}_{\bullet}
\]
and note that the left hand side is a graded module supported in degrees $\geq a+1$.

Therefore, we have a canonical map
\[
A_\bullet\otimes_{A_0}\overline{M}_a\left\{-a\right\}\xrightarrow{\simeq}A_\bullet\otimes_{A_0}M_a\left\{-a\right\}\to M_\bullet 
\]
whose cofiber $M'_\bullet$ is supported in degrees $\leq a$. Graded base-change to $A_0$ yields a cofiber sequence
\[
\overline{M}_a\left\{-a\right\}\to \overline{M}\to \overline{M}',
\]
which shows that $\overline{M}_i\xrightarrow{\simeq}\overline{M}'_i$ for $i\geq a+1$ and $\overline{M}_i\simeq 0$ for $i<a+1$.

By our inductive hypothesis, $M'_\bullet$ admits a filtration $\Fil^\bullet_{\mathrm{wt}}M'_\bullet$ with 
\[
\gr^{-i}_{\mathrm{wt}}M'_\bullet \simeq A_\bullet\otimes_{A_0}\gr^i\overline{M}'_{\bullet}\left\{-i\right\}.
\]

We now obtain our desired filtration on $M_\bullet$ by setting
\[
\Fil^i_{\mathrm{wt}}M_\bullet = \begin{cases}
0&\text{if $i>-a$}\\
\Fil^i_{\mathrm{wt}}\overline{M}_\bullet\times_{\overline{M}_\bullet}M_\bullet&\text{if $i\leq -a$}.
\end{cases} 
\]
\end{proof}

As an immediate consequence, we obtain:
\begin{lemma}
\label{lem:associated_graded}
Suppose that $\Fil^\bullet A'$ is a non-negatively filtered animated commutative ring. Set $A = \gr^0A'$. Let $\Fil^\bullet M'\in \mathrm{FilMod}_{\Fil^\bullet A'}$ be such that $\gr^{k}M' \simeq 0$ for $-k$ sufficiently large. Set
\[
\Fil^\bullet M = A\otimes_{\Fil^\bullet A'}\Fil^\bullet M',
\]
where $A$ is equipped with the trivial filtration. Then there is a canonical filtration $\mathrm{Fil}^\bullet_{\mathrm{wt}}$ on $\gr^\bullet M'$ in $\mathrm{GrMod}_{\gr^\bullet A'}$ such that
\[
\gr^{-i}_{\mathrm{wt}}\gr^\bullet M' \simeq \gr^\bullet A'\otimes_A\gr^iM\{-i\}.
\]
In particular, if $\gr^{a-1}M' \simeq 0$, then we have $\gr^aM' \simeq \gr^aM$.
\end{lemma}

\subsection{}\label{subsec:internal_hom_filt}
Let $\Fil^\bullet A$ be a filtered animated commutative ring.\footnote{By this, we mean that $\gr^iA \simeq 0$ for $i<0$.} For $\Fil^\bullet M, \Fil^\bullet N$ in $\mathrm{FilMod}_{\Fil^\bullet A}$ we have an \defnword{internal Hom} object $\Fil^\bullet\underline{\Map}(\Fil^\bullet M,\Fil^\bullet N)\in \mathrm{FilMod}_{\Fil^\bullet A}$ characterized by the fact there is a natural equivalence of contravariant functors
\[
\Map_{\Fil^\bullet A}(\Fil^\bullet M\otimes_{\Fil^\bullet A}\_\_,\Fil^\bullet N)\simeq \Map_{\Fil^\bullet A}(\_\_,\Fil^\bullet\underline{\Map}(\Fil^\bullet M,\Fil^\bullet N))
\]
out of $\mathrm{FilMod}_{\Fil^\bullet A}$.

We set $\Fil^\bullet M^\vee = \underline{\Map}(\Fil^\bullet M,\Fil^\bullet A)$.

Completely analogously, if $A_\bullet$ is a graded animated commutative ring, we have, for every $M_\bullet,N_\bullet\in \mathrm{GrMod}_{A_\bullet}$ an internal Hom object $\underline{\Map}(M_\bullet,N_\bullet)_\bullet\in \mathrm{GrMod}_{A_\bullet}$ and we set $M^\vee_\bullet = \underline{\Map}(M_\bullet,A_\bullet)$.

\begin{lemma}
\label{lem:fil_int_hom}
Suppose that $\Fil^\bullet A'$ is a non-negatively filtered animated commutative ring, and that we have $\Fil^\bullet M'$ and $\Fil^\bullet N'$ in $\mathrm{FilMod}_{\Fil^\bullet A'}$.
\begin{enumerate}
  \item For any $i\in \Int$, we have
  \[
\Fil^i\underline{\Map}(\Fil^\bullet M',\Fil^\bullet N') \simeq \underline{\Map}_{\Fil^\bullet A'}(\Fil^\bullet M'\{-i\},\Fil^\bullet N')
  \]
  and the map
  \[
  \Fil^{i+1}\underline{\Map}(\Fil^\bullet M',\Fil^\bullet N')\to \Fil^{i}\underline{\Map}(\Fil^\bullet M',\Fil^\bullet N') 
  \] 
  corresponds under this equivalence to restriction along the natural map
  \[
    \Fil^\bullet M'\{-i\} \to \Fil^\bullet M'\{-(i+1)\}.
  \]

  \item There are canonical equivalences
  \[
   \gr^i\underline{\Map}(\Fil^\bullet M',\Fil^\bullet N')\simeq \underline{\Map}_{\Fil^\bullet A'}(\gr^\bullet M'\{-(i+1)\}[-1],\Fil^\bullet N')\xleftarrow{\simeq} \underline{\Map}_{\gr^\bullet A'}(\gr^\bullet M'\{-i\},\gr^\bullet N').
  \]

  \item Suppose that $\Fil^\bullet A'$ and $\Fil^\bullet M'$ satisfy the conditions in Lemma~\ref{lem:associated_graded}, and suppose in addition that there exist integers $a,b,c\in \Int$ such that $\gr^i M\simeq 0$ for $i\notin [a,b]$, and that $\gr^iN'\simeq 0$ for $i<c$. Suppose also that $\gr^iM$ is perfect for every $i$. Then
  \[
    \gr^i\underline{\Map}(\Fil^\bullet M',\Fil^\bullet N') \simeq 0
  \]
  for $i<c-b$.

  \item Let the hypotheses be as in (3). Then the natural map
  \[
   \Fil^\bullet M^{',\vee}\otimes_{\Fil^\bullet A}\Fil^\bullet N'\to \Fil^\bullet\underline{\Map}(\Fil^\bullet M',\Fil^\bullet N')
  \]
  is an equivalence.

  \item Let the hypotheses be as in (3). for any map $\Fil^\bullet A'\to \Fil^\bullet B'$ of filtered animated commutative rings, the natural map
  \[
\Fil^\bullet B'\otimes_{\Fil^\bullet A'}\Fil^\bullet\underline{\Map}(\Fil^\bullet M',\Fil^\bullet N')\to \Fil^\bullet\underline{\Map}(\Fil^\bullet B'\otimes_{\Fil^\bullet A'}\Fil^\bullet M',\Fil^\bullet B'\otimes_{\Fil^\bullet A'}\Fil^\bullet N')
  \]
  is an equivalence.
\end{enumerate}
\end{lemma}
\begin{proof}
(1) follows from the fact that, for any $\Fil^\bullet P'$ in $\mathrm{FilMod}_{A'}$, we have
\[
\underline{\Map}_{\Fil^\bullet A'}(\Fil^\bullet A\{-i\},\Fil^\bullet P') \simeq \Fil^iP',
\]
and the first equivalence in (2) is a consequence of (1) and the fiber sequence
\[
\Fil^\bullet M'\{-i\}\to \Fil^\bullet M'\{-(i+1)\}\to \gr^\bullet M'\{-(i+1)\}.
\]

For the second equivalence in (2), we first use the fiber sequence
\[
\Fil^\bullet N'\to \Fil^\bullet N'\{-1\}\to \gr^\bullet N'\{-1\}
\]
to obtain, for any $P'_\bullet\in \mathrm{GrMod}_{\gr^\bullet A'}$, a canonical map
\begin{align}\label{eqn:can_arrow_int_hom_assoc_graded}
\underline{\Map}_{\gr^\bullet A'}(P'_\bullet,\gr^\bullet N')\xrightarrow{\simeq}\underline{\Map}_{\gr^\bullet A'}(P'_\bullet\{-1\},\gr^\bullet N'\{-1\})\to \underline{\Map}_{\Fil^\bullet A'}(P'_\bullet\{-1\}[-1],\Fil^\bullet N').
\end{align}
To show that this is an equivalence, it suffices to consider the case $P'_\bullet = \gr^\bullet A'$, in which case both sides are naturally equivalent to $\gr^0N'$.

For (3), we have to show 
\[
\gr^i\underline{\Map}(\Fil^\bullet M',\Fil^\bullet N')\simeq \underline{\Map}_{\gr^\bullet A'}(\gr^\bullet M'\{-i\},\gr^\bullet N')\simeq 0
\]
for $i<c-b$. For this, we will use the filtration $\Fil^\bullet_{\mathrm{wt}}\gr^\bullet M'$ from Lemma~\ref{lem:associated_graded}, which is a \emph{finite} filtration by our hypotheses. It is now enough to show that
\begin{align*}
\underline{\Map}_{\gr^\bullet A'}(\gr^\bullet A'\otimes_A\gr^k M\{-(i+k)\},\gr^\bullet N')&\simeq (\gr^kM)^\vee\otimes_{A}\underline{\Map}_{\gr^\bullet A'}(\gr^\bullet A'\{-(i+k)\},\gr^\bullet N')\\
&\simeq (\gr^kM)^\vee\otimes_{A}\gr^{i+k}N'
\end{align*}
is zero for $i<c-b$. Indeed, for this last object to be non-zero, we need both $k\leq b$ and $i+k\geq c$.

For (4), the natural arrow is obtained via adjunction from the arrow
\[
\Fil^\bullet M\otimes_{\Fil^\bullet A'}\Fil^\bullet M^{',\vee}\otimes_{\Fil^\bullet A'}\Fil^\bullet N'\to \Fil^\bullet N'
\]
that arises in turn from the natural arrow
\[
\Fil^\bullet M\otimes_{\Fil^\bullet A'}\Fil^\bullet M^{',\vee}\to \Fil^\bullet A'
\]
corresponding via adjunction to the identity on $\Fil^\bullet M^{',\vee}$.

To see that the natural arrow is an equivalence in the situation of (3), one first sees that the objects $\Fil^\bullet M'$ satisfying the conditions in (3) are all obtained via the operations of finite colimits and their retracts, beginning with the unit object $\Fil^\bullet A'$. Therefore, it is enough to show that the arrow in (4) is an equivalence when $\Fil^\bullet M' \simeq \Fil^\bullet A'$ where it is clear.

For (5), we now only have to show that the natural map
\[
\Fil^\bullet B'\otimes_{\Fil^\bullet A'}\Fil^\bullet M^{',\vee}\to \Fil^\bullet(\Fil^\bullet B'\otimes_{\Fil^\bullet A'}M')^\vee
\]
is an equivalence. This follows because $\Fil^\bullet M^{',\vee}$ is the dual of the dualizable object $\Fil^\bullet M'$ (as is easily checked using (4)), and because any symmetric monoidal functor preserves dualizable objects and their duals.
\end{proof}

\section{Module objects}\label{sec:modules}

\subsection{}\label{subsec:module_objects}
Let $D$ be a discrete (not necessarily commutative) ring, and write $\mathrm{Latt}_D$ for the category of finite free $D$-modules. In a mild generalization of~\cite[\S 1.2]{lurie:spec_AB}, given an $\infty$-category $\mathcal{C}$ that admits finite products, we define a \defnword{$D$-module object of }$\mathcal{C}$ to be a functor
\[
F:\mathrm{Latt}_D^{\op}\to \mathcal{C}
\]
that preserves finite products. If $F(D) = M$, then we will say that $M$ \defnword{underlies} a $D$-module object of $\mathcal{C}$, though we might sometimes abuse notation and conflate $M$ with the $D$-module object itself.

These objects are organized into an $\infty$-category $\mathrm{Mod}_D(\mathcal{C})$. If $D = \Int$, we will also call them \defnword{abelian group objects} in $\mathcal{C}$, and write $\mathrm{Ab}(\mathcal{C})$ for the corresponding $\infty$-category.

\subsection{}\label{subsec:module_objects_sets}
When $\mathcal{C} = \mathrm{Set}$ is the ordinary category of sets, this is just the category of $D$-modules. The case of $D = \Int$ is~\cite[Proposition 1.2.7]{lurie:spec_AB}, and the general case is similar: Given a $D$-module $M$, we obtain the functor $\Hom_D(\_,M)$ on $\mathrm{Latt}_D^{\op}$. Conversely, given a functor $F$, we set $M = F(D)$; applying $F$ to the diagonal $D\xrightarrow{\Delta}D^2$ gives us the addition on $M$; applying it to the map $D\to 0$ gives us the zero element; applying it to $D\xrightarrow{x\mapsto xd}D$ for $d\in D$ gives us scalar multiplication by $d$. One now checks that this gives the inverse functor, and establishes the claimed equivalence.

In general, if $\mathcal{C}$ is an ordinary category, then a similar argument shows that a $D$-module object in $\mathcal{C}$ is an object $M\in \mathcal{C}$ along with a lift of its representable functor $h_M$ to one valued in the discrete category of $D$-modules; see~\cite[Example 1.2.12]{lurie:spec_AB}.

\subsection{}\label{subsec:module_objects_spaces}
When $\mathcal{C} = \mathrm{Spc}$, the argument in~\cite[Example 1.2.9]{lurie:spec_AB} shows that the $\infty$-category of $D$-module objects in $\mathrm{Spc}$ can be identified with the animation of the category of $D$-modules, and is thus the $\infty$-category $\Mod[\mathrm{cn}]{D}$ underlying the simplicial category of simplicial $D$-modules.

For any $\infty$-category, we can identify (see~\cite[(1.2.11)]{lurie:spec_AB})
 \[
    \mathrm{Mod}_D(\mathrm{PSh}(\mathcal{C})) \simeq \mathrm{Fun}(\mathcal{C}^{\op},\mathrm{Mod}_D(\mathrm{Spc})) \simeq \mathrm{Fun}(\mathcal{C}^{\op},\Mod[\mathrm{cn}]{D}).
  \]
That is we can identify $D$-module objects in $\mathrm{PSh}(\mathcal{C})$ with presheaves on $\mathcal{C}$ valued in $\Mod[\mathrm{cn}]{D}$.

\begin{remark}\label{rem:module_objects_additive}
If $\mathcal{C}$ is a fully faithful subcategory of $\Mod{R}$ for some $R\in \mathrm{CRing}$ (or more generally a stable $\infty$-category with the structure of a $\Mod{\Int}$-module), then any object $M$ of $\mathcal{C}$ underlies a canonical abelian group object of $\mathcal{C}$: this is just a rephrasing of the fact that, for any other object $N$, the space $\Map_{\mathcal{C}}(N,M)$ underlies an object in $\Mod[\mathrm{cn}]{\Int}$.
\end{remark}

\subsection{}\label{subsec:module_objects_hom_spaces}
Suppose that we have $F$ in $\Mod{D}(\mathcal{C})$ and $G$ in $\mathcal{C}$; then the functor
\begin{align*}
\mathrm{Latt}^{\mathrm{op}}_D &\to \mathrm{Spc}\\
\Lambda&\mapsto \Map_{\mathcal{C}}(G,F(\Lambda))
\end{align*}
shows that $\Map_{\mathcal{C}}(G,F(D))$ underlies a canonical $D$-module object of $\mathrm{Spc}$; that is, it has a lift to $\Mod[\mathrm{cn}]{D}$.

\section{Square zero thickenings and cotangent complexes}\label{sec:square_zero}

\subsection{}\label{subsec:square_zero}
One of the many aspects of algebraic geometry that is clarified by the derived perspective is that of square-zero thickenings of commutative rings. Suppose that we have $R\in \mathrm{CRing}$ and $C\in \mathrm{CRing}_{R/}$; given $M\in \mathrm{Mod}^{\mathrm{cn}}_R$, we obtain a new animated ring $C\oplus M\in \mathrm{CRing}_C$ equipped with a section $C\oplus M\to C$ in $\mathrm{CRing}_C$. Concretely, if $C$ is represented by a simplicial commutative ring (again denoted $C$) and $M$ by a simplicial module over $C$ (again denoted $M$), then $C\oplus M$ is represented by the direct sum with multiplication given by the usual formula: 
\[
(c_1,m_1)\cdot (c_2,m_2) = (c_1c_2,c_1m_2+c_2m_1).
\]

The space $\mathrm{Der}_R(C,M)$  of $R$-\defnword{derivations} of $C$ with values in $M$ is the fiber over the identity of the natural map
\[
\mathrm{Map}_{\mathrm{CRing}_{R/}}(C,C\oplus M)\to \mathrm{Map}_{\mathrm{CRing}_{R/}}(C,C)
\]
In particular, the tautological map $C\to C\oplus M$ corresponds to a derivation $d_{\mathrm{triv}}:C\to C\oplus M$.

The theory of the cotangent complex (see~\cite[\S 1.4]{TV2},~\cite[Ch. 7]{Lurie2017-oh}) now gives us $\mathbb{L}_{C/R}\in \mathrm{Mod}^{\mathrm{cn}}_C$ equipped with a canonical derivation $d_{\mathrm{univ}}:C\to C\oplus \mathbb{L}_{C/R}$ such that, for any $M\in \mathrm{Mod}_C^{\mathrm{cn}}$, composition with $d_{\mathrm{univ}}$ induces an equivalence:
\[
\mathrm{Map}_{\mathrm{Mod}_C}(\mathbb{L}_{C/R},M)\xrightarrow{\simeq}\mathrm{Der}_R(C,M).
\]
When $R$ is discrete and $C$ is a smooth $R$-algebra, then we can identify $\mathbb{L}_{C/R}$ with the module of differentials $\Omega^1_{C/R}$. In general, $\mathbb{L}_{C/R}$---or more precisely, the pair $(C,\mathbb{L}_{C/R})$---is obtained by animating the functor
\[
\mathrm{Poly}^{\mathrm{op}}_R\xrightarrow{C\mapsto (C,\Omega^1_{C/R})}\mathrm{Mod}^{\mathrm{cn}}
\]
on finitely generated polynomial $R$-algebras.

If we have arrows $R\to C\to C'$, then we have a canonical cofiber sequence
\[
C'\otimes_C\mathbb{L}_{C/R}\to \mathbb{L}_{C'/R}\to \mathbb{L}_{C'/C}
\]
in $\mathrm{Mod}_{C'}^{\mathrm{cn}}$; see~\cite[(25.3.2.5)]{Lurie2018-kh}. Also, if $R\to R'$ is any map, we have a canonical equivalence
\[
R'\otimes_R\mathbb{L}_{C/R}\xrightarrow{\simeq}\mathbb{L}_{R'\otimes_RC/R'}.
\]

All square zero thickenings of $C$ of interest to us can now be obtained as follows: For any $M\in \mathrm{Mod}_C^{\mathrm{cn}}$, and any map $\eta:\mathbb{L}_{C/R}\to M[1]$, we obtain a map $C_{\eta}\to C$ that is a square zero thickening with kernel $M$: this means that it is a surjection on $\pi_0$, and that its homotopy fiber is in the image of $\mathrm{Mod}_C^{\mathrm{cn}}$ and is in fact equivalent to $M$.

The construction of $C_{\eta}$ is quite simple: It sits in a homotopy pullback diagram
\[
\begin{diagram}
C_{\eta}&\rTo&C\\
\dTo&&\dTo_{d_{\eta}}\\
C&\rTo_{d_{\mathrm{triv}}}&C\oplus M[1]
\end{diagram}
\]
where $d_{\eta}$ is the derivation associated with $\eta$.

This applies in particular to the Postnikov tower $\{\tau_{\leq n}R\}_{n\in \Int_{\geq 0}}$ associated with any $C\in \mathrm{CRing}_{R/}$: For any $n\geq 1$, there exists a derivation $\eta_{n,C}:\mathbb{L}_{\tau_{\leq n}C/R}\to \pi_{n+1}(C)[n+2]$ such that we have a Cartesian diagram
\[
\begin{diagram}
\tau_{\leq n+1}C&\rTo&\tau_{\leq n}C\\
\dTo&&\dTo_{d_{\eta_{n,C}}}\\
\tau_{\leq n}C&\rTo_{d_{\mathrm{triv}}}&C\oplus \pi_{n+1}(C)[n+2].
\end{diagram}
\]



\section{Lurie's representability theorem and derived mapping stacks}\label{sec:lurie}

Suppose that $\mathcal{C}$ is an $\infty$-category admitting all finite and sequential limits, totalizations of cosimplicial objects, as well as filtered colimits. A $\mathcal{C}$-\defnword{valued prestack over} $R\in \mathrm{CRing}$ is a functor
\[
F:\mathrm{CRing}_{R/}\to \mathcal{C}.
\]

When $\mathcal{C} = \mathrm{Spc}$, we will write $\mathrm{PStk}_R$ for the $\infty$-category of $\mathrm{Spc}$-valued prestacks over $R$: these can be seen as presheaves of spaces on $\mathrm{Aff}_R = \mathrm{CRing}_{R/}^{\op}$, the $\infty$-category of \defnword{affine derived schemes} over $R$. When $R = \Int$, we will simply write $\mathrm{PStk}$ for this $\infty$-category.

Fix a $\mathcal{C}$-valued prestack $F$ over $R$. The following definitions are taken from~\cite{lurie_thesis}, except that we are allowing more general target categories.

\begin{definition}
\label{defn:almost_fp}
$F$ is \defnword{locally of finite presentation} if for every filtered system $\{C_i\}_{i\in I}$ in $\mathrm{CRing}_{R/}$ with colimit $C\in \mathrm{CRing}_{R/}$, the natural map
the natural map
\[
\colim_{i\in I}F(C_i)\to F(C)
\]
is an equivalence.

$F$ is \defnword{almost locally of finite presentation} if for every $k\in \Int_{\ge 0}$, and for every filtered system $\{C_i\}_{i\in I}$ in $\mathrm{CRing}_{\leq k,R/}$ with colimit $C\in \mathrm{CRing}_{\leq k,R/}$, the natural map
\[
\colim_{i\in I}F(C_i)\to F(C)
\]
is an equivalence.
\end{definition}

\begin{definition}\label{defn:nilcomplete}
The prestack $F$ is \defnword{nilcomplete} if for every $S\in \mathrm{CRing}_{R/}$, the natural arrow
\[
F(S)\to \varprojlim_n F(\tau_{\leq n}S)
\]
is an equivalence.
\end{definition}

\begin{definition}\label{defn:inf_cohesive}
The prestack $F$ is \defnword{infinitesimally cohesive} if, for every Cartesian diagram
\[
\begin{diagram}
A&\rTo &A'\\
\dTo&&\dTo\\
B&\rTo&B'
\end{diagram}
\]
in $\mathrm{CRing}_{R/}$ where all the arrows are square-zero thickenings as in~\eqref{subsec:square_zero}, the diagram
\begin{align}\label{eqn:inf_coh_diag}
\begin{diagram}
F(A)&\rTo&F(A')\\
\dTo&&\dTo\\
F(B)&\rTo&F(B')
\end{diagram}
\end{align}
in $\mathcal{C}$ is once again Cartesian.
\end{definition}

\begin{definition}\label{defn:integrable}
The prestack $F$ is \defnword{integrable} if for every discrete complete local Noetherian ring $C\in \mathrm{CRing}_{\heartsuit,R/}$ with maximal ideal $\mathfrak{m}\subset C$ the map
\[
F(C)\to \varprojlim_n F(C/\mathfrak{m}^n)
\]
is an equivalence.
\end{definition}

\subsection{}\label{subsec:cotangent_complex}
For any prestack $F\in \mathrm{PStk}_{R}$, we have a stable $\infty$-category $\mathrm{QCoh}(F)$ of \defnword{quasi-coherent sheaves on $F$}. The precise definition can be found in~\cite[\S 6.2.2]{Lurie2018-kh}: roughly speaking, it is obtained by right Kan extension of the contravariant functor sending $S\in \mathrm{CRing}_{R/}$ to $\Mod{S}$. One can think of an object $\mathcal{M}$ in $\mathrm{QCoh}(F)$ as a way of assigning to every point $x\in F(S)$ an object $\mathcal{M}_x\in \mathrm{Mod}_S$ compatible with base-change.

We will say that $\mathcal{M}$ is \defnword{connective} if $\mathcal{M}_x$ belongs to $\mathrm{Mod}^{\mathrm{cn}}_S$ for each $x\in F(S)$ as above. We will say that it is \defnword{almost connective} if, for every $x\in F(S)$, there exists $n\in \Int_{\geq 0}$ such that $\mathcal{M}_x[n]$ is connective. We will say that it is \defnword{perfect} if, for every $x\in F(S)$, $\mathcal{M}_x$ is perfect.

Write $\mathrm{QCoh}^{\mathrm{cn}}_X$ (resp. $\mathrm{QCoh}^{\mathrm{acn}}_X$) for the  $\infty$-category spanned by the connective (resp. almost connective) objects in $\mathrm{QCoh}_F$.

\begin{definition}
\label{defn:cotangent_complex}
Following~\cite[\S 17.2.4]{Lurie2018-kh}, we will say that a morphism $f:F\to G$ in $\mathrm{PStk}_{R}$ \defnword{admits a cotangent complex} if there exists $\mathbb{L}_{F/G}\in \mathrm{QCoh}^{\mathrm{acn}}_X$ such that, for every $C\in \mathrm{CRing}_{R/}$, every $M\in \mathrm{Mod}^{\mathrm{cn}}_R$, and every $x\in F(C)$, we have a canonical equivalence
\[
\mathrm{Map}_{\mathrm{Mod}_C}(\mathbb{L}_{F/G,x},M)\xrightarrow{\simeq}\mathrm{fib}_{(f(x)[M],x)}(F(C\oplus M)\to G(C\oplus M)\times_{G(C)}F(C)).
\]
Here, $f(x)[M]\in G(C\oplus M)$ is the image of $f(x)$ along the natural section $G(C)\to G(C\oplus M)$.

If $F = \Spec C$ and $G = \Spec D$, then by Yoneda, any morphism $f:F\to G$ corresponds to an arrow $D\to C$ in $\mathrm{CRing}_{R/}$, and $f$ admits a cotangent complex, namely $\mathbb{L}_{C/D}$.
\end{definition}

\begin{definition}\label{defn:differential_conditions}
An object $C\in \mathrm{CRing}_{R/}$ is \defnword{finitely presented} (over $R$) if the functor $S\mapsto \mathrm{Map}_{\mathrm{CRing}_{R/}}(C,S)$ respects filtered colimits. For any such finitely presented $C$, the cotangent complex $\mathbb{L}_{C/R}\in \mathrm{Mod}_C^{\mathrm{cn}}$ is perfect; see~\cite[17.4.3.18]{Lurie2017-oh}.

If in addition $\mathbb{L}_{C/R}$ is $1$-connective, we say that $C$ is \defnword{unramified} over $R$; if $\mathbb{L}_{C/R}\simeq 0$, we say that $C$ is \defnword{\'etale} over $R$.

We say that a finitely presented $C\in \mathrm{CRing}_{R/}$ is \defnword{smooth} over $R$ if $\mathbb{L}_{C/R}\in \mathrm{Mod}^{\mathrm{cn}}_C$ is locally free of finite rank.
\end{definition}

\begin{definition}
\label{defn:etale_sheaf}
We will say that $F$ is an \defnword{\'etale sheaf} if for every \'etale map $S\to T$\footnote{This just means that $T$ is flat over $S$ with $\pi_0(T)$ \'etale over $\pi_0(S)$} in $\mathrm{CRing}_{R/}$ the map
\[
F(S)\to \mathrm{Tot}\left(F(\check{C}^{(\bullet)}(T/S))\right)
\]
is an equivalence. Here, $\check{C}^{\bullet}(T/S)$ is the \u{C}ech conerve of $T$ in $\mathrm{CRing}_{S/}$, so that $\check{C}^{(n)}(T/S)\simeq T^{\otimes_S n}$.
\end{definition}






\begin{definition}
An animated commutative ring $R$ is a \defnword{derived $G$-ring} if $\pi_0(R)$ is a $G$-ring in the sense of~\cite[\href{https://stacks.math.columbia.edu/tag/07GH}{Tag 07GH}]{stacks-project}.
\end{definition}

\begin{definition}
The \defnword{classical truncation} $F^{\mathrm{cl}}$ of the prestack $F$ is its restriction to $\mathrm{CRing}_{\heartsuit,R/}$
\end{definition}

We finally have (almost) all the terminology needed to state the following result, which is (a special case of)~\cite[Theorem 7.1.6]{lurie_thesis}.
\begin{theorem}[Lurie representability]\label{thm:lurie_representability}
Suppose that a prestack $F$ over a derived $G$-ring $R\in \mathrm{CRing}$ has the following properties:
\begin{enumerate}
  \item $F$ is almost locally of finite presentation.
  \item $F$ is an \'etale sheaf.
  \item $F$ is integrable
  \item $F$ is infinitesimally cohesive.
  \item $F$ is nilcomplete.
  \item $F$ admits a cotangent complex over $\Spec R$.
  \item $F^{\mathrm{cl}}$ takes values in $1$-truncated (resp $0$-truncated or discrete) objects.
\end{enumerate}
Then $F$ is represented by a derived Deligne-Mumford stack (resp. derived algebraic space) over $R$.
\end{theorem}

\begin{remark}
Here, we do not have to know what a derived algebraic space (resp. derived Deligne-Mumford stack) is, precisely---for our purposes, we could just as well take the theorem above to be their \emph{definition}.\footnote{Though this would only account for those derived algebraic spaces or Deligne-Mumford stacks that are almost locally finitely presented. See~\cite{lurie_thesis} for the actual definitions.} It is enough for us to know that the classical truncation is a classical algebraic space as given in\cite[\href{https://stacks.math.columbia.edu/tag/025X}{Tag 025X}]{stacks-project} (resp. a classical Deligne-Mumford stack).

A \defnword{derived scheme} is now simply a derived algebraic space whose restriction to $\mathrm{CRing}_{\heartsuit,R/}$ is represented by a classical scheme. The basic example is the affine derived scheme $\Spec C$ with $C\in \mathrm{CRing}_{R/}$.

As explained in~\cite[(2.1.1)]{Khan2020-pd}, we can think of a (quasi-compact) derived scheme (resp. derived algebraic space, resp. derived Deligne-Mumford stack) over $R$ as an \'etale sheaf of prestacks admitting a Zariski (resp. Nisnevich, resp. \'etale) cover by an affine derived scheme.
\end{remark}

When we already know that the classical truncation of $F$ is representable, we need much less. The following result can be found in~\cite[Appendix C]{TV2}.
\begin{theorem}[`Easy' Lurie representability]\label{thm:lurie_representability_easy}
A prestack $F$ over $R\in \mathrm{CRing}$ is represented by a derived Deligne-Mumford stack (resp. derived algebraic space, resp. derived scheme) over $R$ if and only if:
\begin{enumerate}
	\item The classical truncation of $F$ is represented by a classical Deligne-Mumford stack (resp. algebraic space, resp. scheme) over $\pi_0(R)$.
	\item $F$ is infinitesimally cohesive.
  \item $F$ is nilcomplete.
	\item $F$ admits a cotangent complex over $\Spec R$.
\end{enumerate}
\end{theorem}

\subsection{}\label{subsec:mapping_functor}
An example of a functor to which the above theorem applies is the functor of maps between (derived) schemes over $R$; this is considered in~\cite[(2.2.6.3)]{TV2}. Suppose that we have two presheaves $X,Y\in \mathrm{PStk}_{R}$; then we can consider the mapping functor
\[
\shvMap(X,Y): \mathrm{CRing}_{R/} \to \mathrm{Spc}
\]
whose points on $C\in \mathrm{CRing}_{R/}$ are given by 
\[
\shvMap(X,Y)(C) = \Map_{\mathrm{PStk}_C}(X_C,Y_C).
\]

Now, suppose that $X$ is a derived Deligne-Mumford stack over $R$ with an \'etale cover by $U = \Spec A$ with $A\in \mathrm{CRing}_{R/}$. This gives rise to the cosimplicial object 
\[
A^{(\bullet)}\in \Fun(\bm{\Delta},\mathrm{CRing}_{R/})
\]
with $\underbrace{U\times_X\times \cdots \times_XU}_{n} = \Spec A^{(n)}$. 

Suppose that $Y$ satisfies \'etale descent;  then we see that we have
\[
\Map_{\mathrm{PStk}_C}(X_C,Y_C) = \mathrm{Tot}( Y(C\otimes_RA^{(\bullet)}) ).
\]

If, in addition, $Y$ admits a cotangent complex $\mathbb{L}_{Y/R}$, and if we have a trivial square-zero thickening $C\oplus M\to C$ in $\mathrm{CRing}_{R/}$, then, for any point $f\in \Map_{\mathrm{PStk}_C}(X_C,Y_C)$ mapping to $f^{(\bullet)}\in Y(C\otimes_RA^{(\bullet)})$, we find that
\[
\mathrm{fib}_f\left(\Map_{\mathrm{PStk}_C}(X_{C\oplus M},Y_{C\oplus M})\to \Map_{\mathrm{PStk}_C}(X_C,Y_C)\right)
\]
is equivalent to
\[
\mathrm{Tot}\left(\mathrm{fib}_{f^{(\bullet)}}\left(Y((C\otimes M)\otimes_RA^{(\bullet)})\to Y(C\otimes_RA^{(n)})\right)\right)\simeq \mathrm{Tot}\left(\Map_{\mathrm{Mod}_{C\otimes_RA^{(\bullet)}}}((f^{(\bullet)})^*\mathbb{L}_{Y/R},M\otimes_RA^{(\bullet)})\right).
\]

Suppose now that the structure map $g:X\to \Spec R$ is proper and of finite Tor dimension; for instance, it could be the base change over $R$ of a proper and flat scheme over some discrete ring mapping to $R$. . Then the above limit is equivalent to
\[
\Map_{\mathrm{QCoh}_{X_C}}(f^*\mathbb{L}_{Y/R},g_C^*M) = \Map_{\Mod{C}}(g_{C,+}f^*\mathbb{L}_{Y/R},M),
\]
where $g_{C,+}:\mathrm{QCoh}_{X_C}\to \Spec C$ is the left adjoint to $g_C^*$ given by Grothendieck-Serre duality; see~\cite[Prop. 6.4.5.3]{Lurie2018-kh}.

\begin{proposition}
\label{prop:deform_maps}
Suppose the following hold: 
\begin{enumerate}
  \item $X$ is of proper and of finite Tor dimension over $R$;
  \item $Y$ is nilcomplete and infinitesimally cohesive, and admits a perfect cotangent complex over $R$;
\end{enumerate}
Then the mapping functor $\shvMap(X,Y)$ is also nilcomplete and infinitesimally cohesive, and admits a cotangent complex that at a point $f\in \Map(X_C,Y_C)$ is equivalent to $g_{C,+}f^*\mathbb{L}_{Y/R}$. In particular, if the restriction of $\shvMap(X,Y)$ to $\mathrm{CRing}_{\heartsuit,R/}$ is represented by a classical Deligne-Mumford stack (scheme) over $\pi_0(R)$, then $\shvMap(X,Y)$ is represented by a derived Deligne-Mumford stack (scheme) over $R$.
\end{proposition}
\begin{proof}
 Under our hypotheses, To\"en and Vezzosi show that the mapping functor is nilcomplete and infinitesimally cohesive (see Lemma 2.2.6.13 and the following paragraph in \emph{loc. cit.}). 
Therefore, the Proposition follows from Theorem~\ref{thm:lurie_representability_easy} and the discussion above.
\end{proof}	

\begin{remark}
\label{rem:mapping_stack_criteria}
Suppose that $R$ is discrete. The conditions above hold when $X$ and $Y$ are projective schemes over $R$ with $Y$ a local complete intersection over $R$: this follows from the theory of Hilbert schemes. 
\end{remark}

\begin{remark}
\label{rem:abelian_scheme_dualizing}
If $\mathbb{L}_{Y/R} \simeq \Reg{Y}\otimes_R\kappa_Y$ for some \emph{perfect} object $\kappa_Y\in \Mod[\mathrm{perf}]{R}$, then by~\cite[Prop. 6.4.5.4]{Lurie2018-kh}, we have
\[
g_{C,+}f^*\mathbb{L}_{Y/R} \simeq (g_{C,*}f^*\mathbb{L}_{Y/R}^\vee)^\vee\simeq C\otimes_R(\kappa_Y\otimes_R R\Gamma(X,\Reg{X})^\vee).
\]
This shows that the cotangent complex for $\shvMap(X,Y)$ is the pullback of the object
\[
\kappa_Y\otimes_R R\Gamma(X,\Reg{X})^\vee\in \Mod{R}.
\]

This applies in particular when $R$ is discrete and $Y$ is an abelian scheme over $R$, in which case, we can take $\kappa_Y = \omega_{Y/R}$ to be the module of invariant differentials on $Y$. It also applies when $Y = B\Gm$, the classifying stack of line bundles, in which case we can take $\kappa_Y = R[1]$.
\end{remark}

\section{Quasi-smooth morphisms}\label{sec:qsmooth}

\begin{definition}
Suppose that we have $R\in \mathrm{CRing}$. If $C$ is finitely presented over $R$ and if $\mathbb{L}_{C/R}$ has Tor amplitude in $[-1,0]$, we say that $C$ is \defnword{quasi-smooth} over $R$. 

A derived Deligne-Mumford stack $X$ over $R$ is \defnword{locally quasi-smooth (resp. smooth, \'etale, resp. unramified)} if it admits an \'etale cover $\{U_i\}$ with $U_i \simeq \Spec C_i$ for a quasi-smooth (resp. smooth, \'etale, resp. unramified) object in $\mathrm{CRing}_{R/}$. This is equivalent to saying that $X$ is locally finitely presented and $\mathbb{L}_{X/R}\in \mathrm{QCoh}_X$ has Tor amplitude in $[-1,0]$ (resp. $\mathbb{L}_{X/R}$ is a vector bundle, resp. $\mathbb{L}_{X/R}\simeq 0$, resp. $\mathbb{L}_{X/R}$ is $1$-connective).

A representable map of prestacks $f:X\to Y$ over $R$ is \defnword{locally quasi-smooth (resp. smooth, resp. \'etale, resp. unramified)} if for every $C\in \mathrm{CRing}_{R/}$ and $y\in Y(C)$, the fiber product
\[
X\times_{Y,y}\Spec C\to \Spec C
\]
is a locally quasi-smooth (resp. smooth, resp. \'etale, resp. unramified) derived scheme.
\end{definition}

\begin{definition}
If $X\to Y$ is locally quasi-smooth and unramified, then $\mathbb{L}_{X/Y}[-1]$ is a vector bundle over $X$ of finite rank, and we will call this rank the \defnword{virtual codimension of $X$ in $Y$}.
\end{definition}

\subsection{}
For every $n\geq 0$, we have a symmetric power functor
  \[
    \mathrm{Sym}^n: \Mod[\mathrm{cn}]{}\to \Mod[\mathrm{cn}]{}
  \]
  fibered over $\mathrm{CRing}$: this is obtained by animating the functor
  \[
      (R,N)\mapsto (R,\Sym^n_R(N))
  \]
  where $R$ is a finitely generated polynomial algebra over $\Int$ and $N$ is a finite free $R$-module; for more on this, see~\cite[\S 25.2]{Lurie2018-kh}.

  The direct sum 
  \[
      \Sym_R(N) = \bigoplus_{n=0}^\infty \Sym_R^n(N)
  \]
  is naturally an object in $\mathrm{CRing}_{R/}$, equipped with a canonical equivalence
  \[
     \mathbb{L}_{\Sym_R(N)/R}\simeq \Sym_R(N)\otimes_RN.
  \]
  See for example~\cite[(25.3.2.2)]{Lurie2018-kh}. 

Therefore, we see:
\begin{itemize}
  \item When $N$ is finite locally free, $\Sym_R(N)$ is smooth over $R$;
  \item When $N$ is perfect with Tor amplitude $[-1,0]$, $\Sym_R(N)$ is quasi-smooth over $R$.
\end{itemize}

Globalizing this construction, for any connective complex $\mathcal{F}\in \mathrm{QCoh}^{\mathrm{cn}}_X$ over a derived Deligne-Mumford stack $X$, we obtain a map
\[
\mathbf{V}(\mathcal{F})\to X
\]
whose base-change over any \'etale map $x:\Spec R\to X$ satisfies
\[
\mathbf{V}(\mathcal{F})\times_X\Spec R \simeq \Spec \Sym_R(\mathcal{F}_x)
\]
Note that we have 
\[
\mathbb{L}_{\mathbf{V}(\mathcal{F})/X}\simeq \Reg{\mathbf{V}(\mathcal{F})}\otimes_{\Reg{X}}\mathcal{F}.
\]

\begin{remark}
\label{rem:relative_functor_of_points}
Let $\mathrm{Aff}_R = \mathrm{CRing}_{R/}^{\op}$ be the $\infty$-category of affine derived schemes over $R$. Given any prestack $X$ over $R$, we can consider the comma category $\mathrm{Aff}_{/X}$ of arrows $\Spec C\to X$ of prestacks over $R$ with $C\in \mathrm{CRing}_{R/}$; equivalently, this is the opposite category to that of pairs $(C,x)$ where $C\in \mathrm{CRing}_{R/}$ and $x\in X(C)$.

By general considerations giving a map of prestacks $Y\to X$ is equivalent to specifying the associated presheaf of spaces on $\mathrm{Aff}_{/X}$ given by $(C,x)\mapsto \mathrm{fib}_x(Y(C)\to X(C))$.
\end{remark}

\begin{remark}
\label{rem:zero_section_cotangent_complex}
Via the previous remark, the map $\mathbf{V}(\mathcal{F})\to X$ is determined by the associated presheaf $\mathrm{Aff}_{/X}^{\op}\to \mathrm{Spc}$ given by
\[
(C,x)\mapsto \Map_{\Mod{C}}(\mathcal{F}_x,C).
\]

In particular, there is a canonical zero section $0:X\to \mathbf{V}(\mathcal{F})$ given locally by the map $\Sym_C(\mathcal{F}_x)\to C$ in $\mathrm{CRing}_{C/}$ carrying $\mathcal{F}_x$ to $0$.

For any map $s:W\to \mathbf{V}(\mathcal{F})$, we can consider the fiber product
\[
Z(s) = W\times_{\mathbf{V}(\mathcal{F}),0}X\to W.
\]
This is the \defnword{zero locus} of $s$ and is a closed immersion\footnote{This is just a map that \'etale locally on the target looks like $\Spec C'\to \Spec C$ with $C\to C'$ surjective on $\pi_0$; in the terminology of~\eqref{subsec:anipair}, this is an arrow in $\mathrm{AniPair}$.} that has cotangent complex
\[
\mathbb{L}_{Z(s)/W} \simeq \Reg{Z(s)}\otimes_{\Reg{W}}\mathcal{F}_s[1].
\]

This follows from the local fiber sequence (for any $N\in \Mod[\mathrm{cn}]{R}$)
 \[
     (N\simeq R\otimes_{\Sym_R(N)}\mathbb{L}_{\Sym_R(N)/R})\to (\mathbb{L}_{R/R}\simeq 0)\to \mathbb{L}_{R/\Sym_R(N)}.
  \]

In particular, if $\mathcal{F}$ is a \emph{vector bundle} of rank $n$, $Z(s)\to W$ is a quasi-smooth closed immersion of virtual codimension $n$.
\end{remark}

Every quasi-smooth closed immersion locally looks like the example given above:
\begin{lemma}
[Rydh-Khan]
A closed immersion $X\to Y$ of derived Deligne-Mumford stacks over $R$ is quasi-smooth of virtual codimension $n$ if and only if \'etale locally on $X$ there exists a map $Y\to \mathbb{V}(\Reg{Y}^n) = \mathbb{A}^n_Y$ such that we have a Cartesian diagram of prestacks
\[
\begin{diagram}
X&\rTo&Y\\
\dTo&&\dTo_{0}\\
Y&\rTo&\mathbb{A}^n_{Y}.
\end{diagram}
\]
\end{lemma}
\begin{proof}
See~\cite[Prop. 2.3.8]{Khan2018-dk} and its proof. Note that the definition of quasi-smooth in \emph{loc. cit.} is in terms of the local criterion stated in the lemma, and what they show is that that definition agrees with our differential one here.
\end{proof}

\begin{remark}
What the above lemma is saying is that \'etale locally on $X$, the structure sheaf for $X$ is represented by a Koszul complex on $Y$ associated with a sequence of $n$ global sections.
\end{remark}

The following lemma will be used in Section~\ref{sec:cycle_classes}.
\begin{lemma}
\label{lem:quasi-smooth_with_section}
Suppose that $X$ is a classical Deligne-Mumford stack, and that $Y\to X$ is a finite unramified, quasi-smooth map of derived Deligne-Mumford stacks, equipped with a section $s:X\to Y$ inducing an equivalence $X\xrightarrow{\simeq}Y^{\mathrm{cl}}$. Then there is a canonical equivalence
\[
X\times_{0,\mathbf{V}(s^*\mathbb{L}_{Y/X}[-1]),0}X\xrightarrow{\simeq}Y
\]
of quasi-smooth schemes over $X$. 
\end{lemma}
\begin{proof}
We immediately reduce to the local situation where $X = \Spec R$ with $R$ discrete and $Y = \Spec C$ with $C$ finite, unramified and quasi-smooth in $\mathrm{CRing}_{R/}$, equipped with a section $C\to R$. We now have to show that there is a canonical equivalence
\[
R\otimes_{\Sym_R(N)}R \xrightarrow{\simeq}C
\]
in $\mathrm{CRing}_{R/}$. Here, the left hand side is the (derived) base-change over $R$ of the zero section $\Sym_R(N)\to R$ along itself, where 
\[
N = R\otimes_C\mathbb{L}_{C/R}[-1].
\]

For this we argue as in~\cite[Proposition 2.3.8]{Khan2018-dk}. Set $
F = \mathrm{hker}(R\to C)$; then the existence of the section shows that we have $C \simeq R\oplus F[1]\in \mathrm{Mod}^{\mathrm{cn}}_R$

By~\cite[(25.3.6.1)]{Lurie2018-kh}, there is a natural equivalence $\pi_0(C\otimes_RF)\xrightarrow{\simeq} \pi_1(\mathbb{L}_{C/R})$, where the left hand side in turn is equivalent to
\[
\pi_0(F[1]\otimes_RF)\oplus \pi_0(F)\simeq \pi_0(F)\simeq \pi_1(C).
\]
Since the fiber of the section $C\to R$ is $1$-connective, we see that $\pi_1(\mathbb{L}_{C/R})\simeq N$, where $N$ is as in the statement of the lemma. Therefore, we have a canonical equivalence $N\xrightarrow{\simeq}\pi_1(C)$.

Now, the point is that $R\otimes_{\Sym_R(N)}R$ parameterizes loops from the zero section to itself. More precisely, by~\cite[Lemma 2.3.5]{Khan2018-dk}, we have
\[
\pi_0\left(\mathrm{Map}_{\mathrm{CRing}_{R/}}(R\otimes_{\Sym_R(N)}R,C)\right)\simeq \Hom_R(N,\pi_1(C)).
\]
Therefore we have a canonical (homotopy class of) arrow(s)
\[{}
R\otimes_{\Sym_R(N)}R\to C,
\]
of quasi-smooth unramified animated $R$-algebras, which by construction is an isomorphism on $\pi_0$ as well as on cotangent complexes. We now conclude by~\cite[(25.3.6.6)]{Lurie2018-kh}.
\end{proof}

\section{Chow groups and $K$-groups}\label{app:cycle_classes}
\subsection{}
Suppose that $X$ is a classical, regular, equidimensional Deligne-Mumford stack of finite type over $\Int$. As explained in~\cite[\S 2]{Gillet2009-tw}, we have the rational $K$-group $K_0(X)_{\Rat}$, which is endowed with a ring structure arising from tensor product of complexes, and is equipped with the \defnword{coniveau filtration} $F^\bullet K_0(X)_{\Rat}$. See~\cite[Appendix A.2]{HMP:mod_codim} for details. In particular, for any coherent sheaf $\mathcal{F}$ over $X$ supported on a closed substack of codimension $d$, we obtain a canonical class $[\mathcal{F}]\in F^dK_0(X)_{\Rat}$.

We set
\[
\mathrm{Gr}^d_{\gamma}K_0(X)_{\Rat} = F^dK_0(X)_{\Rat}/F^{d+1}(X)_{\Rat}.
\] 

\begin{theorem}[Gillet-Soul\'e, Khan]
\label{thm:lambda}
\begin{enumerate}
  \item The exterior power operations $\mathcal{F}\mapsto \wedge^k\mathcal{F}$ give rise to ring homomorphisms, the Adams operations $\psi^k:K_0(X)_{\Rat}\rightarrow K_0(X)_{\Rat}$. If 
  \[
   K_0(X)_{\Rat}^{(d)} = \{\alpha\in K_0(X)_{\Rat}^{(d)}:\;\psi^k(\alpha) = k^d\alpha\text{ for all $d\geq 1$}\},
  \]
  then we have $K_0(X)_{\Rat}^{(d)}\subset F^dK_0(X)_{\Rat}$, inducing an isomorphism
  \[
    K_0(X)_{\Rat}^{(d)}\xrightarrow{\simeq}\gr^d_{\gamma}K_0(X)_{\Rat}.
  \]

  \item If $\pi:Z\to X$ is a quasi-smooth finite unramified morphism of derived stacks of virtual codimension $d$, then we have
  \[
   [\pi_*\Reg{Z}]\in K_0(X)_{\Rat}^{(d)}.
  \]

\end{enumerate}

\end{theorem}
\begin{proof}
The existence of Adams operations, and their relationship with the coniveau filtration reduces to the case of schemes, where it is explained in~\cite[Theorem 4.6]{Gillet1987-ny}.

As for (2), once again, it reduces to the case where $X$ is a scheme. If $X$ is a quasi-compact and quasi-projective algebraic space over a field, then this is shown in~\cite[Theorem 6.21]{khan:gtheory}. However, an analysis of the proof shows that it only requires the existence of Adams operations on $K_0(X)_{\Rat}$---which is assertion (1)---as well as the excess intersection formula for the derived blowups constructed in~\cite{Khan2018-dk}. This argument is essentially already contained in ~\cite{sga6}.
\end{proof}

\subsection{}
For any equidimensional classical Noetherian Deligne-Mumford stack $Z$, we also have the rational Chow group $\mathrm{CH}^\bullet(Z)_{\Rat}$, which are covariant for finite morphisms; see the discussion in~\cite[\S A.1]{HMP:mod_codim}. 

Given a finite map $\pi:Z\to X$ of such stacks with $\dim Z = \dim X - r$, we have the \defnword{fundamental class} $[Z] \in \mathrm{CH}^r(X)_\Rat$; see~\cite[A.1.2]{HMP:mod_codim}.

The following result is due to Gillet-Soul\'e and Gillet; see~\cite{Gillet1987-ny},~\cite{Gillet1984-tk},~\cite{Gillet2009-tw}. For more details, see also the discussion in~\cite[Appendix A]{HMP:mod_codim}.
\begin{theorem}
\label{thm:GS}
Assume once again that $X$ is a classical, regular Noetherian Deligne-Mumford stack over $\Int$.
\begin{enumerate}
 \item There are natural isomorphisms
\[
\mathrm{CH}^{d}(X)_\Rat \xrightarrow{\simeq} \mathrm{Gr}^d_{\gamma} K_0(X)_\Rat  \xleftarrow{\simeq} K_0(X)_{\Rat}^{(d)}.
\]
\item If $\pi:Z\to X$ is a finite morphism with $Z$ classical of dimension $\dim X - d$, then the first isomorphism above carries the fundamental class $[Z]\in \mathrm{CH}^d(X)_\Rat$ to the image of the class $[\pi_*\Reg{Z}]\in F^dK_0(X)_{\Rat}$. 
\item Multiplication in $K$-theory induces maps
  \[
     K_0(X)_{\Rat}^{(d_1)}\times K_0(X)_{\Rat}^{(d_2)}\to K_0(X)_{\Rat}^{d_1+d_2}
  \]
  which give rise via the isomorphisms in (1) to the intersection product
  \[
     \mathrm{CH}^{d_1}(X)_{\Rat}\times \mathrm{CH}^{d_2}(X)_{\Rat}\to \mathrm{CH}^{d_1+d_2}(X)_{\Rat}.
  \]
\item There is a commuting diagram:
\[
\begin{diagram}
\Pic(X)_\Rat &\rTo^{\mathcal{L}\mapsto [\Reg{M}]-[\mathcal{L}^{-1}]}& K_0(X)_{\Rat}^{(1)}\\
\dTo^{c_1}_\simeq &\ruTo_{\simeq}\\
\mathrm{CH}^1(X)_\Rat,
\end{diagram}
\]
where $c_1$ is the Chern class map sending the isomorphism class of a line bundle to the class of the Weil divisor of any rational section.
\end{enumerate}
\end{theorem}

\begin{definition}
\label{defn:qs_cycle_class}
For every quasi-smooth finite unramified morphism $\pi:Z\to X$ of derived stacks with $Z$ of virtual codimension $d$, we get a \defnword{fundamental class}
\[
[Z/X]\in \mathrm{CH}^d(X)_{\Rat}
\]
carried to the class $[\pi_*\Reg{Z}]\in K_0(X)_{\Rat}^{(d)}$ (see (2) of Theorem~\ref{thm:lambda}) under the isomorphisms of Theorem~\ref{thm:GS}(1).
\end{definition}

\begin{remark}
\label{rem:product_classes}
If $\pi_1:Z_1\to X$ and $\pi_2:Z_2\to X$ are two quasi-smooth finite unramified maps of virtual codimensions $d_1$ and $d_2$, then we have
  \[
     [Z_1/X]\cdot [Z_2/X] = [(Z_1\times_XZ_2)/X]\in \mathrm{CH}^{d_1+d_2}(X)_{\Rat},
  \]
  where on the right hand side we are taking the derived fiber product 
  \[
  \pi:Z_1\times_XZ_2\to X.
  \]

  Note that $Z_1\times_XZ_2\to X$ is once again quasi-smooth and finite unramified of virtual codimension $d_1+d_2$: its relative cotangent complex is $p_1^*\mathbb{L}_{Z_1/X}\oplus p_2^*\mathbb{L}_{Z_2/X}$, where $p_1:Z_1\times_XZ_2\to Z_1$ and $p_2:Z_1\times_XZ_2\to Z_2$ are the two projections. 

  So the desired equality is immediate from the fact that we have a canonical equivalence
  \[
    \pi_{1,*}\Reg{Z_1}\otimes_{\Reg{X}}\pi_{2,*}\Reg{Z_2}\simeq \pi_*\Reg{Z_1\times_XZ_2}\in \mathrm{QCoh}_X.
  \]
\end{remark}

\begin{remark}
\label{rem:top_chern_class}
Suppose that we have a vector bundle $\mathcal{E}$ over $X$ of rank $r$. Consider the `trivial' quasi-smooth morphism $Y = X\times_{0,\mathbf{V}(\mathcal{E})}\{0\}\to X$ with relative cotangent complex $\mathcal{E}$. Then the class
\[
[Y/X]\in \mathrm{CH}^r(X)_{\Rat}
\]
is equal to $c_{\mathrm{top}}(\mathcal{E}^\vee) = (-1)^r c_{\mathrm{top}}(\mathcal{E})$. 

Indeed, the class in question is represented in $K_0(X)_{\Rat}^{(r)}$ by the Koszul complex associated with the zero section of $\mathcal{E}$; in other words, it is
\[
\sum_{i=0}^{r}(-1)^i[\wedge^i\mathcal{E}]\in K_0(X)_{\Rat}^{(r)}.
\]
We want to see that this is carried to $c_{\mathrm{top}}(\mathcal{E}^\vee)$ under the inverse of the isomorphism in Theorem~\ref{thm:GS}(1). Using the splitting principle and the definition of the top Chern class, it suffices to check this for the Koszul complex associated with the zero section of a line bundle, where it follows from Theorem~\ref{thm:GS}(4)
\end{remark}

\begin{remark}
\label{rem:adeel_construction}
 There is another way of obtaining cycle classes, using the theory in~\cite{khan:virtual}: Here, Khan associates with every quasi-smooth morphism $Z\to X$ of virtual codimension $d$, and every \'etale motivic spectrum $\mathcal{F}$ over $\Int$, a canonical virtual fundamental class
  \[
      [Z]_{\mathcal{F}}\in H^{2d}(X,\mathcal{F}(d)),
  \]
  where the right hand side involves motivic cohomology; see~\cite[(3.21)]{khan:virtual}. There is a natural cup product on motivic cohomology, and given another quasi-smooth morphism $Z'\to X$ of virtual codimension $d'$ we have 
  \[
     [Z]_{\mathcal{F}}\cup [Z']_{\mathcal{F}} = [Z\times_XZ']_{\mathcal{F}}\in H^{2(d+d')}(X,\mathcal{F}(d+d')).
  \]

  As explained in~\cite[Example 2.10]{khan:virtual}, if we take $\mathcal{F}$ to be the rational motivic cohomology spectrum $\mathbf{Q}$, and $X$ to be a smooth scheme over a field, then there is an identification
  \[
     H^{2d}(X,\mathbf{Q}(d))\simeq \mathrm{CH}^d(X)_{\Rat},
  \]
  and we once again recover virtual fundamental classes in Chow groups associated with quasi-smooth morphisms, and compatible with products in a natural way.

    This excludes the case of interest to us, which is of a flat, regular scheme (or Deligne-Mumford stack) over $\Int$.
\end{remark}

\printbibliography

\end{document}